\documentclass[12pt,a4paper,reqno,oneside]{amsart}

\usepackage{mathabx}

\DeclareMathAlphabet{\mathsl}{OT1}{cmss}{m}{sl}
\SetMathAlphabet{\mathsl}{bold}{OT1}{cmss}{bx}{sl}

\DeclareFontEncoding{LS1}{}{}
\DeclareFontSubstitution{LS1}{stix}{m}{n}
\DeclareMathAlphabet{\mathscr}{LS1}{stixscr}{m}{n} 

\usepackage
{geometry}

\usepackage[dvipsnames]{xcolor}
\usepackage[english]{babel} 

\definecolor{bred}{rgb}{0.8,0,0}

\usepackage{amsthm,amsmath,amssymb,bbm,geometry,epsfig,enumerate,xcolor,wasysym,tikz}
\usepackage[pagewise]{lineno}[4.41]

\usepackage{marginnote}\reversemarginpar
\usepackage[sort,nocompress]{cite}
\usepackage[colorlinks,linkcolor=red,citecolor=red]{hyperref}

\usepackage[notref, notcite,final]{showkeys}
\usepackage[capitalise,compress]{cleveref} 

\crefname{equation}{}{}
\newtheorem{lemma}{Lemma}[section]
\newtheorem{remark}[lemma]{Remark}
\newtheorem{proposition}[lemma]{Proposition}
\newtheorem{theorem}[lemma]{Theorem}

\newtheorem{corollary}[lemma]{Corollary}

\newtheorem{setting}[lemma]{Setting}
\crefname{subsection}{Subsection}{Subsections}
\crefname{enumi}{item}{items}
\creflabelformat{enumi}{(#2\textup{#1}#3)}

\newcommand{\1}{\ensuremath{\mathbbm{1}}}
\providecommand{\N}{{\ensuremath{\mathbbm{N}}}}
\providecommand{\Z}{{\ensuremath{\mathbbm{Z}}}}
\providecommand{\R}{{\ensuremath{\mathbbm{R}}}}

\renewcommand{\P}{{\ensuremath{\mathbbm{P}}}}
\providecommand{\E}{{\ensuremath{\mathbbm{E}}}}
\providecommand{\F}{{\ensuremath{\mathbbm{F}}}}

\newcommand{\xeqref}[1]{}

\newcommand{\threenorm}[1]{{\left\vert\kern-0.25ex\left\vert\kern-0.25ex\left\vert #1 
    \right\vert\kern-0.25ex\right\vert\kern-0.25ex\right\vert}}
\newcommand{\exponentV}{{p_{\mathrm{v}}}}
\newcommand{\exponentZ}{{p_{\mathrm{z}}}}
\newcommand{\exponentX}{{p_{\mathrm{x}}}}
\newcommand{\exponentFirstNorm}{{q_1}}
\newcommand{\pr}{\mathrm{pr}}
\newcommand{\var}{\mathrm{var}}
\newcommand{\rdown}[1]{\lfloor #1\rfloor}
\renewcommand{\gets}{\curvearrowleft}

\title[]{Multilevel Picard approximations overcome the curse of dimensionality\\ in the numerical approximation of general semilinear PDEs\\ with gradient-dependent nonlinearities}

\author[A. Neufeld]{Ariel Neufeld$^{1}$}
\address{$^1$  Division of Mathematical Sciences, School of Physical and Mathematical Sciences, Nanyang Technological University, Singapore}
\email{ariel.neufeld@ntu.edu.sg}

\author[T.A. Nguyen]{Tuan Anh Nguyen$^{2}$}
\address{$^2$  Faculty of Mathematics, University of Bielefeld, Germany}
\email{tnguyen@math.uni-bielefeld.de}

\author[S. Wu]{Sizhou Wu$^{3}$}
\address{$^3$  School of Mathematics, Shanghai University of Finance and Economics, Shanghai, China}
\email{wusizhou@sufe.edu.cn}


\thanks{
Financial support by the Nanyang Assistant Professorship Grant (NAP Grant) \emph{Machine Learning based Algorithms in
Finance and Insurance} 
and "the Fundamental Research Funds for the Central Universities"
is gratefully acknowledged.}
\allowdisplaybreaks
\begin{document}
	
\begin{abstract}
Neufeld and Wu (arXiv:2310.12545) developed a multilevel Picard (MLP) algorithm which can approximately solve \textit{general} semilinear parabolic PDEs with gradient-dependent nonlinearities, allowing also for coefficient functions of the corresponding PDE to be non-constant.  By introducing a particular stochastic fixed-point equation (SFPE) motivated by the Feynman-Kac representation and the Bismut-Elworthy-Li formula and identifying the first and second component of the unique fixed-point of the SFPE with the unique viscosity solution of the PDE and its gradient, they proved convergence of their algorithm. However, it remained an open question whether the proposed MLP schema in arXiv:2310.12545  does not suffer from the curse of dimensionality.
In this paper, we prove that the MLP algorithm in arXiv:2310.12545 indeed can overcome the curse of dimensionality, i.e. that its computational complexity only grows polynomially in the dimension $d\in \mathbb{N}$ and the reciprocal of the accuracy $\varepsilon$, under some suitable assumptions on the nonlinear part of the corresponding PDE.

\end{abstract}
\vspace*{-0.7cm} 
\maketitle
\section{Introduction}
Partial differential equations (PDEs) are important tools to analyze many real world phenomena, e.g., in financial engineering, economics, quantum mechanics, or statistical physics to name but a few. 
In most of the cases such high-dimensional nonlinear PDEs cannot be solved explicitly. It is one of
the most challenging problems in applied mathematics to approximately solve high-dimensional nonlinear PDEs.
In particular, it is very difficult to find approximation schemata for nonlinear PDEs for which one can
rigorously prove that they do overcome the so-called \emph{curse of dimensionality} in the sense that the computational complexity only grows polynomially in the space dimension $d$ of the PDE and the reciprocal
$\frac{1}{\varepsilon}$ of the accuracy $\varepsilon$.

In recent years, 
there are two types of approximation methods which are quite
successful in the numerical approximation of solutions of high-dimensional nonlinear  PDEs:  neural network based 
approximation methods for PDEs, cf.,
\cite{al2022extensions,beck2020deep,beck2021deep,beck2019machine,
beck2020overview,berner2020numerically,castro2022deep,cioica2022deep,
ew2017deep,ew2018deep,weinan2021algorithms,frey2022convergence,frey2022deep,GPW2022,
gnoatto2022deep,gonon2023random,gonon2021deep,gonon2023deep,grohs2023proof,
han2018solving,han2020convergence, 
han2019solving,hure2020deep,hutzenthaler2020proof,ito2021neural,
jacquier2023deep,jacquier2023random,jentzen2018proof,
lu2021deepxde,nguwi2022deep,nguwi2022numerical,nguwi2023deep,
raissi2019physics,reisinger2020rectified,sirignano2018dgm,zhang2020learning} 
and multilevel Monte-Carlo based approximation methods for PDEs, cf.,
\cite{beck2020overcomingElliptic,beck2020overcoming,becker2020numerical,
hutzenthaler2019multilevel,hutzenthaler2021multilevel,
giles2019generalised,HJK2022,HJKNW2020,
HJKN2020,
hutzenthaler2020overcoming,HK2020,hutzenthaler2022multilevel,
hutzenthaler2022multilevel1,NW2022}.

Neural networks based algorithms are very efficient in practice. However, a rigorous convergence analysis for them is often missing because when training neural networks the corresponding optimization problems  are typically non-convex. On the other hand for multilevel Monte-Carlo based algorithms it is often possible to provide a complete convergence and complexity analysis.
It has been proven that under some suitable assumptions, e.g., Lipschitz continuity
on the linear part, the nonlinear part, and the initial (or terminal) condition function of the PDE under consideration, 
the multilevel Picard approximation algorithms can overcome the curse of
dimensionality in the sense that the number of computational operations of
the proposed Monte-Carlo based approximation method grows at most polynomially in both the reciprocal
$\frac{1}{\varepsilon}$ of the prescribed approximation accuracy $\varepsilon\in (0,1)$ and the PDE dimension $d\in \N$, 
see
\cite{beck2020overcomingElliptic,beck2020overcoming,hutzenthaler2019multilevel,
hutzenthaler2021multilevel,giles2019generalised,HJK2022,HJKNW2020,
HJKN2020,
hutzenthaler2020overcoming,HK2020,hutzenthaler2022multilevel,
hutzenthaler2022multilevel1,NW2022,NNW2023}.

Nevertheless, for semilinear PDEs whose nonlinear part depends also on the gradient the development of numerical schemes as well as their complexity analysis are still at their infancy. 
In \cite{HJK2022,HK2020} multilevel Picard (MLP) approximation algorithms together with their convergence and complexity analysis have been developed for semilinear heat equations with  nonlinear parts depending on the gradients of the solutions. Recently, \cite{NW2023} developed an MLP algorithm for \textit{general} semilinear PDEs with gradient-dependent nonlinearities, allowing also for coefficient functions of the corresponding PDE to be non-constant.  The main idea of the algorithm in  \cite{NW2023}  is to introduce a particular stochastic fixed-point equation (SFPE) motivated by the Feynman-Kac representation and the Bismut-Elworthy-Li formula and to identify the first and second component of the unique fixed-point of the SFPE with the unique viscosity solution of the PDE and its gradient, allowing to prove convergence of their algorithm. However, it remained an open question whether the proposed MLP schema in \cite{NW2023}  does not suffer from the curse of dimensionality.
The main goal of this paper is to prove that the MLP schema in  \cite{NW2023} indeed  can overcome the curse of dimensionality under some suitable assumptions on the nonlinear part, namely, \eqref{c60} and \eqref{a01i} in \cref{c34b} below, which is the main result of our paper.
\subsection{Notations}Throughout the paper we use the following notations.
Let 
$\R$ denote the set of all real numbers. Let
$\Z, \N_0, \N $ denote the sets which satisfy that $\Z=\{\ldots,-2,-1,0,1,2,\ldots\}$, $\N=\{1,2,\ldots\}$, $\N_0=\N\cup\{0\}$. Let $\operatorname{D}\!$ denote the total derivative, $\nabla$ denote the gradient,  and $\operatorname{Hess}$ denote the Hessian matrix. For every matrix $A$ let $A^\top$ denote the transpose of $A$ and let $\mathrm{trace}(A)$ denote the trace of $A$ when $A$ is a square matrix. 
Let $\langle\cdot ,\cdot\rangle $ denote the standard scalar product on $\R^d$, $d\in \N$.
Let
$\mathrm{B}(\cdot, \cdot)$ denote the Beta function. 
For $c\in \R$, for a set $X$, and for a real-valued function $f\colon X\to \R$  we write $c\leq f$ if $c\leq f(x)$ for all $x\in X$.
When applying a result we often use a phrase like ``Lemma 3.8 with $d\gets (d-1)$''
that should be read as ``Lemma 3.8 applied with $d$ (in the notation of Lemma 3.8) replaced
by $(d-1 )$ (in the current notation)'' and we often omit a trivial replacement to lighten the
notation, e.g., we rarely write, e.g., ``Lemma 3.28 with $d \gets d$''.
\subsection{Main result}
\begin{theorem}\label{c34b}
Let  $\Theta=\cup_{n\in \N}\Z^n$,
$T\in (0,\infty)$, $\mathbf{k}\in [0,\infty)$,
$c\in [1,\infty)$.
Let $\lVert \cdot\rVert\colon \cup_{k,\ell\in\N}\R^{k\times \ell}\to[0,\infty)$ satisfy for all $k,\ell\in\N$, $s=(s_{ij})_{i\in[1,k]\cap\N,j\in [1,\ell]\cap\N}\in\R^{k\times \ell}$ that
$\lVert s\rVert^2=\sum_{i=1}^{k}\sum_{j=1}^{\ell}\lvert s_{ij}\rvert^2$.
For every $d\in \N$
let $(L^d_i)_{i\in [0,d]\cap \Z}\in \R^{d+1}$ satisfy that \begin{align}\label{c60}
\sum_{i=0}^{d}L^d_i\leq c.
\end{align} 
For every
 $K\in \N$ 
let $\rdown{\cdot}_K\colon \R\to\R$ satisfy for all $t\in \R$ that 
$\rdown{t}_K=\max( \{0,\frac{T}{K},\frac{2T}{K},\ldots,T\} \cap (    (-\infty,t)\cup\{0\} )  )$.
For every $d\in \N$
let 
 $\Lambda^d=(\Lambda^d_{\nu})_{\nu\in [0,d]\cap\Z}\colon [0,T]\to \R^{1+d}$ satisfy for all $t\in [0,T]$ that $\Lambda^d(t)=(1,\sqrt{t},\ldots,\sqrt{t})$.
For every $d\in \N$
let $\pr^d=(\pr^d_\nu)_{\nu\in [0,d]\cap\Z}\colon \R^{d+1}\to\R$ satisfy
for all 
$w=(w_\nu)_{\nu\in [0,d]\cap\Z}$,
$i\in [0,d]\cap\Z$ that
$\pr^d_i(w)=w_i$. 
For every $d\in \N$, $k\in [1,d]\cap\Z$ let $e^d_k\in \R^d$ denote the $d$-dimensional vector with a $1$ in the $k$-th coordinate and $0$'s elsewhere.
For every $d\in \N$ let $f_d\in C( [0,T)\times\R^d\times \R^{d+1},\R)$, $g_d\in C(\R^d,\R)$, 
$\mu_d\in C^3(\R^d,\R^d)$, 
$\sigma_d\in C^3(\R^d,\R^{d\times d})$.
To shorten the notation we write 
for all $d\in \N$,
$t\in [0,T)$, $x\in \R^d$,
$w\colon[0,T)\times\R^d\to \R^{d+1} $ that
\begin{align}
(F_d(w))(t,x)=f_d(t,x,w(t,x)).
\end{align}
Assume
for all $d\in \N$,
$i\in [0,d]\cap\Z$,
$s\in [0,T)$,
$t\in [s,T)$, $r\in (t,T]$,
 $x,y,h\in \R^d$, $w_1,w_2\in \R^{d+1}$
 that $\sigma_d$ is invertible,
\begin{align}
\lVert\mu_d(0)\rVert+\lVert\sigma_d(0)\rVert\leq cd^c,\label{f04g}
\end{align}
\begin{align}
\max\left\{\left\lVert\left((\operatorname{D}\!\mu_d)(x)\right)\!(h)\right\rVert  ,\left\lVert \left((\operatorname{D}\!\sigma_d)(x)
\right)\!(h)
\right\rVert\right\}\leq c\lVert h\rVert,\label{f04f}
\end{align}
\begin{align}\label{c37}
y^\top \sigma_d(x)(\sigma_d(x))^\top y\geq \frac{1}{c}\lVert y\rVert^2,
\end{align}
\begin{align}
&
\max \left\{
\left\lVert
\left(
(\operatorname{D}\!\mu_d)(x)-
(\operatorname{D}\!\mu_d)(y)\right)\!(h)\right\rVert
,
\left\lVert
\left(
(\operatorname{D}\!\sigma_d)(x)-
(\operatorname{D}\!\sigma_d)(y)\right)\!(h)\right\rVert
\right\}
\leq cd^c\lVert x-y\rVert\lVert h\rVert,
\label{f10f}
\end{align}
\begin{align}\label{b04a}
\lvert g_d(x)\rvert+ \lvert Tf_d(t,x,0)\rvert\leq \left[(cd^c)^2+c^2\lVert x\rVert^2\right]^\frac{1}{2},
\end{align}
\begin{align}
&
\lvert
f_d(t,x,w_1)-f_d(t,y,w_2)\rvert\leq \sum_{\nu=0}^{d}\left[
L_\nu^d\Lambda^d_\nu(T)
\lvert\pr^d_\nu(w_1-w_2) \rvert\right]
+\frac{1}{T}c\frac{\lVert x-y\rVert}{\sqrt{T}}
,\label{a01i}
\end{align} and
\begin{align}
\lvert g_d(x)-g_d(y)\rvert\leq c\frac{\lVert x-y\rVert}{\sqrt{T}}.\label{a19g}
\end{align}
Let $\varrho\colon \{(\tau,\sigma)\in [0,T)^2\colon\tau<\sigma \}\to\R$ satisfy for all $t\in [0,T)$, $s\in (t,T)$ that
\begin{align}
\varrho(t,s)=\frac{1}{\mathrm{B}(\tfrac{1}{2},\tfrac{1}{2})}\frac{1}{\sqrt{(T-s)(s-t)}}.\label{a98h}
\end{align} 
Let $ (\Omega,\mathcal{F},\P, (\F_t)_{t\in [0,T]})$ be a filtered probability space which satisfies the usual conditions.
For every random variable $\mathfrak{X}\colon \Omega\to\R$ let $\lVert\mathfrak{X}\rVert_2\in[0,\infty]$ satisfy that $\lVert\mathfrak{X}\rVert_2^2=\E [\lVert\mathfrak{X}\rVert^2]$.
Let $\mathfrak{r}^\theta\colon \Omega\to (0,1) $, $\theta\in \Theta$, be independent and identically distributed random variables and satisfy for all
$b\in (0,1)$
that 
\begin{align}
\P(\mathfrak{r}^0\leq b)=
\frac{1}{\mathrm{B}(\tfrac{1}{2},\tfrac{1}{2})}
\int_{0}^b\frac{dr}{\sqrt{r(1-r)}}.\label{c57b}
\end{align}
For every $d\in \N$
let $W^{d,\theta}\colon [0,T]\times \Omega\to \R^d$, $\theta\in \Theta$, be independent $(\F_t)_{t\in [0,T]}$-Brownian motions with continuous sample paths. Assume for every $d\in \N$ that $(W^{d,\theta})_{\theta\in \Theta}$ and $(\mathfrak{r}^\theta)_{\theta\in \Theta}$ are independent.
For every $\theta\in \Theta$, $d,K\in \N$, $s\in [0,T]$, $x\in \R^d$,
$k\in [1,d]\cap\Z$
let $(\mathcal{X}^{d,\theta,K,s,x}_t)_{t\in [s,T]},(\mathcal{D}^{d,\theta,K,s,x,k}_t)_{t\in[s,T]}\colon [s,T]\times \Omega\to \R^d$
be  $(\F_t)_{t\in[s,T]}$-adapted stochastic processes with continuous sample paths which satisfy for all $t\in [s,T]$ that $\P$-a.s.\ we have that
\begin{align}
\mathcal{X}^{d,\theta,K,s,x}_t=x+\int_{s}^{t} \mu_d(\mathcal{X}^{d,\theta,K,s,x}_{\max \{s, \rdown{r}_K\}})\,dr
+
\int_{s}^{t} \sigma_d(\mathcal{X}^{d,\theta,K,s,x}_{\max \{s, \rdown{r}_K\}})\,dW_r^{d,\theta}\label{c57a}
\end{align}
and
\begin{align}
\mathcal{D}^{d,\theta,K,s,x,k}_t&=e^d_k+\int_{s}^{t}
\left( (\operatorname{D}\!\mu_d)(\mathcal{X}^{d,\theta,K,s,x}_{\max \{s, \rdown{r}_K\}})
\right)\!
\left(\mathcal{D}^{d,\theta,K,s,x,k}_{\max \{s, \rdown{r}_K\}}\right)
dr\nonumber\\
&\quad 
+
\int_{s}^{t}\left( (\operatorname{D}\!\sigma_{d})(\mathcal{X}^{d,\theta,K,s,x}_{\max \{s, \rdown{r}_K\}})
\right)
\!
\left(\mathcal{D}^{d,\theta,K,s,x,k}_{\max \{s, \rdown{r}_K\}}\right)
dW_r^{d,\theta}.
\end{align}
For every 
$\theta\in \Theta$,
$d,K\in \N$,
$s\in [0,T)$, $t\in (s,T]$, $x\in \R^d$ let 
$\mathcal{V}^{d,\theta,K,s,x}_t= (\mathcal{V}_t^{d,\theta,K,s,x,k})_{k\in [1,d]\cap\Z}\colon \Omega\to\R^d $,
$\mathcal{Z}_t^{d,\theta,K,s,x}
=(\mathcal{Z}_t^{d,\theta,K,s,x,k})_{k\in [0,d]\cap\Z}
\colon \Omega\to \R^{d+1}$
 satisfy that
\begin{align}
\mathcal{V}^{d,\theta,K,s,x}_t=\frac{1}{t-s}\int_{s}^{t}
\left(
\sigma^{-1}_d(\mathcal{X}^{d,\theta,K,s,x}_{\max\{s,\rdown{r}_K\}})
\mathcal{D}^{d,\theta,K,s,x}_{\max\{s,\rdown{r}_K\}}
\right)^\top dW_r^{d,\theta}\label{c56b}
\end{align}
and $\mathcal{Z}_t^{d,\theta,K,s,x}=(1,\mathcal{V}_t^{d,\theta,K,s,x} )$.
Let $
U^{d,\theta}_{n,m,K}\colon [0,T)\times \R^d\to \R^{d+1}$, 
$d,K\in \N$,
$n,m\in \Z$, $\theta\in \Theta$, satisfy for all $d,n,m,K\in \N$, 
$\theta\in \Theta$,
$t\in [0,T)$, $x\in\R^d$ that
$U_{-1,m,K}^{d,\theta}(t,x)=U^{d,\theta}_{0,m,K}(t,x)=0$ and
\begin{align} \begin{split} 
&
U^{d,\theta}_{n,m,K}(t,x)= (g_d(x),0)+\sum_{i=1}^{m^n}
\frac{g_d(\mathcal{X}^{d,(\theta,0,-i),K,t,x}_T )-g_d(x) }{m^n}\mathcal{Z}^{d,(\theta,0,-i),K,t,x}_{T}\\
& +\sum_{\ell=0}^{n-1}\sum_{i=1}^{m^{n-\ell}} \frac{\left(F_d(U^{d,(\theta,\ell,i)}_{\ell,m,K})-
\1_\N(\ell)
F_d(U^{d,(\theta,\ell,-i)}_{\ell-1,m,K})\right)\!\left(t+(T-t) \mathfrak{r}^{(\theta,\ell,i)},
\mathcal{X}^{d,(\theta,\ell,i),K,t,x}_{t+(T-t) \mathfrak{r}^{(\theta,\ell,i)}}\right)
\mathcal{Z}^{d,(\theta,\ell,i),K,t,x}_{t+(T-t) \mathfrak{r}^{(\theta,\ell,i)}}}{m^{n-\ell}\varrho(t,t+(T-t)\mathfrak{r}^{(\theta,\ell,i)})}.\end{split}\label{a04k}
\end{align}
For every $d\in \N$ let $\mathfrak{e}_d,\mathfrak{f}_d, \mathfrak{g}_d\in [0,\infty)$ satisfy that 
\begin{align}
\mathfrak{e}_d+\mathfrak{f}_d+\mathfrak{g}_d\leq cd^c.\label{c25a}
\end{align} Let $\mathfrak{C}^d_{n,m,K}\in [0,\infty)$, $n,m\in \Z$, $d,K\in \N$, satisfy for all 
$n\in \Z$,
$d,m,K\in \N$ that 
\begin{align}\label{c22a}
\mathfrak{C}^d_{n,m,K}\leq m^n(K\mathfrak{e}_d+\mathfrak{g}_d)\1_{\N}(n)
+\sum_{\ell=0}^{n-1}\left[m^{n-\ell}(K\mathfrak{e}_d+\mathfrak{f}_d+\mathfrak{C}_{\ell,m,K}^d+\mathfrak{C}^d_{\ell-1,m,K})\right].
\end{align}
Then
the following items hold.

\begin{enumerate}[(i)]
\item \label{k36} For all $d\in \N$ there exists a unique continuous function $u_d\colon [0,T)\times \R^d\to \R^{d+1}$ such that $v_d:=\pr_0^d(u_d)$ is the unique viscosity solution to the following 
semilinear PDE of parabolic type:
\begin{align}
&
\frac{\partial v_d}{\partial t}(t,x)
+\left\langle (\nabla_x v_d)(t,x),\mu_d(x)\right\rangle
+\frac{1}{2}\operatorname{trace}\!
\left(\sigma_d(x)[\sigma_d(x)]^\top (\operatorname{Hess}_xv_d)(t,x)\right)\nonumber
\\
&\quad \qquad\qquad+f_d(t,x,v_d(t,x), (\nabla_xv_d)(t,x))
=0
\quad 
\forall\, t\in (0,T), x\in \R^d,\label{c36}
\\
&
v_d(T,x)=g_d(x) \quad \forall\, x\in \R^d\label{c38}
\end{align}
and such that
$\nabla_xv_d=(\pr^d_1(u_d), \pr^d_2(u_d),\ldots, \pr^d_d(u_d))$.
 
\item  For all $d\in \N$ we have that

\begin{align}
\limsup_{n\to\infty}\sup_{\nu\in [0,d]\cap\Z,t\in [0,T), x\in [-\mathbf{k},\mathbf{k}]^d}\left[
\Lambda_\nu^d(T-t)\left\lVert\pr^d_{\nu}\!\left(U^{d,0}_{n,n^\frac{1}{3},n^\frac{n}{3}}(t,x)-{u}_d(t,x)\right)\right\rVert_2
\right]=0.
\end{align}
\item\label{k36b}  There exist
$(C_\delta)_{\delta\in (0,1)}\subseteq (0,\infty)$,
 $\eta\in(0,\infty)$,
$(N_{d,\varepsilon})_{d\in \N,\varepsilon\in (0,1)}\subseteq\N$
 such that for all 
$d\in \N$, $\delta,\varepsilon\in (0,1)$ we have that 
\begin{align}
\sup_{\nu\in [0,d]\cap\Z,t\in [0,T), x\in [-\mathbf{k},\mathbf{k}]^d}\left[
\Lambda_\nu^d(T-t)\left\lVert\pr^d_{\nu}\!\left(U^{d,0}_{N_{d,\varepsilon},\lvert N_{d,\varepsilon}\rvert^\frac{1}{3},\lvert N_{d,\varepsilon}\rvert^\frac{N_{d,\varepsilon}}{3}}(t,x)-{u}_d(t,x)\right)\right\rVert_2\right]<\varepsilon
\end{align}
and
\begin{align}
\mathfrak{C}^d_{N_{d,\epsilon}, \lvert N_{d,\varepsilon}\rvert^\frac{1}{3}, \lvert N_{d,\varepsilon}\rvert^{\frac{N_{d,\varepsilon}}{3}}}\leq C_\delta
\varepsilon^{-(4+\delta)} \eta d^\eta.
\end{align}
\end{enumerate}

\begin{remark}
The conditions
$\mu_d\in C^3(\R^d,\R^d)$, 
$\sigma_d\in C^3(\R^d,\R^{d\times d})$
are only needed to ensure the existence and uniqueness of the viscosity
solution of the PDE \eqref{c36}-\eqref{c38} satisfying a stochastic representation of Feynman-Kac and Bismut-Elworthy-Li type , 
see \cite[Proposition~5.1, Proposition~5.2]{NW2023}. 
Later, in the recent work \cite{HP2023} (see also their accompanying paper \cite{pohl2023existence}), 
the authors only assume that $\mu_d\in C^1(\R^d,\R^d)$, 
$\sigma_d\in C^1(\R^d,\R^{d\times d})$ in order to obtain the analog result as \cite[Proposition~5.1, Proposition~5.2]{NW2023}, see \cite[Theorem~1.1]{HP2023}. However it comes with the price that in \cite[Theorem~1.1]{HP2023} they need instead to additionally assume that
the nonlinear term $f_d$ and the terminal condition $g_d$ satisfy
that $f_d\in L^2([0,T]\times\R^{2d+1},\R)$ 
and $g_d\in L^2(\R^d,\R)$, which
is not satisfied, for example, by PDEs with (piece-wise) linear $f_d$ and/or $g_d$.

Hence,  while we have decided to impose our conditions $\mu_d\in C^3(\R^d,\R^d)$, 
$\sigma_d\in C^3(\R^d,\R^{d\times d})$ so that we can directly apply  \cite[Proposition~5.1, Proposition~5.2]{NW2023}, one could instead 
use the conditions $\mu_d\in C^1(\R^d,\R^d)$, 
$\sigma_d\in C^1(\R^d,\R^{d\times d})$ together with additionally assuming that $f_d\in L^2([0,T]\times\R^{2d+1},\R)$ 
and $g_d\in L^2(\R^d,\R)$ 
and then use \cite[Theorem~1.1]{HP2023}  instead. Of course in that case, all the remaining conditions imposed in our Theorem~\ref{c34b}, such as, e.g., \eqref{a01i} combined with \eqref{c60}, would need to be kept in order to guarantee that our Theorem~\ref{c34b}  and its proof remain valid.
\end{remark}

\begin{remark}
Note that in the MLP approximation \eqref{a04k} we employ random variables
$\mathfrak{r}^\theta$ with Beta distribution, 
whereas uniform random variables are usually used 
in the MLP approximations for PDEs without gradient-dependent nonlinearities.
(see, e.g., \cite{
HJKNW2020,
HJKN2020,
hutzenthaler2020overcoming,
hutzenthaler2022multilevel1,NW2022}.)
The MLP approximation \eqref{a04k} is based on the Bismut-Elworthy-Li formula
(see \cite[Equation (5.3)]{NW2023}), and the key point of approximating the time-dependent integral is to deal with the singular integral arising in
the second term on the right hand side of \cite[Equation (5.3)]{NW2023}.
More precisely, we see, e.g., in \eqref{a26} and \eqref{a31} that we would lose integrability of the corresponding integrals if we had chosen random variables
$\mathfrak{r}^\theta$ being uniformly distributed. Similarly, we see for the choice of $\mathfrak{r}^\theta$ satisfying $\P(\mathfrak{r}^\theta\leq b)=\sqrt{b}$ for all $b\in (0,1)$ as in \cite{HJK2022} that we would lose the integrability, e.g., in \eqref{a31}.
\end{remark}

\end{theorem}
\cref{c34b} follows directly from \cref{c34} (see the proof of \cref{c34b} after the proof of \cref{c34} in \cref{s06}). 
Let us make some comments on the mathematical objects in \cref{c34b}. For every $d\in \N$
we want to approximately solve the PDE \eqref{c36} with terminal condition \eqref{c38} where $(\mu_d, \sigma_{d})$ is the linear part, $f_d$ is the nonlinear part, and $g_d$ is the terminal condition.
To make sure that for all $d\in \N$, \eqref{c36}--\eqref{c38} have a unique viscosity solution we need \eqref{f04g}--\eqref{a19g} which are Lipschitz and linear growth conditions.
Here, \eqref{f04g} tells us that the initial values of
$\mu_d,\sigma_d$ grow at most polynomially in $d$.
Next, \eqref{f04f} assumes that the operator norms of $(\mathrm{D}\mu_d)(x)$ and $(\mathrm{D}\sigma_d)(x)$ are uniformly bounded by $c$. Moreover, \eqref{c37} gives a lower bound condition on $\sigma_{d}$ and \eqref{f10f} is a Lipschitz condition of $\mathrm{D}\mu_d$ and 
$\mathrm{D}\sigma_d$ with respect to the operator norm. In addition,
\eqref{b04a} is a growth condition on $g_d$ and the initial values of $f_d$ and \eqref{a01i} as well as \eqref{a19g} are Lipschitz condition of $f_d$ and $g_d$. Note that the factors
$T$, $\sqrt{T}$, and 
$\Lambda_\nu^d(T)$, $\nu\in [0,d]\cap \Z$, $d\in \N$, 
in \eqref{a01i} as well as \eqref{a19g}
are only constants depending on $T$ which are introduced to later simplify the calculations. 
To approximate the exact solutions and its derivatives we introduce the MLP approximation in \eqref{a04k} based on Euler-Maruyama schemata \eqref{c57b}--\eqref{c56b} that approximate the forward processes. 
The motivation for \eqref{a04k} as well as \eqref{c57b}--\eqref{c56b} is the so-called Bismut-Elworthy-Li formula (see the discussion in the introduction of \cite{NW2023}). 
To describe the computational complexity, 
for each $d\in\N$,
$n\in\N_0$, 
 $M\in\N$ 
we introduce $\mathfrak{C}^{d}_{n,m,K}\in \N$ to denote the sum of: 
the number of function evaluations of $g_d$,
the number of function evaluations of $(\mu_d, (\operatorname{D}\!\mu_d))$,
the number of function evaluations of $(\sigma_d,(\operatorname{D}\!\sigma_d),\sigma^{-1}_d)$,
and the number of realizations of scalar random variables used to obtain
one realization of the MLP approximation algorithm in \eqref{a04k}. 
Moreover, for each $d\in\N$ we use
$\mathfrak{g}_{d}$ to denote the number of function evaluations of $g_d$, we use
$\mathfrak{f}_{d}$ to denote the number of function evaluations of $f_d$, and we use
$\mathfrak{e}_{d}$ to denote the sum of:
the number of realizations of scalar random variables generated,
the number of function evaluations of $(\mu_d, (\operatorname{D}\!\mu_d))$,
and the number of function evaluations of $(\sigma_d,(\operatorname{D}\!\sigma_d),\sigma^{-1}_d)$. 
Here, we count the number of evaluations of $(\sigma_d)^{-1}$ since in practice we often calculate and save $(\sigma_d)^{-1}$ first before we start with multilevel Picard approximations. Assumption \eqref{c25a} implies that the computational effort to evaluate the input functions to \eqref{c36}--\eqref{c38} grows only polynomially in the dimension $d$ which is a reasonable assumption for MLP approximation \eqref{a04k} to overcome the curse of dimensionality when approximately solving \eqref{c36}--\eqref{c38}. 
We highlight that condition \eqref{a01i} combined with \eqref{c60} is  stronger than condition (2.2) in \cite{NW2023}, which allows the numerical schema to overcome the curse of dimensionality.

Our paper is organized as follows. In \cref{s01} we study existence, uniqueness, and the spatial and temporal  regularity of solutions to SFPEs.  \cref{s02} is a perturbation result that estimates the difference between two fixed points. \cref{s04} provides an abstract framework for the study of MLP approximations. \cref{s05} contains some results for solutions to stochastic differential equations (SDEs) and their discretizations with explicit constants independent of the dimension $d$. \cref{s06} combines the results in Sections \ref{s02}--\ref{s05} to prove the main result, \cref{c34b} above.

\section{Stochastic fixed point equations}\label{s01}Lemmata \ref{a15}--\ref{a11c} are some
simple but useful  auxiliary results.
\subsection{A Gr\"onwall-type inequality}
\begin{lemma}\label{a15}
Let $T\in (0,\infty)$. Then for all $t\in [0,T)$ we have that 
$
\int_{t}^{T}\frac{dr}{\sqrt{(T-r)(r-t)}}=
\mathrm{B}(\tfrac{1}{2},\tfrac{1}{2})\leq 4.
$
\end{lemma}
\begin{proof}[Proof of \cref{a15}]The substitution $s=\frac{r-t}{T-t}$, $ds=\frac{dr}{T-t}$ shows 
for all $t\in [0,T)$ that
\begin{align} \begin{split} 
\int_{t}^{T}\frac{dr}{\sqrt{(T-r)(r-t)}}&=
 \int_{t}^{T}\frac{dr}{(T-r)^\frac{1}{2}(r-t)^\frac{1}{2}}
=
 \int_{0}^{1}\frac{(T-t)ds}{[(1-s)(T-t)]^\frac{1}{2}[s(T-t)]^\frac{1}{2}}
\\
&=
\int_0^1\frac{ds}{[s(1-s)]^\frac{1}{2}}
=
\mathrm{B}(\tfrac{1}{2},\tfrac{1}{2})\leq 4
.\end{split}
\end{align}
This completes the proof of \cref{a15}.
\end{proof}
\begin{lemma}\label{b02}
Let $T\in (0,\infty)$, $p\in (1,2)$, $q\in (2,\infty)$ satisfy that
$\frac{1}{p}+\frac{1}{q}=1$. Let
$H\colon [0,T)\to [0,\infty)$ be measurable. Then for all $t\in [0,T)$ we have that 
\begin{align}
T\int_{t}^{T}\frac{H(r)\,dr}{\sqrt{(T-r)(r-t)}}\leq T^{\frac{1}{p}}
\left(\mathrm{B}(1-\tfrac{p}{2},1-\tfrac{p}{2})
\right)^\frac{1}{p}\left(
\int_{t}^{T}\lvert H(r)\rvert^{q}\,dr\right)^\frac{1}{q}.
\end{align}
\end{lemma}
\begin{proof}[Proof of \cref{b02}]
H\"older's inequality and the substitution $s=\frac{r-t}{T-t}$, $ds=\frac{dr}{T-t}$ imply that for all $t\in [0,T)$
we have that
\begin{align} 
T\int_{t}^{T}\frac{H(r)\,dr}{\sqrt{(T-r)(r-t)}}&\leq
T
\left(
 \int_{t}^{T}\frac{dr}{(T-r)^\frac{p}{2}(r-t)^\frac{p}{2}}\right)^\frac{1}{p}
\left(
\int_{t}^{T}\lvert H(r)\rvert^{q}\,dr\right)^\frac{1}{q}\nonumber\\
&
=T
\left(
 \int_{0}^{1}\frac{(T-t)ds}{[(1-s)(T-t)]^\frac{p}{2}[s(T-t)]^\frac{p}{2}}\right)^\frac{1}{p}\left(
\int_{t}^{T}\lvert H(r)\rvert^{q}\,dr\right)^\frac{1}{q}\nonumber
\\
&\leq T(T-t)^{\frac{1}{p}-1}
\left(
\int_0^1\frac{ds}{[s(1-s)]^\frac{p}{2}}
\right)^\frac{1}{p}\left(
\int_{t}^{T}\lvert H(r)\rvert^{q}\,dr\right)^\frac{1}{q}\nonumber\\
&\leq T^{\frac{1}{p}}
\left(\mathrm{B}(1-\tfrac{p}{2},1-\tfrac{p}{2})
\right)^\frac{1}{p}\left(
\int_{t}^{T}\lvert H(r)\rvert^{q}\,dr\right)^\frac{1}{q}.
\end{align}This completes the proof of  \cref{b02}.
\end{proof}

\begin{lemma}[Gr\"onwall-type inequality]\label{a11c}
Let $T\in (0,\infty)$, $a,b\in [0,\infty)$, 
$p\in (1,2)$, $q\in (2,\infty)$ satisfy that
$\frac{1}{p}+\frac{1}{q}=1$. Let
$H\colon [0,T)\to [0,\infty)$ be measurable.  Assume for all
$t\in [0,T)$ that 
$\int_{0}^{T}\lvert H(s)\rvert^{q}\,ds<\infty$ and
\begin{align}\label{a11}
H(t)\leq a+bT\int_{t}^{T}\frac{H(r)\,dr}{\sqrt{(T-r)(r-t)}}.
\end{align}
Then
 for all
$t\in [0,T)$ we have that
\begin{align} 
H(t)
\leq 
2^{\frac{q-1}{q}} a\exp \left(\frac{
2^{q-1}
\lvert b\rvert^{q} T^{\frac{q}{p}}
\left(\mathrm{B}(1-\tfrac{p}{2},1-\tfrac{p}{2})
\right)^\frac{q}{p}}{q}(T-t)
\right).
\end{align}
\end{lemma}
\begin{proof}[Proof of \cref{a11c}]
\cref{b02} and \eqref{a11} prove for all
$t\in [0,T)$ that 
\begin{align}\label{a11b} \begin{split} 
H(t)\leq a+bT^{\frac{1}{p}}
\left(\mathrm{B}(1-\tfrac{p}{2},1-\tfrac{p}{2})
\right)^\frac{1}{p}\left(
\int_{t}^{T}\lvert H(r)\rvert^{q}\,dr\right)^\frac{1}{q}\end{split}.
\end{align}
This and the fact that $\forall\, x,y\in \R\colon \lvert x+y\rvert^{q}\leq 2^{q-1}\lvert x\rvert^{q}+2^{q-1}\lvert y\rvert^{q}$
show for all
$t\in [0,T)$ that  
\begin{align}
 \begin{split} 
\lvert H(t)\rvert^{q}\leq 2^{q-1}\lvert a\rvert^{q}+
2^{q-1}
\lvert b\rvert^{q} T^{\frac{q}{p}}
\left(\mathrm{B}(1-\tfrac{p}{2},1-\tfrac{p}{2})
\right)^\frac{q}{p}
\int_{t}^{T}\lvert H(r)\rvert^{q}\,dr.\end{split}
\end{align}
Hence, the fact that 
$\int_{0}^{T}\lvert H(s)\rvert^{q}\,ds<\infty$ and Gr\"onwall's lemma imply for all
$t\in [0,T)$ that
\begin{align}
\lvert H(t)\rvert^{q}
\leq 
2^{q-1}\lvert a\rvert^{q}\exp \left(
2^{q-1}
\lvert b\rvert^{q} T^{\frac{q}{p}}
\left(\mathrm{B}(1-\tfrac{p}{2},1-\tfrac{p}{2})
\right)^\frac{q}{p}(T-t)
\right).
\end{align}
This proves  for all
$t\in [0,T)$ that
\begin{align} 
H(t)
\leq 
2^{\frac{q-1}{q}} a\exp \left(\frac{
2^{q-1}
\lvert b\rvert^{q} T^{\frac{q}{p}}
\left(\mathrm{B}(1-\tfrac{p}{2},1-\tfrac{p}{2})
\right)^\frac{q}{p}}{q}(T-t)
\right).
\end{align}
The proof of \cref{a11c} is thus completed.
\end{proof}
\begin{corollary}\label{a16}
Let $T\in (0,\infty)$, $a,b\in [0,\infty)$. Let
$H\colon [0,T)\to [0,\infty)$ be measurable. Assume for all
$t\in [0,T)$ that 
$\sup_{s\in [0,T)}\lvert H(s)\rvert<\infty$ and
$
H(t)\leq a+bT\int_{t}^{T}\frac{H(r)\,dr}{\sqrt{(T-r)(r-t)}}
$. Then we have for all $t\in [0,T)$ that 
$H(t)\leq 2ae^{86b^3T^2(T-t)}$.
\end{corollary}
\begin{proof}[Proof of \cref{a16}]
\cref{a11c}  (with
$p\gets\frac{3}{2}$, $q\gets 3$ in the notation of \cref{a11c}) and the fact that $\sup_{s\in [0,T)}\lvert H(s)\rvert<\infty$ show for all
$t\in [0,T)$ that
\begin{align}  \begin{split} 
H(t)
&\leq 
2^{\frac{3-1}{3}} a\exp \left(\frac{
2^{3-1}
\lvert b\rvert^{3} T^{2}
\left(\mathrm{B}(1-\tfrac{3}{4},1-\tfrac{3}{4})
\right)^2}{3}(T-t)
\right)\\
&\leq 2a\exp\left(\frac{4 \lvert b\rvert^3T^2 8^2 }{3}(T-t)\right) \\
&\leq 2ae^{86b^3T^2(T-t)}
.\end{split}
\end{align}
This completes the proof of \cref{a16}.
\end{proof}
\subsection{Existence and uniqueness of solutions to SFPEs}
In \cref{a03} below we establish existence, uniqueness, and a growth property of solutions to SFPEs.
\begin{setting}\label{z01}
Let 
$T\in (0,\infty)$,
$d\in \N$, 
$\exponentV,\exponentZ\in (1,\infty)$,
$c\in [1,\infty)$, $(L_i)_{i\in [0,d]\cap \Z}\in \R^{d+1}$ satisfy that $
\sum_{i=0}^{d}L_i\leq c$ and $
 \frac{1}{\exponentV}+\frac{1}{\exponentZ}\leq 1$. Let $\lVert\cdot \rVert\colon \R^d\to [0,\infty)$ be a norm on $\R^d$.
Let 
$\Lambda=(\Lambda_{i})_{i\in [0,d]\cap\Z}\colon [0,T]\to \R^{1+d}$ satisfy for all $t\in [0,T]$ that $\Lambda(t)=(1,\sqrt{t},\ldots,\sqrt{t})$.
Let $\pr=(\pr_\nu)_{\nu\in [0,d]\cap\Z}\colon \R^{d+1}\to\R$ satisfy
for all 
$w=(w_\nu)_{\nu\in [0,d]\cap\Z}$,
$i\in [0,d]\cap\Z$ that
$\pr_i(w)=w_i$. Let $f\in C( [0,T]\times\R^d\times \R^{d+1},\R)$, $g\in C(\R^d,\R)$, $V\in C([0,T]\times \R^d, [0,\infty))$.
Let $ (\Omega,\mathcal{F},\P)$ be a probability space.
For every random variable $\mathfrak{X}\colon \Omega\to\R$,
$s\in [1,\infty)$ let $\lVert \mathfrak{X}\rVert_s\in [0,\infty]$ satisfy that 
$ \lVert \mathfrak{X}\rVert_s= (\E[\lvert \mathfrak{X}\rvert^s])^\frac{1}{s}$.
Let $(X^{s,x}_t)_{s\in [0,T],t\in[s,T],x\in\R^d}\colon \{(\tau,\sigma)\in [0,T]^2\colon \tau\leq \sigma\}\times \R^d\times \Omega\to\R^d $,
$(Z^{s,x}_t)_{s\in [0,T],t\in[s,T],x\in\R^d}\colon \{(\tau,\sigma)\in [0,T]^2\colon \tau< \sigma\}\times \R^d\times \Omega\to\R^{d+1} $ be measurable. Assume
for all 
$i\in [0,d]\cap\Z$,
$t\in [0,T)$, $r\in (t,T]$,
 $x\in \R^d$, $w_1,w_2\in \R^{d+1}$ that
\begin{align}\label{a04}
\lvert g(x)\rvert\leq V(T,x),\quad \lvert Tf(t,x,0)\rvert\leq V(t,x),
\end{align}
\begin{align}\label{a01b}
\lvert
f(t,x,w_1)-f(t,x,w_2)\rvert\leq \sum_{\nu=0}^{d}\left[
L_\nu\Lambda_\nu(T)
\lvert\pr_\nu(w_1-w_2) \rvert\right],
\end{align}
\begin{align}\label{a02b}
\left\lVert
V(r,X^{t,x}_r)
\right\rVert_{\exponentV}\leq V(t,x),\quad 
\left\lVert
\pr_i(Z^{t,x}_r)\right\rVert_{\exponentZ}\leq \frac{c}{\Lambda_{i}(r-t)}.
\end{align}
\end{setting}

\begin{lemma}\label{a03}
Assume \cref{z01}.
Then the following items hold.
\begin{enumerate}[(i)]
\item \label{a17}
There exists a unique measurable function $u\colon [0,T)\times \R^d\to \R^{d+1}$ such that for all $t\in [0,T)$, $x\in \R^d$
we have that
\begin{align}\label{b01}
\max_{\nu\in [0,d]\cap\Z}\sup_{\tau\in [0,T), \xi\in \R^d}
\left[\Lambda_\nu(T-\tau)\frac{\lvert\pr_\nu(u(\tau,\xi))\rvert}{V(\tau,\xi)}\right]<\infty,
\end{align}
\begin{align}\max_{\nu\in [0,d]\cap\Z}\left[
\E\!\left [\left\lvert g(X^{t,x}_{T} )\pr_\nu(Z^{t,x}_{T})\right\rvert \right] + \int_{t}^{T}
\E \!\left[\left\lvert
f(r,X^{t,x}_{r},u(r,X^{t,x}_{r}))\pr_\nu(Z^{t,x}_{r})\right\rvert\right]dr\right]<\infty,
\end{align}
and
\begin{align}
u(t,x)=\E\!\left [g(X^{t,x}_{T} )Z^{t,x}_{T} \right] + \int_{t}^{T}
\E \!\left[
f(r,X^{t,x}_{r},u(r,X^{t,x}_{r}))Z^{t,x}_{r}\right]dr.\label{a05}
\end{align}

\item \label{z03} For all $t\in [0,T)$ we have that
\begin{align}
\max_{\nu\in [0,d]\cap\Z}\sup_{y\in \R^d}
\left[\Lambda_\nu(T-t)\frac{\lvert\pr_\nu(u(t,y))\rvert}{V(t,y)}\right]
\leq 
6c e^{86c^6T^2(T-t)}.
\end{align}
\end{enumerate}
\end{lemma}
\begin{proof}[Proof of \cref{a03}]
Denote by $M$ the set of all measurable functions
$w\colon [0,T)\times \R^d\to \R^{d+1}$
 and by $B\subseteq M$ the set which satisfies that
\begin{align}
B=\left\{w\in M\colon 
\sup_{t\in [0,T),x\in \R^d}\max_{\nu\in [0,d]\cap\Z}
\frac{\left\lvert
\Lambda_\nu(T-t)\pr_\nu (w(t,x))\right\rvert}{V(t,x)}<\infty
\right\}.\label{a18}
\end{align} 
For every $\lambda\in [0,\infty)$ let $\threenorm{\cdot }_\lambda\colon M\to [0,\infty]$ satisfy for all $w\in  M$ that
\begin{align}
\threenorm{w}_\lambda=\sup_{t\in [0,T),x\in \R^d}\max_{\nu\in [0,d]\cap\Z}
\frac{e^{\lambda t}
\left\lvert\Lambda_\nu(T-t)\pr_\nu (w(t,x))\right\rvert}{V(t,x)}.\label{a12}
\end{align}
Then it is easy to show for all $\lambda\in [0,\infty)$ that $(B, \threenorm{\cdot}_\lambda\upharpoonright_B)$ is an $\R$-Banach space.

 Next,
H\"older's inequality, \eqref{a04}, and \eqref{a02b} imply for all 
$i\in [0,d]\cap\Z$, $t\in [0,T)$, $x\in \R^d$ that
\begin{align} \begin{split} 
\left\lVert
\Lambda_{i}(T-t)g(X^{t,x}_{T} )\pr_i(Z^{t,x}_{T})
\right\rVert_1
&\leq \xeqref{a04}
\left\lVert
\Lambda_{i}(T-t)V(T,X^{t,x}_{T} )\pr_i(Z^{t,x}_{T})
\right\rVert_1\\
&\leq 
\Lambda_{i}(T-t)\left\lVert
V(T,X^{t,x}_{T} )
\right\rVert_{\exponentV}\left\lVert
\pr_i(Z^{t,x}_{T})
\right\rVert_{\exponentZ}\\
&\leq 
\Lambda_{i}(T-t)\xeqref{a02b}
V(t,x)\xeqref{a02b}\frac{c}{\Lambda_{i}(T-t)}\\
&\leq c V(t,x).
\end{split}\label{a06}
\end{align}
In addition,
H\"older's inequality, the fact that 
$\frac{1}{\exponentV}+\frac{1}{\exponentZ}\leq 1$,
\eqref{a04},  \eqref{a02b}, and the fact that
$\forall\,t\in [0,T)\colon \int_{t}^{T}\frac{dr}{\sqrt{r-t}}=2\sqrt{r-t}|_{r=t}^T=2\sqrt{T-t}$
 show for all 
$i\in [0,d]\cap\Z$, $t\in [0,T)$, $x\in \R^d$ that
\begin{align}
\int_{t}^{T}
\left\lVert
\Lambda_{i}(T-t)
f(r,X^{t,x}_{r},0)\pr_i(Z^{t,x}_{r})\right\rVert_1
dr
&\leq \xeqref{a04}
\int_{t}^{T}
\left\lVert
\Lambda_{i}(T-t)
\frac{1}{T}V(r,X^{t,x}_{r})\pr_i(Z^{t,x}_{r})\right\rVert_1
dr\nonumber
\\
&\leq 
\int_{t}^{T}
\Lambda_{i}(T-t)
\frac{1}{T}\left\lVert
V(r,X^{t,x}_{r})\right\rVert_{\exponentV}\left\lVert\pr_i(Z^{t,x}_{r})\right\rVert_{\exponentZ}
dr\nonumber\\
&\leq \xeqref{a02b}\int_{t}^{T}
\Lambda_{i}(T-t)
\frac{1}{T}V(t,x)
\frac{c}{\Lambda_i(r-t)}\,dr\nonumber\\
&\leq \int_{t}^{T}
\frac{c}{T}V(t,x)\frac{\sqrt{T-t}}{\sqrt{r-t}}
\,dr\nonumber\\
&\leq 2cV(t,x).
\label{a07}
\end{align}
Next, H\"older's inequality, the fact that
$\frac{1}{\exponentV}+\frac{1}{\exponentZ}\leq 1$, \eqref{a02b}, and \cref{a15} imply for all 
$i\in [0,d]\cap\Z$, $t\in [0,T)$, $x\in \R^d$, $w\in B$ that 
\begin{align} 
&
\int_{t}^{T}
\Lambda_{i}(T-t)
\sum_{\nu=0}^d\left[
L_\nu\Lambda_\nu(T)
\left\lVert \pr_\nu (w(r,X_r^{t,x}))
\pr_i(Z^{t,x}_{r})\right\rVert_1\right]
dr\nonumber\\
&\leq 
\int_{t}^{T}
\Lambda_{i}(T-t)
\sum_{\nu=0}^d\left[
L_\nu\frac{\sqrt{T}}{\sqrt{T-r}}\Lambda_\nu(T-r)
\left\lVert \pr_{\nu}(w(r,X_r^{t,x}))
\pr_i(Z^{t,x}_{r})\right\rVert_1\right]
dr\nonumber\\
&\leq 
\int_{t}^{T}
\Lambda_{i}(T-t)
\frac{c\sqrt{T}}{\sqrt{T-r}}\left[\max_{\nu\in [0,d]\cap\Z}
\sup_{y\in \R^d}
\frac{\Lambda_\nu(T-r)\lvert \pr_\nu(w(r,y))\rvert}{V(r,y)}\right]
\left\lVert V(r,X_r^{t,x})
\pr_i(Z^{t,x}_{r})\right\rVert_1
dr\nonumber\\
&\leq 
\int_{t}^{T}
\Lambda_{i}(T-t)
\frac{c\sqrt{T}}{\sqrt{T-r}}\left[\max_{\nu\in [0,d]\cap\Z}
\sup_{\tau\in [0,T),y\in \R^d}
\frac{\Lambda_\nu(T-\tau)\lvert\pr_{\nu}( w(\tau,y))\rvert}{V(\tau,y)}\right]
\left\lVert V(r,X_r^{t,x})\right\rVert_{\exponentV}
\left\lVert
\pr_i(Z^{t,x}_{r})\right\rVert_{\exponentZ}
dr\nonumber\\
&\leq 
\int_{t}^{T}
\Lambda_{i}(T-t)
\frac{c\sqrt{T}}{\sqrt{T-r}}\left[\max_{\nu\in [0,d]\cap\Z}
\sup_{\tau\in [0,T),y\in \R^d}
\frac{\Lambda_\nu(T-\tau)\lvert\pr_{\nu}( w(\tau,y))\rvert}{V(\tau,y)}\right]
V(t,x)
\frac{c}{\Lambda_{i}(r-t)}
dr\nonumber\\
&\leq 
\left[\max_{\nu\in [0,d]\cap\Z}
\sup_{\tau\in [0,T),y\in \R^d}
\frac{\Lambda_\nu(T-\tau)\lvert \pr_{\nu}(w(\tau,y))\rvert}{V(\tau,y)}\right]
V(t,x)\int_{t}^{T}\frac{c^2\sqrt{T}}{\sqrt{T-r}}
\frac{\Lambda_{i}(T-t)}{\Lambda_{i}(r-t)}\,dr\nonumber\\
&\leq 
\left[\max_{\nu\in [0,d]\cap\Z}
\sup_{\tau\in [0,T),y\in \R^d}
\frac{\Lambda_\nu(T-\tau)\lvert \pr_{\nu}(w(\tau,y))\rvert}{V(\tau,y)}\right]
V(t,x)\int_{t}^{T}\frac{c^2\sqrt{T}}{\sqrt{T-r}}
\frac{\sqrt{T-t}}{\sqrt{r-t}}\,dr<\infty.
\end{align}
This, the triangle inequality,  \eqref{a01b}, and \eqref{a07} prove 
 for all 
$i\in [0,d]\cap\Z$, $t\in [0,T)$, $x\in \R^d$, $w\in B$ that 
\begin{align} 
&
\int_{t}^{T}
\Lambda_{i}(T-t)
\left\lVert
f(r,X^{t,x}_{r},w(r,X^{t,x}_{r}))\pr_i(Z^{t,x}_{r})\right\rVert_1dr\nonumber\\
&\leq 
\int_{t}^{T}
\left\lVert
\Lambda_{i}(T-t)
f(r,X^{t,x}_{r},0)\pr_i(Z^{t,x}_{r})\right\rVert_1
dr\nonumber\\
&\quad +
\int_{t}^{T}
\Lambda_{i}(T-t)
\sum_{\nu=0}^d\left[
L_\nu\Lambda_\nu(T)
\left\lVert \pr_{\nu}(w(r,X_r^{t,x}))
\pr_i(Z^{t,x}_{r})\right\rVert_1\right]
dr<\infty.
\label{a21e}\end{align}
Let $\Phi\colon B\to (\R^{d+1})^{[0,T)\times \R^d}$
satisfy for all $w\in B$, $t\in [0,T)$, $x\in \R^d$ that
\begin{align}
(\Phi(w))(t,x)=\E\!\left [g(X^{t,x}_{T} )Z^{t,x}_{T} \right] + \int_{t}^{T}
\E \!\left[
f(r,X^{t,x}_{r}, w(r,X^{t,x}_{r}))Z^{t,x}_{r}\right]dr,\label{a09}
\end{align}
where the expectations are well-defined due to \eqref{a21e} and \eqref{a06}. Moreover, Fubini's theorem implies  for all $w\in B$ that $\Phi(w)\in M$.
Next, \eqref{a09}, the triangle inequality,  \eqref{a06}, and \eqref{a07} show for all $i\in [0,d]\cap\Z$,
$t\in [0,T)$, $x\in \R^d$ that 
\begin{align} 
\Lambda_{i}(T-t)\left\lvert \pr_i \left( (\Phi(0))(t,x)\right)\right\rvert&\leq \Lambda_{i}(T-t)\left \lVert g(X^{t,x}_{T} )Z^{t,x}_{T} \right\rVert_{1} + \int_{t}^{T}
\Lambda_{i}(T-t)\left\lVert
f(r,X^{t,x}_{r},0)Z^{t,x}_{r}\right\rVert_1 dr\nonumber\\
&\leq \xeqref{a06}cV(t,x)+\xeqref{a07}2cV(t,x)=3c V(t,x).
\end{align}
This implies for all 
$\lambda\in [0,\infty)$ that
\begin{align}
\threenorm{\Phi(0)}_\lambda\leq 3ce^{\lambda T}<\infty.\label{a13}
\end{align}
Moreover,
\eqref{a09}, \eqref{a01b}, H\"older's inequality, the fact that
$\frac{1}{\exponentV}+\frac{1}{\exponentZ}\leq 1$, and \eqref{a02b}
prove for all
$i\in [0,d]\cap\Z $, 
$w,v\in B$,
$t\in [0,T)$, $x\in \R^d$, $\lambda\in [0,\infty)$ that
\begin{align} 
&\left\lvert\Lambda_{i}(T-t)\pr_i\left(
(\Phi(w))(t,x)-(\Phi(v))(t,x)\right)\right\rvert\nonumber\\
&=\xeqref{a09}\left\lvert\Lambda_{i}(T-t)
 \int_{t}^{T}
\E \!\left[
\left(
f(r,X^{t,x}_{r},w(r,X^{t,x}_{r}))
-f(r,X^{t,x}_{r}, v(r,X^{t,x}_{r}))\right)
\pr_i(Z^{t,x}_{r})\right]dr\right\rvert\nonumber\\
&\leq 
\Lambda_{i}(T-t)
 \int_{t}^{T}
\left\lVert\left\lvert
\left(
f(r,X^{t,x}_{r},w(r,X^{t,x}_{r}))
-f(r,X^{t,x}_{r}, v(r,X^{t,x}_{r}))\right)\right\rvert\left\lvert
\pr_i(Z^{t,x}_{r})\right\rvert\right\rVert_1 dr\nonumber\\
&\leq \xeqref{a01b}
\Lambda_{i}(T-t)
 \int_{t}^{T}
\left\lVert\sum_{\nu=0}^dL_\nu\Lambda_\nu(T)
\left\lvert\pr_\nu \left(
w(r,X^{t,x}_{r})
-v(r,X^{t,x}_{r})\right)\right\rvert\left\lvert
\pr_i(Z^{t,x}_{r})\right\rvert\right\rVert_1 dr\nonumber\\
&\leq 
 \int_{t}^{T}
\sum_{\nu=0}^dL_\nu\Lambda_\nu(T)
\left\lVert
\pr_\nu \left(
w(r,X^{t,x}_{r})
-v(r,X^{t,x}_{r})\right)
\right\rVert_{\exponentV}
\Lambda_{i}(T-t)
\left\lVert
\pr_i(Z^{t,x}_{r})\right\rVert_{\exponentZ} dr\nonumber\\
&\leq 
 \int_{t}^{T}
\sum_{\nu=0}^dL_\nu
\frac{\sqrt{T}}{\sqrt{T-r}}
\Lambda_\nu(T-r)
\left\lVert
\pr_\nu \left(
w(r,X^{t,x}_{r})
-v(r,X^{t,x}_{r})\right)
\right\rVert_{\exponentV}
\Lambda_{i}(T-t)
\left\lVert
\pr_i(Z^{t,x}_{r})\right\rVert_{\exponentZ} dr\nonumber\\
&\leq 
 \int_{t}^{T}
\sum_{\nu=0}^dL_\nu
\frac{\sqrt{T}}{\sqrt{T-r}}
\Lambda_\nu(T-r)\left[\sup_{y\in \R^d}
\frac{\left\lvert
\pr_\nu \left(
w(r,y)
-v(r,y)\right)
\right\rvert}{V(r,y)}\right]\left\lVert V(r,X^{t,x}_r) \right\rVert_{\exponentV}\xeqref{a02b}
\frac{c\Lambda_{i}(T-t)}{\Lambda_{i}(r-t)}
 dr\nonumber\\
&\leq 
 \int_{t}^{T}
\sum_{\nu=0}^dL_\nu
\frac{\sqrt{T}}{\sqrt{T-r}}
\Lambda_\nu(T-r)\left[\sup_{y\in \R^d}
\frac{\left\lvert
\pr_\nu \left(
w(r,y)
-v(r,y)\right)
\right\rvert}{V(r,y)}\right]\xeqref{a02b}V(t,x)
\frac{c\sqrt{T-t}}{\sqrt{r-t}}
 dr\nonumber\\
&\leq 
 \int_{t}^{T}
\sum_{\nu=0}^dL_\nu\left[
\sup_{y\in \R^d}
\frac{e^{\lambda r}\Lambda_\nu(T-r)\left\lvert
\pr_\nu \left(
w(r,y)
-v(r,y)\right)
\right\rvert}{V(r,y)}\right]
V(t,x)
\frac{cTe^{-\lambda r}}{\sqrt{(T-r)(r-t)}}
 dr\nonumber\\
&\leq c^2\threenorm{w-v}_{\lambda}V(t,x)\int_{t}^{T}
\frac{Te^{-\lambda r}}{\sqrt{(T-r)(r-t)}}
 dr.
\label{a10}
\end{align}
Next, \cref{b02}  (with $p\gets \frac{3}{2}$, $q\gets 3$)
and the fact that
$\forall\,t\in [0,T),\lambda\in (0,\infty)\colon 
\int_{t}^{T} e^{-3\lambda  r}\,dr=
\frac{e^{-3\lambda  r}}{-3\lambda}\Bigr|_{r=t}^T=\frac{e^{-3\lambda  t} - e^{-3\lambda  T}}{3\lambda }\leq \frac{e^{-3\lambda  t}}{3\lambda }
$ establish for all $t\in [0,T)$ that
\begin{align} \begin{split} 
T\int_{t}^{T}\frac{e^{-\lambda r}\,dr}{\sqrt{(T-r)(r-t)}}
&\leq 
T^{\frac{2}{3}}
\left(\mathrm{B}(1-\tfrac{3}{4},1-\tfrac{3}{4})
\right)^\frac{2}{3}\left(
\int_{t}^{T}e^{-3\lambda  r}\,dr\right)^\frac{1}{3}
\\
&\leq 
T^{\frac{2}{3}}
\left(\mathrm{B}(\tfrac{1}{4},\tfrac{1}{4})
\right)^\frac{2}{3}\frac{e^{-\lambda t}}{(3\lambda)^\frac{1}{3}}.
\end{split}
\end{align}
This and \eqref{a10} imply for all
$i\in [0,d]\cap\Z $, 
$w,v\in B$,
$t\in [0,T)$, $x\in \R^d$, $\lambda\in (0,\infty)$ that
\begin{align}
&\left\lvert\Lambda_{i}(T-t)\pr_i\left(
(\Phi(w))(t,x)-(\Phi(v))(t,x)\right)\right\rvert\leq 
c^2\threenorm{w-v}_{\lambda}V(t,x)T^{\frac{2}{3}}
\left(\mathrm{B}(\tfrac{1}{4},\tfrac{1}{4})
\right)^\frac{2}{3}\frac{e^{-\lambda t}}{(3\lambda)^\frac{1}{3}}
 .
\end{align}
Hence,  \eqref{a12} proves for all $w,v\in B$, $\lambda\in (0,\infty)$ that
\begin{align}
\threenorm{\Phi(w)-\Phi(v)}_\lambda\leq 
c^2\threenorm{w-v}_{\lambda}
\frac{T^{\frac{2}{3}}
\left(\mathrm{B}(\tfrac{1}{4},\tfrac{1}{4})
\right)^\frac{2}{3}}{(3\lambda)^\frac{1}{3}}.
\end{align}
Therefore, there exists $\lambda_0\in (0,\infty)$ such that
for all $w,v\in B$
we have that
\begin{align}
\threenorm{\Phi(w)-\Phi(v)}_{\lambda_0}\leq\frac{1}{2}
\threenorm{w-v}_{\lambda_0}.\label{a14}
\end{align}
This, the triangle inequality, and \eqref{a13} imply for all $w\in B$ that 
\begin{align}
\threenorm{\Phi(w)}_{\lambda_0}\leq 
\threenorm{\Phi(0)}_{\lambda_0}+
\threenorm{\Phi(w)-\Phi(0)}_{\lambda_0}\leq
\threenorm{\Phi(0)}_{\lambda_0}+
\frac{1}{2}
\threenorm{w}_{\lambda_0}<\infty.
\end{align}
Thus, $\Phi(w)\in B$. This, \eqref{a14}, and the Banach fixed point theorem show that there exists $u\in B$ such that
$\Phi(u)=u$. Therefore, \eqref{a09} and \eqref{a18} imply \eqref{a17}.

Next, H\"older's inequality and \eqref{a02b} 
 show for all 
$i\in [0,d]\cap\Z$, $t\in [0,T)$, $x\in \R^d$ that
\begin{align}
&
\int_{t}^{T}
\left\lVert
\Lambda_{i}(T-t)
\sum_{\nu=0}^{d}
L_\nu\Lambda_\nu(T)
\lvert\pr_{\nu}(u(r,X^{t,x}_{r}))\rvert\lvert\pr_i(Z^{t,x}_{r})\rvert\right\rVert_1
dr\nonumber\\
&\leq 
\int_{t}^{T}
\sum_{\nu=0}^{d}\left[
L_\nu\Lambda_\nu(T)
\left
\lVert\pr_{\nu}(u(r,X^{t,x}_{r}))\right\rVert_{\exponentV}\right]
\left[
\Lambda_{i}(T-t)
\left\lVert
\pr_i(Z^{t,x}_{r})
\right\rVert_{\exponentZ}\right]dr\nonumber\\
&\leq 
\int_{t}^{T}
\sum_{\nu=0}^{d}\left[
L_\nu
\frac{\sqrt{T}}{\sqrt{T-r}}
\Lambda_\nu(T-r)
\left
\lVert\pr_{\nu}(u(r,X^{t,x}_{r}))\right\rVert_{\exponentV}\right]
\left[
\Lambda_{i}(T-t)
\left\lVert
\pr_i(Z^{t,x}_{r})
\right\rVert_{\exponentZ}\right]dr\nonumber\\
&\leq 
\int_{t}^{T}
c\frac{\sqrt{T}}{\sqrt{T-r}}\left[\max_{\nu\in [0,d]\cap\Z}
\left[\Lambda_\nu(T-r)
\left
\lVert\pr_{\nu}(u(r,X^{t,x}_{r}))\right\rVert_{\exponentV}
\right]
\right]\left[
\frac{c\Lambda_{i}(T-t)}{\Lambda_{i}(r-t)}
\right]dr\nonumber\\
&\leq 
\int_{t}^{T}
c\frac{\sqrt{T}}{\sqrt{T-r}}\left[\max_{\nu\in [0,d]\cap\Z}
\left[\Lambda_\nu(T-r)
\left
\lVert\pr_{\nu}(u(r,X^{t,x}_{r}))\right\rVert_{\exponentV}
\right]\right]
\frac{c\sqrt{T-t}}{\sqrt{r-t}}
dr\nonumber\\
&\leq 
\int_{t}^{T}
c\frac{\sqrt{T}}{\sqrt{T-r}}\left[\max_{\nu\in [0,d]\cap\Z}\sup_{y\in \R^d}
\left[\Lambda_\nu(T-r)\frac{\lvert\pr_\nu(u(r,y))\rvert}{V(r,y)}
\right]\right]
\left
\lVert V(r,X^{t,x}_{r})\right\rVert_{\exponentV}
\frac{c\sqrt{T-t}}{\sqrt{r-t}}
dr\nonumber\\
&
\leq \int_{t}^{T}c^2T\left[
\max_{\nu\in [0,d]\cap\Z}\sup_{y\in \R^d}
\left[\Lambda_\nu(T-r)\frac{\lvert\pr_\nu(u(r,y))\rvert}{V(r,y)}
\right]
\right]
\frac{V(t,x)}{\sqrt{(T-r)(r-t)}}\,dr.
\label{a08}
\end{align}
In addition,
the triangle inequality and \eqref{a01b} imply  for all
$t\in [0,T]$, $x\in \R^d$, $w\colon[0,T]\times \R^d\to \R^{d+1}  $ that
\begin{align} \begin{split} 
\left\lvert f(t,x,w(t,x))\right\rvert
&\leq 
\left\lvert f(t,x,0)\right\rvert
+
\left\lvert f(t,x,w(t,x))-f(t,x,0)\right\rvert
\\&
\leq 
\left\lvert f(t,x,0)\right\rvert
+\sum_{\nu=0}^{d}\left[
L_\nu\Lambda_\nu(T)
\lvert\pr_\nu(w)\rvert\right].
\end{split}
\end{align}
This, \eqref{a05}, the triangle inequality, \eqref{a06}, \eqref{a07}, and \eqref{a08}  prove for all 
$i\in [0,d]\cap\Z$, $t\in [0,T]$, $x\in \R^d$ that
\begin{align} 
&\lvert\Lambda_{i}(T-t)\pr_i(u(t,x))\rvert\nonumber\\
&=\left\lvert
\E\!\left[\Lambda_{i}(T-t)g(X^{t,x}_{T} )\pr_i(Z^{t,x}_{T}) \right] + \int_{t}^{T}
\E \!\left[
\Lambda_{i}(T-t)
f(r,X^{t,x}_{r},u(r,X^{t,x}_{r}))\pr_i(Z^{t,x}_{r})\right]dr\right\rvert\nonumber\\
&\leq 
\left\lVert
\Lambda_{i}(T-t)g(X^{t,x}_{T} )\pr_i(Z^{t,x}_{T})
\right\rVert_1
+
 \int_{t}^{T}
\left\lVert
\Lambda_{i}(T-t)
f(r,X^{t,x}_{r},u(r,X^{t,x}_{r}))\pr_i(Z^{t,x}_{r})\right\rVert_1
dr\nonumber\\
&\leq 
\left\lVert
\Lambda_{i}(T-t)g(X^{t,x}_{T} )\pr_i(Z^{t,x}_{T})
\right\rVert_1
+
 \int_{t}^{T}
\left\lVert
\Lambda_{i}(T-t)
f(r,X^{t,x}_{r},0)\pr_i(Z^{t,x}_{r})\right\rVert_1
dr\nonumber\\
&\quad +
\int_{t}^{T}
\left\lVert
\Lambda_{i}(T-t)
\sum_{\nu=0}^{d}
L_\nu\Lambda_\nu(T)
\lvert\pr_{\nu}(u(r,X^{t,x}_{r}))\rvert\lvert\pr_i(Z^{t,x}_{r})\rvert\right\rVert_1
dr\nonumber\\
&\leq \xeqref{a06}cV(t,x)+\xeqref{a07}2cV(t,x)+\xeqref{a08}\int_{t}^{T}c^2T
\left[\max_{\nu\in [0,d]\cap\Z}\sup_{y\in \R^d}
\left[\Lambda_\nu(T-r)\frac{\lvert\pr_\nu(u(r,y))\rvert}{V(r,y)}
\right]\right]\frac{V(t,x)}{\sqrt{(T-r)(r-t)}}\,dr.
\end{align}
Dividing by $V(t,x)$ then proves for all $t\in [0,T)$ that
\begin{align}
\begin{split} 
&
\max_{\nu\in [0,d]\cap\Z}\sup_{y\in \R^d}
\left[\Lambda_\nu(T-t)\frac{\lvert\pr_\nu(u(t,y))\rvert}{V(t,y)}
\right]\\
&\leq 3c+\int_{t}^{T}c^2T\left[
\max_{\nu\in [0,d]\cap\Z}\sup_{y\in \R^d}
\left[\Lambda_\nu(T-r)\frac{\lvert\pr_\nu(u(r,y))\rvert}{V(r,y)}\right]
\right]\frac{dr}{\sqrt{(T-r)(r-t)}}
\end{split}
\end{align}
Thus, \eqref{b01} and \cref{a16} 
imply for all $t\in [0,T)$ that
\begin{align}
\max_{\nu\in [0,d]\cap\Z}\sup_{y\in \R^d}
\left[\Lambda_\nu(T-t)\frac{\lvert\pr_\nu(u(t,y))\rvert}{V(t,y)}\right]
\leq 
6c e^{86c^6T^2(T-t)}.
\end{align}
This completes the proof of \cref{a03}.
\end{proof}
\subsection{Spatial Lipschitz continuity of solutions to SFPEs}
In \cref{a37} below we prove spatial Lipschitz continuity of solutions to SFPEs. The key assumptions here are the Lipschitz-type conditions 
\eqref{a01}--\eqref{a39} together with the so-called flow property \eqref{a40}, which is satisfied, e.g., when the corresponding process $X$ is a solution to an SDE.

\begin{setting}\label{z04}Assume \cref{z01}. 
Suppose that $\max \{ c,48e^{86c^6T^3}\}\leq V$,
$
 \frac{2}{\exponentV}+\frac{1}{\exponentX}+\frac{1}{\exponentZ}\leq 1$, 
and $\frac{1}{2}+\frac{1}{\exponentZ}\leq 1$. Furthermore, suppose
for all 
$i\in [0,d]\cap\Z$,
$s\in [0,T)$,
$t\in [s,T)$, $r\in (t,T]$,
 $x\in \R^d$, $w_1,w_2\in \R^{d+1}$,
$A\in (\mathcal{B}(\R^d))^{\otimes \R^d}$,
$B\in 
(\mathcal{B}(\R^d))^{\otimes ([t,T)\times\R^d)}$
 that
\begin{align}
\lvert
f(t,x,w_1)-f(t,y,w_2)\rvert\leq \sum_{\nu=0}^{d}\left[
L_\nu\Lambda_\nu(T)
\lvert\pr_\nu(w_1-w_2) \rvert\right]
+\frac{1}{T}\frac{V(t,x)+V(t,y)}{2}\frac{\lVert x-y\rVert}{\sqrt{T}}
,\label{a01}
\end{align}
\begin{align}
\lvert g(x)-g(y)\rvert\leq \frac{V(T,x)+V(T,y)}{2}
\frac{\lVert x-y\rVert}{\sqrt{T}},\label{a19}
\end{align}
\begin{align}
\left\lVert\left\lVert
X^{t,{x}}_{r}-
X^{t,{y}}_{r}\right\rVert
\right\rVert_{\exponentX}\leq 
\frac{V(t,x)+V(t,y)}{2}
\lVert x-y\rVert,\label{a21}
\end{align}
\begin{align}
\left\lVert
\pr_i\! \left(
Z^{t,{x}}_{r} 
-Z^{t,{y}}_{r}\right)\right\rVert_{\exponentZ}\leq
\frac{V(t,x)+V(t,y)}{2} \frac{\lVert x-y\rVert}{\sqrt{T}\Lambda_{i}(r-t)},\label{a39}
\end{align}
\begin{align}
\P \!\left( X_r^{t, X_t^{s,x}} = X^{s,x}_r\right)=1,\quad 
\P\!\left(X_t^{s,(\cdot)} \in A, X^{t,(\cdot)}_{(\cdot)} \in B\right) =\P\!\left (X_t^{s,(\cdot)} \in A\right)\P\!\left(X^{t,(\cdot)}_{(\cdot)} \in B\right),\label{a40}
\end{align}
and $\P(X^{s,x}_s=x)=1$.
Moreover,
let   $u\colon [0,T)\times \R^d\to \R^{d+1}$ be the 
unique measurable function (cf. \cref{a03})
such that for all $t\in [0,T)$, $x\in \R^d$
we have that
\begin{align}\label{b01a}
\max_{\nu\in [0,d]\cap\Z}\sup_{\tau\in [0,T), \xi\in \R^d}
\left[\Lambda_\nu(T-\tau)\frac{\lvert\pr_\nu(u(\tau,\xi))\rvert}{V(\tau ,\xi)}\right]<\infty,
\end{align}
\begin{align}\max_{\nu\in [0,d]\cap\Z}\left[
\E\!\left [\left\lvert g(X^{t,x}_{T} )\pr_\nu(Z^{t,x}_{T})\right\rvert \right] + \int_{t}^{T}
\E \!\left[\left\lvert
f(r,X^{t,x}_{r},u(r,X^{t,x}_{r}))\pr_\nu(Z^{t,x}_{r})\right\rvert\right]dr\right]<\infty,
\end{align}
and
\begin{align}
u(t,x)=\E\!\left [g(X^{t,x}_{T} )Z^{t,x}_{T} \right] + \int_{t}^{T}
\E \!\left[
f(r,X^{t,x}_{r},u(r,X^{t,x}_{r}))Z^{t,x}_{r}\right]dr.\label{a05a}
\end{align}
\end{setting}

\begin{lemma}
[Lipschitz continuity of the fixed point]\label{a37}Assume \cref{z04}.  Then for all 
$x,y\in\R^d$, $t\in [0,T)$ we have that
\begin{align} \begin{split} 
&
\max_{i\in [0,d]\cap\Z}\left[
\Lambda_{i}(T-t)
\left\lvert \pr_i\left(u(t,x)-
u(t,y)\right)\right\rvert\right]\leq e^{c^2T}
\frac{V^6(t,x)+V^6(t,y)}{2}
\frac{\lVert x-y\rVert}{\sqrt{T}}.
\end{split}\label{a59}
\end{align}
\end{lemma}

\begin{proof}
[Proof of \cref{a37}]
\cref{a03} and the assumption of \cref{a37} prove that
for all $t\in [0,T)$ we have that
\begin{align}
\max_{\nu\in [0,d]\cap\Z}\sup_{y\in \R^d}
\left[\Lambda_\nu(T-t)\frac{\lvert\pr_\nu(u(t,y))\rvert}{V(t,y)}\right]
\leq 
6c e^{86c^6T^2(T-t)}.\label{a38}
\end{align}
Next,
H\"older's inequality, the fact that
$\frac{1}{\exponentV}+\frac{1}{\exponentX}+\frac{1}{\exponentZ}\leq 1$, \eqref{a02b}, \eqref{a21}, the fact that
$\forall\, x,y,p,q\in [0,\infty)\colon \frac{x^p+y^p}{2}\frac{x^q+y^q}{2}\leq \frac{x^{p+q}+y^{p+q}}{2}$ imply for all $i\in[0,d]\cap\Z$, 
$\tilde{x}, \tilde{y}\in \R^d$, $t\in [0,T)$, $r\in (t,T]$ that
\begin{align}
&\Lambda_{i}(T-t)
\left\lVert 
\frac{V(r,X^{t,\tilde{x}}_{r} )+
V(r,X^{t,\tilde{y}}_{r} )}{2}
\frac{\left\lVert
X^{t,\tilde{x}}_{r}-
X^{t,\tilde{y}}_{r}
\right\rVert}{\sqrt{T}}
\pr_i(
Z^{t,\tilde{x}}_{r})
\right\rVert_{1} \nonumber\\
&\leq \Lambda_{i}(T-t)
\frac{\left\lVert V(r,X^{t,\tilde{x}}_{r} )\right\rVert_{\exponentV}+\left\lVert
V(r,X^{t,\tilde{y}}_{r} )\right\rVert_{\exponentV}}{2}
\frac{\left\lVert\left\lVert
X^{t,\tilde{x}}_{r}-
X^{t,\tilde{y}}_{r}
\right\rVert\right\rVert_{\exponentX}}{\sqrt{T}}
\left\lVert
\pr_i(
Z^{t,\tilde{x}}_{r})
\right\rVert_{\exponentZ}\nonumber\\
&\leq\Lambda_{i}(T-t) \xeqref{a02b}
\frac{V(t,\tilde{x})+V(t,\tilde{y})}{2}\xeqref{a21}
\frac{V(t,\tilde{x})+V(t,\tilde{y})}{2}
\frac{\lVert\tilde{x}-\tilde{y}\rVert}{\sqrt{T}}\xeqref{a02b}\frac{c}{\Lambda_{i}(r-t)}\nonumber\\
&\leq c
\frac{V^2(t,\tilde{x})+V^2(t,\tilde{y})}{2}
\frac{\lVert\tilde{x}-\tilde{y}\rVert }{\sqrt{T}}
\frac{\sqrt{T-t}}{\sqrt{r-t}}
.\label{a39a}
\end{align}
Hence,  H\"older's inequality, the fact that
$\frac{2}{\exponentV}+\frac{1}{\exponentX}\leq 1$, the fact that $c\leq V$, \eqref{a02b}, \eqref{a21}, and the fact that
$\forall\, x,y,p,q\in [0,\infty)\colon \frac{x^p+y^p}{2}\frac{x^q+y^q}{2}\leq \frac{x^{p+q}+y^{p+q}}{2}$ show for all $i\in[0,d]\cap\Z$, 
${x}, {y}\in \R^d$, 
$s\in [0,T)$,
$t\in [s,T)$, $r\in (0,T]$ that
\begin{align}
&
\Lambda_{i}(T-t)\left\lVert
\left\lVert 
\frac{V(r,X^{t,\tilde{x}}_{r} )+
V(r,X^{t,\tilde{y}}_{r} )}{2}
\frac{\left\lVert
X^{t,\tilde{x}}_{r}-
X^{t,\tilde{y}}_{r}
\right\rVert}{\sqrt{T}}
\pr_i(
Z^{t,\tilde{x}}_{r})
\right\rVert_{1}\Bigr|_{\substack{\tilde{x} =X^{s,x}_t,\tilde{y} =X^{s,y}_t}}\right\rVert_{2}\nonumber\\
&\leq\xeqref{a39a} \left\lVert c
\frac{V^2(t,\tilde{x})+V^2(t,\tilde{y})}{2}
\frac{\lVert\tilde{x}-\tilde{y}\rVert}{\sqrt{T}}\frac{\sqrt{T-t}}{\sqrt{r-t}}\Bigr|_{\substack{\tilde{x} =X^{s,x}_t,\tilde{y} =X^{s,y}_t}}\right\rVert_{2}\nonumber\\
&=
\left\lVert c
\frac{V^2(t,X^{s,x}_t)+V^2(t,X^{s,y}_t)}{2}
\frac{\lVert X^{s,x}_t-X^{s,y}_t\rVert }{\sqrt{T}}
\frac{\sqrt{T-t}}{\sqrt{r-t}}
\right\rVert_{2}\nonumber\\
&\leq c
\frac{\left\lVert V^2(t,X^{s,x}_t)\right\rVert_{\frac{\exponentV}{2}}+\left\lVert V^2(t,X^{s,y}_t)\right\rVert_{\frac{\exponentV}{2}}}{2}
\frac{\left\lVert\left
\lVert X^{s,x}_t-X^{s,y}_t\right\rVert \right\rVert_{\exponentX}}{\sqrt{T}}\frac{\sqrt{T-t}}{\sqrt{r-t}}\nonumber\\
&\leq \frac{V(s,x)+V(s,y)}{2}\xeqref{a02b}\frac{V^2(s,x)+V^2(s,y)}{2}\xeqref{a21}
\frac{V(s,x)+V(s,y)}{2}\frac{\lVert x-y\rVert}{\sqrt{T}}\frac{\sqrt{T-t}}{\sqrt{r-t}}\nonumber\\
&\leq \frac{V^4(s,x)+V^4(s,y)}{2}\frac{\lVert x-y\rVert}{\sqrt{T}}\frac{\sqrt{T-t}}{\sqrt{r-t}}.
\label{a41}\end{align}
This and  \eqref{a19} imply for all $i\in[0,d]\cap\Z$, 
${x}, {y}\in \R^d$, $s\in [0,T)$,
$t\in [s,T)$ that
\begin{align} 
&
\Lambda_{i}(T-t)
\left\lVert
\left\lVert \left(
g(X^{t,\tilde{x}}_{T} )
-g(X^{t,\tilde{y}}_{T} )\right)
\pr_i(
Z^{t,\tilde{x}}_{T})
\right\rVert_{1}\Bigr|_{\substack{\tilde{x} =X^{s,x}_t,\tilde{y} =X^{s,y}_t}}\right\rVert_{2}\nonumber\\
&\leq\xeqref{a19} \Lambda_{i}(T-t)\left\lVert
\left\lVert 
\frac{V(T,X^{t,\tilde{x}}_{T} )+
V(T,X^{t,\tilde{y}}_{T} )}{2}
\frac{\left\lVert
X^{t,\tilde{x}}_{T}-
X^{t,\tilde{y}}_{T}
\right\rVert}{\sqrt{T}}
\pr_i(
Z^{t,\tilde{x}}_{T})
\right\rVert_{1}\Bigr|_{\substack{\tilde{x} =X^{s,x}_t,\tilde{y} =X^{s,y}_t}}\right\rVert_{2}\nonumber\\
&\leq \frac{V^4(s,x)+V^4(s,y)}{2}\frac{\lVert x-y\rVert}{\sqrt{T}}.
\label{a45}
\end{align}
Next,  H\"older's inequality, the fact that
$\frac{1}{\exponentV}+\frac{1}{\exponentZ}\leq 1$, \eqref{a02b},  \eqref{a39}, and the fact that
$\forall\, x,y,p,q\in [0,\infty)\colon \frac{x^p+y^p}{2}\frac{x^q+y^q}{2}\leq \frac{x^{p+q}+y^{p+q}}{2}$  prove for all $i\in[0,d]\cap\Z$, 
$\tilde{x}, \tilde{y}\in \R^d$, $t\in [0,T)$, $r\in (t,T]$  that
\begin{align}
& \Lambda_{i}(T-t)
\left\lVert V(r,X^{t,\tilde{y}}_{r})
\pr_i\! \left(
Z^{t,\tilde{x}}_{r} 
-Z^{t,\tilde{y}}_{r}\right)
\right\rVert_{1}\nonumber\\
&\leq \Lambda_{i}(T-t)\left\lVert V(r,X^{t,\tilde{y}}_{r})\right\rVert_{\exponentV}
\left\lVert
\pr_i\! \left(
Z^{t,\tilde{x}}_{r} 
-Z^{t,\tilde{y}}_{r}\right)
\right\rVert_{\exponentZ}\nonumber\\
&\leq\Lambda_{i}(T-t)\xeqref{a02b} V(t,\tilde{y})
\xeqref{a39}\frac{V(t,\tilde{x})+V(t,\tilde{y})}{2} \frac{\lVert \tilde{x}-\tilde{y}\rVert}{\sqrt{T}\Lambda_{i}(r-t)}\nonumber\\
&\leq 2
\frac{V(t,\tilde{x})+V(t,\tilde{y})}{2}\frac{V(t,\tilde{x})+V(t,\tilde{y})}{2}
\frac{\lVert \tilde{x}-\tilde{y}\rVert}{\sqrt{T}}\frac{\sqrt{T-t}}{\sqrt{r-t}}\nonumber\\
&\leq 2
\frac{V^2(t,\tilde{x})+V^2(t,\tilde{y})}{2}
\frac{\lVert \tilde{x}-\tilde{y}\rVert}{\sqrt{T}}\frac{\sqrt{T-t}}{\sqrt{r-t}}.
\end{align}
Therefore, \eqref{a41} and the fact that $1\leq c$ imply for all $i\in[0,d]\cap\Z$, 
${x}, {y}\in \R^d$, $s\in [0,T)$,
$t\in [s,T)$, $r\in (0,T]$ that
\begin{align}
&
\Lambda_{i}(T-t)\left\lVert
\left\lVert V(r,X^{t,\tilde{y}}_{r})
\pr_i\! \left(
Z^{t,\tilde{x}}_{r} 
-Z^{t,\tilde{y}}_{r}\right)
\right\rVert_{1}
\Bigr|_{\substack{\tilde{x} =X^{s,x}_t,\tilde{y} =X^{s,y}_t}}
\right\rVert_{2}\nonumber\\
&\leq \left\lVert 2
\frac{V^2(t,\tilde{x})+V^2(t,\tilde{y})}{2}
\frac{\lVert \tilde{x}-\tilde{y}\rVert}{\sqrt{T}}\frac{\sqrt{T-t}}{\sqrt{r-t}}\Bigr|_{\substack{\tilde{x} =X^{s,x}_t,\tilde{y} =X^{s,y}_t}}\right\rVert_{2}\nonumber\\
&\leq\xeqref{a41} 2\frac{V^4(s,x)+V^4(s,y)}{2}\frac{\lVert x-y\rVert}{\sqrt{T}}\frac{\sqrt{T-t}}{\sqrt{r-t}}.
\label{a46}
\end{align}
This and \eqref{a04} imply
for all $i\in[0,d]\cap\Z$, 
${x}, {y}\in \R^d$, $s\in [0,T)$, $t\in [s,T)$ that
\begin{align} 
&
\Lambda_{i}(T-t)\left\lVert\left\lVert
g(X^{t,\tilde{y}}_{T})
\pr_i\! \left(
Z^{t,\tilde{x}}_{T} 
-Z^{t,\tilde{y}}_{T}\right)
\right\rVert_{1}\Bigr|_{\substack{\tilde{x} =X^{s,x}_t,\tilde{y} =X^{s,y}_t}}\right\rVert_{2}\nonumber\\
&\leq \xeqref{a04}
\Lambda_{i}(T-t)\left\lVert
\left\lVert V(T,X^{t,\tilde{y}}_{T})
\pr_i\! \left(
Z^{t,\tilde{x}}_{T} 
-Z^{t,\tilde{y}}_{T}\right)
\right\rVert_{1}
\Bigr|_{\substack{\tilde{x} =X^{s,x}_t,\tilde{y} =X^{s,y}_t}}
\right\rVert_{2}\nonumber\\
&\leq\xeqref{a46} 2\frac{V^4(s,x)+V^4(s,y)}{2}\frac{\lVert x-y\rVert}{\sqrt{T}}.
\label{a48}
\end{align}
Next,
  the triangle inequality, H\"older's inequality, the fact that
$\frac{1}{2}+\frac{1}{\exponentZ}\leq 1$, and \eqref{a02b}
 establish for all $i\in[0,d]\cap\Z$, 
$\tilde{x}, \tilde{y}\in \R^d$, $t\in [0,T)$, $r\in (t,T]$  that
\begin{align} 
& \Lambda_{i}(T-t)
\left\lVert
\sum_{\nu=0}^d
L_\nu\Lambda_\nu(T)\left\lvert
\pr_{\nu}\!
\left(
u(r,X^{t,\tilde{x}}_{r})
-
u(r,X^{t,\tilde{y}}_{r})
\right)\pr_i(
Z^{t,\tilde{y}}_{r})\right\rvert\right\rVert_1\nonumber\\
&\leq 
\Lambda_{i}(T-t)
\sum_{\nu=0}^d
\left\lVert L_\nu\Lambda_\nu(T)
\pr_{\nu}
\left(
u(r,X^{t,\tilde{x}}_{r})
-
u(r,X^{t,\tilde{y}}_{r})
\right)\pr_i(
Z^{t,\tilde{y}}_{r})\right\rVert_1\nonumber\\
&\leq \Lambda_{i}(T-t)
\sum_{\nu=0}^d\left[
L_\nu\frac{\sqrt{T}}{\sqrt{T-r}}
\left\lVert
\Lambda_\nu(T-r)
\pr_{\nu}
\left(
u(r,X^{t,\tilde{x}}_{r})
-
u(r,X^{t,\tilde{y}}_{r})
\right)\pr_i(
Z^{t,\tilde{y}}_{r})\right\rVert_1\right]\nonumber\\
&\leq \Lambda_{i}(T-t)
\sum_{\nu=0}^d\left[
L_\nu\frac{\sqrt{T}}{\sqrt{T-r}}
\left\lVert
\Lambda_\nu(T-r)
\pr_{\nu}
\left(
u(r,X^{t,\tilde{x}}_{r})
-
u(r,X^{t,\tilde{y}}_{r})
\right)\right\rVert_2
\left\lVert
\pr_i(
Z^{t,\tilde{y}}_{r})\right\rVert_{\exponentZ}\right]\nonumber\\
&= 
\sum_{\nu=0}^d\left[
L_\nu
\left\lVert
\Lambda_\nu(T-r)
\pr_{\nu}
\left(
u(r,X^{t,\tilde{x}}_{r})
-
u(r,X^{t,\tilde{y}}_{r})
\right)\right\rVert_2
\frac{\sqrt{T}}{\sqrt{T-r}}
\Lambda_{i}(T-t)
\left\lVert
\pr_i(
Z^{t,\tilde{y}}_{r})\right\rVert_{\exponentZ}\right]\nonumber\\
&\leq \xeqref{a02b} 
\sum_{\nu=0}^d\left[
L_\nu
\left\lVert
\Lambda_\nu(T-r)
\pr_{\nu}
\left(
u(r,X^{t,\tilde{x}}_{r})
-
u(r,X^{t,\tilde{y}}_{r})
\right)\right\rVert_2
\frac{\sqrt{T}}{\sqrt{T-r}}
\frac{c\Lambda_{i}(T-t)}{\Lambda_i(r-t)}
\right]\nonumber\\
&\leq 
\sum_{\nu=0}^d\left[
L_\nu
\left\lVert
\Lambda_\nu(T-r)
\pr_{\nu}
\left(
u(r,X^{t,\tilde{x}}_{r})
-
u(r,X^{t,\tilde{y}}_{r})
\right)\right\rVert_2
\frac{\sqrt{T}}{\sqrt{T-r}}
\frac{c\sqrt{T-t}}{\sqrt{r-t}}\right].
\end{align}
This, the triangle inequality, the disintegration theorem,   \eqref{a40}, and the fact that
$\sum_{\nu=0}^dL_\nu\leq c$
prove
for all $i\in[0,d]\cap\Z$, 
${x}, {y}\in \R^d$, $s\in [0,T)$, $t\in [s,T)$, $r\in (t,T]$ that
\begin{align} 
&\Lambda_{i}(T-t)\left\lVert
\left\lVert
\sum_{\nu=0}^d
L_\nu\Lambda_\nu(T)\left\lvert
\pr_{\nu}\!
\left(
u(r,X^{t,\tilde{x}}_{r})
-
u(r,X^{t,\tilde{y}}_{r})
\right)\pr_i(
Z^{t,\tilde{y}}_{r})\right\rvert\right\rVert_1\Bigr|_{\substack{\tilde{x} =X^{s,x}_t,\tilde{y} =X^{s,y}_t}}\right\rVert_2\nonumber\\
&
\leq 
\left\lVert
\sum_{\nu=0}^d\left[
L_\nu
\left\lVert
\Lambda_\nu(T-r)
\pr_{\nu}
\left(
u(r,X^{t,\tilde{x}}_{r})
-
u(r,X^{t,\tilde{y}}_{r})
\right)\right\rVert_2
\frac{\sqrt{T}}{\sqrt{T-r}}
\frac{c\sqrt{T-t}}{\sqrt{r-t}}\right]
\Bigr|_{\substack{\tilde{x} =X^{s,x}_t,\tilde{y} =X^{s,y}_t}}\right\rVert_2\nonumber\\
&\leq
\sum_{\nu=0}^d\left[
L_\nu\left\lVert
\left\lVert
\Lambda_\nu(T-r)
\pr_{\nu}
\left(
u(r,X^{t,\tilde{x}}_{r})
-
u(r,X^{t,\tilde{y}}_{r})
\right)\right\rVert_2
\Bigr|_{\substack{\tilde{x} =X^{s,x}_t,\tilde{y} =X^{s,y}_t}}\right\rVert_2
\frac{cT}{\sqrt{(T-r)(r-t)}}\right]\nonumber\\
&\leq \left[
\max_{\nu\in [0,d]\cap\Z}
\left\lVert
\Lambda_\nu(T-r)
\pr_{\nu}
\left(
u(r,X^{s,{x}}_{r})
-
u(r,X^{s,{y}}_{r})
\right)\right\rVert_2
\right]
\frac{c^2T}{\sqrt{(T-r)(r-t)}}.
\label{a47}
\end{align}
Hence, the triangle inequality, \eqref{a01}, and \eqref{a41}
imply
for all ${i}\in[0,d]\cap\Z$, 
${x}, {y}\in \R^d$, $s\in [0,T)$, $t\in [s,T)$, $r\in (t,T]$ that
\begin{align} 
&\Lambda_{i}(T-t)
\left\lVert
\left\lVert\left(
f(r,X^{t,\tilde{x}}_{r}, u(r,X^{t,\tilde{x}}_{r}))
-
f(r,X^{t,\tilde{y}}_{r},u(r,X^{t,\tilde{y}}_{r}))
\right)
\pr_i(Z^{t,\tilde{x}}_{r} )
\right\rVert_{1}\Bigr|_{\substack{\tilde{x} =X^{s,x}_t,\tilde{y} =X^{s,y}_t}}\right\rVert_{2}\nonumber
\\
&\leq \xeqref{a01}\Lambda_{i}(T-t)\left\lVert
\left\lVert
\sum_{\nu=0}^d
L_\nu\Lambda_\nu(T)\left\lvert
\pr_{\nu}\!
\left(
u(r,X^{t,\tilde{x}}_{r})
-
u(r,X^{t,\tilde{y}}_{r})
\right)\pr_i(
Z^{t,\tilde{y}}_{r})\right\rvert\right\rVert_1\Bigr|_{\substack{\tilde{x} =X^{s,x}_t,\tilde{y} =X^{s,y}_t}}\right\rVert_2\nonumber\\
&\quad+
\Lambda_{i}(T-t)\frac{1}{T}\left\lVert
\left\lVert 
\frac{V(r,X^{t,\tilde{x}}_{r} )+
V(r,X^{t,\tilde{y}}_{r} )}{2}
\frac{\left\lVert
X^{t,\tilde{x}}_{r}-
X^{t,\tilde{y}}_{r}
\right\rVert}{\sqrt{T}}
\pr_i(
Z^{t,\tilde{x}}_{r})
\right\rVert_{1}\Bigr|_{\substack{\tilde{x} =X^{s,x}_t,\tilde{y} =X^{s,y}_t}}\right\rVert_{2}\nonumber\\
&\leq \xeqref{a47}\max_{\nu\in [0,d]\cap\Z}
\left\lVert
\Lambda_\nu(T-r)
\pr_{\nu}
\left(
u(r,X^{s,{x}}_{r})
-
u(r,X^{s,{y}}_{r})
\right)\right\rVert_2
\frac{c^2T}{\sqrt{(T-r)(r-t)}}\nonumber\\
&\quad +\xeqref{a41}\frac{1}{T}
\frac{V^4(s,x)+V^4(s,y)}{2}\frac{\lVert x-y\rVert}{\sqrt{T}}\frac{\sqrt{T-t}}{\sqrt{r-t}}.
\end{align}
This and the fact that
$\forall\,t\in [0,T)\colon \int_{t}^{T}\frac{dr}{\sqrt{r-t}}=2\sqrt{r-t}|_{r=t}^T=2\sqrt{T-t}$  prove
for all $i\in[0,d]\cap\Z$, 
${x}, {y}\in \R^d$, $s\in [0,T)$, $t\in [s,T)$ that
\begin{align}
&
\int_{t}^{T}
\max_{i\in [0,d]\cap\Z} \left[\Lambda_{i}(T-t)
\left\lVert
\left\lVert\left(
f(r,X^{t,\tilde{x}}_{r}, u(r,X^{t,\tilde{x}}_{r}))
-
f(r,X^{t,\tilde{y}}_{r},u(r,X^{t,\tilde{y}}_{r}))
\right)
\pr_i(Z^{t,\tilde{x}}_{r} )
\right\rVert_{1}\Bigr|_{\substack{\tilde{x} =X^{s,x}_t,\tilde{y} =X^{s,y}_t}}\right\rVert_{2}\right]
dr\nonumber\\
& 
\leq \int_{t}^{T}\max_{\nu\in [0,d]\cap\Z}
\left\lVert
\Lambda_\nu(T-r)
\pr_{\nu}\!
\left(
u(r,X^{s,{x}}_{r})
-
u(r,X^{s,{y}}_{r})
\right)\right\rVert_2
\frac{c^2T}{\sqrt{(T-r)(r-t)}}\,dr\nonumber\\
&\quad +\frac{1}{T}
\frac{V^4(s,x)+V^4(s,y)}{2}\frac{\lVert x-y\rVert}{\sqrt{T}}\int_{t}^{T}\frac{\sqrt{T-t}}{\sqrt{r-t}}\,dr\nonumber\\
&\leq \int_{t}^{T}\max_{\nu\in [0,d]\cap\Z}
\left\lVert
\Lambda_\nu(T-r)
\pr_{\nu}\!
\left(
u(r,X^{s,{x}}_{r})
-
u(r,X^{s,{y}}_{r})
\right)\right\rVert_2
\frac{c^2T}{\sqrt{(T-r)(r-t)}}\,dr\nonumber\\
&\quad +2
\frac{V^4(s,x)+V^4(s,y)}{2}\frac{\lVert x-y\rVert}{\sqrt{T}}.\label{a52}
\end{align}
Next, \eqref{a04}, \eqref{a46}, 
and the fact that
$\forall\,t\in [0,T)\colon \int_{t}^{T}\frac{dr}{\sqrt{r-t}}=2\sqrt{r-t}|_{r=t}^T=2\sqrt{T-t}$ show
for all $i\in[0,d]\cap\Z$, 
${x}, {y}\in \R^d$, $s\in [0,T)$, $t\in [s,T)$ that
\begin{align} 
&\int_{t}^{T}
\Lambda_{i}(T-t)
\left\lVert
\left\lVert
f(r,X^{t,\tilde{y}}_{r},0)
\pr_i\!\left(Z^{t,\tilde{x}}_{r} -Z^{t,\tilde{y}}_{r} \right)
\right\rVert_{1}\Bigr|_{\substack{\tilde{x} =X^{s,x}_t,\tilde{y} =X^{s,y}_t}}\right\rVert_{2}dr\nonumber\\
&\leq \xeqref{a04}\int_{t}^{T}
\Lambda_{i}(T-t)\frac{1}{T}
\left\lVert
\left\lVert
V(r,X^{t,\tilde{y}}_{r})
\pr_i\!\left(Z^{t,\tilde{x}}_{r} -Z^{t,\tilde{y}}_{r} \right)
\right\rVert_{1}\Bigr|_{\substack{\tilde{x} =X^{s,x}_t,\tilde{y} =X^{s,y}_t}}\right\rVert_{2}dr\nonumber\\
&\leq \int_{t}^{T}\frac{1}{T}\xeqref{a46}\cdot 2\frac{V^4(s,x)+V^4(s,y)}{2}\frac{\lVert x-y\rVert}{\sqrt{T}}\frac{\sqrt{T-t}}{\sqrt{r-t}}\,dr\nonumber\\
&\leq 4\frac{V^4(s,x)+V^4(s,y)}{2}\frac{\lVert x-y\rVert}{\sqrt{T}}.\label{a53}
\end{align}
Moreover, \eqref{a38}, the triangle inequality, the fact that
$\sum_{\nu=0}^{d}L_\nu\leq c$,  \eqref{a46},  \cref{a15},  the fact that $48e^{86c^6T^3}\leq V$, and the fact that $\forall\, x,y,p,q\in [0,\infty)\colon \frac{x^p+y^p}{2}\frac{x^q+y^q}{2}\leq \frac{x^{p+q}+y^{p+q}}{2}$  imply
for all $i\in[0,d]\cap\Z$, 
${x}, {y}\in \R^d$, $s\in [0,T)$, $t\in [s,T)$ that
\begin{align}
&\int_{t}^{T}
\Lambda_{i}(T-t)
\sum_{\nu=0}^{d}\left[L_\nu 
\left\lVert
\left\lVert 
\Lambda_\nu(T)
\pr_\nu (u(r,X^{t,\tilde{y}}_{r}))
\pr_i\!\left(Z^{t,\tilde{x}}_{r} -Z^{t,\tilde{y}}_{r} \right)
\right\rVert_{1}\Bigr|_{\substack{\tilde{x} =X^{s,x}_t,\tilde{y} =X^{s,y}_t}}\right\rVert_{2}\right]dr\nonumber\\
&\leq \int_{t}^{T}
\Lambda_{i}(T-t)
\sum_{\nu=0}^{d}\left[L_\nu 
\left\lVert
\left\lVert 
\frac{\sqrt{T}}{\sqrt{T-r}}
\Lambda_\nu(T-r)
\pr_\nu (u(r,X^{t,\tilde{y}}_{r}))
\pr_i\!\left(Z^{t,\tilde{x}}_{r} -Z^{t,\tilde{y}}_{r} \right)
\right\rVert_{1}\Bigr|_{\substack{\tilde{x} =X^{s,x}_t,\tilde{y} =X^{s,y}_t}}\right\rVert_{2}\right]dr\nonumber\\
&\leq \int_{t}^{T}
\Lambda_{i}(T-t)
\sum_{\nu=0}^{d}\left[L_\nu 
\left\lVert
\left\lVert 
\frac{\sqrt{T}}{\sqrt{T-r}}
6c e^{86c^6T^3}V(r,X^{t,\tilde{y}}_{r} )
\pr_i\!\left(Z^{t,\tilde{x}}_{r} -Z^{t,\tilde{y}}_{r} \right)
\right\rVert_{1}\Bigr|_{\substack{\tilde{x} =X^{s,x}_t,\tilde{y} =X^{s,y}_t}}\right\rVert_{2}\right]dr\nonumber\\
&\leq \int_{t}^{T}
\Lambda_{i}(T-t)
\left\lVert
\left\lVert V(r,X^{t,\tilde{y}}_{r} )
\pr_i\!\left(Z^{t,\tilde{x}}_{r} -Z^{t,\tilde{y}}_{r} \right)
\right\rVert_{1}\Bigr|_{\substack{\tilde{x} =X^{s,x}_t,\tilde{y} =X^{s,y}_t}}\right\rVert_{2}
\frac{\sqrt{T}}{\sqrt{T-r}}
6c^2 e^{86c^6T^3}\,dr\nonumber\\
&\leq \int_{t}^{T}\xeqref{a46}
2\frac{V^4(s,x)+V^4(s,y)}{2}\frac{\lVert x-y\rVert}{\sqrt{T}}\frac{\sqrt{T-t}}{\sqrt{r-t}}
\frac{\sqrt{T}}{\sqrt{T-r}}
6c^2 e^{86c^6T^3}dr\nonumber\\
&\leq 2\frac{V^4(s,x)+V^4(s,y)}{2}\frac{\lVert x-y\rVert}{\sqrt{T}}\int_{t}^{T}
\frac{6c^2 Te^{86c^6T^3}}{\sqrt{(T-r)(r-t)}}dr\nonumber\\
&\leq 2\frac{V^4(s,x)+V^4(s,y)}{2}\frac{\lVert x-y\rVert}{\sqrt{T}}6c^2 Te^{86c^6T^3}\cdot 4\nonumber\\
&\leq \frac{V^4(s,x)+V^4(s,y)}{2}\frac{V(s,x)+V(s,y)}{2}
\frac{\lVert x-y\rVert}{\sqrt{T}}c^2T\nonumber\\
&\leq \frac{V^5(s,x)+V^5(s,y)}{2}
\frac{\lVert x-y\rVert}{\sqrt{T}}c^2T.
\label{a54}\end{align}
This, the triangle inequality,  \eqref{a53}, and the fact that
$1+c^2T\leq e^{c^2T}$ prove 
for all $i\in[0,d]\cap\Z$, 
${x}, {y}\in \R^d$, $s\in [0,T)$, $t\in [s,T)$ that
\begin{align} 
&\int_{t}^{T}
\Lambda_{i}(T-t)
\left\lVert
\left\lVert
f(r,X^{t,\tilde{y}}_{r},u(r,X^{t,\tilde{y}}_{r}))
\pr_i\!\left(Z^{t,\tilde{x}}_{r} -Z^{t,\tilde{y}}_{r} \right)
\right\rVert_{1}\Bigr|_{\substack{\tilde{x} =X^{s,x}_t,\tilde{y} =X^{s,y}_t}}\right\rVert_{2}dr\nonumber\\
&\leq \int_{t}^{T}
\Lambda_{i}(T-t)
\left\lVert
\left\lVert
f(r,X^{t,\tilde{y}}_{r},0)
\pr_i\!\left(Z^{t,\tilde{x}}_{r} -Z^{t,\tilde{y}}_{r} \right)
\right\rVert_{1}\Bigr|_{\substack{\tilde{x} =X^{s,x}_t,\tilde{y} =X^{s,y}_t}}\right\rVert_{2}dr\nonumber\\
&\quad +\int_{t}^{T}
\Lambda_{i}(T-t)
\sum_{\nu=0}^{d}\left[L_\nu 
\left\lVert
\left\lVert 
\Lambda_\nu(T)
\pr_\nu (u(r,X^{t,\tilde{y}}_{r}))
\pr_i\!\left(Z^{t,\tilde{x}}_{r} -Z^{t,\tilde{y}}_{r} \right)
\right\rVert_{1}\Bigr|_{\substack{\tilde{x} =X^{s,x}_t,\tilde{y} =X^{s,y}_t}}\right\rVert_{2}\right]dr\nonumber\\
&\leq \xeqref{a53}4\frac{V^4(s,x)+V^4(s,y)}{2}\frac{\lVert x-y\rVert}{\sqrt{T}}+\xeqref{a54}\frac{V^5(s,x)+V^5(s,y)}{2}
\frac{\lVert x-y\rVert}{\sqrt{T}}c^2T\nonumber\\
&\leq 4\frac{V^5(s,x)+V^5(s,y)}{2}
\frac{\lVert x-y\rVert}{\sqrt{T}}e^{c^2T}.
\label{a55}
\end{align}
Next, \eqref{a05a} and the triangle inequality show for all  ${x}, {y}\in \R^d$, $s\in [0,T)$, $t\in [s,T)$ that
{\small
\begin{align}
&
\max_{i\in [0,d]\cap\Z}\left[
\Lambda_{i}(T-t)
\left\lVert \pr_i\left(u(t,X^{s,x}_t)-
u(t,X^{s,y}_t)\right)\right\rVert_{2}\right]\nonumber\\
&
=
\max_{i\in [0,d]\cap\Z}\left[
\Lambda_{i}(T-t)
\left\lVert \pr_i(u(t,\tilde{x})-
u(t,\tilde{y}))|_{\substack{\tilde{x} =X^{s,x}_t,\tilde{y} =X^{s,y}_t}}
\right\rVert_{2}\right]\nonumber
\\
&=
\max_{i\in [0,d]\cap\Z}\Biggl[\Lambda_{i}(T-t)
\Biggl\lVert\biggl[\E\!\left [\pr_i\! \left(g(X^{t,\tilde{x}}_{T} )Z^{t,\tilde{x}}_{T} 
-g(X^{t,\tilde{y}}_{T} )Z^{t,\tilde{y}}_{T}
\right)
\right]\nonumber\\&\quad\quad   + \int_{t}^{T}
\E \!\left[\pr_i\left(
f(r,X^{t,\tilde{x}}_{r},u(r,X^{t,\tilde{x}}_{r}))Z^{t,\tilde{x}}_{r}
-
f(r,X^{t,\tilde{y}}_{r},u(r,X^{t,\tilde{y}}_{r}))Z^{t,\tilde{y}}_{r}\right)
\right]dr\biggr]\biggr|_{\substack{\tilde{x} =X^{s,x}_t,\tilde{y} =X^{s,y}_t}}
\Biggr \rVert_{2}\Biggr]\nonumber\\
&\leq \max_{i\in [0,d]\cap\Z}\left[\Lambda_{i}(T-t)
\left\lVert
\left\lVert\pr_i\! \left(
g(X^{t,\tilde{x}}_{T} )Z^{t,\tilde{x}}_{T} 
-g(X^{t,\tilde{y}}_{T} )Z^{t,\tilde{y}}_{T}\right)
\right\rVert_{1}\Bigr|_{\substack{\tilde{x} =X^{s,x}_t,\tilde{y} =X^{s,y}_t}}\right\rVert_{2}\right]\nonumber\\
&\quad 
+
\max_{i\in [0,d]\cap\Z}\int_{t}^{T}\left[\Lambda_{i}(T-t)
\left\lVert
\left\lVert\pr_i\! \left(
f(r,X^{t,\tilde{x}}_{r}, u(r,X^{t,\tilde{x}}_{r}))Z^{t,\tilde{x}}_{r}
-
f(r,X^{t,\tilde{y}}_{r},u(r,X^{t,\tilde{y}}_{r}))Z^{t,\tilde{y}}_{r}\right)\right\rVert_{1}\Bigr|_{\substack{\tilde{x} =X^{s,x}_t,\tilde{y} =X^{s,y}_t}}\right\rVert_{2}\right]
dr\nonumber\\
&\leq \max_{i\in [0,d]\cap\Z}\left[\Lambda_{i}(T-t)
\left\lVert
\left\lVert \left(
g(X^{t,\tilde{x}}_{T} )
-g(X^{t,\tilde{y}}_{T} )\right)
\pr_i(
Z^{t,\tilde{x}}_{T})
\right\rVert_{1}\Bigr|_{\substack{\tilde{x} =X^{s,x}_t,\tilde{y} =X^{s,y}_t}}\right\rVert_{2}\right]\nonumber\\
&\quad 
+
\max_{i\in [0,d]\cap\Z}\left[\Lambda_{i}(T-t)
\left\lVert
\left\lVert
g(X^{t,\tilde{y}}_{T})
\pr_i\! \left(
Z^{t,\tilde{x}}_{T} 
-Z^{t,\tilde{y}}_{T}\right)
\right\rVert_{1}\Bigr|_{\substack{\tilde{x} =X^{s,x}_t,\tilde{y} =X^{s,y}_t}}\right\rVert_{2}\right]\nonumber\\
&\quad +
\max_{i\in [0,d]\cap\Z}\int_{t}^{T}\left[\Lambda_{i}(T-t)
\left\lVert
\left\lVert\left(
f(r,X^{t,\tilde{x}}_{r}, u(r,X^{t,\tilde{x}}_{r}))
-
f(r,X^{t,\tilde{y}}_{r},u(r,X^{t,\tilde{y}}_{r}))
\right)
\pr_i(Z^{t,\tilde{x}}_{r} )
\right\rVert_{1}\Bigr|_{\substack{\tilde{x} =X^{s,x}_t,\tilde{y} =X^{s,y}_t}}\right\rVert_{2}\right]
dr\nonumber\\
&\quad +
\max_{i\in [0,d]\cap\Z}\int_{t}^{T}\left[\Lambda_{i}(T-t)
\left\lVert
\left\lVert
f(r,X^{t,\tilde{y}}_{r},u(r,X^{t,\tilde{y}}_{r}))
\pr_i\!\left(Z^{t,\tilde{x}}_{r} -Z^{t,\tilde{y}}_{r} \right)
\right\rVert_{1}\Bigr|_{\substack{\tilde{x} =X^{s,x}_t,\tilde{y} =X^{s,y}_t}}\right\rVert_{2}\right]
dr.
\end{align}}%
This, \eqref{a45}, \eqref{a48}, \eqref{a52},  \eqref{a55}, and the fact that $1\leq V$ prove for all  ${x}, {y}\in \R^d$, $s\in [0,T)$, $t\in [s,T)$ that
\begin{align} 
&
\max_{i\in [0,d]\cap\Z}\left[
\Lambda_{i}(T-t)
\left\lVert \pr_i\left(u(t,X^{s,x}_t)-
u(t,X^{s,y}_t)\right)\right\rVert_{2}\right]\nonumber\\
&\leq\xeqref{a45}\frac{V^4(s,x)+V^4(s,y)}{2}\frac{\lVert x-y\rVert}{\sqrt{T}}+\xeqref{a48} 2\frac{V^4(s,x)+V^4(s,y)}{2}\frac{\lVert x-y\rVert}{\sqrt{T}}\nonumber\\
&\quad \xeqref{a52}+
\int_{t}^{T}\max_{\nu\in [0,d]\cap\Z}
\left\lVert
\Lambda_\nu(T-r)
\pr_{\nu}
\left(
u(r,X^{s,{x}}_{r})
-
u(r,X^{s,{y}}_{r})
\right)\right\rVert_2
\frac{c^2T}{\sqrt{(T-r)(r-t)}}\,dr\nonumber\\
&\quad +2
\frac{V^4(s,x)+V^4(s,y)}{2}\frac{\lVert x-y\rVert}{\sqrt{T}}
+\xeqref{a55}4\frac{V^5(s,x)+V^5(s,y)}{2}
\frac{\lVert x-y\rVert}{\sqrt{T}}e^{c^2T}\nonumber\\
&\leq 9e^{c^2T}\frac{V^5(s,x)+V^5(s,y)}{2}\frac{\lVert x-y\rVert}{\sqrt{T}}\nonumber\\&\quad +\int_{t}^{T}\max_{\nu\in [0,d]\cap\Z}
\left\lVert
\Lambda_\nu(T-r)
\pr_{\nu}
\left(
u(r,X^{s,{x}}_{r})
-
u(r,X^{s,{y}}_{r})
\right)\right\rVert_2
\frac{c^2T}{\sqrt{(T-r)(r-t)}}\,dr.
\label{a57}
\end{align}
Moreover, \eqref{a38}, Jensen's inequality, the fact that $2\leq \exponentV$, and \eqref{a02b} imply
for all  ${x}\in \R^d$, $s\in [0,T)$, $t\in [s,T)$ that
\begin{align}
\max_{i\in [0,d]\cap\Z}\left[
\Lambda_{i}(T-t)
\left\lVert \pr_i\left(u(t,X^{s,x}_t)\right)\right\rVert_{2}\right]
\leq 6c e^{86c^6T^3}\left\lVert V(t,X^{s,x}_t)\right\rVert_{\exponentV}\leq 6c e^{86c^6T^3}V(s,x).
\end{align}
Therefore, the Gr\"onwall-type inequality (see \cref{a16}),  \eqref{a57}, the fact that $48e^{86c^6T^3}\leq V$, and the fact that
$\forall\, x,y,p,q\in [0,\infty)\colon \frac{x^p+y^p}{2}\frac{x^q+y^q}{2}\leq \frac{x^{p+q}+y^{p+q}}{2}$
  show for all  ${x}, {y}\in \R^d$, $s\in [0,T)$, $t\in [s,T)$ that 
\begin{align}
&
\max_{i\in [0,d]\cap\Z}\left[
\Lambda_{i}(T-t)
\left\lVert \pr_i\left(u(t,X^{s,x}_t)-
u(t,X^{s,y}_t)\right)\right\rVert_{2}\right]\nonumber\\
&
\leq 18e^{c^2T}\frac{V^5(s,x)+V^5(s,y)}{2}\frac{\lVert x-y\rVert}{\sqrt{T}}e^{86c^6T^3}\nonumber\\
&\leq e^{c^2T}
\frac{V^5(s,x)+V^5(s,y)}{2}
\frac{V(s,x)+V(s,y)}{2}
\frac{\lVert x-y\rVert}{\sqrt{T}}\nonumber\\
&\leq e^{c^2T}
\frac{V^6(s,x)+V^6(s,y)}{2}
\frac{\lVert x-y\rVert}{\sqrt{T}}.
\end{align}
This and the fact that
$\forall\,s\in [0,T),x \in \R^d\colon \P(X^{s,x}_s=x)=1$
 imply \eqref{a59}. The proof of \cref{a37} is thus completed.
\end{proof}
\subsection{Temporal regularity of the fixed point}
After establishing the spatial Lipschitz continuity
we establish the temporal regularity for solutions to SFPEs, see 
\cref{b37} below. 
Together with the spatial Lipschitz continuity the temporal regularity implies that the fixed point is continuous. Later, continuity of the fixed point is  an important property to ensure that the fixed point is the unique viscosity solution to the corresponding semilinear parabolic PDE.
\begin{lemma}
[Temporal regularity of the fixed point]\label{b37}
Assume \cref{z04}. Suppose that $
 \frac{6}{\exponentV}+\frac{1}{\exponentX}+\frac{1}{\exponentZ}\leq 1$ 
and $\frac{1}{2}+\frac{1}{\exponentZ}\leq 1$. Assume
for all 
$i\in [0,d]\cap\Z$,
$s\in [0,T)$,
$t\in [s,T)$, $r\in (t,T]$,
 $x\in \R^d$, $w_1,w_2\in \R^{d+1}$
 that
\begin{align}
\left\lVert\left\lVert
X^{t,{x}}_{r}-
X^{t,{y}}_{r}\right\rVert
\right\rVert_{\exponentX}\leq 
c
\lVert x-y\rVert,\label{b21}
\end{align}
\begin{align}\label{c74}
\left\lVert
\pr_i(Z^{t,x}_r-Z^{s,x}_r)\right\rVert_{\exponentZ}\leq \frac{V(t,x)+V(s,x)}{2}\frac{\sqrt{t-s}}{\sqrt{r-t}\Lambda_i(r-s)},
\end{align}
\begin{align}\label{c73}
\left\lVert
\left\lVert
X^{s,x}_t-x
\right\rVert\right\rVert_{\exponentX}\leq V(s,x)\sqrt{t-s}.
\end{align}
Moreover, let 
 $u\colon [0,T)\times \R^d\to \R^{d+1}$ be
the unique measurable function (cf. \cref{a03})
such that for all $t\in [0,T)$, $x\in \R^d$
we have that
\begin{align}\label{c01a}
\max_{\nu\in [0,d]\cap\Z}\sup_{\tau\in [0,T), \xi\in \R^d}
\left[\Lambda_\nu(T-\tau)\frac{\lvert\pr_\nu(u(\tau,\xi))\rvert}{V(\tau ,\xi)}\right]<\infty,
\end{align}
\begin{align}\max_{\nu\in [0,d]\cap\Z}\left[
\E\!\left [\left\lvert g(X^{t,x}_{T} )\pr_\nu(Z^{t,x}_{T})\right\rvert \right] + \int_{t}^{T}
\E \!\left[\left\lvert
f(r,X^{t,x}_{r},u(r,X^{t,x}_{r}))\pr_\nu(Z^{t,x}_{r})\right\rvert\right]dr\right]<\infty,
\end{align}
and
\begin{align}
u(t,x)=\E\!\left [g(X^{t,x}_{T} )Z^{t,x}_{T} \right] + \int_{t}^{T}
\E \!\left[
f(r,X^{t,x}_{r},u(r,X^{t,x}_{r}))Z^{t,x}_{r}\right]dr.\label{b05a}
\end{align}
Then $u$ is continuous.
\end{lemma}
\begin{proof}
[Proof of \cref{b37}]
First, 
\cref{a03,a37} and the assumptions of \cref{b37} prove that
for all $t\in [0,T)$ we have that
\begin{align}\label{b38}
\max_{\nu\in [0,d]\cap\Z}\sup_{y\in \R^d}
\left[\Lambda_\nu(T-t)\frac{\lvert\pr_\nu(u(t,y))\rvert}{V(t,y)}\right]
\leq 
6c e^{86c^6T^2(T-t)}
\end{align} and that
 for all 
$x,y\in\R^d$, $t\in [0,T)$ we have that
\begin{align} \label{b59}\begin{split} 
&
\max_{i\in [0,d]\cap\Z}\left[
\Lambda_{i}(T-t)
\left\lvert \pr_i\left(u(t,x)-
u(t,y)\right)\right\rvert\right]\leq e^{c^2T}
\frac{V^6(t,x)+V^6(t,y)}{2}
\frac{\lVert x-y\rVert}{\sqrt{T}}.
\end{split}
\end{align}
Next,
\eqref{a40}, the disintegration theorem,
\eqref{b21}, the fact that $c\leq V$, \eqref{c73}, and the fact that $\forall\, x,y,p,q\in [0,\infty)\colon \frac{x^p+y^p}{2}\frac{x^q+y^q}{2}\leq \frac{x^{p+q}+y^{p+q}}{2}$
prove for all $t_1\in [0,T)$, $t_2\in [t_1,T)$, $r\in (t_2, T]$, $x\in \R^d$ that
\begin{align}
\left\lVert\left\lVert
X^{t_1,x}_r-X^{t_2,x}_r\right\rVert
\right\rVert_{\exponentX}
&=\xeqref{a40}
\left\lVert\left\lVert
X^{t_2,X^{t_1,x}_{t_2}}_r-X^{t_2,x}_r\right\rVert
\right\rVert_{\exponentX}\nonumber\\
&
=\left\lVert
\left\lVert\left\lVert
X^{t_2,y}_r-X^{t_2,x}_r\right\rVert
\right\rVert_{\exponentX}\Bigr|_{y=X^{t_1,x}_{t_2}}\right\rVert_{\exponentX}\nonumber\\
&
\leq \xeqref{b21}\left\lVert c
\left\lVert y-x\right\rVert
\Bigr|_{y=X^{t_1,x}_{t_2}}\right\rVert_{\exponentX}\nonumber\\
&=
\left\lVert c
\left\lVert X^{t_1,x}_{t_2}-x\right\rVert
\right\rVert_{\exponentX}\nonumber\\
&\leq\xeqref{c73} 2 \frac{V(t_1,x)+V(t_2,x)}{2}\frac{V(t_1,x)+V(t_2,x)}{2}\sqrt{t_2-t_1}\nonumber\\
&\leq 2 \frac{V^2(t_1,x)+V^2(t_2,x)}{2}\sqrt{t_2-t_1}.\label{b79}
\end{align}
Hence,
H\"older's inequality, the fact that
$\frac{1}{\exponentV}+\frac{1}{\exponentX}+\frac{1}{\exponentZ}\leq 1$,
\eqref{a02b},
the fact that $c\leq V$, and the fact that $\forall\, x,y,p,q\in [0,\infty)\colon \frac{x^p+y^p}{2}\frac{x^q+y^q}{2}\leq \frac{x^{p+q}+y^{p+q}}{2}$
imply for all 
$i\in [0,d]\cap\Z$,
$t_1\in [0,T)$, $t_2\in [t_1,T)$, $r\in (t_2, T]$, $x\in \R^d$ that
\begin{align}
& 
\Lambda_i (T-t_1) \left\lVert
\frac{V(r,X^{t_1,x}_r)+V(r,X^{t_2,x}_r)}{2}
\frac{\left\lVert X^{t_1,x}_r-X^{t_2,x}_r
\right\rVert}{\sqrt{T}}
 \pr_i(Z^{t_1,x}_r)\right\rVert_{1}\nonumber\\
&\leq 
\Lambda_i (T-t_1)
\frac{\left\lVert V(r,X^{t_1,x}_r)\right\rVert_{\exponentV}+\left \lVert V(r,X^{t_2,x}_r)\right\rVert_{\exponentV}}{2}
\frac{\left\lVert\left\lVert X^{t_1,x}_r-X^{t_2,x}_r
\right\rVert\right\rVert_{\exponentX}}{\sqrt{T}}
\left\lVert\pr_i(Z^{t_1,x}_r)\right\rVert_{\exponentZ}\nonumber\\
&\leq 
\Lambda_i (T-t_1)\xeqref{a02b}\frac{V(t_1,x)+V(t_2,x)}{2}\xeqref{b79}
 2 \frac{V^2(t_1,x)+V^2(t_2,x)}{2}\frac{\sqrt{t_2-t_1}}{\sqrt{T}}\xeqref{a02b}\frac{c}{\Lambda_i(r-t_1)}\nonumber\\
&\leq 2\frac{V^4(t_1,x)+V^4(t_2,x)}{2}\frac{\sqrt{t_2-t_1}}{\sqrt{T}}
\frac{\sqrt{T-t_1}}{\sqrt{r-t_1}}.\label{c81}
\end{align}
This and \eqref{a19} prove for all $i\in [0,d]\cap\Z$, $t_1\in [0,T)$, $t_2\in [t_1,T)$,  $x\in \R^d$ that
\begin{align}
&\Lambda_i (T-t_1) \left\lVert(g(X^{t_1,x}_T)-g(X^{t_2,x}_T)) \pr_i(Z^{t_1,x}_T)\right\rVert_{1}\nonumber\\
&
\leq 
\Lambda_i (T-t_1) \left\lVert\xeqref{a19}
\frac{V(T,X^{t_1,x}_T)+V(T,X^{t_1,x}_T)}{2}
\frac{\left\lVert X^{t_1,x}_T-X^{t_2,x}_T
\right\rVert}{\sqrt{T}}
 \pr_i(Z^{t_1,x}_T)\right\rVert_{1}\nonumber\\
&\leq\xeqref{c81} 2\frac{V^4(t_1,x)+V^4(t_2,x)}{2}\frac{\sqrt{t_2-t_1}}{\sqrt{T}}. \label{c82}
\end{align}
Next, H\"older's inequality, the fact that $\frac{1}{\exponentV}+\frac{1}{\exponentZ}\leq 1$, \eqref{a02b}, and \eqref{c74}
show for all 
$i\in [0,d]\cap\Z$,
$t_1\in [0,T)$, $t_2\in [t_1,T)$, $r\in (t_2, T]$, $x\in \R^d$ that
\begin{align}
&
 \Lambda_i(T-t_1) \left\lVert V(r,X^{t_2,x}_r)\pr_i(Z^{t_1,x}_r-Z^{t_2,x}_r)\right\rVert_1\nonumber
\\
&\leq \Lambda_i(T-t_1) \left\lVert
V(r,X^{t_2,x}_r)\right\rVert_{\exponentV}
\left\lVert\pr_i( Z^{t_1,x}_r-Z^{t_2,x}_r)\right\rVert_{\exponentZ}\nonumber\\
&\leq  \Lambda_i(T-t_1)\xeqref{a02b}2\frac{V(t_1,x)+V(t_2,x)}{2}
\xeqref{c74}
\frac{V(t_1,x)+V(t_2,x)}{2}\frac{\sqrt{t_2-t_1}}{\sqrt{r-t_2}\Lambda_i(r-t_1)}\nonumber\\
&\leq 2\frac{V^2(t_1,x)+V^2(t_2,x)}{2}\frac{\sqrt{t_2-t_1}}{\sqrt{r-t_2}}
\frac{\sqrt{T-t_1}}{\sqrt{r-t_1}}.\label{c83}
\end{align}
This and \eqref{a04} prove for all $i\in [0,d]\cap\Z$,
$t_1\in [0,T)$, $t_2\in [t_1,T)$,  $x\in \R^d$ that
\begin{align}
&
\Lambda_i(T-t_1) \left\lVert g(X^{t_2,x}_T)\pr_i(Z^{t_1,x}_T-Z^{t_2,x}_T)\right\rVert_1\nonumber\\
&
\leq \Lambda_i(T-t_1) \left\lVert\xeqref{a04} V(T,X^{t_2,x}_T)\pr_i(Z^{t_1,x}_T-Z^{t_2,x}_T)\right\rVert_1
\leq\xeqref{c83} 2\frac{V^2(t_1,x)+V^2(t_2,x)}{2}\frac{\sqrt{t_2-t_1}}{\sqrt{T-t_2}}.
\label{c85a}
\end{align}
Therefore, the triangle inequality and \eqref{c82} imply that for all $i\in [0,d]\cap\Z$, $(t_n)_{n\in \N}\subseteq [0,T)$, $t\in [0,T)$ with $\lim_{n\to\infty}t_n=t$ we have that
\begin{align}
&\limsup_{n\to\infty}
\left\lVert
g(X_T^{t_n,x})\pr_i(Z_T^{t_n,x})
-
g(X_T^{t,x})\pr_i(Z_T^{t,x})\right\rVert_1\nonumber\\
&
\leq \limsup_{n\to\infty}
\left\lVert
(g(X_T^{t_n,x})-g(X_T^{t,x}))\pr_i(Z_T^{t_n,x})
\right\rVert_1
+\limsup_{n\to\infty}\left\lVert
g(X_T^{t,x})(\pr_i(Z_T^{t_n,x})- \pr_i(Z_T^{t,x}))
\right\rVert_1\nonumber\\
&\leq \limsup_{n\to\infty}\left[\xeqref{c82}
2\frac{V^4(t_n,x)+V^4(t,x)}{2}\frac{\sqrt{\lvert t_n-t\rvert}}{\sqrt{T}}\frac{1}{\Lambda_i(T-\min \{t,t_n\})}\right]\nonumber\\
&\quad+\limsup_{n\to\infty}\left[\xeqref{c85a}
2\frac{V^2(t_n,x)+V^2(t,x)}{2}\frac{\sqrt{\lvert t_n-t\rvert}}{\sqrt{T-\max\{t,t_n\}}}
\frac{1}{\Lambda_i(T-\min \{t,t_n\})}\right]\nonumber\\
&=0.
\end{align}
This and Jensen's inequality prove for all $i\in [0,d]\cap\Z$, $x\in \R^d$ that 
\begin{align}
\left(
[0,T) \ni t\mapsto
\E\!\left[
g(X^{t,x}_T)\pr_i(Z^{t,x}_T)\right]\in \R\right)\in C([0,T), \R).\label{c86}
\end{align}
Next,
H\"older's inequality, the fact that $\frac{1}{\exponentV}+\frac{1}{\exponentZ}\leq 1$, \eqref{a04}, \eqref{a02b}, the fact that 
$c\leq V$,
and the fact that
$
\forall\, t_1\in [0,T),t_2\in [t_1,T)\colon  \int_{t_1}^{t_2}\frac{dr}{\sqrt{r-t_1}}=2\sqrt{r-t_1}|_{r=t_1}^{t_2}=2\sqrt{t_2-t_1}$ show for all $i\in [0,d]\cap\Z$, $t_1\in [0,T)$, $t_2\in [t_1,T)$, $x\in \R^d$ that
\begin{align}
&
\int_{t_1}^{t_2}
\Lambda_i(T-t_1)
\left\lVert
f(r,X^{t_1,x}_{r},0)\pr_i(Z^{t_1,x}_{r})\right\rVert_1dr\nonumber\\
&\leq \xeqref{a04}
\int_{t_1}^{t_2}\frac{1}{T}
\Lambda_i(T-t_1)
\left\lVert
V(r,X^{t_1,x}_{r})\pr_i(Z^{t_1,x}_{r})\right\rVert_1dr\nonumber\\
&\leq \int_{t_1}^{t_2}\frac{1}{T}
\Lambda_i(T-t_1)
\left\lVert
V(r,X^{t_1,x}_{r})\right\rVert_{\exponentV}\left\lVert\pr_i(Z^{t_1,x}_{r})\right\rVert_{\exponentZ}dr\nonumber\\
&\leq 
\int_{t_1}^{t_2}\frac{1}{T}\Lambda_i(T-t_1)\xeqref{a02b}V(t_1,x)
\xeqref{a02b}
\frac{c}{\Lambda_{i}(r-t_1)}\,dr\nonumber\\
&\leq \frac{V^2(t_1,x)}{T}\int_{t_1}^{t_2}\frac{\sqrt{T-t_1}}{\sqrt{r-t_1}}\,dr\nonumber\\
&=
\frac{V^2(t_1,x)\sqrt{T-t_1}}{T}2\sqrt{t_2-t_1}\nonumber\\
&\leq 2V^2(t_1,x)\frac{\sqrt{t_2-t_1}}{\sqrt{T}}.\label{c85}
\end{align}
Furthermore, the triangle inequality,
\eqref{b38}, H\"older's inequality, the fact that
$\frac{1}{\exponentV}+\frac{1}{\exponentX}\leq 1$,
\eqref{a02b}, the fact that $\max \{c,6e^{86c^6T^3}\}\leq V$, and the fact that 
$\forall\,t_1\in [0,T), t_2\in [t_1,T)\colon \int_{t_1}^{t_2}\frac{dr}{r-t_1}=2\sqrt{r-t_1}|_{r=t_1}^{t_2}=2\sqrt{t_2-t_1} $
imply for all 
$i\in [0,d]\cap\Z$,
$t_1\in [0,T)$, $t_2\in [t_1,T)$, $x\in \R^d$ that
\begin{align}
&
\int_{t_1}^{t_2}
\Lambda_i(T-t_1)
\left\lVert \sum_{\nu=0}^dL_\nu \Lambda_\nu(T)
\lvert
\pr_{\nu}(u(r,X^{t_1,x}_{r}))\rvert\lvert
\pr_i(Z^{t_1,x}_{r})\rvert\right\rVert_1dr\nonumber\\
&\leq 
\int_{t_1}^{t_2}
\Lambda_i(T-t_1)
 \sum_{\nu=0}^d\left[L_\nu \Lambda_\nu(T)
\left\lVert
\pr_{\nu}(u(r,X^{t_1,x}_{r}))
\pr_i(Z^{t_1,x}_{r})\right\rVert_1\right]dr\nonumber\\
&\leq 
\int_{t_1}^{t_2}
\Lambda_i(T-t_1)\sum_{\nu=0}^d\left[L_\nu \frac{\sqrt{T}}{\sqrt{T-r}}\Lambda_\nu(T-r)
\left\lVert
\pr_{\nu}(u(r,X^{t_1,x}_{r}))
\pr_i(Z^{t_1,x}_{r})\right\rVert_1\right]dr\nonumber\\
&\leq \xeqref{b38}
\int_{t_1}^{t_2}
\Lambda_i(T-t_1)\sum_{\nu=0}^d\left[L_\nu \frac{\sqrt{T}}{\sqrt{T-r}}
6ce^{86c^6T^3}
\left\lVert
V(r,X^{t_1,x}_{r})
\pr_i(Z^{t_1,x}_{r})\right\rVert_1\right]dr\nonumber\\
&\leq 
\int_{t_1}^{t_2}
\Lambda_i(T-t_1) \frac{c\sqrt{T}}{\sqrt{T-r}}
6ce^{86c^6T^3}
\left\lVert
V(r,X^{t_1,x}_{r})\right\rVert_{\exponentV}
\left\lVert
\pr_i(Z^{t_1,x}_{r})\right\rVert_{\exponentZ}dr\nonumber
\\
&\leq \xeqref{a02b}\int_{t_1}^{t_2}
\Lambda_i(T-t_1) \frac{c\sqrt{T}}{\sqrt{T-r}}
6ce^{86c^6T^3} V(t_1,x)\frac{c}{\Lambda_{i}(r-t_1)}\nonumber\\
&\leq c^2V^3(t_1,x)\int_{t_1}^{t_2}\frac{\sqrt{T}}{\sqrt{T-r}}\frac{\sqrt{T-t_1}}{\sqrt{r-t_1}}\,dr\nonumber\\
&\leq c^2TV^3(t_1,x)\frac{1}{\sqrt{T-t_2}}\int_{t_1}^{t_2}\frac{dr}{\sqrt{r-t_1}}\nonumber\\
&\leq 2c^2TV^3(t_1,x)\frac{\sqrt{t_1-t_2}}{\sqrt{T-t_2}}.\label{c89a}
\end{align}
This, the triangle inequality,  \eqref{a01},  \eqref{c85}, the fact that $1\leq V$, and the fact that $1+c^2T\leq e^{c^2T}$
 prove for all $i\in [0,d]\cap\Z$, $t_1\in [0,T)$, $t_2\in [t_1,T)$, $x\in \R^d$ that
\begin{align}
&\Lambda_i(T-t_1)
\int_{t_1}^{t_2}
\left\lVert
f(r,X^{t_1,x}_{r},u(r,X^{t_1,x}_{r}))\pr_i(Z^{t_1,x}_{r})\right\rVert_1 dr\nonumber\\
&\leq \xeqref{a01}\int_{t_1}^{t_2}
\Lambda_i(T-t_1)
\left\lVert
f(r,X^{t_1,x}_{r},0)\pr_i(Z^{t_1,x}_{r})\right\rVert_1dr\nonumber\\
&\quad +
\int_{t_1}^{t_2}
\Lambda_i(T-t_1)
\left\lVert \sum_{\nu=0}^dL_\nu \Lambda_\nu(T)
\lvert
\pr_{\nu}(u(r,X^{t_1,x}_{r})\rvert\lvert
\pr_i(Z^{t_1,x}_{r})\rvert\right\rVert_1dr\nonumber\\
&\leq\xeqref{c85} 2V^2(t_1,x)\frac{\sqrt{t_2-t_1}}{\sqrt{T}}+\xeqref{c89a}2c^2TV^3(t_1,x)\frac{\sqrt{t_1-t_2}}{\sqrt{T-t_2}}\nonumber\\
&\leq 4 \frac{V^{3}(t_1,x)+V^{3}(t_2,x)}{2}
e^{c^2T}\frac{\sqrt{t_1-t_2}}{\sqrt{T-t_2}}.\label{c89}
\end{align}
Next, the triangle inequality, \eqref{b59}, the fact that
$\sum_{\nu=0}^{d}L_i\leq c$,
H\"older's inequality, the fact that 
$\frac{6}{\exponentV}+\frac{1}{\exponentX}+\frac{1}{\exponentZ}\leq 1$,
\eqref{a02b}, \eqref{b79},   the fact that $\forall\, x,y,p,q\in [0,\infty)\colon \frac{x^p+y^p}{2}\frac{x^q+y^q}{2}\leq \frac{x^{p+q}+y^{p+q}}{2}$, and \cref{a15}
 show for all $i\in [0,d]\cap\Z$, $t_1\in [0,T)$, $t_2\in [t_1,T)$, $x\in \R^d$ that
\begin{align}
&\int_{t_2}^{T}
\Lambda_{i}(T-t_1)
\left\lVert
\sum_{\nu=0}^dL_\nu \Lambda_{\nu}(T)\lvert
\pr_\nu( u(r,X^{t_1,x}_{r})-u(r,X^{t_2,x}_{r}) )\rvert
\lvert\pr_i (Z^{t_1,x}_{r})\rvert
\right\rVert_1dr\nonumber\\
&\leq \int_{t_2}^{T}
\Lambda_{i}(T-t_1)
\sum_{\nu=0}^d\left[L_\nu \frac{\sqrt{T}}{\sqrt{T-r}}\left\lVert\Lambda_{\nu}(T-r)
\pr_\nu( u(r,X^{t_1,x}_{r})-u(r,X^{t_2,x}_{r}) )
\pr_i (Z^{t_1,x}_{r})
\right\rVert_1\right]dr\nonumber\\
&\leq\xeqref{b59} \int_{t_2}^{T}
\Lambda_{i}(T-t_1)\frac{c\sqrt{T}}{\sqrt{T-r}}
\left\lVert
e^{c^2T}
\frac{V^6(r,X^{t_1,x}_{r})+V^6(r,X^{t_2,x}_{r})}{2}
\frac{\left\lVert X^{t_1,x}_{r}-X^{t_2,x}_{r}\right\rVert}{\sqrt{T}}
\pr_i (Z^{t_1,x}_{r})\right\rVert_1dr\nonumber\\
&\leq \int_{t_2}^{T}
\Lambda_{i}(T-t_1)\frac{c\sqrt{T}}{\sqrt{T-r}}
e^{c^2T}\nonumber\\
&\qquad\qquad
\frac{\left\lVert V^6(r,X^{t_1,x}_{r})\right\rVert_{\frac{\exponentV}{6}}+\left\lVert V^6(r,X^{t_2,x}_{r})\right\rVert_{\frac{\exponentV}{6}}}{2}
\frac{\left\lVert \left\lVert X^{t_1,x}_{r}-X^{t_2,x}_{r}\right\rVert\right\rVert_{\exponentX}}{\sqrt{T}}\left\lVert
\pr_i (Z^{t_1,x}_{r})\right\rVert_{\exponentZ}dr\nonumber\\
&\leq \int_{t_2}^{T}
\Lambda_{i}(T-t_1)\frac{c\sqrt{T}}{\sqrt{T-r}}
e^{c^2T}\xeqref{a02b} \frac{V(t_1,x)+V(t_2,x)}{2}\xeqref{b79}
2 \frac{V^2(t_1,x)+V^2(t_2,x)}{2}\frac{\sqrt{t_2-t_1}}{\sqrt{T}}
\frac{c}{\Lambda_{i}(r-t_1)}dr\nonumber\\
&\leq \int_{t_2}^{T}2c^2\frac{V^3(t_1,x)+V^3(t_2,x)}{2}
\frac{\sqrt{t_2-t_1}}{\sqrt{T}}
\frac{\sqrt{T}}{\sqrt{T-r}}\frac{\sqrt{T-t_1}}{\sqrt{r-t_1}}dr\nonumber\\
&\leq 2c^2 T\frac{V^3(t_1,x)+V^3(t_2,x)}{2}
\frac{\sqrt{t_2-t_1}}{\sqrt{T}}\int_{t_1}^{T}\frac{dr}{\sqrt{T-r}\sqrt{r-t_1}}\nonumber\\
&\leq 8
c^2 T\frac{V^3(t_1,x)+V^3(t_2,x)}{2}
\frac{\sqrt{t_2-t_1}}{\sqrt{T}}.\label{c88}
\end{align}
Furthermore, \eqref{c81} and the fact that
$\forall\,t_1\in [0,T), t_2\in [0,T)\colon \int_{t_2}^{T}\frac{dr}{\sqrt{r-t_1}}=2\sqrt{r-t_1}|_{r=t_2}^T\leq 2 \sqrt{T-t_1}$ imply for all $i\in [0,d]\cap\Z$, $t_1\in [0,T)$, $t_2\in [t_1,T)$, $x\in \R^d$ that
\begin{align}
&
\int_{t_2}^{T}
\Lambda_{i}(T-t_1)\left\lVert\frac{1}{T}
\frac{V(r,X^{t_1,x}_{r})+V(r,X^{t_2,x}_{r})}{2}
\frac{\left\lVert
X^{t_1,x}_{r}-X^{t_2,x}_{r}
\right\rVert}{\sqrt{T}}\pr_i(Z^{t_1,x}_{r})
\right\rVert_1dr\nonumber\\
&\leq\xeqref{c81} \frac{2}{T}\frac{V^4(t_1,x)+V^4(t_2,x)}{2}\frac{\sqrt{t_2-t_1}}{\sqrt{T}}
\int_{t_2}^{T}
\frac{\sqrt{T-t_1}}{\sqrt{r-t_1}}\,dr\nonumber\\
&\leq  \frac{2}{T}\frac{V^4(t_1,x)+V^4(t_2,x)}{2}\frac{\sqrt{t_2-t_1}}{\sqrt{T}}2(T-t_1)\nonumber\\
&\leq 4\frac{V^4(t_1,x)+V^4(t_2,x)}{2}\frac{\sqrt{t_2-t_1}}{\sqrt{T}}.\label{c88b}
\end{align}
This, the triangle inequality, \eqref{a01}, \eqref{c88}, the fact that $1\leq V$, and the fact that $1+c^2T\leq e^{c^2T}$ prove 
for all $i\in [0,d]\cap\Z$, $t_1\in [0,T)$, $t_2\in [t_1,T)$, $x\in \R^d$ that
\begin{align}
&\int_{t_2}^{T}
\Lambda_{i}(T-t_1)
\left\lVert
\left[
f(r,X^{t_1,x}_{r},u(r,X^{t_1,x}_{r}))
-f(r,X^{t_2,x}_{r},u(r,X^{t_2,x}_{r}))\right]
\pr_i (Z^{t_1,x}_{r})
\right\rVert_1 dr\nonumber\\
&\leq \xeqref{a01}\int_{t_2}^{T}
\Lambda_{i}(T-t_1)
\left\lVert
\sum_{\nu=0}^dL_\nu \Lambda_{\nu}(T)\lvert u(r,X^{t_1,x}_{r})-u(r,X^{t_2,x}_{r}) \rvert
\lvert\pr_i (Z^{t_1,x}_{r})\rvert
\right\rVert_1dr\nonumber\\
&\quad +\int_{t_2}^{T}
\Lambda_{i}(T-t_1)\left\lVert\frac{1}{T}
\frac{V(r,X^{t_1,x}_{r})+V(r,X^{t_2,x}_{r})}{2}
\frac{\left\lVert
X^{t_1,x}_{r}-X^{t_2,x}_{r}
\right\rVert}{\sqrt{T}}\pr_i(Z^{t_1,x}_{r})
\right\rVert_1dr\nonumber\\
&\leq \xeqref{c88}8
c^2 T\frac{V^3(t_1,x)+V^3(t_2,x)}{2}
\frac{\sqrt{t_2-t_1}}{\sqrt{T}}+\xeqref{c88b}4\frac{V^4(t_1,x)+V^4(t_2,x)}{2}\frac{\sqrt{t_2-t_1}}{\sqrt{T}}\nonumber\\
&\leq 8e^{c^2T}
\frac{V^4(t_1,x)+V^4(t_2,x)}{2}
\frac{\sqrt{t_2-t_1}}{\sqrt{T}}.\label{c92}
\end{align}
Next, \eqref{a04} and \eqref{c83} imply 
for all $i\in [0,d]\cap\Z$, $t_1\in [0,T)$, $t_2\in [t_1,T)$, $x\in \R^d$ that
\begin{align}
&
\int_{t_2}^{T}
\Lambda_{i}(T-t_1)
\left\lVert
f(r,X^{t_2,x}_{r},0)\pr_i(Z^{t_1,x}_{r}-Z^{t_2,x}_{r})
\right\rVert_1dr\nonumber\\
&\leq \xeqref{a04}
\int_{t_2}^{T}
\Lambda_{i}(T-t_1)
\left\lVert \frac{1}{T}V(r,X^{t_2,x}_{r})\pr_i(Z^{t_1,x}_{r}-Z^{t_2,x}_{r})
\right\rVert_1dr\nonumber\\
&\leq \xeqref{c83}\int_{t_2}^{T}
\frac{2}{T}\frac{V^2(t_1,x)+V^2(t_2,x)}{2}\frac{\sqrt{t_2-t_1}}{\sqrt{r-t_2}}
\frac{\sqrt{T-t_1}}{\sqrt{r-t_1}}\,dr.\label{c94}
\end{align}
In addition, the triangle inequality, \eqref{b38}, H\"older's inequality, the fact that $\frac{1}{\exponentV}+\frac{1}{\exponentZ}\leq 1$, 
\eqref{a02b}, \eqref{c74},
the fact that
$\max\{c,6e^{86c^6T^3}\}\leq V$, and the fact that 
$\forall\, x,y,p,q\in [0,\infty)\colon \frac{x^p+y^p}{2}\frac{x^q+y^q}{2}\leq \frac{x^{p+q}+y^{p+q}}{2}$ show for all $i\in [0,d]\cap\Z$, $t_1\in [0,T)$, $t_2\in [t_1,T)$, $x\in \R^d$ that
\begin{align}
&
\int_{t_2}^{T}
\Lambda_{i}(T-t_1)
\left\lVert
\sum_{\nu=0}^dL_\nu\Lambda_{\nu}(T)
\lvert\pr_\nu( u(r,X^{t_2,x}_{r}))\rvert\lvert
\pr_i(Z^{t_1,x}_{r}-Z^{t_2,x}_{r})
\rvert
\right\rVert_1dr\nonumber\\
&\leq 
\int_{t_2}^{T}
\Lambda_{i}(T-t_1)
\sum_{\nu=0}^d\left[L_\nu\Lambda_{\nu}(T)
\left\lVert\pr_\nu( u(r,X^{t_2,x}_{r}))
\pr_i(Z^{t_1,x}_{r}-Z^{t_2,x}_{r})
\right\rVert_1\right]dr\nonumber\\
&
\leq \int_{t_2}^{T}
\Lambda_{i}(T-t_1)
\sum_{\nu=0}^d\left[L_\nu\frac{\sqrt{T}}{\sqrt{T-r}}\Lambda_{\nu}(T-r)
\left\lVert\pr_\nu( u(r,X^{t_2,x}_{r}))
\pr_i(Z^{t_1,x}_{r}-Z^{t_2,x}_{r})
\right\rVert_1\right]dr\nonumber\\
&
\leq \xeqref{b38}\int_{t_2}^{T}
\Lambda_{i}(T-t_1)
\sum_{\nu=0}^d\left[L_\nu\frac{\sqrt{T}}{\sqrt{T-r}}
\left\lVert 6c e^{86c^6T^3}  V(r,X^{t_2,x}_{r})
\pr_i(Z^{t_1,x}_{r}-Z^{t_2,x}_{r})
\right\rVert_1\right]dr\nonumber\\
&\leq 
\int_{t_2}^{T}
\Lambda_{i}(T-t_1)
c\frac{\sqrt{T}}{\sqrt{T-r}}
 6c e^{86c^6T^3} \left\lVert V(r,X^{t_2,x}_{r})\right\rVert_{\exponentV}
\left\lVert
\pr_i(Z^{t_1,x}_{r}-Z^{t_2,x}_{r})\right\rVert_{\exponentZ}dr
\nonumber\\
&\leq 
\int_{t_2}^{T}
\Lambda_{i}(T-t_1) \frac{\sqrt{T}}{\sqrt{T-r}}
\xeqref{a02b}
V^4(t_2,x)\xeqref{c74}
\frac{V(t_1,x)+V(t_2,x)}{2}\frac{\sqrt{t_2-t_1}}{\sqrt{r-t_2}\Lambda_i(r-t_1)}\,dr\nonumber\\
&\leq 
\int_{t_2}^{T}2\frac{V^5(t_1,x)+V^5(t_2,x)}{2}\frac{\sqrt{T}}{\sqrt{T-r}}
\frac{\sqrt{T-t_1}\sqrt{t_2-t_1}}{\sqrt{r-t_1}\sqrt{r-t_2}}\,dr.
\label{c95}
\end{align}
Thus, the triangle inequality,  \eqref{a01}, and \eqref{c94} prove for all
$i\in [0,d]\cap\Z$, $t_1\in [0,T)$, $t_2\in [t_1,T)$, $x\in \R^d$ that
\begin{align}
&
\int_{t_2}^{T}
\Lambda_{i}(T-t_1)\left\lVert
f(r,X^{t_2,x}_{r},u(r,X^{t_2,x}_{r}))
\pr_i (Z^{t_1,x}_{r}-Z^{t_2,x}_{r})
\right\rVert_1dr\nonumber\\
&\leq \xeqref{a01}\int_{t_2}^{T}
\Lambda_{i}(T-t_1)
\left\lVert
f(r,X^{t_2,x}_{r},0)\pr_i(Z^{t_1,x}_{r}-Z^{t_2,x}_{r})
\right\rVert_1dr\nonumber\\
&\quad +
\int_{t_2}^{T}
\Lambda_{i}(T-t_1)
\left\lVert
\sum_{\nu=0}^dL_\nu\Lambda_{\nu}(T)
\lvert\pr_\nu( u(r,X^{t_2,x}_{r}))\rvert\lvert
\pr_i(Z^{t_1,x}_{r}-Z^{t_2,x}_{r})
\rvert
\right\rVert_1dr\nonumber\\
&\leq \int_{t_2}^{T}\xeqref{c94}
\frac{2}{T}\frac{V^2(t_1,x)+V^2(t_2,x)}{2}\frac{\sqrt{t_2-t_1}}{\sqrt{r-t_2}}
\frac{\sqrt{T-t_1}}{\sqrt{r-t_1}}\,dr\nonumber\\
&\quad +\int_{t_2}^{T}\xeqref{c95}2\frac{V^5(t_1,x)+V^5(t_2,x)}{2}\frac{\sqrt{T}}{\sqrt{T-r}}
\frac{\sqrt{T-t_1}\sqrt{t_2-t_1}}{\sqrt{r-t_1}\sqrt{r-t_2}}\,dr.\label{c96}
\end{align}
Therefore, the triangle inequality, \eqref{c89}, and \eqref{c92} imply for all
$i\in [0,d]\cap\Z$, $t_1\in [0,T)$, $t_2\in [t_1,T)$, $x\in \R^d$ that
\begin{align}
&
\left\lvert
\int_{t_1}^{T}\E \!\left[
f(r,X^{t_1,x}_{r},u(r,X^{t_1,x}_{r}))\pr_i(Z^{t_1,x}_{r})\right]dr
-
\int_{t_2}^{T}\E \!\left[
f(r,X^{t_2,x}_{r},u(r,X^{t_2,x}_{r}))\pr_i(Z^{t_2,x}_{r})\right]dr\right\rvert\nonumber\\
&\leq 
\int_{t_1}^{t_2}
\left\lVert
f(r,X^{t_1,x}_{r},u(r,X^{t_1,x}_{r}))\pr_i(Z^{t_1,x}_{r})
\right\rVert_1dr\nonumber\\
&\quad +
\int_{t_2}^T
\left\lVert
f(r,X^{t_1,x}_{r},u(r,X^{t_1,x}_{r}))\pr_i(Z^{t_1,x}_{r})
-
f(r,X^{t_2,x}_{r},u(r,X^{t_2,x}_{r}))\pr_i(Z^{t_2,x}_{r})\right\rVert_1dr\nonumber\\
&
\leq \int_{t_1}^{t_2}
\left\lVert
f(r,X^{t_1,x}_{r},u(r,X^{t_1,x}_{r}))\pr_i(Z^{t_1,x}_{r})
\right\rVert_1dr\nonumber\\
&\quad +\int_{t_2}^{T}
\left\lVert
\left[
f(r,X^{t_1,x}_{r},u(r,X^{t_1,x}_{r}))
-f(r,X^{t_2,x}_{r},u(r,X^{t_2,x}_{r}))\right]
\pr_i (Z^{t_1,x}_{r})
\right\rVert_1 dr\nonumber\\
&\quad +\int_{t_2}^{T}
\left\lVert
f(r,X^{t_2,x}_{r},u(r,X^{t_2,x}_{r}))
\pr_i (Z^{t_1,x}_{r}-Z^{t_2,x}_{r})
\right\rVert_1dr\nonumber\\
&\leq  \Biggl[\xeqref{c89}4 \frac{V^{3}(t_1,x)+V^{3}(t_2,x)}{2}
e^{c^2T}\frac{\sqrt{t_2-t_1}}{\sqrt{T-t_2}}\nonumber\\
&\quad +\xeqref{c92}8e^{c^2T}
\frac{V^4(t_1,x)+V^4(t_2,x)}{2}
\frac{\sqrt{t_2-t_1}}{\sqrt{T}}\nonumber\\
&\quad +\xeqref{c96}\int_{t_2}^{T}
\frac{2}{T}\frac{V^2(t_1,x)+V^2(t_2,x)}{2}\frac{\sqrt{t_2-t_1}}{\sqrt{r-t_2}}
\frac{\sqrt{T-t_1}}{\sqrt{r-t_1}}\,dr\nonumber\\
&\quad +\int_{t_2}^{T}2\frac{V^5(t_1,x)+V^5(t_2,x)}{2}\frac{\sqrt{T}}{\sqrt{T-r}}
\frac{\sqrt{T-t_1}\sqrt{t_2-t_1}}{\sqrt{r-t_1}\sqrt{r-t_2}}\,dr\Biggr]\frac{1}{\Lambda_{i}(T-t_1)}.
\end{align}
This shows for all $i\in [0,d]\cap\Z$, $(t_n)_{n\in \N}\subseteq [0,T)$, $t\in [0,T)$, $x\in \R^d$ that
\begin{align}
&
\left\lvert
\int_{t_n}^{T}\E \!\left[
f(r,X^{t_n,x}_{r},u(r,X^{t_n,x}_{r}))\pr_i(Z^{t_n,x}_{r})\right]dr
-
\int_{t}^{T}\E \!\left[
f(r,X^{t,x}_{r},u(r,X^{t,x}_{r}))\pr_i(Z^{t,x}_{r})\right]dr\right\rvert\nonumber\\
&\leq  \Biggl[4 \frac{V^{3}(t_n,x)+V^{3}(t,x)}{2}
e^{c^2T}\frac{\sqrt{\lvert t_n-t\rvert}}{\sqrt{T-\max\{t,t_n\}}}\nonumber\\
&\quad +8e^{c^2T}
\frac{V^4(t_1,x)+V^4(t_2,x)}{2}
\frac{\sqrt{\lvert t_n-t\rvert}}{\sqrt{T}}\nonumber\\
&\quad +\int_{\max\{t,t_n\}}^{T}
\frac{2}{T}\frac{V^2(t_n,x)+V^2(t,x)}{2}\frac{\sqrt{\lvert t_n-t\rvert}}{\sqrt{r-t_n}}
\frac{\sqrt{T-\min\{t,t_n\}}}{\sqrt{r-t}}\,dr\nonumber\\
&\quad +\int_{\max\{t,t_n\}}^{T}2\frac{V^5(t_1,x)+V^5(t_2,x)}{2}\frac{\sqrt{T}}{\sqrt{T-r}}
\frac{\sqrt{T-\min\{t,t_n\}}\sqrt{\lvert t_n-t\rvert}}{\sqrt{r-t_n}\sqrt{r-t}}\,dr\Biggr]\frac{1}{\Lambda_{i}(T-\min\{t,t_n\})}.\label{c97}
\end{align}
Next, note that for all $(t_n)_{n\in \N}\subseteq [0,T)$, $t\in [0,T)$ with $\lim_{n\to\infty}t_n=t$ we have that 
\begin{align}
\sup_{n\in \N}
\int_{0}^{T}\left[\frac{\sqrt{\lvert t_n-t\rvert}}{\sqrt{r-t_n}\sqrt{r-t}}\1_{(\max\{t,t_n\},T)}(r)\right]^\frac{3}{2}dr
&=\sup_{n\in \N}
\int_{\max\{t,t_n\}}^{T}\left[\frac{\sqrt{\lvert t_n-t\rvert}}{\sqrt{r-\min\{t,t_n\}}\sqrt{r-\max\{t,t_n\}}}\right]^\frac{3}{2}dr\nonumber\\
&\leq \sup_{n\in \N}
\int_{\max\{t,t_n\}}^{T}\left[\frac{\sqrt{\lvert t_n-t\rvert}}{\sqrt{\lvert t_n-t\rvert}\sqrt{r-\max\{t,t_n\}}}\right]^\frac{3}{2}dr\nonumber\\
&=\sup_{n\in \N}\int_{\max\{t,t_n\}}^{T}(r-\max\{t,t_n\})^{-\frac{3}{4}}\,dr\nonumber\\
&=\sup_{n\in \N}4(r-\max\{t,t_n\})^\frac{1}{4}\Bigr|_{r=\max\{t,t_n\}}^T\nonumber\\
&\leq T^\frac{1}{4}.
\end{align}
This implies that for all $(t_n)_{n\in \N}\subseteq [0,T)$, $t\in [0,T)$ with $\lim_{n\to\infty}t_n=t$ we have that 
$([0,T)\ni r\mapsto
\frac{\sqrt{\lvert t_n-t\rvert}}{\sqrt{r-t_n}\sqrt{r-t}}\1_{(\max\{t,t_n\},T)}(r)\in \R)_{n\in \N}$ 
is uniformly integrable. Hence, for all $(t_n)_{n\in \N}\subseteq [0,T)$, $t\in [0,T)$ with $\lim_{n\to\infty}t_n=t$ we have that 
\begin{align}
\lim_{n\to\infty}\int_{\max\{t,t_n\}}^{T}
\frac{\sqrt{\lvert t_n-t\rvert}}{\sqrt{r-t_n}\sqrt{r-t}}\,dr&=
\lim_{n\to\infty}\int_{0}^{T}
\frac{\sqrt{\lvert t_n-t\rvert}}{\sqrt{r-t_n}\sqrt{r-t}}\1_{(\max\{t,t_n\},T)}(r)\,dr\nonumber\\
&
=
\int_{0}^{T}\lim_{n\to\infty}
\frac{\sqrt{\lvert t_n-t\rvert}}{\sqrt{r-t_n}\sqrt{r-t}}\1_{(\max\{t,t_n\},T)}(r)\,dr=0.
\label{c99}
\end{align}
Next, the substitution $s=\frac{r-a}{T-a}$, $dr=(T-a)ds$ and the definition of the Beta function prove for all $a\in [0,T)$ that
\begin{align}
\int_{a}^{T}\frac{dr}{(T-r)^\frac{3}{4} (r-a)^\frac{3}{4} }
=\int_{0}^{1}\frac{(T-a)ds}{[(1-s)(T-a)]^\frac{3}{4}[s(T-a)]^\frac{3}{4} }=\mathrm{B}(\tfrac{1}{4},\tfrac{1}{4})(T-a)^{-\frac{1}{2}}.
\end{align}
This shows  that for all $(t_n)_{n\in \N}\subseteq [0,T)$, $t\in [0,T)$ with $\lim_{n\to\infty}t_n=t$ we have that 
\begin{align}
&\sup_{n\in \N}
\int_{0}^{T}\left[\frac{\sqrt{\lvert t_n-t\rvert}}{\sqrt{T-r}\sqrt{r-t_n}\sqrt{r-t}}\1_{(\max\{t,t_n\},T)}(r)
\right]^\frac{3}{2}dr\nonumber\\
&=\sup_{n\in \N}
\int_{\max\{t,t_n\}}^{T}
\left[
\frac{\sqrt{\lvert t_n-t\rvert}}{\sqrt{T-r}\sqrt{r-\min\{t,t_n\}}\sqrt{r-\max\{t,t_n\}}}
\right]^\frac{3}{2}dr\nonumber\\
&\leq \sup_{n\in \N}
\int_{\max\{t,t_n\}}^{T}
\left[
\frac{\sqrt{\lvert t_n-t\rvert}}{\sqrt{T-r}\sqrt{\lvert t_n-t\rvert}\sqrt{r-\max\{t,t_n\}}}
\right]^\frac{3}{2}dr\nonumber\\
&\leq \sup_{n\in \N}
\int_{\max\{t,t_n\}}^{T}\frac{dr}{(T-r)^\frac{3}{4}(r-\max\{t,t_n\})^\frac{3}{4}}\nonumber\\
&=\sup_{n\in \N}\mathrm{B}(\tfrac{1}{4},\tfrac{1}{4})(T-\max\{t,t_n\})^{-\frac{1}{2}}<\infty.
\end{align}
Therefore, for all $(t_n)_{n\in \N}\subseteq [0,T)$, $t\in [0,T)$ with $\lim_{n\to\infty}t_n=t$ we have that 
$([0,T)\ni r\mapsto
\frac{\sqrt{\lvert t_n-t\rvert}}{\sqrt{T-r}\sqrt{r-t_n}\sqrt{r-t}}\1_{(\max\{t,t_n\},T)}(r)\in \R )_{n\in \N}
$ is uniformly integrable. Hence, for all $(t_n)_{n\in \N}\subseteq [0,T)$, $t\in [0,T)$ with $\lim_{n\to\infty}t_n=t$ we have that 
\begin{align}
\lim_{n\to\infty}\int_{\max\{t,t_n\}}^{T}
\frac{\sqrt{\lvert t_n-t\rvert}}{\sqrt{T-r}\sqrt{r-t_n}\sqrt{r-t}}\,dr&=
\lim_{n\to\infty}
\int_{0}^{T}
\frac{\sqrt{\lvert t_n-t\rvert}}{\sqrt{T-r}\sqrt{r-t_n}\sqrt{r-t}}
\1_{(\max\{t,t_n\},T)}(r)\,dr\nonumber\\
&=
\int_{0}^{T}\lim_{n\to\infty}
\frac{\sqrt{\lvert t_n-t\rvert}}{\sqrt{T-r}\sqrt{r-t_n}\sqrt{r-t}}
\1_{(\max\{t,t_n\},T)}(r)\,dr\nonumber\\
&=0.
\end{align}
This,  \eqref{c99}, and \eqref{c97} 
prove that for all $i\in [0,d]\cap\Z$, $(t_n)_{n\in \N}\subseteq [0,T)$, $t\in [0,T)$, $x\in \R^d$ with 
$t=\lim_{n\to\infty}t_n$
we have that
\begin{align}
\lim_{n\to\infty}
\left\lvert
\int_{t_n}^{T}\E \!\left[
f(r,X^{t_n,x}_{r},u(r,X^{t_n,x}_{r}))\pr_i(Z^{t_n,x}_{r})\right]dr
-
\int_{t}^{T}\E \!\left[
f(r,X^{t,x}_{r},u(r,X^{t,x}_{r}))\pr_i(Z^{t,x}_{r})\right]dr\right\rvert=0.
\end{align}
Hence, for all $x\in \R^d$, $i\in [0,d]\cap\Z$ that 
$[0,T)\ni t\mapsto \int_{t}^{T}\E \!\left[
f(r,X^{t,x}_{r},u(r,X^{t,x}_{r}))\pr_i(Z^{t,x}_{r})\right]dr\in \R$ is continuous. This,  \eqref{c86}, and \eqref{b05a} show for all 
$x\in \R^d$ that $[0,T)\ni t\mapsto u(t,x)\in \R^{d+1} $ is continuous. Therefore, \eqref{b59} and the fact that $V$ and $\Lambda$ are continuous imply that for all 
$i\in [0,d]\cap\Z$,
$(t_n)_{n\in \N}\subseteq [0,T)$, $t\in [0,T)$,
$(x_n)_{n\in \N}\subseteq \R^d$, $x\in \R^d$ with $\lim_{n\to\infty}(t_n,x_n)= (t,x)$ we have that
\begin{align}
&\limsup_{n\to\infty} \lvert\pr_i(
u(t_n,x_n)-u(t,x))\rvert\nonumber\\
&\leq \limsup_{n\in \N}\lvert\pr_i( u(t_n,x_n)-u(t_n,x))\rvert+
\limsup_{n\to\infty} \lvert \pr_i(u(t_n,x)-u(t,x))\rvert\nonumber\\
& 
\leq\limsup_{n\to\infty} \left[ e^{c^2T}
\frac{V^6(t_n,x_n)+V^6(t_n,x)}{2}
\frac{\lVert x_n-x\rVert}{\sqrt{T}}\frac{1}{\Lambda_{i}(T-t_n)}\right]=0.
\end{align}
Hence, $u$ is continuous. This completes the proof of \cref{b37}.
\end{proof}

\section{Perturbation lemma}\label{s02}
In \cref{a37a} below we estimate the difference between two fixed points, roughly speaking, one generated from the solution to an SDE and one from an approximation schema (e.g. Euler-Maruyama schema), see \eqref{a05c} and \eqref{a05d}.
\begin{lemma}
[Perturbation lemma]\label{a37a}
Assume \cref{z04}. Let $\delta\in (0,1)$. Suppose that
$
 \frac{8}{\exponentV}+\frac{1}{\exponentX}+\frac{1}{\exponentZ}\leq 1$ and $\frac{1}{2}+\frac{1}{\exponentZ}\leq 1$. 
Let $
(\mathcal{X}^{s,x}_t)_{s\in [0,T],t\in[s,T],x\in\R^d}\colon \{(\sigma,\tau)\in [0,T]^2\colon \sigma\leq \tau\}\times \R^d\times \Omega\to\R^d $,
$
(\mathcal{Z}^{s,x}_t)_{s\in [0,T),t\in(s,T],x\in\R^d}\colon \{(\sigma,\tau)\in [0,T]^2\colon \sigma< \tau\}\times \R^d\times \Omega\to\R^{d+1} $ be measurable. Assume
for all 
$t\in [0,T)$, $r\in (t,T]$,
 $x\in \R^d$
 that
\begin{align}
\left\lVert
V(r,\mathcal{X}^{t,x}_r)
\right\rVert_{\exponentV}\leq V(t,x),\label{a02d}
\end{align}
\begin{align}
\left\lVert\left\lVert
{X}^{t,x}_{r}-\mathcal{X}^{t,x}_{r}\right\rVert
\right\rVert_{\exponentX}\leq {\delta}^\frac{1}{2} V(t,x),\quad \left\lVert
\pr_i({Z}^{t,x}_{r}-\mathcal{Z}^{t,x}_{r})\right\rVert_{\exponentZ}\leq \frac{\delta^\frac{1}{2}V(t,x)}{\sqrt{T}\Lambda_{i}(r-t)}
.\label{a70}
\end{align}
Then the following items hold.
\begin{enumerate}[(i)]
\item \label{a17b}
There exist unique measurable functions $u,\mathfrak{u}\colon [0,T)\times \R^d\to \R^{d+1}$ such that for all $t\in [0,T)$, $x\in \R^d$
we have that
\begin{align}\label{b01c}
\max_{\nu\in [0,d]\cap\Z}\sup_{\tau\in [0,T), \xi\in \R^d}
\left[\Lambda_\nu(T-\tau)\frac{\lvert\pr_\nu(u(\tau,\xi))\rvert}{V(\tau ,\xi)}\right]<\infty,
\end{align}
\begin{align}\label{b01d}
\max_{\nu\in [0,d]\cap\Z}\sup_{\tau\in [0,T), \xi\in \R^d}
\left[\Lambda_\nu(T-\tau)\frac{\lvert\pr_\nu(\mathfrak{u}(\tau,\xi))\rvert}{V(\tau ,\xi)}\right]<\infty,
\end{align}
\begin{align}\max_{\nu\in [0,d]\cap\Z}\left[
\E\!\left [\left\lvert g(X^{t,x}_{T} )\pr_\nu(Z^{t,x}_{T})\right\rvert \right] + \int_{t}^{T}
\E \!\left[\left\lvert
f(r,X^{t,x}_{r},u(r,X^{t,x}_{r}))\pr_\nu(Z^{t,x}_{r})\right\rvert\right]dr\right]<\infty,
\end{align}
\begin{align}\max_{\nu\in [0,d]\cap\Z}\left[
\E\!\left [\left\lvert g(\mathcal{X}^{t,x}_{T} )\pr_\nu(\mathcal{Z}^{t,x}_{T})\right\rvert \right] + \int_{t}^{T}
\E \!\left[\left\lvert
f(r,\mathcal{X}^{t,x}_{r},\mathfrak{u}(r,\mathcal{X}^{t,x}_{r}))\pr_\nu(\mathcal{Z}^{t,x}_{r})\right\rvert\right]dr\right]<\infty,
\end{align}
\begin{align}
u(t,x)=\E\!\left [g(X^{t,x}_{T} )Z^{t,x}_{T} \right] + \int_{t}^{T}
\E \!\left[
f(r,X^{t,x}_{r},u(r,X^{t,x}_{r}))Z^{t,x}_{r}\right]dr,\label{a05c}
\end{align} and
\begin{align}
\mathfrak{u}(t,x)=\E\!\left [g(\mathcal{X}^{t,x}_{T} )\mathcal{Z}^{t,x}_{T} \right] + \int_{t}^{T}
\E \!\left[
f(r,\mathcal{X}^{t,x}_{r},\mathfrak{u}(r,\mathcal{X}^{t,x}_{r}))\mathcal{Z}^{t,x}_{r}\right]dr.\label{a05d}
\end{align}
\item \label{a92}
For all
$\nu\in [0,d]\cap \Z$, $t\in [0,T)$, $y\in \R^d$ we have that
\begin{align}
\Lambda_\nu(T-t)\lvert
\pr_{\nu}( {u}(t,y) - \mathfrak{u}(t,y))\rvert
\leq  \frac{\delta^\frac{1}{2}e^{c^2T}}{\sqrt{T}}V^9(t,y).
\end{align}
\end{enumerate}
\end{lemma}
\begin{proof}[Proof of \cref{a37a}]
First, 
\cref{a03} and the assumptions of \cref{a37a} prove \eqref{a17b} and imply that
for all $t\in [0,T)$ we have that
\begin{align}\label{a38c}
\max_{\nu\in [0,d]\cap\Z}\sup_{y\in \R^d}
\left[\Lambda_\nu(T-t)\frac{\lvert\pr_\nu(\mathfrak{u}(t,y))\rvert}{V(t,y)}\right]
\leq 
6c e^{86c^6T^2(T-t)}.
\end{align}
Furthermore,
\cref{a37} and the assumptions of \cref{a37a} show for all 
$x,y\in\R^d$, $t\in [0,T)$ that
\begin{align} \begin{split} 
&
\max_{i\in [0,d]\cap\Z}\left[
\Lambda_{i}(T-t)
\left\lvert \pr_i\left(u(t,x)-
u(t,y)\right)\right\rvert\right]\leq e^{c^2T}
\frac{V^6(t,x)+V^6(t,y)}{2}
\frac{\lVert x-y\rVert}{\sqrt{T}}.
\end{split}\label{a59a}
\end{align}
Next,
H\"older's inequality, the fact that
$\frac{6}{\exponentV}+\frac{1}{\exponentX}+\frac{1}{\exponentZ}\leq 1$, \eqref{a02b}, \eqref{a02d}, \eqref{a70}, and the fact that $c\leq V$ imply for all 
$i\in[0,d]\cap\Z$,
$t\in [0,T)$, $r\in (t,T]$, $x\in \R^d$ that
\begin{align}
&
\Lambda_{i}(T-t)
\left\lVert\frac{
V(r,{X}^{t,x}_{r})+V(r,\mathcal{X}^{t,x}_{r})}{2}
\frac{\left\lVert
{X}^{t,x}_{r}-\mathcal{X}^{t,x}_{r}
\right\rVert}{\sqrt{T}}\pr_i({Z}^{t,x}_{r})
\right\rVert_1\nonumber\\
&\leq 
\Lambda_{i}(T-t)
\left\lVert\frac{
V^6(r,{X}^{t,x}_{r})+V^6(r,\mathcal{X}^{t,x}_{r})}{2}
\frac{\left\lVert
{X}^{t,x}_{r}-\mathcal{X}^{t,x}_{r}
\right\rVert}{\sqrt{T}}\pr_i({Z}^{t,x}_{r})
\right\rVert_1\nonumber\\
&\leq 
\frac{\left\lVert
V^6(r,{X}^{t,x}_{r})\right\rVert_{\frac{\exponentV}{6}}+\left\lVert V^6(r,\mathcal{X}^{t,x}_{r})\right\rVert_{\frac{\exponentV}{6}}}{2}
\frac{\left\lVert
{X}^{t,x}_{r}-\mathcal{X}^{t,x}_{r}
\right\rVert_{\exponentX}}{\sqrt{T}}
\Lambda_{i}(T-t)\left\lVert
\pr_i({Z}^{t,x}_{r})
\right\rVert_{\exponentZ}\nonumber\\
&\leq \xeqref{a02b}\xeqref{a02d}V^6(t,x)\xeqref{a70}\frac{{\delta}^\frac{1}{2} V(t,x)}{\sqrt{T}}\xeqref{a02b}
\frac{c\Lambda_{i}(T-t)}{\Lambda_{i}(r-t)}\nonumber\\
&\leq \frac{c{\delta}^\frac{1}{2} V^7(t,x)}{\sqrt{T}}\frac{\sqrt{T-t}}{\sqrt{r-t}}\nonumber\\&\leq \frac{{\delta}^\frac{1}{2} V^8(t,x)}{\sqrt{T}}\frac{\sqrt{T-t}}{\sqrt{r-t}}.
\label{a81}
\end{align}
This and  \eqref{a19} prove for all $i\in[0,d]\cap\Z$, $t\in [0,T)$,  $x\in \R^d$ that
\begin{align} \begin{split} 
&
\Lambda_{i}(T-t)\left\lVert
\left(
g({X}^{t,x}_{T} ) 
-g(\mathcal{X}^{t,x}_{T} )\right)\pr_i
({Z}^{t,x}_{T})
\right\rVert_1\\
&\leq \Lambda_{i}(T-t)
\left\lVert\frac{
V(T,{X}^{t,x}_{T})+V(T,\mathcal{X}^{t,x}_{T})}{2}
\frac{\left\lVert
{X}^{t,x}_{T}-\mathcal{X}^{t,x}_{T}
\right\rVert}{\sqrt{T}}\pr_i({Z}^{t,x}_{T})
\right\rVert_1
\leq \frac{{\delta}^\frac{1}{2} V^8(t,x)}{\sqrt{T}}.
\end{split}\label{a79}
\end{align}
Next, H\"older's inequality, the fact that
$\frac{1}{\exponentV}+\frac{1}{\exponentZ}\leq 1$,
\eqref{a02b}, and \eqref{a70} imply for all $i\in[0,d]\cap\Z$, $t\in [0,T)$,  $x\in \R^d$  that
\begin{align} 
\Lambda_{i}(T-t)
\left\lVert
g({X}^{t,x}_{T} )\pr_i\!\left({Z}^{t,x}_{T}- \mathcal{Z}^{t,x}_{T}\right)\right\rVert_1&\leq 
\Lambda_{i}(T-t)
\left\lVert
V(T,{X}^{t,y}_{T} )\pr_i\!\left({Z}^{t,x}_{T}- \mathcal{Z}^{t,x}_{T}\right)\right\rVert_1\nonumber\\
&\leq 
\Lambda_{i}(T-t)
\left\lVert
V(T,{X}^{t,y}_{T} )
\right\rVert_{\exponentV}
\left\lVert
\pr_i\!\left({Z}^{t,x}_{T}- \mathcal{Z}^{t,x}_{T}\right)\right\rVert_{\exponentZ}\nonumber\\
&
\leq \Lambda_{i}(T-t)\xeqref{a02b}V(t,x)\xeqref{a70}
\frac{\delta^\frac{1}{2}V(t,x)}{\sqrt{T}\Lambda_{i}(T-t)}\nonumber\\
&
\leq 
\frac{\delta^\frac{1}{2}V^2(t,x)}{\sqrt{T}}.
\label{a80}\end{align}
In addition, the triangle inequality, the fact that $\sum_{\nu=0}^{d}L_\nu\leq c$, \eqref{a59a}, \eqref{a81}
prove for all $i\in[0,d]\cap\Z$, $t\in [0,T)$, 
$r\in (0,T]$,
 $x\in \R^d$ that
\begin{align}  
&\Lambda_{i}(T-t)
\left\lVert\sum_{\nu=0}^{d}L_\nu\Lambda_\nu(T) \left\lvert\pr_\nu(u(r,{X}^{t,x}_{r}) - u(r,\mathcal{X}^{t,x}_{r}))
\pr_i({Z}^{t,x}_{r})\right\rvert
\right\rVert_1\nonumber\\
&\leq \Lambda_{i}(T-t)
\sum_{\nu=0}^{d}\left\lVert L_\nu\Lambda_\nu(T) \pr_\nu(u(r,{X}^{t,x}_{r}) - u(r,\mathcal{X}^{t,x}_{r}))
\pr_i({Z}^{t,x}_{r})
\right\rVert_1\nonumber\\
&\leq \Lambda_{i}(T-t)
\sum_{\nu=0}^{d}\left[ L_\nu\frac{\sqrt{T}}{\sqrt{T-r}}\left\lVert\Lambda_\nu(T-r) \pr_\nu(u(r,{X}^{t,x}_{r}) - u(r,\mathcal{X}^{t,x}_{r}))
\pr_i({Z}^{t,x}_{r})
\right\rVert_1\right]\nonumber\\
 &\leq\xeqref{a59a} 
 \frac{c\sqrt{T}}{\sqrt{T-r}}\Lambda_{i}(T-t)\left\lVert e^{c^2T}
\frac{V^6(r,{X}^{t,x}_{r}) + V^6(r,\mathcal{X}^{t,x}_{r})}{2}
\frac{\left\lVert
{X}^{t,x}_{r}-\mathcal{X}^{t,x}_{r}
\right\rVert
}{\sqrt{T}}
\pr_i({Z}^{t,x}_{r})
\right\rVert_1\nonumber\\
&\leq  \frac{e^{c^2T}c\sqrt{T}}{\sqrt{T-r}}\xeqref{a81}\frac{c{\delta}^\frac{1}{2} V^7(t,x)}{\sqrt{T}}\frac{\sqrt{T-t}}{\sqrt{r-t}}\nonumber\\
&\leq \frac{\delta^\frac{1}{2}e^{c^2T}c^2T V^7(t,x)}{\sqrt{T}}
\frac{1}{\sqrt{(T-r)(r-t)}}.
\label{a83}
\end{align}
This, \eqref{a01}, the triangle inequality, the fact that $\forall\,t\in [0,T)\colon \int_{t}^{T}\frac{dr}{\sqrt{r-t}}=2\sqrt{r-t}|_{r=t}^T=2\sqrt{T-t}$, \cref{a15}, the fact that $1\leq V$, the fact that
$1+c^2T\leq e^{c^2T}$ imply for all $i\in[0,d]\cap\Z$, $t\in [0,T)$, 
 $x\in \R^d$ that
\begin{align} 
&\int_{t}^{T}
\Lambda_{i}(T-t)\left\lVert
\left(
f(r,{X}^{t,x}_{r}, u(r,{X}^{t,x}_{r}))-
f(r,\mathcal{X}^{t,x}_{r}, {u}(r,\mathcal{X}^{t,x}_{r}))\right)\pr_i({Z}^{t,x}_{r})
\right\rVert_1dr\nonumber\\
&
\leq \int_{t}^{T}\Lambda_{i}(T-t)\frac{1}{T}
\left\lVert\frac{
V(r,{X}^{t,x}_{r})+V(r,\mathcal{X}^{t,x}_{r})}{2}
\frac{\left\lVert
{X}^{t,x}_{r}-\mathcal{X}^{t,x}_{r}
\right\rVert}{\sqrt{T}}\pr_i({Z}^{t,x}_{r})
\right\rVert_1dr\nonumber\\
&\quad +\int_{t}^{T}\Lambda_{i}(T-t)
\left\lVert\sum_{\nu=0}^{d}L_\nu\Lambda_\nu(T)\left\lvert \pr_\nu(u(r,{X}^{t,x}_{r}) - u(r,\mathcal{X}^{t,x}_{r}))
\pr_i({Z}^{t,x}_{r})\right\rvert
\right\rVert_1dr\nonumber\\
&\leq \frac{1}{T}\int_{t}^{T}\xeqref{a81}\frac{{\delta}^\frac{1}{2}V^8(t,x)}{\sqrt{T}}\frac{\sqrt{T-t}}{\sqrt{r-t}}\,dr+\int_{t}^{T}\xeqref{a83}\frac{\delta^\frac{1}{2}e^{c^2T}c^2T V^7(t,x)}{\sqrt{T}}
\frac{dr}{\sqrt{(T-r)(r-t)}}\nonumber\\
&\leq 2\frac{{\delta}^\frac{1}{2}V^8(t,x)}{\sqrt{T}}+
\frac{\delta^\frac{1}{2}e^{c^2T}c^2T V^7(t,x)}{\sqrt{T}}
\leq 2\frac{\delta^\frac{1}{2}e^{2c^2T} V^8(t,x)}{\sqrt{T}}.
\label{a82}
\end{align}
Next,
\eqref{a01}, the triangle inequality, H\"older's inequality,  the fact that $\frac{8}{\exponentV}+\frac{1}{\exponentZ}\leq 1$,  \eqref{a02d}, and \eqref{a02b}
show for all $i\in[0,d]\cap\Z$, $t\in [0,T)$, 
 $x\in \R^d$ that
\begin{align} 
&
\int_{t}^{T}\Lambda_{i}(T-t)\left\lVert
\left(
f(r,\mathcal{X}^{t,x}_{r}, {u}(r,\mathcal{X}^{t,x}_{r}))
-f(r,\mathcal{X}^{t,x}_{r}, \mathfrak{u}(r,\mathcal{X}^{t,x}_{r}))
\right)\pr_i({Z}^{t,x}_{r})
\right\rVert_1dr\nonumber\\
&\leq 
\int_{t}^{T}\Lambda_{i}(T-t)
\sum_{\nu=0}^{d}\left[
L_\nu \Lambda_\nu(T)
\left\lVert
\pr_{\nu}
\left( {u}(r,\mathcal{X}^{t,x}_{r})
- \mathfrak{u}(r,\mathcal{X}^{t,x}_{r})
\right)\pr_i({Z}^{t,x}_{r})
\right\rVert_1\right]dr\nonumber\\
&\leq 
\int_{t}^{T}\Lambda_{i}(T-t)
\sum_{\nu=0}^{d}\left[
L_\nu \frac{\sqrt{T}}{\sqrt{T-r}}\Lambda_\nu(T-r)
\left\lVert
\pr_{\nu}
\left( {u}(r,\mathcal{X}^{t,x}_{r})
- \mathfrak{u}(r,\mathcal{X}^{t,x}_{r})
\right)\pr_i({Z}^{t,x}_{r})
\right\rVert_1\right]dr\nonumber\\
&\leq 
\int_{t}^{T}\Lambda_{i}(T-t)
\sum_{\nu=0}^{d}\left[
L_\nu \frac{\sqrt{T}}{\sqrt{T-r}}
\sup_{y\in \R^d}
\frac{\Lambda_\nu(T-r)\lvert
\pr_{\nu}( {u}(r,y) - \mathfrak{u}(r,y))\rvert}{V^8(r,y)}
\left\lVert
V^8(r,\mathcal{X}^{t,x}_{r})\pr_i({Z}^{t,x}_{r})
\right\rVert_1\right]dr\nonumber\\
&\leq 
\int_{t}^{T}\Lambda_{i}(T-t)
 \frac{c\sqrt{T}}{\sqrt{T-r}}\nonumber\\&\qquad\qquad\left[
\max_{\nu\in [0,d]\cap\Z}
\sup_{y\in \R^d}
\frac{\Lambda_\nu(T-r)\lvert
\pr_{\nu}( {u}(r,y) - \mathfrak{u}(r,y))\rvert}{V^8(r,y)}\right]\left\lVert
V^8(r,\mathcal{X}^{t,x}_{r})
\right\rVert_{\frac{\exponentV}{8}}
\left\lVert
\pr_i({Z}^{t,x}_{r})\right\rVert_{\exponentZ}
dr\nonumber\\
&\leq 
\int_{t}^{T}\Lambda_{i}(T-t)
 \frac{c\sqrt{T}}{\sqrt{T-r}}\nonumber\\&\qquad\qquad \left[
\max_{\nu\in [0,d]\cap\Z}
\sup_{y\in \R^d}
\frac{\Lambda_\nu(T-r)\lvert
\pr_{\nu}( {u}(r,y) - \mathfrak{u}(r,y))\rvert}{V^8(r,y)}\right]\xeqref{a02d}V^8(t,x)
\xeqref{a02b}
\frac{c}{\Lambda_{i}(r-t)}
dr\nonumber\\
 &\leq 
\int_{t}^{T}\left[
\max_{\nu\in [0,d]\cap\Z}
\sup_{y\in \R^d}
\frac{\Lambda_\nu(T-r)\lvert
\pr_{\nu}( {u}(r,y) - \mathfrak{u}(r,y))\rvert}{V^8(r,y)}\right]V^8(t,x)
 \frac{c^2\sqrt{T}}{\sqrt{T-r}}
\frac{\Lambda_{i}(T-t)}{\Lambda_{i}(r-t)}\,dr\nonumber
\\
 &\leq 
\int_{t}^{T}\left[
\max_{\nu\in [0,d]\cap\Z}
\sup_{y\in \R^d}
\frac{\Lambda_\nu(T-r)\lvert
\pr_{\nu}( {u}(r,y) - \mathfrak{u}(r,y))\rvert}{V^8(r,y)}\right]V^8(t,x)
 \frac{c^2\sqrt{T}}{\sqrt{T-r}}
\frac{\sqrt{T-t}}{\sqrt{r-t}}\,dr\nonumber\\
&\leq 
\int_{t}^{T}\left[
\max_{\nu\in [0,d]\cap\Z}
\sup_{y\in \R^d}
\frac{\Lambda_\nu(T-r)\lvert
\pr_{\nu}( {u}(r,y) - \mathfrak{u}(r,y))\rvert}{V^8(r,y)}\right]V^8(t,x)
 \frac{c^2T}{\sqrt{(T-r)(r-t)}}\,dr.
\label{a84a}
\end{align}
Furthermore,
\eqref{a04}, H\"older's inequality, the fact that
$\frac{1}{\exponentV}+\frac{1}{\exponentZ}\leq 1$, \eqref{a02d}, \eqref{a70}, and the fact that $\forall\,t\in [0,T)\colon \int_{t}^{T}\frac{dr}{\sqrt{r-t}}=2\sqrt{r-t}|_{r=t}^T=2\sqrt{T-t}$ imply for all 
$i\in[0,d]\cap\Z$, $t\in [0,T)$, 
 $x\in \R^d$
 that
\begin{align} 
&
\int_{t}^{T}\Lambda_{i}(T-t)\left\lVert
f(r,\mathcal{X}^{t,x}_{r},0)
\pr_i({Z}^{t,x}_{r}-\mathcal{Z}^{t,x}_{r})
\right\rVert_1dr\nonumber\\&\leq\xeqref{a04} \int_{t}^{T}\Lambda_{i}(T-t)\frac{1}{T}\left\lVert
V(r,\mathcal{X}^{t,x}_{r})
\pr_i({Z}^{t,x}_{r}-\mathcal{Z}^{t,x}_{r})
\right\rVert_1dr\nonumber\\
&\leq \frac{1}{T}\int_{t}^{T}\Lambda_{i}(T-t)\left\lVert
V(r,\mathcal{X}^{t,x}_{r})\right\rVert_{\exponentV}
\left\lVert
\pr_i({Z}^{t,x}_{r}-\mathcal{Z}^{t,x}_{r})
\right\rVert_{\exponentZ}dr\nonumber\\
&\leq 
\frac{1}{T}\int_{t}^{T}\Lambda_{i}(T-t)\xeqref{a02d}V(t,x)\xeqref{a70}\frac{\delta^\frac{1}{2}V(t,x)}{\sqrt{T}\Lambda_{i}(r-t)}\,dr\nonumber\\
&\leq \frac{\delta^\frac{1}{2}V^2(t,x)}{T\sqrt{T}}\int_{t}^{T}\frac{\sqrt{T-t}}{\sqrt{r-t}}\,dr\nonumber\\
&\leq 2\frac{\delta^\frac{1}{2}V^2(t,x)}{\sqrt{T}}.
\label{a83a}
\end{align}
Next,
the triangle inequality, \eqref{a38c}, the fact that
$\sum_{\nu=0}^{d}L_i\leq c$,
 H\"older's inequality, the fact that $\frac{1}{\exponentV}+\frac{1}{\exponentZ}\leq 1$, \eqref{a02d}, \eqref{a70}, and
\cref{a15}  prove for all 
$i\in[0,d]\cap\Z$, $t\in [0,T)$, 
 $x\in \R^d$ that
\begin{align} 
&
\int_{t}^{T}\Lambda_{i}(T-t)\sum_{\nu=0}^{d}\left[L_\nu
\left\lVert\Lambda_\nu(T)
\pr_\nu (
\mathfrak{u}(r,\mathcal{X}^{t,x}_{r}))
\pr_i({Z}^{t,x}_{r}-\mathcal{Z}^{t,x}_{r})
\right\rVert_1\right]dr\nonumber\\
&\leq 
\int_{t}^{T}
\Lambda_{i}(T-t)\sum_{\nu=0}^{d}\left[L_\nu
\frac{\sqrt{T}}{\sqrt{T-r}}
\left\lVert\Lambda_\nu(T-r)
\pr_\nu (
\mathfrak{u}(r,\mathcal{X}^{t,x}_{r}))
\pr_i({Z}^{t,x}_{r}-\mathcal{Z}^{t,x}_{r})
\right\rVert_1\right]dr\nonumber\\
&\leq \xeqref{a38c}
\int_{t}^{T}
\Lambda_{i}(T-t)
\frac{c\sqrt{T}}{\sqrt{T-r}}
\left\lVert 6ce^{86c^6T^3}
V(r,\mathcal{X}^{t,x}_{r})
\pr_i({Z}^{t,x}_{r}-\mathcal{Z}^{t,x}_{r})
\right\rVert_1 dr\nonumber\\
&\leq \int_{t}^{T}
\Lambda_{i}(T-t)
\frac{6c^2e^{86c^6T^3}\sqrt{T}}{\sqrt{T-r}}
\left\lVert V(r,\mathcal{X}^{t,x}_{r})\right\rVert_{\exponentV}
\left\lVert
\pr_i({Z}^{t,x}_{r}-\mathcal{Z}^{t,x}_{r})
\right\rVert_{\exponentZ}\nonumber\\
&\leq \int_{t}^{T}
\Lambda_{i}(T-t)
\frac{6c^2e^{86c^6T^3}\sqrt{T}}{\sqrt{T-r}}
\xeqref{a02d}
V(t,x)
\xeqref{a70}\frac{\delta^\frac{1}{2}V(t,x)}{\sqrt{T}\Lambda_i(r-t)}\nonumber\\
&\leq \int_{t}^{T}\frac{6\delta^\frac{1}{2}c^2e^{86c^6T^3}V^2(t,x)}{\sqrt{T}}
\frac{\sqrt{T}{\Lambda_{i}(T-t)}}{\sqrt{T-r}\Lambda_{i}(r-t)}\,dr\nonumber\\
&\leq \int_{t}^{T}\frac{6\delta^\frac{1}{2}c^2e^{86c^6T^3}V^2(t,x)}{\sqrt{T}}\frac{\sqrt{T}\sqrt{T-t}}{\sqrt{T-r}\sqrt{r-t}}\,dr\nonumber\\&\leq \frac{24\delta^\frac{1}{2}c^2Te^{86c^6T^3}V^2(t,x)}{\sqrt{T}}.
\label{a84}
\end{align}
This,
the triangle inequality, \eqref{a83a}, and the fact that
$
24e^{86c^6T^3}\leq V$ show for all 
$i\in[0,d]\cap\Z$, $t\in [0,T)$, 
$r\in (0,T]$,
 $x\in \R^d$ that
\begin{align} 
&
\int_{t}^{T}\Lambda_{i}(T-t)\left\lVert
f(r,\mathcal{X}^{t,x}_{r}, \mathfrak{u}(r,\mathcal{X}^{t,x}_{r}))
\pr_i({Z}^{t,x}_{r}-\mathcal{Z}^{t,x}_{r})
\right\rVert_1dr\nonumber\\
&\leq 
\int_{t}^{T}\Lambda_{i}(T-t)\left\lVert
f(r,\mathcal{X}^{t,x}_{r},0)
\pr_i({Z}^{t,x}_{r}-\mathcal{Z}^{t,x}_{r})
\right\rVert_1dr\nonumber\\
&\quad 
+
\int_{t}^{T}\Lambda_{i}(T-t)\sum_{\nu=0}^{d}\left[L_\nu
\Lambda_\nu(T)
\left\lVert
\pr_\nu (
\mathfrak{u}(r,\mathcal{X}^{t,x}_{r}))
\pr_i({Z}^{t,x}_{r}-\mathcal{Z}^{t,x}_{r})
\right\rVert_1\right]dr\nonumber
\\
&
\leq \xeqref{a83a}2\frac{\delta^\frac{1}{2}V^2(t,x)}{\sqrt{T}}+\xeqref{a84}\frac{24\delta^\frac{1}{2}c^2Te^{86c^6T^3}V^2(t,x)}{\sqrt{T}}\nonumber\\
&\leq 
\frac{24\delta^\frac{1}{2}e^{c^2T}e^{86c^6T^3}V^2(t,x)}{\sqrt{T}}\nonumber\\&\leq 
\frac{\delta^\frac{1}{2}e^{c^2T}V^3(t,x)}{\sqrt{T}}.
\label{a87}
\end{align}
Next, 
\eqref{a05c}, \eqref{a05d}, and the triangle inequality
prove for all 
$i\in[0,d]\cap\Z$, $t\in [0,T)$, 
 $x\in \R^d$ that
\begin{align} 
&
\Lambda_{i}(T-t)
\left \lvert\pr_i (u(t,x)-\mathfrak{u}(t,x))\right\rvert\nonumber
\\&\leq 
\Lambda_{i}(T-t)\left \lvert\E\!\left[g({X}^{t,x}_{T} )\pr_i({Z}^{t,x}_{T})
-g(\mathcal{X}^{t,x}_{T} )\pr_i(\mathcal{Z}^{t,x}_{T}) \right] \right\rvert\nonumber\\
&\quad + \int_{t}^{T}\left\lvert
\Lambda_{i}(T-t)
\E \!\left[
f(r,{X}^{t,x}_{r}, u(r,{X}^{t,x}_{r}))
\pr_i({Z}^{t,x}_{r})-
f(r,\mathcal{X}^{t,x}_{r}, \mathfrak{u}(r,\mathcal{X}^{t,x}_{r}))
\pr_i(\mathcal{Z}^{t,x}_{r})\right]\right\rvert dr\nonumber\\
&\leq \Lambda_{i}(T-t)\left\lVert
\left(
g({X}^{t,x}_{T} ) 
-g(\mathcal{X}^{t,x}_{T} )\right)\pr_i
({Z}^{t,x}_{T})
\right\rVert_1
+\Lambda_{i}(T-t)
\left\lVert
g(\mathcal{X}^{t,y}_{T} )\pr_i\!\left({Z}^{t,x}_{T}-                \mathcal{Z}^{t,x}_{T}\right)\right\rVert_1\nonumber\\
&\quad +\int_{t}^{T}\Lambda_{i}(T-t)\left\lVert
\left(
f(r,{X}^{t,x}_{r}, u(r,{X}^{t,x}_{r}))-
f(r,\mathcal{X}^{t,x}_{r}, \mathfrak{u}(r,\mathcal{X}^{t,x}_{r}))\right)\pr_i({Z}^{t,x}_{r})
\right\rVert_1dr\nonumber\\
&\quad+\int_{t}^{T}\Lambda_{i}(T-t)\left\lVert
f(r,\mathcal{X}^{t,x}_{r}, \mathfrak{u}(r,\mathcal{X}^{t,x}_{r}))
\pr_i({Z}^{t,x}_{r}-\mathcal{Z}^{t,x}_{r})
\right\rVert_1dr\nonumber\\
&\leq \Lambda_{i}(T-t)\left\lVert
\left(
g({X}^{t,x}_{T} ) 
-g(\mathcal{X}^{t,x}_{T} )\right)\pr_i
({Z}^{t,x}_{T})
\right\rVert_1
+\Lambda_{i}(T-t)
\left\lVert
g(\mathcal{X}^{t,y}_{T} )\pr_i\!\left({Z}^{t,x}_{T}-                \mathcal{Z}^{t,x}_{T}\right)\right\rVert_1\nonumber\\
&\quad +\int_{t}^{T}\Lambda_{i}(T-t)\left\lVert
\left(
f(r,{X}^{t,x}_{r}, u(r,{X}^{t,x}_{r}))-
f(r,\mathcal{X}^{t,x}_{r}, {u}(r,\mathcal{X}^{t,x}_{r}))\right)\pr_i({Z}^{t,x}_{r})
\right\rVert_1dr\nonumber\\
&\quad +\int_{t}^{T}\Lambda_{i}(T-t)\left\lVert
\left(
f(r,\mathcal{X}^{t,x}_{r}, {u}(r,\mathcal{X}^{t,x}_{r}))
-f(r,\mathcal{X}^{t,x}_{r}, \mathfrak{u}(r,\mathcal{X}^{t,x}_{r}))
\right)\pr_i({Z}^{t,x}_{r})
\right\rVert_1dr\nonumber\\
&\quad+\int_{t}^{T}\Lambda_{i}(T-t)\left\lVert
f(r,\mathcal{X}^{t,x}_{r}, \mathfrak{u}(r,\mathcal{X}^{t,x}_{r}))
\pr_i({Z}^{t,x}_{r}-\mathcal{Z}^{t,x}_{r})
\right\rVert_1dr.
\end{align}
Thus, \eqref{a79}, \eqref{a80}, \eqref{a82}, \eqref{a84a}, and \eqref{a87} imply for all 
$i\in[0,d]\cap\Z$, $t\in [0,T)$, 
 $x\in \R^d$ that
\begin{align} 
&
\Lambda_{i}(T-t)
\left \lvert\pr_i (u(t,x)-\mathfrak{u}(t,x))\right\rvert\nonumber\\&\leq\xeqref{a79}\frac{{\delta}^\frac{1}{2} V^8(t,x)}{\sqrt{T}}+
\xeqref{a80}
\frac{\delta^\frac{1}{2}V^2(t,x)}{\sqrt{T}}+\xeqref{a82}2\frac{\delta^\frac{1}{2}e^{2c^2T} V^8(t,x)}{\sqrt{T}}\nonumber\\
&\quad +\xeqref{a84a}\int_{t}^{T}\left[
\max_{\nu\in [0,d]\cap\Z}
\sup_{y\in \R^d}
\frac{\Lambda_\nu(T-r)
\pr_{\nu}( {u}(r,y) - \mathfrak{u}(r,y))}{V^8(r,y)}\right]V^8(t,x)
 \frac{c^2T}{\sqrt{(T-r)(r-t)}}\,dr\nonumber\\
&\quad +\xeqref{a87}
\frac{\delta^\frac{1}{2}e^{c^2T}V^3(t,x)}{\sqrt{T}}\nonumber\\
&\leq 5 \frac{\delta^\frac{1}{2}e^{c^2T}V^8(t,x)}{\sqrt{T}} +\int_{t}^{T}\left[
\max_{\nu\in [0,d]\cap\Z}
\sup_{y\in \R^d}
\frac{\Lambda_\nu(T-r)
\pr_{\nu}( {u}(r,y) - \mathfrak{u}(r,y))}{V^8(r,y)}\right]V^8(t,x)
 \frac{c^2T}{\sqrt{(T-r)(r-t)}}\,dr.
\end{align}
This shows for all $i\in[0,d]\cap\Z$, $t\in [0,T)$  that
\begin{align} \begin{split} 
&
\max_{\nu\in [0,d]\cap\Z}
\sup_{y\in \R^d}
\frac{\Lambda_\nu(T-t)\lvert
\pr_{\nu}( {u}(t,y) - \mathfrak{u}(t,y))\rvert}{V^8(t,y)}\\
&\leq 5 \frac{\delta^\frac{1}{2}e^{c^2T}}{\sqrt{T}} +\int_{t}^{T}\left[
\max_{\nu\in [0,d]\cap\Z}
\sup_{y\in \R^d}
\frac{\Lambda_\nu(T-r)\lvert
\pr_{\nu}( {u}(r,y) - \mathfrak{u}(r,y))\rvert}{V^8(r,y)}\right]
 \frac{c^2T}{\sqrt{(T-r)(r-t)}}\,dr.
\end{split}
\end{align}
Therefore, \eqref{b01c}, \eqref{b01d}, and the Gr\"onwall-type inequality in \cref{a16} imply
for all 
 $t\in [0,T)$ that
\begin{align}
\max_{\nu\in [0,d]\cap\Z}
\sup_{y\in \R^d}
\frac{\Lambda_\nu(T-r)\lvert
\pr_{\nu}( {u}(t,y) - \mathfrak{u}(t,y))\rvert}{V^8(t,y)}\leq 10 \frac{\delta^\frac{1}{2}e^{c^2T}}{\sqrt{T}}e^{86c^6T^3}.
\end{align}
Hence, the fact that $10 e^{86c^6T^3}\leq V$
proves for all
$\nu\in [0,d]\cap \Z$, $t\in [0,T)$, $y\in \R^d$ that
\begin{align}
\Lambda_\nu(T-t)\lvert
\pr_{\nu}( {u}(t,y) - \mathfrak{u}(t,y))\rvert
\leq 10 \frac{\delta^\frac{1}{2}e^{c^2T}}{\sqrt{T}}e^{86c^6T^3}V^8(t,y)\leq \frac{\delta^\frac{1}{2}e^{c^2T}}{\sqrt{T}}V^9(t,y).
\end{align}
This shows \eqref{a92} and completes the proof of \cref{a37a}. 
\end{proof}
\section{MLP approximations}\label{s04}
In \cref{a41b} below we introduce MLP approximations for solutions to SFPEs (see \eqref{a04d}).

\begin{setting}\label{a41b}
Let $d\in \N$, $\Theta=\cup_{n\in \N}\Z^n$,
$T\in (0,\infty)$,
$\exponentV,\exponentZ, \exponentX\in (1,\infty)$,
$c\in [1,\infty)$, $(L_i)_{i\in [0,d]\cap \Z}\in \R^{d+1}$ satisfy that $
\sum_{i=0}^{d}L_i\leq c$. 
Let $\lVert\cdot \rVert\colon \R^d\to [0,\infty)$ be a norm on $\R^d$. 
Let 
 $\Lambda=(\Lambda_{\nu})_{\nu\in [0,d]\cap\Z}\colon [0,T]\to \R^{1+d}$ satisfy for all $t\in [0,T]$ that $\Lambda(t)=(1,\sqrt{t},\ldots,\sqrt{t})$. 
Let $\pr=(\pr_\nu)_{\nu\in [0,d]\cap\Z}\colon \R^{d+1}\to\R$ satisfy
for all 
$w=(w_\nu)_{\nu\in [0,d]\cap\Z}$,
$i\in [0,d]\cap\Z$ that
$\pr_i(w)=w_i$.
Let $f\in C( [0,T)\times\R^d\times \R^{d+1},\R)$, $g\in C(\R^d,\R)$, 
$V\in C([0,T)\times \R^d, [0,\infty) )$ satisfy that
$\max \{ c,48e^{86c^6T^3}\}\leq V$. 
To shorten the notation we write 
for all 
$t\in [0,T)$, $x\in \R^d$,
$w\colon[0,T)\times\R^d\to \R^{d+1} $ that
\begin{align}
(F(w))(t,x)=f(t,x,w(t,x)).
\end{align}
Let $\varrho\colon \{(\tau,\sigma)\in [0,T)^2\colon\tau<\sigma \}\to\R$ satisfy for all $t\in [0,T)$, $s\in (t,T)$ that
\begin{align}
\varrho(t,s)=\frac{1}{\mathrm{B}(\tfrac{1}{2},\tfrac{1}{2})}\frac{1}{\sqrt{(T-s)(s-t)}}.\label{a98}
\end{align}
Let $ (\Omega,\mathcal{F},\P)$ be a probability space.
For every random variable $\mathfrak{X}\colon \Omega\to\R$,
$s\in [1,\infty)$ let $\lVert \mathfrak{X}\rVert_s\in [0,\infty]$ satisfy that 
$ \lVert \mathfrak{X}\rVert_s= (\E[\lvert \mathfrak{X}\rvert^s])^\frac{1}{s}$.
Let $\mathfrak{r}^\theta\colon \Omega\to (0,1) $, $\theta\in \Theta$, be independent and identically distributed\ random variables and satisfy for all
$b\in (0,1)$
that 
\begin{align}
\P(\mathfrak{r}^0\leq b)=
\frac{1}{\mathrm{B}(\tfrac{1}{2},\tfrac{1}{2})}
\int_{0}^b\frac{dr}{\sqrt{r(1-r)}}.
\end{align}
Let $\mathcal{X}^\theta=
(\mathcal{X}^{\theta,s,x}_t)_{s\in [0,T],t\in[s,T],x\in\R^d}\colon \{(\sigma,\tau)\in [0,T]^2\colon \sigma\leq \tau\}\times \R^d\times \Omega\to\R^d $, $\theta\in \Theta$, be measurable.
Let 
$\mathcal{Z}^\theta=
(\mathcal{Z}^{\theta,s,x}_t)_{s\in [0,T),t\in(s,T],x\in\R^d}\colon \{(\sigma,\tau)\in [0,T]^2\colon \sigma< \tau\}\times \R^d\times \Omega\to\R^{d+1} $, $\theta\in \Theta$, be measurable.
Assume that
$(\mathcal{X}^\theta, \mathcal{Z}^\theta)
$, $\theta\in \Theta$, are independent and identically distributed.
Assume that
$
(\mathcal{X}^\theta, \mathcal{Z}^\theta)_{\theta\in \Theta}
$ and
$(\mathfrak{r}^\theta)_{\theta\in \Theta}$  are independent. Assume
for all 
$i\in [0,d]\cap\Z$,
$s\in [0,T)$,
$t\in [s,T)$, $r\in (t,T]$,
 $x\in \R^d$, $w_1,w_2\in \R^{d+1}$
 that
\begin{align}\label{a04e}
\lvert g(x)\rvert\leq V(T,x),\quad \lvert Tf(t,x,0)\rvert\leq V(t,x),
\end{align}
\begin{align}
\lvert
f(t,x,w_1)-f(t,y,w_2)\rvert\leq \sum_{\nu=0}^{d}\left[
L_\nu\Lambda_\nu(T)
\lvert\pr_\nu(w_1-w_2) \rvert\right]
+\frac{1}{T}\frac{V(t,x)+V(t,y)}{2}\frac{\lVert x-y\rVert}{\sqrt{T}}
,\label{a01e}
\end{align}
\begin{align}\label{a02e}
\left\lVert
V(r,\mathcal{X}^{0,t,x}_r)
\right\rVert_{\exponentV}
\leq V(t,x),\quad 
\left\lVert\left\lVert\mathcal{X}^{0,s,x}_{t} -x\right\rVert\right\rVert_\exponentX\leq V(s,x)\sqrt{t-s},
\quad \left\lVert
\pr_i(\mathcal{Z}^{0,t,x}_r)\right\rVert_{\exponentZ}
\leq \frac{c}{\Lambda_{i}(r-t)},
\end{align}
\begin{align}V(T,x)\leq V(t,x),\quad 
\lvert g(x)-g(y)\rvert\leq \frac{V(T,x)+V(T,y)}{2}
\frac{\lVert x-y\rVert}{\sqrt{T}},\label{a19d}
\end{align}
\begin{align}
\P\!\left(\pr_0(\mathcal{Z}_{t}^{0,s,x})=1\right)=1,\quad \E\!\left[\mathcal{Z}_{t}^{0,s,x}\right]=(1,0,\ldots,0).\label{b03}
\end{align} 
Let $
U^{\theta}_{n,m}\colon [0,T)\times \R^d\to \R^{d+1}$, $n,m\in \Z$, $\theta\in \Theta$, satisfy for all $n,m\in \N$, 
$\theta\in \Theta$,
$t\in [0,T)$, $x\in\R^d$ that
$U_{-1,m}^\theta(t,x)=U^\theta_{0,m}(t,x)=0$ and
\begin{align} \begin{split} 
&
U^{\theta}_{n,m}(t,x)= (g(x),0)+\sum_{i=1}^{m^n}
\frac{g(\mathcal{X}^{(\theta,0,-i),t,x}_T )-g(x) }{m^n}\mathcal{Z}^{(\theta,0,-i),t,x}_{T}\\
& +\sum_{\ell=0}^{n-1}\sum_{i=1}^{m^{n-\ell}} \frac{\left(F(U^{(\theta,\ell,i)}_{\ell,m})-
\1_\N(\ell)
F(U^{(\theta,\ell,-i)}_{\ell-1,m})\right)\!\left(t+(T-t) \mathfrak{r}^{(\theta,\ell,i)},
\mathcal{X}^{(\theta,\ell,i),t,x}_{t+(T-t) \mathfrak{r}^{(\theta,\ell,i)}}\right)
\mathcal{Z}^{(\theta,\ell,i),t,x}_{t+(T-t) \mathfrak{r}^{(\theta,\ell,i)}}}{m^{n-\ell}\varrho(t,t+(T-t)\mathfrak{r}^{(\theta,\ell,i)})}.\end{split}\label{a04d}
\end{align}
\end{setting}

In \cref{a02} below we first study independence and distributional properties of MLP approximations.
\begin{lemma}[Independence and distributional properties]\label{a02}
Assume \cref{a41b}. Then the following items hold.
\begin{enumerate}[(i)]
\item \label{i01}We have for all $n\in \N_0$, $m\in \N$, $\theta\in \Theta$ that $U^\theta_{n,m} $ is measurable,
\item\label{i02} We have for all 
$n\in \N_0$, $m\in \N$, $\theta\in \Theta$ that
\begin{align}
&
\sigma( \{U^\theta_{n,m}(t,x)\colon  t\in [0,T), x\in \R^d \})\nonumber\\
&
\subseteq \sigma(
\{
\mathfrak{r}^{(\theta,\nu)},
\mathcal{X}^{(\theta,\nu),s,x}_{t}, 
\mathcal{Z}^{(\theta,\nu),s,x}_{t}
\colon 
\nu\in \Theta, s\in [0,T), t\in (s,T] , x\in \R^d\})
\end{align}
\item\label{i03} We have for all ${\theta\in \Theta}$, $m\in \N$ that
$(U^{\theta,\ell, i}_{\ell,m})_{t\in [0,T),x\in \R^d}$,
$(U^{\theta,\ell, -i}_{\ell-1,m})_{t\in [0,T),x\in \R^d}$,
$
((\mathcal{X}^{(\theta,\ell,i),s,x}_t)_{s\in [0,T],t\in[s,T],x\in\R^d}, 
(\mathcal{Z}^{(\theta,\ell,i),s,x}_t)_{s\in [0,T),t\in(s,T],x\in\R^d}   )
$,
$
\mathfrak{r}^{(\theta,\ell,i)}$, $i\in \N$, $\ell\in \N_0$, are independent,
\item \label{i04}We have for all $n\in \N_0$, $m\in \N$ that 
$(U^\theta_{n,m}(t,x ))_{t\in [0,T),x\in\R^d}$, $\theta\in \Theta$, are identically distributed.
\item\label{i05} We have for all $\theta\in \Theta$, $\ell\in \N_0$, $m\in \N$, $t\in [0,T)$, $x\in \R^d$ that 
\begin{align}
\frac{\left(F(U^{(\theta,\ell,i)}_{\ell,m})-
\1_\N(\ell)
F(U^{(\theta,\ell,-i)}_{\ell-1,m})\right)\!\left(t+(T-t) \mathfrak{r}^{(\theta,\ell,i)},
\mathcal{X}^{(\theta,\ell,i),t,x}_{t+(T-t) \mathfrak{r}^{(\theta,\ell,i)}}\right)
\mathcal{Z}^{(\theta,\ell,i),t,x}_{t+(T-t) \mathfrak{r}^{(\theta,\ell,i)}}}{\varrho(t,t+(T-t)\mathfrak{r}^{(\theta,\ell,i)})}, \quad i\in \N,
\end{align}
are independent and identically distributed and have the same distribution as 
\begin{align}
\frac{\left(F(U^{0}_{\ell,m})-
\1_\N(\ell)
F(U^{1}_{\ell-1,m})\right)\!\left(t+(T-t) \mathfrak{r}^{0},
\mathcal{X}^{0,t,x}_{t+(T-t) \mathfrak{r}^{0}}\right)
\mathcal{Z}^{0,t,x}_{t+(T-t) \mathfrak{r}^{0}}}{\varrho(t,t+(T-t)\mathfrak{r}^{0})}, \quad i\in \N.
\end{align}
\end{enumerate}
\end{lemma}

\begin{proof}[Proof of \cref{a02}]
The assumptions on measurability and distributions, basic properties of measurable functions, and induction prove 
\eqref{i01} and \eqref{i02}. In addition, \eqref{i02} and the assumptions on independence prove \eqref{i03}. Furthermore, \eqref{i03}, the fact that $\forall\theta\in \Theta, m\in \N\colon U^\theta_{0,m}=0$, \eqref{a04d}, the disintegration theorem, the assumptions on distributions, and induction establish \eqref{i04} and \eqref{i05}.
\end{proof}

In \cref{a20} below we establish an
upper bounds for the $L^2$-distances between the exact solutions of the considered stochastic fixed point equations and the proposed MLP approximations. Our main idea here is the use of a family of semi-norms which allows us to get the  recursive inequality 
\eqref{a05f}.

\begin{proposition}[Error analysis by semi-norms]\label{a20}
Assume \cref{a41b}. Let $\exponentFirstNorm\in [3,\infty)$.
Assume that
$\frac{1}{\exponentV}+\frac{1}{\exponentX}+\frac{1}{\exponentZ}\leq \frac{1}{2}$. For every random field $H\colon [0,T)\times \R^d\times \Omega\to \R^{d+1}$ let $\threenorm{H}_s$, $s\in [0,T)$, satisfy for all  $s\in [0,T)$ that
\begin{align}
\threenorm{H}_{s}=\max_{\nu\in [0,d]\cap\Z}\sup_{r\in [s,T),x\in \R^d}\frac{\Lambda_\nu(T-r)\left\lVert\pr_{\nu} (H(r,x))\right\rVert_2}{V^\exponentFirstNorm(r,x)}. \label{a25}
\end{align}
Then the following items hold.
\begin{enumerate}[(i)]
\item \label{a17c}
There exists a unique measurable function $\mathfrak{u}\colon [0,T)\times \R^d\to \R^{d+1}$ such that for all $t\in [0,T)$, $x\in \R^d$
we have that
\begin{align}\label{b01e}
\max_{\nu\in [0,d]\cap\Z}\sup_{\tau\in [0,T), \xi\in \R^d}
\left[\Lambda_\nu(T-\tau)\frac{\lvert\pr_\nu(\mathfrak{u}(\tau,\xi))\rvert}{V(\tau ,\xi)}\right]<\infty,
\end{align}
\begin{align}\max_{\nu\in [0,d]\cap\Z}\left[
\E\!\left [\left\lvert g(\mathcal{X}^{0,t,x}_{T} )\pr_\nu(\mathcal{Z}^{0,t,x}_{T})\right\rvert \right] + \int_{t}^{T}
\E \!\left[\left\lvert
f(r,\mathcal{X}^{0,t,x}_{r},\mathfrak{u}(r,\mathcal{X}^{0,t,x}_{r}))\pr_\nu(\mathcal{Z}^{0,t,x}_{r})\right\rvert\right]dr\right]<\infty,
\end{align}
and
\begin{align}
\mathfrak{u}(t,x)=\E\!\left [g(\mathcal{X}^{0,t,x}_{T} )\mathcal{Z}^{0,t,x}_{T} \right] + \int_{t}^{T}
\E \!\left[
f(r,\mathcal{X}^{0,t,x}_{r},\mathfrak{u}(r,\mathcal{X}^{0,t,x}_{r}))\mathcal{Z}^{0,t,x}_{r}\right]dr.\label{a05e}
\end{align}
\item \label{a42}
For all $n,m\in \N$, $t\in [0,T)$ we have that
\begin{align}\label{a05f}
\threenorm{U^0_{n,m}-\mathfrak{u}}_t\leq 
\frac{4}{\sqrt{m^n}}+\sum_{\ell=0}^{n-1}\left[
\frac{8
c^2T^\frac{5}{6}}{\sqrt{m^{n-\ell-1}}}\
\left[\int_{t}^{T}
\threenorm{U^{0}_{\ell,m}-\mathfrak{u}}_s^{6}\right]^{\frac{1}{6}}
\right].
\end{align}
\item \label{a43}For all
$n,m\in \N$, $t\in [0,T)$ we have that
\begin{align}
\threenorm{U^0_{n,m}-\mathfrak{u}}_t
\leq 
e^\frac{m^3}{6}
m^{-\frac{n}{2}}8^ne^{nc^2T}.
\end{align} 
\item \label{b40} For all
$n,m\in \N$, $t\in [0,T)$, $\nu\in [0,d]\cap\Z$ we have that
\begin{align}
\Lambda_\nu(T-t)\left\lVert\pr_{\nu}\!\left(U^0_{n,m}(t,x)-\mathfrak{u}(t,x)\right)\right\rVert_2\leq 
e^\frac{m^3}{6}
m^{-\frac{n}{2}}8^ne^{nc^2T} V^\exponentFirstNorm(t,x).
\end{align}
\end{enumerate}
\end{proposition}
\begin{proof}
[Proof of \cref{a20}]First,
\cref{a03} and the assumption of \cref{a20} show \eqref{a17c}
and imply that 
for all $t\in [0,T)$ we have that
\begin{align}\label{a22}
\max_{\nu\in [0,d]\cap\Z}\sup_{y\in \R^d}
\left[\Lambda_\nu(T-t)\frac{\lvert\pr_\nu(\mathfrak{u}(t,y))\rvert}{V(t,y)}\right]
\leq 
6c e^{86c^6T^2(T-t)}.
\end{align}
Thus, the fact that $\max \{c,6e^{86c^6T^3}\}\leq V$ implies for all 
$\nu\in[0,d]\cap\Z$, $t\in [0,T)$, $y\in \R^d$ that
$\Lambda_\nu(T-t)\lvert\pr_\nu(\mathfrak{u}(t,y))\rvert\leq V^3(t,y)$. This, \eqref{a25}, and the fact that $\exponentFirstNorm\geq 3$ prove for all $t\in [0,T)$ that 
\begin{align}
\threenorm{\mathfrak{u}}_t\leq 1.\label{a27}
\end{align}
Next, \eqref{a04e}, the fact that $\exponentV\geq 2$, Jensen's inequality, and \eqref{a02e} show for all $t\in [0,T)$, $x\in \R^d$ that 
\begin{align}
\left\lVert g(\mathcal{X}^{0,t,x}_T )\right\rVert_2\leq\xeqref{a04e} 
\left\lVert V(T,\mathcal{X}^{0,t,x}_T )\right\rVert_2
\leq \left\lVert V(T,\mathcal{X}^{0,t,x}_T )\right\rVert_{\exponentV}
\leq\xeqref{a02e} V(t,x).\label{b04}
\end{align}
Furthermore, the definition of $\Lambda$, \eqref{a19d}, H\"older's inequality,
the fact that $\frac{1}{\exponentV}+\frac{1}{\exponentX}+\frac{1}{\exponentZ}\leq \frac{1}{2}$, and \eqref{a02e} prove for all
$\nu\in [1,d]\cap\Z$, $t\in [0,T)$, $x\in \R^d$ that
\begin{align}
&
\left\lVert\Lambda_\nu(T-t)
(g(\mathcal{X}^{0,t,x}_T )-g(x))\pr_\nu(\mathcal{Z}^{0,t,x}_{T})
\right\rVert_2\nonumber\\
&\leq 
\left\lVert\sqrt{T-t}\xeqref{a19d}
\frac{V(T,\mathcal{X}^{0,t,x}_T)+V(T,x)}{2}
\frac{\left\lVert
\mathcal{X}^{0,t,x}_T -x\right\rVert}{\sqrt{T}}\pr_\nu(\mathcal{Z}^{0,t,x}_{T})
\right\rVert_2\nonumber\\
&\leq \sqrt{T-t}
\frac{\left\lVert V(T,\mathcal{X}^{0,t,x}_T)\right\rVert_{\exponentV}+V(t,x)}{2}
\frac{\left\lVert\left\lVert
\mathcal{X}^{0,t,x}_T -x\right\rVert\right\rVert_{\exponentX}}{\sqrt{T}}
\left\lVert\pr_\nu(\mathcal{Z}^{0,t,x}_{T})\right\rVert_{\exponentZ}\nonumber\\
&\leq \sqrt{T-t}V(t,x)\frac{V(t,x)\sqrt{T-t}}{\sqrt{T}}
\frac{c}{\sqrt{T-t}}\nonumber\\
&\leq V^3(t,x).\label{b16}
\end{align}
In addition, \eqref{b03}, the triangle inequality,  \eqref{b04},
the independence and distributional properties (cf. \cref{a02}), and a standard property of the variance imply for all 
$m,n\in \N$, $\theta\in \Theta$,
$t\in [0,T)$, $x\in \R^d$ that 
\begin{align}
&
\left\lVert\Lambda_0(T-t)
\pr_0 \left(
(g(x),0)+\sum_{i=1}^{m^n}
\frac{g(\mathcal{X}^{(\theta,0,-i),t,x}_T )-g(x) }{m^n}\mathcal{Z}^{(\theta,0,-i),t,x}_{T}
\right)\right\rVert_2\nonumber\\
&=\left\lVert
\sum_{i=1}^{m^n}
\frac{g(\mathcal{X}^{(\theta,0,-i),t,x}_T )}{m^n}
\right\rVert_2\leq \left\lVert g(\mathcal{X}^{0,t,x}_T )\right\rVert_2\leq V(t,x)\label{b17}
\end{align}
and 
\begin{align}
&\left(
\var\! \left[\Lambda_0(T-t)
\pr_0 \left(
(g(x),0)+\sum_{i=1}^{m^n}
\frac{g(\mathcal{X}^{(\theta,0,-i),t,x}_T )-g(x) }{m^n}\mathcal{Z}^{(\theta,0,-i),t,x}_{T}
\right)\right]\right)^\frac{1}{2}\nonumber\\
&=\left(\var\!
\left[
\sum_{i=1}^{m^n}
\frac{g(\mathcal{X}^{(\theta,0,-i),t,x}_T )}{m^n}
\right]
\right)^\frac{1}{2}
=\frac{\left(\var\!\left[ g(\mathcal{X}^{0,t,x}_T )\right]\right)^\frac{1}{2}}{\sqrt{m^n}}
\leq 
\frac{\left\lVert g(\mathcal{X}^{0,t,x}_T )\right\rVert_2}{\sqrt{m^n}}\leq \frac{V(t,x)}{\sqrt{m^n}}.\label{b18}
\end{align}
Next, \eqref{b03}, the triangle inequality, \eqref{b16}, the independence and distributional properties (cf. \cref{a02}), and a standard property of the variance show for all
$m,n\in \N$, $\theta\in \Theta$,
$t\in [0,T)$, $x\in \R^d$, $\nu\in [1,d]\cap\Z$ that
\begin{align}
&\left\lVert\Lambda_\nu(T-t)
\pr_\nu \left(
(g(x),0)+\sum_{i=1}^{m^n}
\frac{g(\mathcal{X}^{(\theta,0,-i),t,x}_T )-g(x) }{m^n}\mathcal{Z}^{(\theta,0,-i),t,x}_{T}
\right)\right\rVert_{2}\nonumber\\&=\left\lVert\Lambda_\nu(T-t)\sum_{i=1}^{m^n}
\frac{g(\mathcal{X}^{(\theta,0,-i),t,x}_T )-g(x) }{m^n}\pr_\nu(\mathcal{Z}^{(\theta,0,-i),t,x}_{T})\right\rVert_{2}\nonumber\\
&\leq \left\lVert\Lambda_\nu(T-t)
(g(\mathcal{X}^{0,t,x}_T )-g(x))\pr_\nu(\mathcal{Z}^{0,t,x}_{T})
\right\rVert_2\leq V^3(t,x)\label{a19a}
\end{align}
and
\begin{align}
&\left(\var\!\left[\Lambda_\nu(T-t)
\pr_\nu \left(
(g(x),0)+\sum_{i=1}^{m^n}
\frac{g(\mathcal{X}^{(\theta,0,-i),t,x}_T )-g(x) }{m^n}\mathcal{Z}^{(\theta,0,-i),t,x}_{T}
\right)\right]\right)^\frac{1}{2}\nonumber\\
&=\left(\var\!\left[\Lambda_\nu(T-t)\sum_{i=1}^{m^n}
\frac{g(\mathcal{X}^{(\theta,0,-i),t,x}_T )-g(x) }{m^n}\pr_\nu(\mathcal{Z}^{(\theta,0,-i),t,x}_{T})\right]\right)^\frac{1}{2}\nonumber\\
&=\left(\frac{\var\!\left[\Lambda_\nu(T-t)(g(\mathcal{X}^{0,t,x}_T )-g(x))\pr_\nu(\mathcal{Z}^{0,t,x}_{T})\right]}{m^n}
\right)^\frac{1}{2}
\nonumber
\\
&\leq \frac{\left\lVert\Lambda_\nu(T-t)
(g(\mathcal{X}^{0,t,x}_T )-g(x))\pr_\nu(\mathcal{Z}^{0,t,x}_{T})
\right\rVert_2}{\sqrt{m^n}}\leq\frac{ V^3(t,x)}{\sqrt{m^n}}.\label{a20a}
\end{align}
This, \eqref{b17}, \eqref{b18}, and the fact that $1\leq V$ imply 
for all
$m,n\in \N$, $\theta\in \Theta$,
$t\in [0,T)$, $x\in \R^d$, $\nu\in [0,d]\cap\Z$ that

\begin{align}
&\left\lVert\Lambda_\nu(T-t)
\pr_\nu \left(
(g(x),0)+\sum_{i=1}^{m^n}
\frac{g(\mathcal{X}^{(\theta,0,-i),t,x}_T )-g(x) }{m^n}\mathcal{Z}^{(\theta,0,-i),t,x}_{T}
\right)\right\rVert_{2}\leq V^3(t,x)\label{a19b}
\end{align}
and
\begin{align}
&\left(\var\!\left[\Lambda_\nu(T-t)
\pr_\nu \left(
(g(x),0)+\sum_{i=1}^{m^n}
\frac{g(\mathcal{X}^{(\theta,0,-i),t,x}_T )-g(x) }{m^n}\mathcal{Z}^{(\theta,0,-i),t,x}_{T}
\right)\right]\right)^\frac{1}{2} \leq\frac{ V^3(t,x)}{\sqrt{m^n}}.
\end{align}
Furthermore, H\"older's inequality, the fact that
$\frac{1}{\exponentV}+\frac{1}{\exponentZ}\leq \frac{1}{2}$,  \eqref{a02e}, and the fact that $c\leq V$ show for all
$t\in [0,T)$, 
$s\in (t,T]$, 
$x\in \R^d$, $\nu\in [0,d]\cap\Z$ that
\begin{align}
\left\lVert
V(s,
\mathcal{X}^{0,t,x}_{s})
\pr_\nu(\mathcal{Z}^{0,t,x}_{s})\right\rVert_2&\leq 
\left\lVert V(s,
\mathcal{X}^{0,t,x}_{s})\right\rVert_{\exponentV}
\left\lVert\pr_\nu(\mathcal{Z}^{0,t,x}_{s})\right\rVert_\exponentZ\nonumber\\
&
\leq \xeqref{a02e}V(t,x)\xeqref{a02e}\frac{c}{\Lambda_\nu(s-t)}
\leq \frac{V^2(t,x)}{\Lambda_\nu(s-t)}.\label{a23}
\end{align}
Next,
the substitution $s=t+(T-t)r$, $ds=(T-t)dr$, $r=0\Rightarrow s=t$, $r= \frac{b-t}{T-t}\Rightarrow s =b$,
$r=\frac{s-t}{T-t}$, $1-r=1-\frac{s-t}{T-t}=\frac{T-s}{T-t}$
and \eqref{a98}
prove for all $t\in [0,T)$, $b\in (t,T)$ that
\begin{align}
&
\P (t+(T-t)\mathfrak{r}^0\leq b)=\P(\mathfrak{r}^0
\leq \tfrac{b-t}{T-t})=\frac{1}{\mathrm{B}(\tfrac{1}{2},\tfrac{1}{2})}\int_{0}^{\tfrac{b-t}{T-t}}\frac{dr}{\sqrt{r(1-r)}}
\nonumber\\&=\frac{1}{\mathrm{B}(\tfrac{1}{2},\tfrac{1}{2})}\int_{t}^{b}\frac{\frac{ds}{T-t}}{\sqrt{
\frac{s-t}{T-t}\frac{T-s}{T-t}} }=\frac{1}{\mathrm{B}(\tfrac{1}{2},\tfrac{1}{2})}\int_{t}^{b}\frac{ds}{\sqrt{(T-s)(s-t)}}
=\int_{t}^{b}\varrho(t,s)\,ds.
\end{align}
This shows for all $t\in [0,T)$ and all measurable functions $h\colon(t,T)\to [0,\infty) $
that \begin{align}
\E\!\left [h(t+(T-t)\mathfrak{r}^0)\right]=\int_t^Th(s)\varrho(t,s)\,ds.\label{a24}
\end{align} 
Hence, the independence and distributional properties (cf. \cref{a02}), the disintegration theorem, \eqref{a04e}, \eqref{a23}, \eqref{a98}, the fact that
$\forall\,t\in [0,T)\colon \int_{t}^{T}\frac{dr}{\sqrt{r-t}}=2\sqrt{r-t}|_{r=t}^T=2\sqrt{T-t}$, the fact that 
$\sqrt{2\mathrm{B}(\tfrac{1}{2},\tfrac{1}{2})}\leq 3$
imply for all $t\in [0,T)$, 
$x\in \R^d$, $\nu\in [0,d]\cap\Z$ that
\begin{align}
&\left\lVert \frac{\Lambda_\nu(T-t)\left(F(0)\right)\!\left(t+(T-t) \mathfrak{r}^{0},
\mathcal{X}^{0,t,x}_{t+(T-t) \mathfrak{r}^{0}}\right)
\pr_\nu\! \left(\mathcal{Z}^{0,t,x}_{t+(T-t) \mathfrak{r}^{0}}\right)
}{\varrho(t,t+(T-t)\mathfrak{r}^{0})}
\right\rVert_2\nonumber\\
&=\xeqref{a24}\left[
\int_{t}^{T}\left\lVert \frac{\Lambda_\nu(T-t)\left(F(0)\right)\!\left(s,
\mathcal{X}^{0,t,x}_{s}\right)
\pr_{\nu}(\mathcal{Z}^{0,t,x}_{s})}{\varrho(t,s)}
\right\rVert_2^2\varrho(t,s)\,ds\right]^\frac{1}{2}\nonumber\\
&=\left[
\int_{t}^{T}\frac{\left\lVert \Lambda_\nu(T-t)\left(F(0)\right)\!\left(s,
\mathcal{X}^{0,t,x}_{s}\right)
\pr_\nu(\mathcal{Z}^{0,t,x}_{s})
\right\rVert_2^2}{\varrho(t,s)}\,ds\right]^\frac{1}{2}\nonumber\\
&\leq \xeqref{a04e}
\left[
\int_{t}^{T}\frac{\left\lVert \Lambda_\nu(T-t)\frac{1}{T}V\!\left(s,
\mathcal{X}^{0,t,x}_{s}\right)
\pr_\nu(\mathcal{Z}^{0,t,x}_{s})
\right\rVert_2^2}{\varrho(t,s)}\,ds\right]^\frac{1}{2}\nonumber\\
&\leq\xeqref{a23}
\left[
\int_{t}^{T}\frac{ \Lambda_\nu^2(T-t)\frac{1}{T^2}\frac{V^4(t,x)}{\Lambda_\nu^2(s-t)}
}{\varrho(t,s)}\,ds\right]^\frac{1}{2}\nonumber\\ 
&\leq\xeqref{a98}
\left[
\int_{t}^{T}\frac{ \frac{1}{T^2}V^4(t,x)\frac{T-t}{s-t}
}{\frac{1}{\mathrm{B}(\tfrac{1}{2},\tfrac{1}{2})}\frac{1}{\sqrt{(T-s)(s-t)}}}\,ds\right]^\frac{1}{2}\nonumber\\ 
&=\frac{V^2(t,x)}{T}\sqrt{\mathrm{B}(\tfrac{1}{2},\tfrac{1}{2})}
\left[(T-t)^\frac{3}{2}
\int_{t}^{T}\frac{ds}{\sqrt{s-t}}\right]^\frac{1}{2}\nonumber\\ 
&=
\frac{V^2(t,x)}{T}\sqrt{\mathrm{B}(\tfrac{1}{2},\tfrac{1}{2})}
\left[(T-t)^\frac{3}{2}
2\sqrt{T-t}\right]^\frac{1}{2}
\nonumber\\
&
\leq 3V^2(t,x).\label{a26}
\end{align}
Next,  \eqref{a24},
the independence and distributional properties (cf. \cref{a02}), the disintegration theorem,
\eqref{a01e}, the triangle inequality, \eqref{a25},  \eqref{a23}, \eqref{a98}
prove for all
$\nu\in [0,d]\cap\Z$, $j\in \N_0$, $m\in \N$, $t\in [0,T)$, $x\in \R^d$ that
\begin{align}
&\left\lVert \frac{\Lambda_\nu(T-t)\left(F(U^{0}_{j,m})-
F(\mathfrak{u})\right)\!\left(t+(T-t) \mathfrak{r}^{0},
\mathcal{X}^{0,t,x}_{t+(T-t) \mathfrak{r}^{0}}\right)
\pr_\nu\! \left(\mathcal{Z}^{0,t,x}_{t+(T-t) \mathfrak{r}^{0}}\right)
}{\varrho(t,t+(T-t)\mathfrak{r}^{0})}
\right\rVert_2\nonumber\\
&=\xeqref{a24}\left[
\int_{t}^{T}\left\lVert \frac{\Lambda_\nu(T-t)
\left(F(U^{0}_{j,m})-F(\mathfrak{u})\right)\!\left(s,
\mathcal{X}^{0,t,x}_{s}\right)
\pr_{\nu}(\mathcal{Z}^{0,t,x}_{s})}{\varrho(t,s)}
\right\rVert_2^2\varrho(t,s)\,ds\right]^\frac{1}{2}\nonumber\\
&=\left[
\int_{t}^{T}\frac{\left\lVert \Lambda_\nu(T-t)\left(F(U^{0}_{j,m})-
F(\mathfrak{u})\right)\!\left(s,
\mathcal{X}^{0,t,x}_{s}\right)
\pr_\nu(\mathcal{Z}^{0,t,x}_{s})
\right\rVert_2^2}{\varrho(t,s)}\,ds\right]^\frac{1}{2}\nonumber\\
&
\leq \xeqref{a01e}\left[
\int_{t}^{T}\frac{\left\lVert \Lambda_\nu(T-t)
\sum_{i=0}^{d}L_i\Lambda_i(T)\left\lvert
(U^{0}_{j,m}-\mathfrak{u})(s,
\mathcal{X}^{0,t,x}_{s})
\pr_\nu(\mathcal{Z}^{0,t,x}_{s})
\right\rvert\right\rVert_2}{\varrho(t,s)}\,ds\right]^\frac{1}{2}\nonumber\\
&\leq \left[\int_{t}^{T}\frac{ \left[\Lambda_\nu(T-t)
\sum_{i=0}^{d}L_i\Lambda_i(T)\left\lVert
(U^{0}_{j,m}-\mathfrak{u})(s,
\mathcal{X}^{0,t,x}_{s})\pr_\nu(\mathcal{Z}^{0,t,x}_{s})
\right\rVert_2\right]^2}{\varrho(t,s)}\,ds\right] ^\frac{1}{2}\nonumber\\
&\leq \left[\int_{t}^{T}\frac{\left[ \Lambda_\nu(T-t)
\sum_{i=0}^{d}L_i\frac{\sqrt{T}}{\sqrt{T-s}}\left\lVert
\Lambda_i(T-s)
(U^{0}_{j,m}-\mathfrak{u})(s,
\mathcal{X}^{0,t,x}_{s})\pr_\nu(\mathcal{Z}^{0,t,x}_{s})
\right\rVert_2\right]^2}{\varrho(t,s)}\,ds\right]^\frac{1}{2}\nonumber\\
&\leq \left[\int_{t}^{T}\frac{\left[ \Lambda_\nu(T-t)
\sum_{i=0}^{d}L_i\frac{\sqrt{T}}{\sqrt{T-s}}\left\lVert
\left\lVert
\Lambda_i(T-s)
(U^{0}_{j,m}-\mathfrak{u})(s,
\tilde{x})\pr_\nu(\tilde{z})
\right\rVert_2
\Bigr|_{\tilde{x}=\mathcal{X}^{0,t,x}_{s}, \tilde{z}=\mathcal{Z}^{0,t,x}_{s}}
\right\rVert_2\right]^2}{\varrho(t,s)}\,ds\right]^\frac{1}{2}\nonumber\\
&\leq\xeqref{a25} \left[\int_{t}^{T}\frac{\left[ \Lambda_\nu(T-t)
c\frac{\sqrt{T}}{\sqrt{T-s}}
\threenorm{U^{0}_{j,m}-\mathfrak{u}}_s\left\lVert
V^{\exponentFirstNorm}(s,
\mathcal{X}^{0,t,x}_{s})
\pr_\nu(\mathcal{Z}^{0,t,x}_{s})
\right\rVert_2\right]^2}{\varrho(t,s)}\,ds\right]^\frac{1}{2}\nonumber\\
&\leq\xeqref{a23} \left[\int_{t}^{T}\frac{\left[ \Lambda_\nu(T-t)
c\frac{\sqrt{T}}{\sqrt{T-s}}
\threenorm{U^{0}_{j,m}-\mathfrak{u}}_s
V^{\exponentFirstNorm}(t,x)\frac{c}{\Lambda_\nu(s-t)}
\right]^2
}{\varrho(t,s)}\,ds\right]^\frac{1}{2}\nonumber\\
&\leq \xeqref{a98}\left[\int_{t}^{T}\frac{\left[ 
c^2
\threenorm{U^{0}_{j,m}-\mathfrak{u}}_s
V^{\exponentFirstNorm}(t,x)\frac{\sqrt{T}}{\sqrt{T-s}}\frac{\sqrt{T-t}}{\sqrt{s-t}}
\right]^2
}{\frac{1}{\mathrm{B}(\tfrac{1}{2},\tfrac{1}{2})}\frac{1}{\sqrt{(T-s)(s-t)}}}\,ds \right]^\frac{1}{2}\nonumber\\
&\leq \left[
\int_{t}^{T}\frac{
c^4\mathrm{B}(\tfrac{1}{2},\tfrac{1}{2})
\left[ 
\threenorm{U^{0}_{j,m}-\mathfrak{u}}_s
V^{\exponentFirstNorm}(t,x)\sqrt{T}\sqrt{T-t}
\right]^2
}{\sqrt{(T-s)(s-t)}}\,ds\right]^\frac{1}{2}\nonumber\\
&=
c^2\sqrt{T}\sqrt{T-t}\sqrt{\mathrm{B}(\tfrac{1}{2},\tfrac{1}{2})}V^\exponentFirstNorm(t,x)
\left[
\int_{t}^{T}\frac{
\threenorm{U^{0}_{j,m}-\mathfrak{u}}_s^2
}{\sqrt{(T-s)(s-t)}}\,ds\right]^\frac{1}{2}.\label{a31}
\end{align}
Furthermore,
the substitution $s=t+(T-t)r$, $ds=(T-t)dr$, $ s=t\Rightarrow r=0$,
$s=T\Rightarrow r=1$,
$T-s=T-t-(T-t)r=(T-t)(1-r)$,
$s-t=(T-t)r$ and the definition of the beta functions imply for all $t\in [0,T)$ that
\begin{align}
\int_{t}^{T}\frac{ds}{(T-s)^{\frac{3}{4}}(s-t)^{\frac{3}{4}}}=
\int_{0}^{1}\frac{(T-t)dr}{[(T-t)(1-r)]^\frac{3}{4}
[(T-t)r)]^\frac{3}{4}
}=(T-t)^{-\frac{1}{2}}\mathrm{B}(\tfrac{1}{4},\tfrac{1}{4}).
\label{a32}
\end{align}
Therefore, \eqref{a31}, H\"older's inequality, the fact that 
$\frac{2}{3}+\frac{1}{3}=1$, and the fact that
$(\mathrm{B}(\frac{1}{4},\frac{1}{4}))^\frac{1}{3}\leq 2$
show for all
$\nu\in [0,d]\cap\Z$, $j\in \N_0$, $m\in \N$, $t\in [0,T)$, $x\in \R^d$ that
\begin{align}
&\left\lVert \frac{\Lambda_\nu(T-t)\left(F(U^{0}_{j,m})-F(\mathfrak{u})\right)\!\left(t+(T-t) \mathfrak{r}^{0},
\mathcal{X}^{0,t,x}_{t+(T-t) \mathfrak{r}^{0}}\right)
\pr_\nu\! \left(\mathcal{Z}^{0,t,x}_{t+(T-t) \mathfrak{r}^{0}}\right)
}{\varrho(t,t+(T-t)\mathfrak{r}^{0})}
\right\rVert_2\nonumber\\
&\leq \xeqref{a31}
c^2\sqrt{T}\sqrt{T-t}\sqrt{\mathrm{B}(\tfrac{1}{2},\tfrac{1}{2})}V^\exponentFirstNorm(t,x)
\left[\int_{t}^{T}\frac{ds}{(T-s)^{\frac{1}{2}\frac{3}{2}}(s-t)^{\frac{1}{2}\frac{3}{2}}}\right]^{\frac{2}{3}\frac{1}{2}}
\left[\int_{t}^{T}
\threenorm{U^{0}_{j,m}-\mathfrak{u}}_s^{2\cdot 3}\right]^{\frac{1}{3}\frac{1}{2}}\nonumber\\
&\leq 
c^2\sqrt{T}\sqrt{T-t}\sqrt{\mathrm{B}(\tfrac{1}{2},\tfrac{1}{2})}V^\exponentFirstNorm(t,x)
\left[\int_{t}^{T}\frac{ds}{(T-s)^{\frac{3}{4}}(s-t)^{\frac{3}{4}}}\right]^{\frac{1}{3}}
\left[\int_{t}^{T}
\threenorm{U^{0}_{j,m}-\mathfrak{u}}_s^{6}\right]^{\frac{1}{6}}\nonumber\\
&\leq 2
c^2\sqrt{T}\sqrt{T-t}V^\exponentFirstNorm(t,x)
\left[\xeqref{a32}
(T-t)^{-\frac{1}{2}}
\mathrm{B}(\tfrac{1}{4},\tfrac{1}{4})\right]^\frac{1}{3}
\left[\int_{t}^{T}
\threenorm{U^{0}_{j,m}-\mathfrak{u}}_s^{6}\right]^{\frac{1}{6}}
\nonumber\\
&\leq 4
c^2T^\frac{5}{6}V^\exponentFirstNorm(t,x)
\left[\int_{t}^{T}
\threenorm{U^{0}_{j,m}-\mathfrak{u}}_s^{6}\right]^{\frac{1}{6}}.\label{a29}
\end{align}
This, the triangle inequality, and the distributional and independence properties (cf. \cref{a02})
prove for all 
$\nu\in [0,d]\cap\Z$,  $\ell,m\in \N$, $t\in [0,T)$, $x\in \R^d$ that
\begin{align}
&\left\lVert \frac{\Lambda_\nu(T-t)\left(F(U^{0}_{\ell,m})-F(U^{1}_{\ell-1,m})\right)\!\left(t+(T-t) \mathfrak{r}^{0},
\mathcal{X}^{0,t,x}_{t+(T-t) \mathfrak{r}^{0}}\right)
\pr_\nu\! \left(\mathcal{Z}^{0,t,x}_{t+(T-t) \mathfrak{r}^{0}}\right)
}{\varrho(t,t+(T-t)\mathfrak{r}^{0})}
\right\rVert_2\nonumber\\
&\leq \sum_{j=\ell-1}^{\ell}\left\lVert \frac{\Lambda_\nu(T-t)\left(F(U^{0}_{j,m})-F(\mathfrak{u})\right)\!\left(t+(T-t) \mathfrak{r}^{0},
\mathcal{X}^{0,t,x}_{t+(T-t) \mathfrak{r}^{0}}\right)
\pr_\nu\! \left(\mathcal{Z}^{0,t,x}_{t+(T-t) \mathfrak{r}^{0}}\right)
}{\varrho(t,t+(T-t)\mathfrak{r}^{0})}
\right\rVert_2\nonumber\\
&\leq 4
c^2T^\frac{5}{6}V^\exponentFirstNorm(t,x)
\sum_{j=\ell-1}^{\ell}
\left[\int_{t}^{T}
\threenorm{U^{0}_{j,m}-\mathfrak{u}}_s^{6}\right]^{\frac{1}{6}}.
\end{align}
Hence, 
\eqref{a04d},
the fact that $\forall\,\theta\in \Theta,m\in\N\colon U^\theta_{0,m}=0$, \eqref{a19b}, \eqref{a26},  \eqref{a27}, 
the independence and distributional properties (cf. \cref{a02}),
and an induction argument prove
for all $n,m\in \N$, $\ell\in [0,n-1]\cap\Z$, $\nu \in [0,d]\cap\Z$, $t\in [0,T)$, $x\in \R^d$ that
\begin{align}
&\threenorm{U_{n,m}^0}_t+\left\lVert \frac{\Lambda_\nu(T-t)(F(U^{0}_{\ell,m}))\!\left(t+(T-t) \mathfrak{r}^{0},
X^{0,t,x}_{t+(T-t) \mathfrak{r}^{0}}\right)
\pr_\nu\! \left(Z^{0,t,x}_{t+(T-t) \mathfrak{r}^{0}}\right)
}{\varrho(t,t+(T-t)\mathfrak{r}^{0})}
\right\rVert_2<\infty.
\end{align}
This, linearity of the expectation, and the independence and distributional properties (cf. \cref{a02}) imply for all 
$n,m\in \N$, $t\in [0,T)$, $x\in \R^d$ that
\begin{align}
&
\E\!\left[
U^{0}_{n,m}(t,x)\right]\nonumber\\
&= (g(x),0)+\sum_{i=1}^{m^n}
\E\!\left[
\frac{g(\mathcal{X}^{(0,0,-i),t,x}_T )-g(x) }{m^n}\mathcal{Z}^{(0,0,-i),t,x}_{T}\right]
\nonumber\\
&\quad  +\sum_{\ell=0}^{n-1}\E\!\left[\sum_{i=1}^{m^{n-\ell}} \frac{\left(F(U^{(0,\ell,i)}_{\ell,m})-
\1_\N(\ell)
F(U^{(0,\ell,-i)}_{\ell-1,m})\right)\!\left(t+(T-t) \mathfrak{r}^{(0,\ell,i)},
\mathcal{X}^{(0,\ell,i),t,x}_{t+(T-t) \mathfrak{r}^{(0,\ell,i)}}\right)
\mathcal{Z}^{(0,\ell,i),t,x}_{t+(T-t) \mathfrak{r}^{(0,\ell,i)}}}{m^{n-\ell}\varrho(t,t+(T-t)\mathfrak{r}^{(0,\ell,i)})}\right]\nonumber\\
&=(g(x),0)+\E\!\left[(g(\mathcal{X}^{0,t,x}_T)-g(x))\mathcal{Z}^{0,t,x}_T\right]\nonumber\\
&\quad +
\sum_{\ell=0}^{n-1}\Biggl(
\E\!\left[\frac{F(U^0_{\ell,m})
(t+(T-t)\mathfrak{r}^0,\mathcal{X}^{0,t,x}_{t+(T-t)\mathfrak{r}^0} )\mathcal{Z}^{0,t,x}_{t+(T-t)\mathfrak{r}^0}}{\varrho(t,t+(T-t)\mathfrak{r}^0)}
\right]\nonumber\\
&\qquad\qquad\qquad -
\1_\N(\ell)\E\!\left[\frac{F(U^0_{\ell-1,m})
(t+(T-t)\mathfrak{r}^0,\mathcal{X}^{0,t,x}_{t+(T-t)\mathfrak{r}^0} )\mathcal{Z}^{0,t,x}_{t+(T-t)\mathfrak{r}^0}}{\varrho(t,t+(T-t)\mathfrak{r}^0)}
\right]\Biggr)\nonumber\\
&=\E\!\left[(g(\mathcal{X}^{0,t,x}_T)\mathcal{Z}^{0,t,x}_T\right]
+
\E\!\left[\frac{F(U^0_{n-1,m})
(t+(T-t)\mathfrak{r}^0,\mathcal{X}^{0,t,x}_{t+(T-t)\mathfrak{r}^0} )\mathcal{Z}^{0,t,x}_{t+(T-t)\mathfrak{r}^0}}{\varrho(t,t+(T-t)\mathfrak{r}^0)}
\right].\label{a33}
\end{align}
Next, 
\eqref{a05e}, \eqref{a24}, the disintegration theorem, 
the independence and distributional properties (cf. \cref{a02}) prove for all
$t\in [0,T)$, $x\in \R^d$ that
\begin{align}
\mathfrak{u}(t,x)&
=\xeqref{a05e}\E\!\left [g(\mathcal{X}^{0,t,x}_{T} )\mathcal{Z}^{0,t,x}_{T} \right] + \int_{t}^{T}
\E \!\left[\frac{
f(r,\mathcal{X}^{0,t,x}_{r},\mathfrak{u}(r,\mathcal{X}^{0,t,x}_{r}))\mathcal{Z}^{0,t,x}_{r}}{\varrho(t,r)}\right]\varrho(t,r)\,dr\nonumber
\\
&=\xeqref{a24}\E\!\left[(g(\mathcal{X}^{0,t,x}_T)\mathcal{Z}^{0,t,x}_T\right]
+
\E\!\left[\frac{F(\mathfrak{u})
(t+(T-t)\mathfrak{r}^0,\mathcal{X}^{0,t,x}_{t+(T-t)\mathfrak{r}^0} )\mathcal{Z}^{0,t,x}_{t+(T-t)\mathfrak{r}^0}}{\varrho(t,t+(T-t)\mathfrak{r}^0)}
\right].
\end{align}
This, \eqref{a33}, Jensen's inequality, and \eqref{a29} 
imply for all
$t\in [0,T)$, $x\in \R^d$, $\nu\in [0,d]\cap\Z$ that
\begin{align}
&
\Lambda_\nu(T-t)\left\lvert\pr_{\nu}\!\left(
\E\!\left[
U^{0}_{n,m}(t,x)\right]-\mathfrak{u}(t,x)\right)\right\rvert\nonumber\\
&
\leq 
\left\lVert
\Lambda_\nu(T-t)
\frac{(F(U^0_{n-1,m})-F(\mathfrak{u}))
(t+(T-t)\mathfrak{r}^0,\mathcal{X}^{0,t,x}_{t+(T-t)\mathfrak{r}^0} )\mathcal{Z}^{0,t,x}_{t+(T-t)\mathfrak{r}^0}}{\varrho(t,t+(T-t)\mathfrak{r}^0)}
\right\rVert_2\nonumber\\
&\leq 4
c^2T^\frac{5}{6}V^\exponentFirstNorm(t,x)
\left[\int_{t}^{T}
\threenorm{U^{0}_{n-1,m}-\mathfrak{u}}_s^{6}\right]^{\frac{1}{6}}.\label{a29b}
\end{align}
Furthermore,
\eqref{a04d}, the triangle inequality, the independence and distributional properties (cf. \cref{a02}), \eqref{a20a}, \eqref{a26},  \eqref{a31}, and the fact that $1\leq V$ prove for all
$t\in [0,T)$, $x\in \R^d$, $\nu\in [0,d]\cap\Z$ that
\begin{align} 
&
\Lambda_\nu(T-t)\left(\var\!\left[
\pr_\nu (U^{0}_{n,m}(t,x))\right]\right)^\frac{1}{2}\nonumber\\
&= 
\Lambda_\nu(T-t)\left(\var\!\left[
\pr_{\nu}\left(
(g(x),0)+\sum_{i=1}^{m^n}
\tfrac{g(\mathcal{X}^{(0,0,-i),t,x}_T )-g(x) }{m^n}\mathcal{Z}^{(0,0,-i),t,x}_{T}\right)
\right]\right)^\frac{1}{2}\nonumber
\\
& \quad +\sum_{\ell=0}^{n-1}
\left(\var\!\left[
\sum_{i=1}^{m^{n-\ell}} \tfrac{\Lambda_\nu(T-t)\left(F(U^{(0,\ell,i)}_{\ell,m})-
\1_\N(\ell)
F(U^{(0,\ell,-i)}_{\ell-1,m})\right)\!\left(t+(T-t) \mathfrak{r}^{(0,\ell,i)},
\mathcal{X}^{(0,\ell,i),t,x}_{t+(T-t) \mathfrak{r}^{(0,\ell,i)}}\right)
\pr_{\nu}\left(\mathcal{Z}^{(0,\ell,i),t,x}_{t+(T-t) \mathfrak{r}^{(0,\ell,i)}}\right)
}{m^{n-\ell}\varrho(t,t+(T-t)\mathfrak{r}^{(0,\ell,i)})}\right]
\right)^\frac{1}{2}
\nonumber\\
&\leq \xeqref{a20a}\tfrac{V^3(t,x)}{\sqrt{m^n}}
+\sum_{\ell=0}^{n-1}\tfrac{1}{\sqrt{m^{n-\ell}}}
\left(\var\!\left[\tfrac{\Lambda_\nu(T-t)\left(F(U^{0}_{\ell,m})-
\1_\N(\ell)
F(U^{1}_{\ell-1,m})\right)\!\left(t+(T-t) \mathfrak{r}^{0},
\mathcal{X}^{0,t,x}_{t+(T-t) \mathfrak{r}^{0}}\right)
\pr_{\nu}\left(\mathcal{Z}^{0,t,x}_{t+(T-t) \mathfrak{r}^{0}}\right)
}{\varrho(t,t+(T-t)\mathfrak{r}^{0})}\right]
\right)^\frac{1}{2}
\nonumber\\
&\leq \tfrac{V^3(t,x)}{\sqrt{m^n}}
+
\sum_{\ell=0}^{n-1}\left[\tfrac{1}{\sqrt{m^{n-\ell}}}
\left\lVert\tfrac{\Lambda_\nu(T-t)\left(F(U^{0}_{\ell,m})-
\1_\N(\ell)
F(U^{1}_{\ell-1,m})\right)\!\left(t+(T-t) \mathfrak{r}^{0},
\mathcal{X}^{0,t,x}_{t+(T-t) \mathfrak{r}^{0}}\right)
\pr_{\nu}\left(\mathcal{Z}^{0,t,x}_{t+(T-t) \mathfrak{r}^{0}}\right)
}{\varrho(t,t+(T-t)\mathfrak{r}^{0})}\right\rVert_2\right]
\nonumber\\
&\leq 
\tfrac{V^3(t,x)}{\sqrt{m^n}}
+
\tfrac{1}{\sqrt{m^n}}
\left\lVert \tfrac{\Lambda_\nu(T-t)\left(F(0)\right)\!\left(t+(T-t) \mathfrak{r}^{0},
\mathcal{X}^{0,t,x}_{t+(T-t) \mathfrak{r}^{0}}\right)
\pr_\nu\! \left(\mathcal{Z}^{0,t,x}_{t+(T-t) \mathfrak{r}^{0}}\right)
}{\varrho(t,t+(T-t)\mathfrak{r}^{0})}
\right\rVert_2\nonumber\\
&\quad +
\sum_{\ell=1}^{n-1}\left[
\tfrac{1}{\sqrt{m^{n-\ell}}}
\left\lVert\tfrac{\Lambda_\nu(T-t)\left(F(U^{0}_{\ell,m})-
F(U^{1}_{\ell-1,m})\right)\!\left(t+(T-t) \mathfrak{r}^{0},
\mathcal{X}^{0,t,x}_{t+(T-t) \mathfrak{r}^{0}}\right)
\pr_{\nu}\left(\mathcal{Z}^{0,t,x}_{t+(T-t) \mathfrak{r}^{0}}\right)
}{\varrho(t,t+(T-t)\mathfrak{r}^{0})}\right\rVert_2\right]\nonumber\\
&\leq \tfrac{V^3(t,x)}{\sqrt{m^n}}+\xeqref{a26}\tfrac{3V^2(t,x)}{\sqrt{m^n}}+\xeqref{a31}\sum_{\ell=1}^{n-1}\left[\tfrac{1}{\sqrt{m^{n-\ell}}}\left[
4
c^2T^\frac{5}{6}V^\exponentFirstNorm(t,x)
\sum_{j=\ell-1}^{\ell}
\left[\int_{t}^{T}
\threenorm{U^{0}_{j,m}-\mathfrak{u}}_s^{6}\right]^{\frac{1}{6}}
\right]\right]\nonumber\\
&\leq \tfrac{4V^3(t,x)}{\sqrt{m^n}}+\sum_{\ell=1}^{n-1}\left[
\sum_{j=\ell-1}^{\ell}
\tfrac{4
c^2T^\frac{5}{6}V^\exponentFirstNorm(t,x)}{\sqrt{m^{n-j-1}}}\
\left[\int_{t}^{T}
\threenorm{U^{0}_{j,m}-\mathfrak{u}}_s^{6}\right]^{\frac{1}{6}}
\right]\nonumber\\
&= \tfrac{4V^3(t,x)}{\sqrt{m^n}}+\sum_{j=0}^{n-1}\left[
\sum_{\ell\in [1,n-1]\cap \{j,j+1\}}
\tfrac{4
c^2T^\frac{5}{6}V^\exponentFirstNorm(t,x)}{\sqrt{m^{n-j-1}}}\
\left[\int_{t}^{T}
\threenorm{U^{0}_{j,m}-\mathfrak{u}}_s^{6}\right]^{\frac{1}{6}}
\right]\nonumber\\
&= \tfrac{4V^3(t,x)}{\sqrt{m^n}}+\sum_{j=0}^{n-1}\left[
(2-\1_{n-1}(j))
\tfrac{4
c^2T^\frac{5}{6}V^\exponentFirstNorm(t,x)}{\sqrt{m^{n-j-1}}}\
\left[\int_{t}^{T}
\threenorm{U^{0}_{j,m}-\mathfrak{u}}_s^{6}\right]^{\frac{1}{6}}
\right].
\label{a04f}
\end{align}
Thus, the triangle inequality, \eqref{a29b}, and the fact that
$V^3\leq V^\exponentFirstNorm$
prove for all 
$\nu\in [0,d]\cap\Z$, $n,m\in \N$, $t\in [0,T)$, $x\in \R^d$ that
\begin{align}
&
\left\lVert
\Lambda_\nu(T-t)
\pr_{\nu}\!\left(
U^0_{n,m}(t,x)-\mathfrak{u}(t,x)\right)\right\rVert_2\nonumber\\
&
\leq 
\Lambda_\nu(T-t)\left(\var\!\left[
\pr_\nu (U^{0}_{n,m}(t,x))\right]\right)^\frac{1}{2}+
\Lambda_\nu(T-t)\left\lvert\pr_{\nu}\!\left(
\E\!\left[
U^{0}_{n,m}(t,x)\right]-\mathfrak{u}(t,x)\right)\right\rvert\nonumber\\
&\leq 
\frac{4V^{\exponentFirstNorm}(t,x)}{\sqrt{m^n}}+\sum_{j=0}^{n-1}\left[
\frac{8
c^2T^\frac{5}{6}V^\exponentFirstNorm(t,x)}{\sqrt{m^{n-j-1}}}\
\left[\int_{t}^{T}
\threenorm{U^{0}_{j,m}-\mathfrak{u}}_s^{6}\right]^{\frac{1}{6}}
\right].
\end{align}
Therefore, \eqref{a25}
implies for all $n,m\in \N$, $t\in [0,T)$ that
\begin{align}
\threenorm{U^0_{n,m}-\mathfrak{u}}_t\leq 
\frac{4}{\sqrt{m^n}}+\sum_{\ell=0}^{n-1}\left[
\frac{8
c^2T^\frac{5}{6}}{\sqrt{m^{n-\ell-1}}}\
\left[\int_{t}^{T}
\threenorm{U^{0}_{\ell,m}-\mathfrak{u}}_s^{6}\right]^{\frac{1}{6}}
\right].\label{a42b}
\end{align}
This shows
\eqref{a42}. 

Next, \eqref{a42b}, \cite[Lemma 3.11]{HJKN2020},   \eqref{a27}, and the fact that $1+c^2T\leq e^{c^2T}$ prove for all
$n,m\in \N$, $t\in [0,T)$ that
\begin{align}
\threenorm{U^0_{n,m}-\mathfrak{u}}_t
&
\leq (4+8
c^2T^\frac{5}{6}\cdot T^\frac{1}{6}\cdot \xeqref{a27}1)
e^\frac{m^3}{6}
m^{-\frac{n}{2}}\left[1+8c^2T^\frac{5}{6}T^\frac{1}{6}\right]^{n-1}\nonumber\\
&\leq 
e^\frac{m^3}{6}
m^{-\frac{n}{2}}(4+8c^2T)^n\nonumber\\
&\leq 
e^\frac{m^3}{6}
m^{-\frac{n}{2}}8^n(1+c^2T)^n\nonumber\\
&\leq 
e^\frac{m^3}{6}
m^{-\frac{n}{2}}8^ne^{nc^2T}.\label{b41}
\end{align}
This shows \eqref{a43}.

Next, \eqref{b41} and \eqref{a25} prove for all
$n,m\in \N$, $t\in [0,T)$, $\nu\in [0,d]\cap\Z$ that
\begin{align}
\Lambda_\nu(T-t)\left\lVert\pr_{\nu}\!\left(U^0_{n,m}(t,x)-\mathfrak{u}(t,x)\right)\right\rVert_2\leq 
e^\frac{m^3}{6}
m^{-\frac{n}{2}}8^ne^{nc^2T} V^\exponentFirstNorm(t,x).
\end{align}
This implies \eqref{b40} and completes the proof of \cref{a20}.
\end{proof}

\section{Euler-Maruyama approximations revisited}\label{s05}

In this section we provide some results of moment estimates, stability, continuity, and discretization
errors for solutions to SDEs with explicit constants independent of the dimension $d\in \N$. 

\begin{setting}\label{a17d}
Let $\lVert \cdot\rVert\colon \cup_{k,\ell\in\N}\R^{k\times \ell}\to[0,\infty)$ satisfy for all $k,\ell\in\N$, $s=(s_{ij})_{i\in[1,k]\cap\N,j\in [1,\ell]\cap\N}\in\R^{k\times \ell}$ that
$\lVert s\rVert^2=\sum_{i=1}^{k}\sum_{j=1}^{\ell}\lvert s_{ij}\rvert^2$.
Let $d\in \N$, 
$T\in (0,\infty)$,
$c,\bar{c},b\in [1,\infty)$,
$p\in[8,\infty)$,
$\mu\in C^2(\R^d,\R^d)$, 
$\sigma\in C^2(\R^d,\R^{d\times d})$,
$\varphi\in C^2(\R^d,[1,\infty))$
satisfy for all $x,y,h\in\R^d$ that
$\sigma(x)$ is invertible,
\begin{align}
\lVert\mu(0)\rVert+\lVert\sigma(0)\rVert+c\lVert x\rVert\leq  (\varphi(x))^{\frac{1}{p}},
\end{align}
\begin{equation}
\bigl\lvert((\operatorname{D} \!\varphi)(x))(y)\bigr\rvert
\leq \bar{c}(\varphi(x))^{\frac{p-1}{p}}\lVert y\rVert,\quad 
\bigl\lvert ((\operatorname{D}^2\! \varphi)(x))(y,y)\bigr\rvert
\leq \bar{c}(\varphi(x))^{\frac{p-2}{p}}\lVert  y\rVert^2,\label{m42}
\end{equation}
\begin{align}
\max\left\{\left\lVert\left((\operatorname{D}\!\mu)(x)\right)\!(h)\right\rVert  ,\left\lVert \left((\operatorname{D}\!\sigma)(x)
\right)\!(h)
\right\rVert,\left\lVert \sigma^{-1}(x)h\right\rVert  \right\}\leq c\lVert h\rVert,\label{f04}
\end{align}and
\begin{align}
&
\max \left\{
\left\lVert
\left(
(\operatorname{D}\!\mu)(x)-
(\operatorname{D}\!\mu)(y)\right)\!(h)\right\rVert
,
\left\lVert
\left(
(\operatorname{D}\!\sigma)(x)-
(\operatorname{D}\!\sigma)(y)\right)\!(h)\right\rVert,\left\lVert \left[
(\sigma(x))^{-1}-(\sigma(y))^{-1}\right]h\right\rVert
\right\}\nonumber\\
&
\leq b\lVert x-y\rVert\lVert h\rVert.
\label{f10}
\end{align}
Let $\iota\colon [0,T] \to[0,T]$ satisfy for all $t\in[0,T]$ that
$\iota(t)=t$.
Let $\mathbbm{S}$ satisfy that
\begin{equation}\begin{split}
\mathbbm{S}= \left \{
\delta\colon [0,T]\to [0,T]\colon 
\begin{aligned}
&\exists\,
n\in\N,
t_0,t_1,\ldots,t_{n}\in [0,T]\colon 0=t_0< t_1<\ldots<t_n=T, \\
&
\delta ([t_0,t_1])=\{t_0\},\delta ((t_1,t_2])=\{t_1\},\ldots,
\delta ((t_{n-1},t_n])=\{t_{n-1}\}
\end{aligned}
\right\}.
\end{split}\end{equation}
Let $\tilde{\mathbbm{S}}=\mathbbm{S}\cup\{\iota\}$.
Let $\lvert\cdot\rvert\colon \tilde{\mathbbm{S}}\to [0,T]$ satisfy for all $\delta \in \mathbbm{S}$ that $\lvert\iota\rvert=0$ and
\begin{equation}
 \lvert\delta\rvert =\max
\Bigl\{\lvert s-t\rvert\colon s,t\in \delta([0,T]), s< t, (s,t)\cap \delta([0,T])=\emptyset  \Bigr\}.
\end{equation}
For every  $k\in [1,d]\cap\Z$ let $e_k\in \R^d$ denote the $d$-dimensional vector with a $1$ in the $k$-th coordinate and $0$'s elsewhere.
Let $(\Omega,\mathcal{F},\P, (\F_t)_{t\in[0,T]})$ be a filtered probability space which satisfies the usual conditions. For every $s\in[1,\infty)$, 
$k,\ell\in\N$
and every random variable $\mathfrak{X}\colon \Omega\to\R^{k\times \ell}$ let $\lVert\mathfrak{X}\rVert_s\in[0,\infty]$ satisfy that $\lVert\mathfrak{X}\rVert_s^s=\E [\lVert\mathfrak{X}\rVert^s]$.
Let $W=(W_t)_{t\in[0,T]}\colon [0,T]\times\Omega\to\R^{d}$ be a standard 
$(\F_t)_{t\in[0,T]}$-Brownian motion with continuous sample paths.
For every 
$\delta\in \tilde{\mathbbm{S}}$, $s\in [0,T]$, $x\in\R^d$ 
 let $(\mathcal{X}^{\delta,s,x}_{t})_{t\in[s,T]} \colon[s,T]  \times\Omega\to\R^d $ 
be an $(\F_t)_{t\in[s,T]}$-adapted stochastic process with continuous sample paths such that
for all $t\in[s,T]$ we have $\P$-a.s.\ that
\begin{align}
\mathcal{X}_{t}^{\delta,s,x}= x+\int_{s}^{t}\mu(\mathcal{X}_{\max\{s,\delta(r)\}}^{\delta,s,x})\,dr+\int_{s}^{t} \sigma(\mathcal{X}_{\max\{s,\delta(r)\}}^{\delta,s,x})\,dW_r.\label{f03}
\end{align}
For every $k\in [1,d]\cap\Z$,
$\delta\in \tilde{\mathbbm{S}}$, $s\in [0,T]$, $x\in\R^d$ 
 let $(\mathcal{D}^{\delta,s,x,k}_{t})_{t\in[s,T]} \colon[s,T]  \times\Omega\to\R^d $ be an $(\F_t)_{t\in[s,T]}$-adapted stochastic process with continuous sample paths such that for all $t\in [s,T]$ we have $\P$-a.s.\ that
\begin{align}
\mathcal{D}^{\delta,s,x,k}_t&=e_k+\int_{s}^{t}
\left((\operatorname{D}\!\mu)(\mathcal{X}^{\delta,s,x}_{\max\{s,\delta(r)\}})
\right)\!\left(
\mathcal{D}^{\delta,s,x,k}_{\max\{s,\delta(r)\}}
\right)dr\nonumber\\
&\quad 
+
\int_{s}^{t}\left((\operatorname{D}\!\sigma)(\mathcal{X}^{\delta,s,x}_{\max\{s,\delta(r)\}})\right)
\!\left(\mathcal{D}^{\delta,s,x,k}_{\max\{s,\delta(r)\}}\right)dW_r.\label{f07}
\end{align}
For every $\delta\in \tilde{\mathbbm{S}}$,
$s\in [0,T]$, $t\in [s,T]$, $x\in \R^d$
 let $\mathcal{D}^{\delta,s,x}_t = (\mathcal{D}^{\delta,s,x,k}_t)_{k\in [1,d]\cap\Z} $.
For every $\delta\in \tilde{\mathbbm{S}}$,
$s\in [0,T)$, $t\in (s,T]$, $x\in \R^d$ let 
$\mathcal{V}^{\delta,s,x}_t=(\mathcal{V}^{\delta,s,x,k}_t)_{k\in [1,d]\cap \Z}\colon \Omega\to\R^d$ satisfy that
\begin{align}
\mathcal{V}^{\delta,s,x}_t=\frac{1}{t-s}\int_{s}^{t}
\left(
\sigma^{-1}(\mathcal{X}^{\delta,s,x}_{\max\{s,\delta(r)\}})
\mathcal{D}^{\delta,s,x}_{\max\{s,\delta(r)\}}
\right)^\top dW_r.\label{f19}
\end{align}

\end{setting}

\begin{remark}
If $\delta=\iota$ in \eqref{f03} and \eqref{f07}
we have for all 
$k\in [0,d]\cap\Z$,
$x\in \R^d$, $s\in [0,T]$, $t\in [s,T]$ that $\P$-a.s.\
\begin{align}
\mathcal{X}_{t}^{\iota,s,x}= x+\int_{s}^{t}\mu(\mathcal{X}_{r}^{\iota,s,x})\,dr+\int_{s}^{t} \sigma(\mathcal{X}_{r}^{\iota,s,x})\,dW_r
\end{align} and 
\begin{align}
\mathcal{D}^{\iota,s,x,k}_t=e_k+\int_s^t
\left((\operatorname{D}\!\mu)(\mathcal{X}^{\iota,s,x}_r)
\right)\!\left(
\mathcal{D}^{\iota,s,x,k}_r\right)dr
+\int_{s}^{t}\left((\operatorname{D}\!\sigma)(\mathcal{X}^{\iota,s,x}_r)
\right)\!\left(
\mathcal{D}^{\iota,s,x,k}_r\right)dW_r.
\end{align}
In this case $\mathcal{D}^{\iota,s,x,k}_t=\frac{\partial}{\partial x_k}\mathcal{X}^{\iota,s,x}_t$ where $\frac{\partial}{\partial x_k}\mathcal{X}^{\iota,s,x}_t$
is the derivative process (cf. \cite[Theorem 3.4]{Kun2004}) defined through the following SDE:
\begin{align}
\frac{\partial}{\partial x_k}\mathcal{X}^{\iota,s,x}_t
=e_k+\int_{s}^{t}
\left((\operatorname{D}\!\mu) (\mathcal{X}^{\iota,s,x}_r)
\right)\!\left(
\frac{\partial}{\partial x_k}\mathcal{X}^{\iota,s,x}_r\right)dr
+\int_{s}^{t}\left(
(\operatorname{D}\!\sigma)(\mathcal{X}^{\iota,s,x}_r)
\right)\!\left(
\frac{\partial}{\partial x_k}\mathcal{X}^{\iota,s,x}_r\right)
dW_r.
\end{align}
\end{remark}

\begin{lemma}\label{f10b}
Assume \cref{a17d}. Then the following items hold.
\begin{enumerate}[(i)]
\item\label{f11}  For all $\delta\in \tilde{\mathbbm{S}}$, $s\in [0,T]$, 
$t\in [s,T]$, $x\in \R^d$ we have that 
$\E\!\left [\varphi(\mathcal{X}^{\delta,s,x}_t)\right]\leq e^{1.5\bar{c}\lvert t-s\rvert}\varphi(x)$.
\item\label{f12}  For all $\delta\in \tilde{\mathbbm{S}}$, $s\in [0,T]$, 
$t\in [s,T]$, $x\in \R^d$ we have that
\begin{align}
\left\lVert
\mathcal{X}^{\delta,s,x}_t
-\mathcal{X}^{\iota,s,x}_t
\right\rVert_p\leq \sqrt{2}c \left[\sqrt{T}+p\right]^2e^{c^2\left[\sqrt{T}+p\right]^2T}
\left(e^{1.5\bar{c}T}\varphi(x)\right)^\frac{1}{p}
\lvert t-s\rvert^\frac{1}{2}\lvert\delta\rvert^\frac{1}{2}.
\end{align}
\item\label{f13}  For all $\delta\in \tilde{\mathbbm{S}}$, $s,\tilde{s}\in [0,T]$, $t\in [s,T]$, $\tilde{t}\in [\tilde{s},T]$, $x,\tilde{x}\in \R^d$ we have that
\begin{align}
&
\left\lVert
\mathcal{X}^{\delta,s,x}_t-
\mathcal{X}^{\delta,\tilde{s},\tilde{x}}_{\tilde{t}}
\right\rVert_p\nonumber\\
&
\leq\sqrt{2} \lVert x-\tilde{x}\rVert e^{c^2\left[\sqrt{T}+p\right]^2T}
+5e^{c^2\left[\sqrt{T}+p\right]^2T}
\left[\sqrt{T}+p\right]
e^\frac{1.5\bar{c}T}{p}\frac{\varphi^\frac{1}{p}(x)+\varphi^\frac{1}{p}(\tilde{x}) }{2}\left[\lvert s-\tilde{s}\rvert^\frac{1}{2}+
\lvert t-\tilde{t}\rvert^\frac{1}{2}
\right].
\end{align}
\item \label{f26}
For all  $k\in [1,d]\cap\Z$,
$\delta\in \tilde{\mathbbm{S}}$, $s\in [0,T]$, $t\in [s,T]$, $x\in \R^d$ we have that
\begin{align}
\left\lVert\mathcal{V}^{\delta,s,x,k}_t\right\rVert_p\leq \frac{\sqrt{2}pce^{ c^2\left[p+\sqrt{T}\right]^2T}}{\sqrt{t-s}}.
\end{align} 
\item \label{f27}
For all
$k\in [1,d]\cap\Z$,
$\delta\in \tilde{\mathbbm{S}}$, $s\in [0,T]$, $t\in [s,T]$, $x,y\in \R^d$ we have that
\begin{align}
\left\lVert
\mathcal{V}^{\delta,s,x,k}_t
-\mathcal{V}^{\delta,s,y,k}_t\right\rVert_{\frac{p}{4}}\leq \frac{2bc
\left[\sqrt{T}+p\right]^3
e^{3c^2\left[\sqrt{T}+p\right]^2T}\lVert x-y\rVert}{\sqrt{t-s}}.
\end{align}

\item \label{f28}
For all
$k\in [1,d]\cap\Z$,
$\delta\in \tilde{\mathbbm{S}}$, $s\in [0,T]$, $t\in [s,T]$, 
$x\in \R^d$ we have that
\begin{align}
\left\lVert
\mathcal{V}^{\delta,s,x,k}_t
-
\mathcal{V}^{\iota,s,x,k}_t\right\rVert_{\frac{p}{2}}
\leq 
\frac{15
c(bc+c^2)
\left[\sqrt{T}+p\right]^6 e^{3c^2\left[\sqrt{T}+p\right]^2T}
\left(e^{1.5\bar{c}T}\varphi(x)\right)^\frac{1}{p}
\lvert\delta\rvert^\frac{1}{2}}{\sqrt{t-s}}
.
\end{align}

\item \label{f39}
For all
$k\in [1,d]\cap\Z$,
 $s\in [0,T]$, 
$\tilde{s}\in [s,T]$,
$t\in [\tilde{s},T]$, 
$x\in \R^d$ we have that
\begin{align}
\left\lVert
\mathcal{V}^{\iota,\tilde{s},x,k}_t-\mathcal{V}^{\iota,s,x,k}_t
\right\rVert_{\frac{p}{2}}
\leq \frac{13(b+c)ce^{3c^2\left[\sqrt{T}+p\right]^2T}
\left[\sqrt{T}+p\right]^4
e^\frac{1.5\bar{c}T}{p}\varphi^\frac{1}{p}(x)\lvert s-\tilde{s}\rvert^\frac{1}{2}}{\sqrt{t-s}\sqrt{t-\tilde{s}}}.
\end{align}
\end{enumerate}
\end{lemma}
\begin{proof}
[Proof of \cref{f10b}]First, \cite[Theorem~3.2]{HN2022} (with $b\gets\infty$, $V\gets \varphi$ in the notation of \cite[Theorem~3.2]{HN2022}) shows \eqref{f11}--\eqref{f13}.

Next,
H\"older's inequality and \eqref{f04} prove for all 
$k\in [1,d]\cap\Z$,
$\delta\in \tilde{\mathbbm{S}}$, $s\in [0,T]$, $t\in [s,T]$, $x\in \R^d$ 
and all stopping times $\tau\colon\Omega \to[0,T]$ 
that 
\begin{align}
&\left\lVert
\int_s^{\max\{s,\delta(\min\{t,\tau\})\} }
\left((\operatorname{D}\!\mu)(\mathcal{X}^{\delta,s,x}_{\max\{s,\delta(r)\}})
\right)\!\left(
\mathcal{D}^{\delta,s,x,k}_{\max\{s,\delta(r)\}}\right)dr
\right\rVert_p\nonumber\\
&=\left\lVert
\int_s^{t }
\left((\operatorname{D}\!\mu)(\mathcal{X}^{\delta,s,x}_{\max\{s,\delta(r)\}})
\right)\!\left(
\mathcal{D}^{\delta,s,x,k}_{\max\{s,\delta(r)\}}\right)\1_{r\leq \max\{s,\delta(\min\{t,\tau\})\} }dr
\right\rVert_p\nonumber\\
&=\left\lVert
\int_s^{t }
\left((\operatorname{D}\!\mu)(\mathcal{X}^{\delta,s,x}_{\max\{s,\delta(\min\{r,\tau\})\}})
\right)\!\left(
\mathcal{D}^{\delta,s,x,k}_{\max\{s,\delta(\min\{r,\tau\})\}}\right)\1_{r\leq \max\{s,\delta(\min\{t,\tau\})\} }dr
\right\rVert_p\nonumber\\
&\leq 
\int_s^{t }
\left\lVert
\left((\operatorname{D}\!\mu)(\mathcal{X}^{\delta,s,x}_{\max\{s,\delta(\min\{r,\tau\})\}})
\right)\!\left(
\mathcal{D}^{\delta,s,x,k}_{\max\{s,\delta(\min\{r,\tau\})\}}\right)\1_{r\leq \max\{s,\delta(\min\{t,\tau\})\} }\right\rVert_p dr
\nonumber\\
&\leq \sqrt{T}
\left[
\int_s^t
\left\lVert
\left((\operatorname{D}\!\mu)(\mathcal{X}^{\delta,s,x}_{\max\{s,\delta(\min \{r,\tau\})\}})
\right)\!\left(
\mathcal{D}^{\delta,s,x,k}_{\max\{s,\delta(\min \{r,\tau\})\}}\right)
\right\rVert_p^2
dr\right]^\frac{1}{2}\nonumber\\
&\leq\xeqref{f04} c\sqrt{T}
\left[
\int_s^t\left\lVert
\mathcal{D}^{\delta,s,x,k}_{\max\{s,\delta(\min \{r,\tau\})\}}
\right\rVert_p^2
dr\right]^\frac{1}{2}.\label{f08}
\end{align}
In addition, the Burkholder-Davis-Gundy inequality (cf. \cite[Lemma~7.7]{DPZ1992}) and  \eqref{f04} imply for all 
$k\in [1,d]\cap\Z$,
$\delta\in \tilde{\mathbbm{S}}$, $s\in [0,T]$, $t\in [s,T]$, $x\in \R^d$
and all stopping times $\tau\colon\Omega \to[0,T]$  that
\begin{align}
&
\left\lVert\int_{s}^{\max\{s,\delta(\min\{t,\tau\})\} }\left((\operatorname{D}\!\sigma)(\mathcal{X}^{\delta,s,x}_{\max\{s,\delta(r)\}})
\right)\!\left(
\mathcal{D}^{\delta,s,x,k}_{\max\{s,\delta(r)\}}\right)dW_r\right\rVert_p\nonumber\\
&=
\left\lVert\int_{s}^{t}\left((\operatorname{D}\!\sigma)(\mathcal{X}^{\delta,s,x}_{\max\{s,\delta(r)\}})
\right)\!\left(
\mathcal{D}^{\delta,s,x,k}_{\max\{s,\delta(r)\}}\right)
\1_{r\leq \max\{s,\delta(\min\{t,\tau\})\}}
dW_r\right\rVert_p\nonumber\\
&=
\left\lVert\int_{s}^{t}\left((\operatorname{D}\!\sigma)(\mathcal{X}^{\delta,s,x}_{\max\{s,\delta(\min \{r,\tau\})\}})
\right)\!\left(
\mathcal{D}^{\delta,s,x,k}_{\max\{s,\delta(\min \{r,\tau\})\}}\right)
\1_{r\leq \max\{s,\delta(\min\{t,\tau\})\}}
dW_r\right\rVert_p\nonumber\\
&\leq p 
\left[
\int_{s}^{t}\left\lVert
\left((\operatorname{D}\!\sigma)(\mathcal{X}^{\delta,s,x}_{\max\{s,\delta(\min \{r,\tau\})\}})
\right)\!\left(
\mathcal{D}^{\delta,s,x,k}_{\max\{s,\delta(\min \{r,\tau\})\}}\right)
\right\rVert_p^2dr\right]^\frac{1}{2}\nonumber\\
&
\leq\xeqref{f04} pc
\left[
\int_{s}^{t}\left\lVert
\mathcal{D}^{\delta,s,x,k}_{\max\{s,\delta(\min \{r,\tau\})\}}
\right\rVert_p^2dr\right]^\frac{1}{2}.\label{f09b}
\end{align}
Thus, the triangle inequality, \eqref{f07}, 
and \eqref{f08} prove for all 
$k\in [1,d]\cap\Z$,
$\delta\in \tilde{\mathbbm{S}}$, $s\in [0,T]$, $t\in [s,T]$, $x\in \R^d$ that
\begin{align}
\left\lVert
\mathcal{D}^{\delta,s,x,k}_{\max\{s,\delta(\min\{t,\tau\})\} }\right\rVert_p&\leq\xeqref{f07} 1+\left\lVert
\int_s^{\max\{s,\delta(\min\{t,\tau\})\} }
\left((\operatorname{D}\!\mu)(\mathcal{X}^{\delta,s,x}_{\max\{s,\delta(r)\}})
\right)\!\left(
\mathcal{D}^{\delta,s,x,k}_{\max\{s,\delta(r)\}}\right)dr
\right\rVert_p\nonumber\\
&\quad 
+\left\lVert\int_{s}^{\max\{s,\delta(\min\{t,\tau\})\} }\left((\operatorname{D}\!\sigma)(\mathcal{X}^{\delta,s,x}_{\max\{s,\delta(r)\}})
\right)\!\left(
\mathcal{D}^{\delta,s,x,k}_{\max\{s,\delta(r)\}}\right)dW_r\right\rVert_p\nonumber\\
&\leq \xeqref{f08}\xeqref{f09b}1+c\left[\sqrt{T}+p\right]\left[
\int_s^t\left\lVert
\mathcal{D}^{\delta,s,x,k}_{\max\{s,\delta(\min \{r,\tau\})\}}
\right\rVert_p^2
dr\right]^\frac{1}{2}.
\end{align}
This and the fact that
$\forall\, x,y\in \R\colon (x+y)^2\leq 2x^2+2y^2$
show for all 
$k\in [1,d]\cap\Z$,
$\delta\in \tilde{\mathbbm{S}}$, $s\in [0,T]$, $t\in [s,T]$, $x\in \R^d$ that
\begin{align}
\left\lVert
\mathcal{D}^{\delta,s,x,k}_{\max\{s,\delta(\min\{t,\tau\})\} }\right\rVert_p^2
\leq 2+2c^2\left[\sqrt{T}+p\right]^2
\int_s^t\left\lVert
\mathcal{D}^{\delta,s,x,k}_{\max\{s,\delta(\min \{r,\tau\})\}}
\right\rVert_p^2
dr.\label{f16}
\end{align}
For every 
$k\in [1,d]\cap\Z$,
$\delta\in \tilde{\mathbbm{S}}$, $s\in [0,T]$, $x\in \R^d$
let 
$\tau_n^{\delta,s,x,k}=\min \Bigl\{\inf\Bigl\{t\in [0,T]\colon \lVert \mathcal{D}^{\delta,s,x,k}_{\max\{s,\delta(t)\}}\rVert\geq n\Bigr\},T\Bigr\}$.
Then for all
$k\in [1,d]\cap\Z$,
$\delta\in \tilde{\mathbbm{S}}$, $s\in [0,T]$, $t\in [s,T]$, $x\in \R^d$, $n\in \N$ 
we have that
$
\left\lVert
\mathcal{D}^{\delta,s,x,k}_{\max\{s,\delta(\min\{t,\tau_n^{\delta,s,x,k}\})\} }\right\rVert\leq n
$.
Then \eqref{f16}
 and Gr\"onwall's lemma 
imply for all  $k\in [1,d]\cap\Z$,
$\delta\in \tilde{\mathbbm{S}}$, $s\in [0,T]$, $t\in [s,T]$, $x\in \R^d$, $n\in \N$ that
\begin{align}
\left\lVert
\mathcal{D}^{\delta,s,x,k}_{\max\{s,\delta(\min\{t,\tau_n^{\delta,s,x,k}\})\} }\right\rVert_p^2
\leq 2e^{2c^2\left[\sqrt{T}+p\right]^2(t-s)}.
\end{align}
This and Fatou's lemma
prove for all  $k\in [1,d]\cap\Z$,
$\delta\in \tilde{\mathbbm{S}}$, $s\in [0,T]$, $t\in [s,T]$, $x\in \R^d$ that
\begin{align}
\left\lVert
\mathcal{D}^{\delta,s,x,k}_{\max\{s,\delta(t)\} }\right\rVert_p^2
\leq {2}e^{2c^2\left[\sqrt{T}+p\right]^2(t-s)}.
\end{align}
Next, the triangle inequality, \eqref{f07}, H\"older's inequality, 
the Burkholder-Davis-Gundy inequality (cf. \cite[Lemma~7.7]{DPZ1992}), and \eqref{f04} prove for all  $k\in [1,d]\cap\Z$,
$\delta\in \tilde{\mathbbm{S}}$, $s\in [0,T]$, $t\in [s,T]$, $x\in \R^d$ that
\begin{align}
\left\lVert
\mathcal{D}^{\delta,s,x,k}_t\right\rVert_p&\leq \xeqref{f07}1+
\left\lVert
\int_{s}^{t}
\left((\operatorname{D}\!\mu)(\mathcal{X}^{\delta,s,x}_{\max\{s,\delta(r)\}})
\right)\!\left(
\mathcal{D}^{\delta,s,x,k}_{\max\{s,\delta(r)\}}
\right)dr\right\rVert_p\nonumber\\
&\quad 
+\left\lVert
\int_{s}^{t}\left((\operatorname{D}\!\sigma)(\mathcal{X}^{\delta,s,x}_{\max\{s,\delta(r)\}})\right)
\!\left(\mathcal{D}^{\delta,s,x,k}_{\max\{s,\delta(r)\}}\right)dW_r\right\rVert_p\nonumber\\
&\leq 
1+\sqrt{T}\left[
\int_{s}^{t}
\left\lVert
\left((\operatorname{D}\!\mu)(\mathcal{X}^{\delta,s,x}_{\max\{s,\delta(r)\}})
\right)\!\left(
\mathcal{D}^{\delta,s,x,k}_{\max\{s,\delta(r)\}}
\right)\right\rVert_p^2 dr\right]^\frac{1}{2}\nonumber\\
&\quad+p
\left[
\int_{s}^{t}
\left\lVert
\left((\operatorname{D}\!\sigma)(\mathcal{X}^{\delta,s,x}_{\max\{s,\delta(r)\}})
\right)\!\left(
\mathcal{D}^{\delta,s,x,k}_{\max\{s,\delta(r)\}}
\right)\right\rVert_p^2 dr\right]^\frac{1}{2}\nonumber\\
&\leq \xeqref{f04}1+c\left[\sqrt{T}+p\right]
\left[
\int_{s}^{t}
\left\lVert
\mathcal{D}^{\delta,s,x,k}_{\max\{s,\delta(r)\}}
\right\rVert_p^2 dr\right]^\frac{1}{2}.
\end{align}
This and the fact that 
$\forall\, x,y\in \R\colon (x+y)^2\leq 2x^2+2y^2$  show for all  $k\in [1,d]\cap\Z$,
$\delta\in \tilde{\mathbbm{S}}$, $s\in [0,T]$, $t\in [s,T]$, $x\in \R^d$ that
\begin{align}
\left\lVert
\mathcal{D}^{\delta,s,x,k}_t\right\rVert^2_p&\leq 2+2c^2\left[\sqrt{T}+p\right]^2
\int_{s}^{t}
\left\lVert
\mathcal{D}^{\delta,s,x,k}_{\max\{s,\delta(r)\}}
\right\rVert_p^2 dr\nonumber\\
&\leq 2+2c^2\left[\sqrt{T}+p\right]^2\int_{s}^{t}2e^{2c^2\left[\sqrt{T}+p\right]^2(r-s)}\,dr\nonumber\\
&=2+\left[2e^{2c^2\left[\sqrt{T}+p\right]^2(r-s)}|_{r=s}^t\right]\nonumber\\
&=2e^{2c^2\left[\sqrt{T}+p\right]^2(t-s)}
\end{align}
and hence
\begin{align}
\left\lVert
\mathcal{D}^{\delta,s,x,k}_t\right\rVert_p\leq \sqrt{2}
e^{ c^2\left[\sqrt{T}+p\right]^2T}.\label{f09}
\end{align}
Therefore, \eqref{f19},
the Burkholder-Davis-Gundy inequality (cf. \cite[Lemma~7.7]{DPZ1992}), and \eqref{f04} imply for all  $k\in [1,d]\cap\Z$,
$\delta\in \tilde{\mathbbm{S}}$, $s\in [0,T]$, $t\in [s,T]$, $x\in \R^d$ that
\begin{align}
\left\lVert \mathcal{V}^{\delta,s,x,k}_t\right\rVert_p&=\xeqref{f19}\frac{1}{t-s}\left\lVert\int_{s}^{t}
\left(
\sigma^{-1}(\mathcal{X}^{\delta,s,x}_{\max\{s,\delta(r)\}})
\mathcal{D}^{\delta,s,x,k}_{\max\{s,\delta(r)\}}
\right)^\top dW_r\right\rVert_p\nonumber\\
&\leq \frac{1}{t-s}p\left[
\int_{s}^{t}
\left\lVert
\sigma^{-1}(\mathcal{X}^{\delta,s,x}_{\max\{s,\delta(r)\}})
\mathcal{D}^{\delta,s,x,k}_{\max\{s,\delta(r)\}}
\right\rVert_p^2 dr
\right]^\frac{1}{2}\nonumber\\
&\leq\xeqref{f04} \frac{1}{t-s}pc\left[
\int_{s}^{t}
\left\lVert
\mathcal{D}^{\delta,s,x,k}_{\max\{s,\delta(r)\}}
\right\rVert_p^2 dr
\right]^\frac{1}{2}\nonumber\\
&\leq \frac{1}{t-s}pc\sqrt{t-s}\xeqref{f09}\sqrt{2}
e^{ c^2\left[p+\sqrt{T}\right]^2T}\nonumber\\
&=\frac{\sqrt{2}pce^{ c^2\left[p+\sqrt{T}\right]^2T}}{\sqrt{t-s}}.
\end{align}
This shows \eqref{f26}.

Next, the triangle inequality, H\"older's inequality, \eqref{f10}, \eqref{f04}, \eqref{f13}, and \eqref{f09} imply for all 
$k\in [1,d]\cap\Z$,
$\delta\in \tilde{\mathbbm{S}}$, $s\in [0,T]$, $r\in [s,T]$, $x,y\in \R^d$ that
\begin{align}
&\left\lVert
\left((\operatorname{D}\!\mu)(\mathcal{X}^{\delta,s,x}_{\max\{s,\delta(r)\}})
\right)\!\left(
\mathcal{D}^{\delta,s,x,k}_{\max\{s,\delta(r)\}}\right)
-
\left((\operatorname{D}\!\mu)(\mathcal{X}^{\delta,s,y}_{\max\{s,\delta(r)\}})
\right)\!\left(
\mathcal{D}^{\delta,s,y,k}_{\max\{s,\delta(r)\}}\right)\right\rVert_{\frac{p}{2}}\nonumber\\
&
\leq \left\lVert
\left((\operatorname{D}\!\mu)(\mathcal{X}^{\delta,s,x}_{\max\{s,\delta(r)\}})
-(\operatorname{D}\!\mu)(\mathcal{X}^{\delta,s,y}_{\max\{s,\delta(r)\}})
\right)\!\left(
\mathcal{D}^{\delta,s,x,k}_{\max\{s,\delta(r)\}}\right)\right\rVert_{\frac{p}{2}}\nonumber\\
&\quad 
+\left\lVert
\left((\operatorname{D}\!\mu)(\mathcal{X}^{\delta,s,y}_{\max\{s,\delta(r)\}})
\right)\!\left(
\mathcal{D}^{\delta,s,x,k}_{\max\{s,\delta(r)\}}-
\mathcal{D}^{\delta,s,y,k}_{\max\{s,\delta(r)\}}\right)\right\rVert_\frac{p}{2}\nonumber\\
&\leq\xeqref{f10} b\left\lVert
\mathcal{X}^{\delta,s,x}_{\max\{s,\delta(r)\}}
-
\mathcal{X}^{\delta,s,y}_{\max\{s,\delta(r)\}}
\right\rVert_{p}
\left\lVert
\mathcal{D}^{\delta,s,x,k}_{\max\{s,\delta(r)\}}
\right\rVert_{p}+\xeqref{f04}c
\left\lVert
\mathcal{D}^{\delta,s,x,k}_{\max\{s,\delta(r)\}}-
\mathcal{D}^{\delta,s,y,k}_{\max\{s,\delta(r)\}}\right\rVert_{\frac{p}{2}}\nonumber\\
&\leq b\xeqref{f13}\sqrt{2}\lVert x-y\rVert e^{c^2\left[\sqrt{T}+p\right]^2T}\xeqref{f09}
 \sqrt{2}
e^{ c^2\left[p+\sqrt{T}\right]^2T}+c\left\lVert
\mathcal{D}^{\delta,s,x,k}_{\max\{s,\delta(r)\}}-
\mathcal{D}^{\delta,s,y,k}_{\max\{s,\delta(r)\}}\right\rVert_{\frac{p}{2}}\nonumber\\
&=2b\lVert x-y\rVert 
 e^{2 c^2\left[p+\sqrt{T}\right]^2T}+c\left\lVert
\mathcal{D}^{\delta,s,x,k}_{\max\{s,\delta(r)\}}-
\mathcal{D}^{\delta,s,y,k}_{\max\{s,\delta(r)\}}\right\rVert_\frac{p}{2}.
\label{f14}
\end{align}
Similarly, we have for all 
$k\in [1,d]\cap\Z$,
$\delta\in \tilde{\mathbbm{S}}$, $s\in [0,T]$, $r\in [s,T]$, $x,y\in \R^d$ that
\begin{align}
&\left\lVert
\left((\operatorname{D}\!\sigma)(\mathcal{X}^{\delta,s,x}_{\max\{s,\delta(r)\}})
\right)\!\left(
\mathcal{D}^{\delta,s,x,k}_{\max\{s,\delta(r)\}}\right)
-
\left((\operatorname{D}\!\sigma)(\mathcal{X}^{\delta,s,y}_{\max\{s,\delta(r)\}})
\right)\!\left(
\mathcal{D}^{\delta,s,y,k}_{\max\{s,\delta(r)\}}\right)\right\rVert_{\frac{p}{2}}\nonumber\\
&\leq 2 b\lVert x-y\rVert 
e^{2 c^2\left[p+\sqrt{T}\right]^2T}+c\left\lVert
\mathcal{D}^{\delta,s,x,k}_{\max\{s,\delta(r)\}}-
\mathcal{D}^{\delta,s,y,k}_{\max\{s,\delta(r)\}}\right\rVert_\frac{p}{2}.
\label{f15}
\end{align}
This, 
\eqref{f07},
the triangle inequality, H\"older's inequality, 
the Burkholder-Davis-Gundy inequality (cf. \cite[Lemma~7.7]{DPZ1992}), and \eqref{f14} prove for all 
$k\in [1,d]\cap\Z$,
$\delta\in \tilde{\mathbbm{S}}$, $s\in [0,T]$, $t\in [s,T]$, $x,y\in \R^d$ that
\begin{align}
&
\left\lVert \mathcal{D}^{\delta,s,x,k}_t
-\mathcal{D}^{\delta,s,y,k}_t\right\rVert_{\frac{p}{2}}\nonumber\\
&\leq\xeqref{f07} \left\lVert\int_s^t
\left((\operatorname{D}\!\mu)(\mathcal{X}^{\delta,s,x}_{\max\{s,\delta(r)\}})
\right)\!\left(
\mathcal{D}^{\delta,s,x,k}_{\max\{s,\delta(r)\}}\right)
-
\left((\operatorname{D}\!\mu)(\mathcal{X}^{\delta,s,y}_{\max\{s,\delta(r)\}})
\right)\!\left(
\mathcal{D}^{\delta,s,y,k}_{\max\{s,\delta(r)\}}\right)
dr\right\rVert_{\frac{p}{2}}\nonumber\\
&\quad+
\left\lVert\int_{s}^{t}\left((\operatorname{D}\!\sigma)(\mathcal{X}^{\delta,s,x}_{\max\{s,\delta(r)\}})
\right)\!\left(
\mathcal{D}^{\delta,s,x,k}_{\max\{s,\delta(r)\}}\right)
-
\left((\operatorname{D}\!\sigma)(\mathcal{X}^{\delta,s,y}_{\max\{s,\delta(r)\}})
\right)\!\left(
\mathcal{D}^{\delta,s,y,k}_{\max\{s,\delta(r)\}}\right)
dW_r\right\rVert_{\frac{p}{2}}\nonumber\\
&\leq \sqrt{T}\left[
\int_s^t
\left\lVert
\left((\operatorname{D}\!\mu)(\mathcal{X}^{\delta,s,x}_{\max\{s,\delta(r)\}})
\right)\!\left(
\mathcal{D}^{\delta,s,x,k}_{\max\{s,\delta(r)\}}\right)
-
\left((\operatorname{D}\!\mu)(\mathcal{X}^{\delta,s,y}_{\max\{s,\delta(r)\}})
\right)\!\left(
\mathcal{D}^{\delta,s,y,k}_{\max\{s,\delta(r)\}}\right)
\right\rVert_{\frac{p}{2}}^2dr
\right]^\frac{1}{2}\nonumber\\
&\quad +
p\left[\int_{s}^{t}\left\lVert\left((\operatorname{D}\!\sigma)(\mathcal{X}^{\delta,s,x}_{\max\{s,\delta(r)\}})
\right)\!\left(
\mathcal{D}^{\delta,s,x,k}_{\max\{s,\delta(r)\}}\right)
-
\left((\operatorname{D}\!\sigma)(\mathcal{X}^{\delta,s,y}_{\max\{s,\delta(r)\}})
\right)\!\left(
\mathcal{D}^{\delta,s,y,k}_{\max\{s,\delta(r)\}}\right)
\right\rVert_{\frac{p}{2}}^2
dr\right]^\frac{1}{2}\nonumber\\
&\leq\xeqref{f14}\xeqref{f15} \left[\sqrt{T}+p\right]\left[\sqrt{t-s}\cdot 2 b\lVert x-y\rVert 
 e^{2 c^2\left[p+\sqrt{T}\right]^2T}
+c\left[
\int_{s}^{t}
\left\lVert
\mathcal{D}^{\delta,s,x,k}_{\max\{s,\delta(r)\}}-
\mathcal{D}^{\delta,s,y,k}_{\max\{s,\delta(r)\}}\right\rVert_\frac{p}{2}^2
\right]^\frac{1}{2}
\right]\nonumber\\
&=2 b
\left[\sqrt{T}+p\right]
e^{2 c^2\left[p+\sqrt{T}\right]^2T}
\sqrt{t-s} \lVert x-y\rVert 
+c\left[\sqrt{T}+p\right]
\left[
\int_{s}^{t}
\left\lVert
\mathcal{D}^{\delta,s,x,k}_{\max\{s,\delta(r)\}}-
\mathcal{D}^{\delta,s,y,k}_{\max\{s,\delta(r)\}}\right\rVert_\frac{p}{2}^2
\right]^\frac{1}{2}.
 \end{align}
Therefore, \eqref{f09}, a
Gr\"onwall-type inequality (cf.  \cite[Corollary~2.2]{HN2022}), and the fact that $2\sqrt{2}\leq 3$ imply for all
$k\in [1,d]\cap\Z$,
$\delta\in \tilde{\mathbbm{S}}$, $s\in [0,T]$, $t\in [s,T]$, $x,y\in \R^d$ that
\begin{align}
\left\lVert \mathcal{D}^{\delta,s,x,k}_t
-\mathcal{D}^{\delta,s,y,k}_t\right\rVert_{\frac{p}{2}}
&
\leq 2\sqrt{2}b
\left[\sqrt{T}+p\right]
e^{2 c^2\left[\sqrt{T}+p\right]^2T}
\sqrt{t-s} \lVert x-y\rVert e^{c^2\left[\sqrt{T}+p\right]^2T}\nonumber\\
&\leq 3b\left[\sqrt{T}+p\right]^2
e^{3c^2\left[\sqrt{T}+p\right]^2T}\lVert x-y\rVert.\label{f17}
\end{align}
Furthermore,
H\"older's inequality, \eqref{f10}, \eqref{f13}, and \eqref{f09} prove
for all
$k\in [1,d]\cap\Z$,
$\delta\in \tilde{\mathbbm{S}}$, $s\in [0,T]$, $r\in [s,T]$, $x,y\in \R^d$ that
\begin{align}
&
\left\lVert
\left(
\sigma^{-1}(\mathcal{X}^{\delta,s,x}_{\max\{s,\delta(r)\}})
-
\sigma^{-1}(\mathcal{X}^{\delta,s,y}_{\max\{s,\delta(r)\}})
\right)
\mathcal{D}^{\delta,s,x,k}_{\max\{s,\delta(r)\}}
\right\rVert_{\frac{p}{4}}\nonumber\\
&\leq \xeqref{f10}b
\left\lVert
\mathcal{X}^{\delta,s,x}_{\max\{s,\delta(r)\}}
-
\mathcal{X}^{\delta,s,y}_{\max\{s,\delta(r)\}}
\right\rVert_\frac{p}{2}\left\lVert
\mathcal{D}^{\delta,s,x,k}_{\max\{s,\delta(r)\}}
\right\rVert_{\frac{p}{2}}\nonumber\\
&\leq b\xeqref{f13}\sqrt{2} \lVert x-y\rVert e^{c^2\left[\sqrt{T}+p\right]^2T}
\xeqref{f09}
\sqrt{2}
e^{ c^2\left[\sqrt{T}+p\right]^2T}\nonumber\\
&=2be^{2 c^2\left[\sqrt{T}+p\right]^2T}\lVert x-y\rVert.\label{f21}
\end{align}
Next, \eqref{f04} and \eqref{f17} 
imply for all
$k\in [1,d]\cap\Z$,
$\delta\in \tilde{\mathbbm{S}}$, $s\in [0,T]$, $r\in [s,T]$, $x,y\in \R^d$ that
\begin{align}
&\left\lVert
\sigma^{-1}(\mathcal{X}^{\delta,s,y}_{\max\{s,\delta(r)\}})
\left(
\mathcal{D}^{\delta,s,x,k}_{\max\{s,\delta(r)\}}
-
\mathcal{D}^{\delta,s,y,k}_{\max\{s,\delta(r)\}}\right)
  \right\rVert_\frac{p}{4}
\nonumber\\
&
\leq \xeqref{f04}c\left\lVert
\mathcal{D}^{\delta,s,x,k}_{\max\{s,\delta(r)\}}
-
\mathcal{D}^{\delta,s,y,k}_{\max\{s,\delta(r)\}}
\right\rVert_\frac{p}{4}\nonumber\\
&\leq\xeqref{f17} 3bc\left[\sqrt{T}+p\right]^2
e^{3c^2\left[\sqrt{T}+p\right]^2T}\lVert x-y\rVert.\label{f20}
\end{align}
Hence,
\eqref{f19}, 
the Burkholder-Davis-Gundy inequality (cf. \cite[Lemma~7.7]{DPZ1992}),
the triangle inequality, and \eqref{f21}  
prove for all
$k\in [1,d]\cap\Z$,
$\delta\in \tilde{\mathbbm{S}}$, $s\in [0,T]$, $t\in [s,T]$, $x,y\in \R^d$ that
\begin{align}
&\left\lVert
\mathcal{V}^{\delta,s,x,k}_t
-\mathcal{V}^{\delta,s,y,k}_t\right\rVert_{\frac{p}{4}}\nonumber
\\
&
=\frac{1}{t-s}\left\lVert\int_{s}^{t}
\left(
\sigma^{-1}(\mathcal{\mathcal{X}}^{\delta,s,x}_{\max\{s,\delta(r)\}})
\mathcal{\mathcal{D}}^{\delta,s,x,k}_{\max\{s,\delta(r)\}}
-
\sigma^{-1}(\mathcal{\mathcal{X}}^{\delta,s,y}_{\max\{s,\delta(r)\}})
\mathcal{\mathcal{D}}^{\delta,s,y,k}_{\max\{s,\delta(r)\}}
\right)^\top dW_r\right\rVert_\frac{p}{4}\nonumber\\
&\leq \frac{\frac{p}{4}}{t-s}\left[
\int_{s}^{t}
\left\lVert
\left(
\sigma^{-1}(\mathcal{\mathcal{X}}^{\delta,s,x}_{\max\{s,\delta(r)\}})
\mathcal{\mathcal{D}}^{\delta,s,x,k}_{\max\{s,\delta(r)\}}
-
\sigma^{-1}(\mathcal{\mathcal{X}}^{\delta,s,y}_{\max\{s,\delta(r)\}})
\mathcal{\mathcal{D}}^{\delta,s,y,k}_{\max\{s,\delta(r)\}}
\right)^\top \right\rVert_\frac{p}{4}^2dr\right]^\frac{1}{2}\nonumber\\
&\leq 
\frac{\frac{p}{4}}{t-s}
\left[
\int_{s}^{t}
\left\lVert
\left(\left(
\sigma^{-1}(\mathcal{\mathcal{X}}^{\delta,s,x}_{\max\{s,\delta(r)\}})
-
\sigma^{-1}(\mathcal{\mathcal{X}}^{\delta,s,y}_{\max\{s,\delta(r)\}})
\right)
\mathcal{\mathcal{D}}^{\delta,s,x,k}_{\max\{s,\delta(r)\}}
\right)^\top \right\rVert_\frac{p}{4}^2 dr\right]^\frac{1}{2}\nonumber\\
&\quad +\frac{\frac{p}{4}}{t-s}
\left[
\int_{s}^{t}\left\lVert
\left(
\sigma^{-1}(\mathcal{\mathcal{X}}^{\delta,s,y}_{\max\{s,\delta(r)\}})
\left(
\mathcal{\mathcal{D}}^{\delta,s,x,k}_{\max\{s,\delta(r)\}}
-
\mathcal{\mathcal{D}}^{\delta,s,y,k}_{\max\{s,\delta(r)\}}\right)
\right)^\top  \right\rVert_\frac{p}{4}^2 dr\right]^\frac{1}{2}\nonumber\\
&\leq \frac{\frac{p}{4}}{\sqrt{t-s}}\left[\xeqref{f21}2be^{2 c^2\left[\sqrt{T}+p\right]^2T}\lVert x-y\rVert+\xeqref{f20}3bc\left[\sqrt{T}+p\right]^2
e^{3c^2\left[\sqrt{T}+p\right]^2T}\lVert x-y\rVert\right]\nonumber\\
&\leq \frac{5bc\frac{p}{4}
\left[\sqrt{T}+p\right]^2
e^{3c^2\left[\sqrt{T}+p\right]^2T}\lVert x-y\rVert}{\sqrt{t-s}}\nonumber\\
&\leq \frac{2bc
\left[\sqrt{T}+p\right]^3
e^{3c^2\left[\sqrt{T}+p\right]^2T}\lVert x-y\rVert}{\sqrt{t-s}}.
\end{align}
This proves \eqref{f27}.

Next, \eqref{f07}, the triangle inequality, H\"older's inequality, 
the Burkholder-Davis-Gundy inequality (cf. \cite[Lemma~7.7]{DPZ1992}),
\eqref{f04}, and \eqref{f09} show for all
$k\in [1,d]\cap\Z$,
$\delta\in \tilde{\mathbbm{S}}$, $s\in [0,T]$, $t\in [s,T]$, 
$\tilde{t}\in [t,T]$,
$x\in \R^d$ that
\begin{align}
&
\left\lVert
\mathcal{D}^{\delta,s,x,k}_{\tilde{t}}
-\mathcal{D}^{\delta,s,x,k}_t\right\rVert_{p}\nonumber\\
&\leq \left\lVert\int_{t}^{\tilde{t}}
\left((\operatorname{D}\!\mu)(\mathcal{X}^{\delta,s,x}_{\max\{s,\delta(r)\}})
\right)\!\left(
\mathcal{D}^{\delta,s,x,k}_{\max\{s,\delta(r)\}}
\right)dr \right\rVert_p
+\left\lVert
\int_{t}^{\tilde{t}}\left((\operatorname{D}\!\sigma)(\mathcal{X}^{\delta,s,x}_{\max\{s,\delta(r)\}})\right)
\!\left(\mathcal{D}^{\delta,s,x,k}_{\max\{s,\delta(r)\}}\right)dW_r \right\rVert_p\nonumber
\\
&\leq \sqrt{T}
\left[
\int_{t}^{\tilde{t}}
\left\lVert
\left((\operatorname{D}\!\mu)(\mathcal{X}^{\delta,s,x}_{\max\{s,\delta(r)\}})
\right)\!\left(
\mathcal{D}^{\delta,s,x,k}_{\max\{s,\delta(r)\}}
\right)\right\rVert_p^2 dr 
\right]^\frac{1}{2}\nonumber\\
&\quad
+p\left[
\int_{t}^{\tilde{t}}
\left\lVert
\left((\operatorname{D}\!\sigma)(\mathcal{X}^{\delta,s,x}_{\max\{s,\delta(r)\}})\right)
\!\left(\mathcal{D}^{\delta,s,x,k}_{\max\{s,\delta(r)\}}\right)\right\rVert_p^2dr\right]^\frac{1}{2}\nonumber \\
&\leq \left[\sqrt{T}+p\right]c
\left[
\int_{t}^{\tilde{t}}
\left\lVert
\mathcal{D}^{\delta,s,x,k}_{\max\{s,\delta(r)\}}\right\rVert_p^2dr\right]^\frac{1}{2}\nonumber\\
&\leq \left[\sqrt{T}+p\right]c
\sqrt{\tilde{t}-t}\xeqref{f09}\sqrt{2}
e^{ c^2\left[\sqrt{T}+p\right]^2T}\nonumber\\
&=\sqrt{2}c\left[\sqrt{T}+p\right]
e^{ c^2\left[\sqrt{T}+p\right]^2T}\sqrt{\tilde{t}-t}.
\label{f07c}
\end{align}
In addition,
\eqref{f10}, H\"older's inequality, the triangle inequality, \eqref{f13}, \eqref{f12}, \eqref{f09}, the fact that
$(5+\sqrt{2})\sqrt{2}\leq 10$, and the fact that $c\geq 1$ imply for all
$k\in [1,d]\cap\Z$,
$\delta\in \tilde{\mathbbm{S}}$, $s\in [0,T]$, $r\in [s,T]$, 
$x\in \R^d$ that
\begin{align}
&
\left\lVert
\left((\operatorname{D}\!\mu)(\mathcal{X}^{\delta,s,x}_{\max\{s,\delta(r)\}})
-
(\operatorname{D}\!\mu)(\mathcal{X}^{\iota,s,x}_{r})
\right)\!\left(
\mathcal{D}^{\delta,s,x,k}_{\max\{s,\delta(r)\}}
\right)
\right\rVert_\frac{p}{2}\nonumber\\
&\leq \xeqref{f10}b
\left\lVert
\mathcal{X}^{\delta,s,x}_{\max\{s,\delta(r)\}}
-\mathcal{X}^{\iota,s,x}_{r}
\right\rVert_{p}
\left\lVert
\mathcal{D}^{\delta,s,x,k}_{\max\{s,\delta(r)\}}
\right\rVert_{p}\nonumber\\
&\leq 
b\left[
\left\lVert
\mathcal{X}^{\delta,s,x}_{\max\{s,\delta(r)\}}
-\mathcal{X}^{\delta,s,x}_{r}
\right\rVert_{p}+
\left\lVert
\mathcal{X}^{\delta,s,x}_{r}
-\mathcal{X}^{\iota,s,x}_{r}
\right\rVert_p
\right]
\left\lVert
\mathcal{D}^{\delta,s,x,k}_{\max\{s,\delta(r)\}}
\right\rVert_{p}\nonumber\\
&\leq b\left[\xeqref{f13}
5e^{c^2\left[\sqrt{T}+p\right]^2T}
\left[\sqrt{T}+p\right]
e^\frac{1.5\bar{c}T}{p}\varphi^\frac{1}{p}(x)
\lvert \delta\rvert^\frac{1}{2}
+\xeqref{f12}
\sqrt{2}c \left[\sqrt{T}+p\right]^3e^{c^2\left[\sqrt{T}+p\right]^2T}
\left(e^{1.5\bar{c}T}\varphi(x)\right)^\frac{1}{p}
\lvert\delta\rvert^\frac{1}{2}
\right]\nonumber\\
&\quad \xeqref{f09}\sqrt{2}
e^{ c^2\left[\sqrt{T}+p\right]^2T}\nonumber\\
&\leq 10bc
\left[\sqrt{T}+p\right]^3e^{2c^2\left[\sqrt{T}+p\right]^2T}
\left(e^{1.5\bar{c}T}\varphi(x)\right)^\frac{1}{p}
\lvert\delta\rvert^\frac{1}{2}.\label{f22}
\end{align}
Furthermore, \eqref{f04},
the triangle inequality, and \eqref{f07c}
 imply for all
$k\in [1,d]\cap\Z$,
$\delta\in \tilde{\mathbbm{S}}$, $s\in [0,T]$, $r\in [s,T]$, 
$x\in \R^d$ that
\begin{align}
&
\left\lVert
(\operatorname{D}\!\mu)(\mathcal{X}^{\iota,s,x}_{r})
\!\left(
\mathcal{D}^{\delta,s,x,k}_{\max\{s,\delta(r)\}}-
\mathcal{D}^{\iota,s,x,k}_{r}\right)
\right\rVert_{\frac{p}{2}}\nonumber\\
&\leq \xeqref{f04}c\left\lVert
\mathcal{D}^{\delta,s,x,k}_{\max\{s,\delta(r)\}}-
\mathcal{D}^{\iota,s,x,k}_{r}\right\rVert_\frac{p}{2}\nonumber\\
&\leq 
c\left\lVert
\mathcal{D}^{\delta,s,x,k}_{\max\{s,\delta(r)\}}-
\mathcal{D}^{\delta,s,x,k}_{r}
\right\rVert_{\frac{p}{2}}
+c
\left\lVert
\mathcal{D}^{\delta,s,x,k}_{r}
-
\mathcal{D}^{\iota,s,x,k}_{r}
\right\rVert_{\frac{p}{2}}\nonumber\\
&\leq c\xeqref{f07c}
\sqrt{2}c\left[\sqrt{T}+p\right]
e^{ c^2\left[\sqrt{T}+p\right]^2T}\lvert\delta\rvert^{\frac{1}{2}}+
c
\left\lVert
\mathcal{D}^{\delta,s,x,k}_{r}
-
\mathcal{D}^{\iota,s,x,k}_{r}
\right\rVert_{\frac{p}{2}}\nonumber\\
&= \sqrt{2}c^2\left[\sqrt{T}+p\right]
e^{ c^2\left[\sqrt{T}+p\right]^2T}\lvert\delta\rvert^{\frac{1}{2}}+
c
\left\lVert
\mathcal{D}^{\delta,s,x,k}_{r}
-
\mathcal{D}^{\iota,s,x,k}_{r}
\right\rVert_{\frac{p}{2}}.\label{f18}
\end{align}
This, the triangle inequality, and \eqref{f22}
prove for all
$k\in [1,d]\cap\Z$,
$\delta\in \tilde{\mathbbm{S}}$, $s\in [0,T]$, $r\in [s,T]$, 
$x\in \R^d$ that
\begin{align}
&\left\lVert
\left((\operatorname{D}\!\mu)(\mathcal{X}^{\delta,s,x}_{\max\{s,\delta(r)\}})
\right)\!\left(
\mathcal{D}^{\delta,s,x,k}_{\max\{s,\delta(r)\}}
\right)-\left((\operatorname{D}\!\mu)(\mathcal{X}^{\iota,s,x}_{r})
\right)\!\left(
\mathcal{D}^{\iota,s,x,k}_{r}
\right)\right\rVert_{\frac{p}{2}}\nonumber\\
&\leq \left\lVert\left((\operatorname{D}\!\mu)(\mathcal{X}^{\delta,s,x}_{\max\{s,\delta(r)\}})
-
(\operatorname{D}\!\mu)(\mathcal{X}^{\iota,s,x}_{r})
\right)\!\left(
\mathcal{D}^{\delta,s,x,k}_{\max\{s,\delta(r)\}}
\right)\right\rVert_{\frac{p}{2}}
+
\left\lVert
(\operatorname{D}\!\mu)(\mathcal{X}^{\iota,s,x}_{r})
\!\left(
\mathcal{D}^{\delta,s,x,k}_{\max\{s,\delta(r)\}}-
\mathcal{D}^{\iota,s,x,k}_{r}
\right)\right\rVert_{\frac{p}{2}}\nonumber\\
&\leq \xeqref{f22}10bc
\left[\sqrt{T}+p\right]^3e^{2c^2\left[\sqrt{T}+p\right]^2T}
\left(e^{1.5\bar{c}T}\varphi(x)\right)^\frac{1}{p}
\lvert\delta\rvert^\frac{1}{2}\nonumber\\
&\quad +\xeqref{f18}
\sqrt{2}c^2\left[\sqrt{T}+p\right]
e^{ c^2\left[\sqrt{T}+p\right]^2T}\lvert\delta\rvert^{\frac{1}{2}}+
c
\left\lVert
\mathcal{D}^{\delta,s,x,k}_{r}
-
\mathcal{D}^{\iota,s,x,k}_{r}
\right\rVert_{\frac{p}{2}}\nonumber\\
&\leq 12(bc+c^2)
\left[\sqrt{T}+p\right]^3e^{2c^2\left[\sqrt{T}+p\right]^2T}
\left(e^{1.5\bar{c}T}\varphi(x)\right)^\frac{1}{p}
\lvert\delta\rvert^\frac{1}{2}
+c
\left\lVert
\mathcal{D}^{\delta,s,x,k}_{r}
-
\mathcal{D}^{\iota,s,x,k}_{r}
\right\rVert_{\frac{p}{2}}.\label{f18b}
\end{align}
Similarly, we have 
for all
$k\in [1,d]\cap\Z$,
$\delta\in \tilde{\mathbbm{S}}$, $s\in [0,T]$, $r\in [s,T]$, 
$x\in \R^d$ that
\begin{align}
&\left\lVert
\left((\operatorname{D}\!\sigma)(\mathcal{X}^{\delta,s,x}_{\max\{s,\delta(r)\}})
\right)\!\left(
\mathcal{D}^{\delta,s,x,k}_{\max\{s,\delta(r)\}}
\right)-\left((\operatorname{D}\!\sigma)(\mathcal{X}^{\iota,s,x}_{r})
\right)\!\left(
\mathcal{D}^{\iota,s,x,k}_{r}
\right)\right\rVert_{\frac{p}{2}}\nonumber\\
&\leq 12(bc+c^2)
\left[\sqrt{T}+p\right]^3e^{2c^2\left[\sqrt{T}+p\right]^2T}
\left(e^{1.5\bar{c}T}\varphi(x)\right)^\frac{1}{p}
\lvert\delta\rvert^\frac{1}{2}
+c
\left\lVert
\mathcal{D}^{\delta,s,x,k}_{r}
-
\mathcal{D}^{\iota,s,x,k}_{r}
\right\rVert_{\frac{p}{2}}.\label{f19a}
\end{align}
Thus, \eqref{f07}, the triangle inequality, H\"older's inequality, 
the Burkholder-Davis-Gundy inequality (cf. \cite[Lemma~7.7]{DPZ1992}), and \eqref{f18b} show for all 
$k\in [1,d]\cap\Z$,
$\delta\in \tilde{\mathbbm{S}}$, $s\in [0,T]$, $t\in [s,T]$, 
$x\in \R^d$  that
\begin{align}
&
\left\lVert
\mathcal{D}_t^{\delta,s,x,k}
-\mathcal{D}_t^{\iota,s,x,k}\right\rVert_{\frac{p}{2}}\nonumber\\
&
\leq \xeqref{f07}\left\lVert\int_{s}^{t}
\left((\operatorname{D}\!\mu)(\mathcal{X}^{\delta,s,x}_{\max\{s,\delta(r)\}})
\right)\!\left(
\mathcal{D}^{\delta,s,x,k}_{\max\{s,\delta(r)\}}
\right)-\left((\operatorname{D}\!\mu)(\mathcal{X}^{\iota,s,x}_{r})
\right)\!\left(
\mathcal{D}^{\iota,s,x,k}_{r}
\right)dr\right\rVert_{\frac{p}{2}}\nonumber\\
&\quad +
\left\lVert\int_{s}^{t}
\left((\operatorname{D}\!\sigma)(\mathcal{X}^{\delta,s,x}_{\max\{s,\delta(r)\}})
\right)\!\left(
\mathcal{D}^{\delta,s,x,k}_{\max\{s,\delta(r)\}}
\right)-\left((\operatorname{D}\!\sigma)(\mathcal{X}^{\iota,s,x}_{r})
\right)\!\left(
\mathcal{D}^{\iota,s,x,k}_{r}
\right)dW_r\right\rVert_{\frac{p}{2}}\nonumber\\
&
\leq \sqrt{T}\left[
\int_{s}^{t}
\left\lVert
\left((\operatorname{D}\!\mu)(\mathcal{X}^{\delta,s,x}_{\max\{s,\delta(r)\}})
\right)\!\left(
\mathcal{D}^{\delta,s,x,k}_{\max\{s,\delta(r)\}}
\right)-\left((\operatorname{D}\!\mu)(\mathcal{X}^{\iota,s,x}_{r})
\right)\!\left(
\mathcal{D}^{\iota,s,x,k}_{r}
\right)\right\rVert_{\frac{p}{2}}^2dr
\right]^\frac{1}{2}\nonumber\\
&\quad 
+p
\left[
\int_{s}^{t}
\left\lVert
\left((\operatorname{D}\!\sigma)(\mathcal{X}^{\delta,s,x}_{\max\{s,\delta(r)\}})
\right)\!\left(
\mathcal{D}^{\delta,s,x,k}_{\max\{s,\delta(r)\}}
\right)-\left((\operatorname{D}\!\sigma)(\mathcal{X}^{\iota,s,x}_{r})
\right)\!\left(
\mathcal{D}^{\iota,s,x,k}_{r}
\right)\right\rVert_{\frac{p}{2}}^2dr
\right]^\frac{1}{2}\nonumber\\
&\leq \left[\sqrt{T}+p\right]\sqrt{t-s}\cdot 
12(bc+c^2)
\left[\sqrt{T}+p\right]^3e^{2c^2\left[\sqrt{T}+p\right]^2T}
\left(e^{1.5\bar{c}T}\varphi(x)\right)^\frac{1}{p}
\lvert\delta\rvert^\frac{1}{2}\nonumber\\
&\quad 
+\left[\sqrt{T}+p\right]c
\left[
\int_{s}^{t}
\left\lVert
\mathcal{D}^{\delta,s,x,k}_{r}
-
\mathcal{D}^{\iota,s,x,k}_{r}
\right\rVert_{\frac{p}{2}}^2dr
\right]^\frac{1}{2}\nonumber\\
&\leq \xeqref{f18}\xeqref{f19a}
12(bc+c^2)
\left[\sqrt{T}+p\right]^5e^{2c^2\left[\sqrt{T}+p\right]^2T}
\left(e^{1.5\bar{c}T}\varphi(x)\right)^\frac{1}{p}
\lvert\delta\rvert^\frac{1}{2}\nonumber\\
&\quad +\left[\sqrt{T}+p\right]c
\left[
\int_{s}^{t}
\left\lVert
\mathcal{D}^{\delta,s,x,k}_{r}
-
\mathcal{D}^{\iota,s,x,k}_{r}
\right\rVert_{\frac{p}{2}}^2dr
\right]^\frac{1}{2}.
\end{align}
This, a Gr\"onwall-type inequality (cf. \cite[Corollory 2.2]{HN2022}) and the fact that $12\sqrt{2}\leq 18$ prove for all
$k\in [1,d]\cap\Z$,
$\delta\in \tilde{\mathbbm{S}}$, $s\in [0,T]$, $t\in [s,T]$, 
$x\in \R^d$ that
\begin{align}
&
\left\lVert
\mathcal{D}_t^{\delta,s,x,k}
-\mathcal{D}_t^{\iota,s,x,k}\right\rVert_{\frac{p}{2}}
\leq \sqrt{2}\cdot 12(bc+c^2)
\left[\sqrt{T}+p\right]^5 e^{2c^2\left[\sqrt{T}+p\right]^2T}
\left(e^{1.5\bar{c}T}\varphi(x)\right)^\frac{1}{p}
\lvert\delta\rvert^\frac{1}{2}\cdot 
e^{c^2\left[\sqrt{T}+p\right]^2T}\nonumber\\
&\leq 18
(bc+c^2)
\left[\sqrt{T}+p\right]^5 e^{3c^2\left[\sqrt{T}+p\right]^2T}
\left(e^{1.5\bar{c}T}\varphi(x)\right)^\frac{1}{p}
\lvert\delta\rvert^\frac{1}{2}.\label{f23}
\end{align}
Next, \eqref{f10}, H\"older's inequality, the triangle inequality, \eqref{f12}, \eqref{f13}, \eqref{f09}, the fact that 
$(5+\sqrt{2})\sqrt{2}\leq 10$
imply for all
$k\in [1,d]\cap\Z$,
$\delta\in \tilde{\mathbbm{S}}$, $s\in [0,T]$, $r\in [s,T]$, 
$x\in \R^d$ that
\begin{align}
&
\left\lVert
\left(
\sigma^{-1}(\mathcal{X}^{\delta,s,x}_{\max\{s,\delta(r)\}})
-
\sigma^{-1}(\mathcal{X}^{\iota,s,x}_{r})
\right)
\mathcal{D}^{\delta,s,x,k}_{\max\{s,\delta(r)\}}
\right\rVert_{\frac{p}{2}}\nonumber\\
&\leq\xeqref{f10} b
\left\lVert
\mathcal{X}^{\delta,s,x}_{\max\{s,\delta(r)\}}-
\mathcal{X}^{\iota,s,x}_{r}
\right\rVert_{p}
\left\lVert
\mathcal{D}^{\delta,s,x,k}_{\max\{s,\delta(r)\}}
\right\rVert_{p}
\nonumber
\\
&\leq b\left[
\left\lVert
\mathcal{X}^{\delta,s,x}_{\max\{s,\delta(r)\}}-
\mathcal{X}^{\delta,s,x}_{r}\right\rVert_p
+
\left\lVert
\mathcal{X}^{\delta,s,x}_{r}-
\mathcal{X}^{\iota,s,x}_{r}
\right\rVert_{p}
\right]
\left\lVert
\mathcal{D}^{\delta,s,x,k}_{\max\{s,\delta(r)\}}
\right\rVert_{p}\nonumber\\
&\leq b\left[\xeqref{f13}
5e^{c^2\left[\sqrt{T}+p\right]^2T}
\left[\sqrt{T}+p\right]
e^\frac{1.5\bar{c}T}{p}\varphi^\frac{1}{p}(x)
\lvert \delta\rvert^\frac{1}{2}
+\xeqref{f12}
\sqrt{2}c \left[\sqrt{T}+p\right]^3e^{c^2\left[\sqrt{T}+p\right]^2T}
\left(e^{1.5\bar{c}T}\varphi(x)\right)^\frac{1}{p}
\lvert\delta\rvert^\frac{1}{2}
\right]\nonumber\\
&\quad \xeqref{f09}\sqrt{2}
e^{ c^2\left[\sqrt{T}+p\right]^2T}\nonumber\\
&\leq 10bc
\left[\sqrt{T}+p\right]^3e^{2c^2\left[\sqrt{T}+p\right]^2T}
\left(e^{1.5\bar{c}T}\varphi(x)\right)^\frac{1}{p}
\lvert\delta\rvert^\frac{1}{2}.\label{f24}
\end{align}
Moreover, \eqref{f04}, the triangle inequality,
\eqref{f07c}, and \eqref{f23}
prove for all
$k\in [1,d]\cap\Z$,
$\delta\in \tilde{\mathbbm{S}}$, $s\in [0,T]$, $r\in [s,T]$, 
$x\in \R^d$ that
\begin{align}
&
\left\lVert
\sigma^{-1}(\mathcal{X}^{\iota,s,x}_{r})
\left(
\mathcal{D}^{\delta,s,x,k}_{\max\{s,\delta(r)\}}
-
\mathcal{D}^{\iota,s,x,k}_{r}
\right)\right\rVert_{\frac{p}{2}}\nonumber\\
&\leq \xeqref{f04}
c
\left\lVert
\mathcal{D}^{\delta,s,x,k}_{\max\{s,\delta(r)\}}
-
\mathcal{D}^{\iota,s,x,k}_{r}
\right\rVert_{p}\nonumber\\
&\leq c\left[
\left\lVert
\mathcal{D}^{\delta,s,x,k}_{\max\{s,\delta(r)\}}
-
\mathcal{D}^{\delta,s,x,k}_{r}\right\rVert_p
+
\left\lVert
\mathcal{D}^{\delta,s,x,k}_{r}
-
\mathcal{D}^{\iota,s,x,k}_{r}
\right\rVert_{p}\right]\nonumber\\
&\leq c\left[\xeqref{f07c}
\sqrt{2}c\left[\sqrt{T}+p\right]
e^{ c^2\left[\sqrt{T}+p\right]^2T}\lvert\delta\rvert^\frac{1}{2}
+\xeqref{f23}
18
(bc+c^2)
\left[\sqrt{T}+p\right]^5 e^{3c^2\left[\sqrt{T}+p\right]^2T}
\left(e^{1.5\bar{c}T}\varphi(x)\right)^\frac{1}{p}
\lvert\delta\rvert^\frac{1}{2}
\right]\nonumber\\
&\leq 
20
c(bc+c^2)
\left[\sqrt{T}+p\right]^5 e^{3c^2\left[\sqrt{T}+p\right]^2T}
\left(e^{1.5\bar{c}T}\varphi(x)\right)^\frac{1}{p}
\lvert\delta\rvert^\frac{1}{2}.\label{f25}
\end{align}
This, the triangle inequality, \eqref{f24}, and the fact that $b,c\geq 1$
show for all
$k\in [1,d]\cap\Z$,
$\delta\in \tilde{\mathbbm{S}}$, $s\in [0,T]$, $r\in [s,T]$, 
$x\in \R^d$ that

\begin{align}
&\left\lVert
\sigma^{-1}(\mathcal{X}^{\delta,s,x}_{\max\{s,\delta(r)\}})
\mathcal{D}^{\delta,s,x,k}_{\max\{s,\delta(r)\}}
-
\sigma^{-1}(\mathcal{X}^{\iota,s,x}_{r})
\mathcal{D}^{\iota,s,x,k}_{r}\right\rVert_{\frac{p}{2}}\nonumber\\
&\leq 
\left\lVert
\left(
\sigma^{-1}(\mathcal{X}^{\delta,s,x}_{\max\{s,\delta(r)\}})
-
\sigma^{-1}(\mathcal{X}^{\iota,s,x}_{r})
\right)
\mathcal{D}^{\delta,s,x,k}_{\max\{s,\delta(r)\}}\right\rVert_{\frac{p}{2}}
+\left\lVert
\sigma^{-1}(\mathcal{X}^{\iota,s,x}_{r})
\left(
\mathcal{D}^{\delta,s,x,k}_{\max\{s,\delta(r)\}}
-
\mathcal{D}^{\iota,s,x,k}_{r}
\right)\right\rVert_{\frac{p}{2}}\nonumber\\
&\leq \xeqref{f24}10bc
\left[\sqrt{T}+p\right]^3e^{2c^2\left[\sqrt{T}+p\right]^2T}
\left(e^{1.5\bar{c}T}\varphi(x)\right)^\frac{1}{p}
\lvert\delta\rvert^\frac{1}{2}\nonumber\\
&\quad 
+\xeqref{f25}20
c(bc+c^2)
\left[\sqrt{T}+p\right]^5 e^{3c^2\left[\sqrt{T}+p\right]^2T}
\left(e^{1.5\bar{c}T}\varphi(x)\right)^\frac{1}{p}
\lvert\delta\rvert^\frac{1}{2}\nonumber\\
&\leq 
30
c(bc+c^2)
\left[\sqrt{T}+p\right]^5 e^{3c^2\left[\sqrt{T}+p\right]^2T}
\left(e^{1.5\bar{c}T}\varphi(x)\right)^\frac{1}{p}
\lvert\delta\rvert^\frac{1}{2}.
\end{align}
Therefore, \eqref{f19} and the Burkholder-Davis-Gundy inequality (cf. \cite[Lemma~7.7]{DPZ1992}) 
imply for all
$k\in [1,d]\cap\Z$,
$\delta\in \tilde{\mathbbm{S}}$, $s\in [0,T]$, $t\in [s,T]$, 
$x\in \R^d$ that
\begin{align}
\left\lVert
\mathcal{V}^{\delta,s,x,k}_t
-
\mathcal{V}^{\iota,s,x,k}_t\right\rVert_{\frac{p}{2}}
&=\left\lVert\frac{1}{t-s}\int_{s}^{t}
\left(
\sigma^{-1}(\mathcal{X}^{\delta,s,x}_{\max\{s,\delta(r)\}})
\mathcal{D}^{\delta,s,x,k}_{\max\{s,\delta(r)\}}
-\sigma^{-1}(\mathcal{X}^{\iota,s,x}_{r})
\mathcal{D}^{\iota,s,x,k}_{r}
\right)^\top dW_r \right\rVert_{\frac{p}{2}}\nonumber\\
&\leq \frac{\frac{p}{2}}{t-s}\left[\int_{s}^{t}
\left\lVert
\sigma^{-1}(\mathcal{X}^{\delta,s,x}_{\max\{s,\delta(r)\}})
\mathcal{D}^{\delta,s,x,k}_{\max\{s,\delta(r)\}}
-\sigma^{-1}(\mathcal{X}^{\iota,s,x}_{r})
\mathcal{D}^{\iota,s,x,k}_{r}
\right\rVert^2_{\frac{p}{2}}dr\right]^\frac{1}{2}\nonumber\\
&\leq 
\frac{\frac{p}{2}}{\sqrt{t-s}}
30
c(bc+c^2)
\left[\sqrt{T}+p\right]^5 e^{3c^2\left[\sqrt{T}+p\right]^2T}
\left(e^{1.5\bar{c}T}\varphi(x)\right)^\frac{1}{p}
\lvert\delta\rvert^\frac{1}{2}\nonumber\\
&\leq 
\frac{15
c(bc+c^2)
\left[\sqrt{T}+p\right]^6 e^{3c^2\left[\sqrt{T}+p\right]^2T}
\left(e^{1.5\bar{c}T}\varphi(x)\right)^\frac{1}{p}
\lvert\delta\rvert^\frac{1}{2}}{\sqrt{t-s}}
.
\end{align}
This shows \eqref{f28}.

Next,
H\"older's inequality, the Burkholder-Davis-Gundy inequality (cf. \cite[Lemma~7.7]{DPZ1992}), \eqref{f04}, and \eqref{f09}
prove for all
$k\in [1,d]\cap\Z$,
$s\in [0,T]$, 
$\tilde{s}\in [s,T]$,
$r\in [\tilde{s},T]$, 
$x\in \R^d$ that
\begin{align}
&
\left\lVert
\int_{s}^{\tilde{s}}
\left((\operatorname{D}\!\mu)(\mathcal{X}^{\iota,s,x}_{r})
\right)\!\left(
\mathcal{D}^{\iota,s,x,k}_{r}
\right)dr\right\rVert_{p}
+
\left\lVert
\int_{s}^{\tilde{s}}\left((\operatorname{D}\!\sigma)(\mathcal{X}^{\iota,s,x}_{r})\right)
\!\left(\mathcal{D}^{\iota,s,x,k}_{r}\right)dW_r \right\rVert_p\nonumber\\
&\leq\sqrt{T}\left[
\int_{s}^{\tilde{s}}
\left\lVert
\left((\operatorname{D}\!\mu)(\mathcal{X}^{\iota,s,x}_{r})
\right)\!\left(
\mathcal{D}^{\iota,s,x,k}_{r}
\right)\right\rVert_{p}^2dr \right]^\frac{1}{2}
+p\left[
\int_{s}^{\tilde{s}}
\left\lVert
\left((\operatorname{D}\!\sigma)(\mathcal{X}^{\iota,s,x}_{r})\right)
\!\left(\mathcal{D}^{\iota,s,x,k}_{r}\right)\right\rVert_p^2 dr \right]^\frac{1}{2}\nonumber\\
&\leq\xeqref{f04} c\left[\sqrt{T}+p\right]
\left[
\int_{s}^{\tilde{s}}
\left\lVert
\mathcal{D}^{\iota,s,x,k}_{r}
\right\rVert_{p}^2dr \right]^\frac{1}{2}\nonumber\\
&\leq c\left[\sqrt{T}+p\right]\sqrt{\tilde{s}-s}\xeqref{f09}\sqrt{2}
e^{ c^2\left[\sqrt{T}+p\right]^2T}
=\sqrt{2}c\left[\sqrt{T}+p\right]
e^{ c^2\left[\sqrt{T}+p\right]^2T}\sqrt{\tilde{s}-s}.\label{f29}
\end{align}
Furthermore, \eqref{f10}, H\"older's inequality, \eqref{f13}, \eqref{f09}, and the fact that 
$5\sqrt{2}\leq 8$ imply for all
$k\in [1,d]\cap\Z$,
$s\in [0,T]$, $r\in [s,T]$, 
$x\in \R^d$ that
\begin{align}
&
\left\lVert
\left((\operatorname{D}\!\mu)(\mathcal{X}^{\iota,s,x}_{r})
-(\operatorname{D}\!\mu)(\mathcal{X}^{\iota,\tilde{s},x}_{r})
\right)
\mathcal{D}^{\iota,s,x,k}_{r}\right\rVert_\frac{p}{2}\nonumber\\
&\leq b\left\lVert \mathcal{X}^{\iota,s,x}_{r} -\mathcal{X}^{\iota,\tilde{s},x}_{r}\right\rVert_{p}
\left\lVert \mathcal{D}^{\iota,s,x,k}_{r}\right\rVert_{p}\nonumber\\
&\leq b\xeqref{f13}\cdot 
5e^{c^2\left[\sqrt{T}+p\right]^2T}
\left[\sqrt{T}+p\right]
e^\frac{1.5\bar{c}T}{p}\varphi^\frac{1}{p}(x)\lvert s-\tilde{s}\rvert^\frac{1}{2}\xeqref{f09}\cdot \sqrt{2}
e^{ c^2\left[\sqrt{T}+p\right]^2T}\nonumber\\
&\leq 8b
e^{2c^2\left[\sqrt{T}+p\right]^2T}
\left[\sqrt{T}+p\right]
e^\frac{1.5\bar{c}T}{p}\varphi^\frac{1}{p}(x)\lvert s-\tilde{s}\rvert^\frac{1}{2}.
\end{align}
This, the triangle inequality, and \eqref{f04}
imply for all
$k\in [1,d]\cap\Z$,
$s\in [0,T]$, 
$\tilde{s}\in [s,T]$,
$r\in [\tilde{s},T]$, 
$x\in \R^d$ that
\begin{align}
&
\left\lVert
\left((\operatorname{D}\!\mu)(\mathcal{X}^{\iota,s,x}_{r})
\right)\!\left(
\mathcal{D}^{\iota,s,x,k}_{r}
\right)-\left((\operatorname{D}\!\mu)(\mathcal{X}^{\iota,\tilde{s},x}_{r})
\right)\!\left(
\mathcal{D}^{\iota,\tilde{s},x,k}_{r}
\right)
\right\rVert_\frac{p}{2}\nonumber\\
&\leq \left\lVert
\left((\operatorname{D}\!\mu)(\mathcal{X}^{\iota,s,x}_{r})
-(\operatorname{D}\!\mu)(\mathcal{X}^{\iota,\tilde{s},x}_{r})
\right)
\mathcal{D}^{\iota,s,x,k}_{r}\right\rVert_\frac{p}{2} +
\left\lVert
(\operatorname{D}\!\mu)(\mathcal{X}^{\iota,\tilde{s},x}_{r})
\left(
\mathcal{D}^{\iota,s,x,k}_{r}-
\mathcal{D}^{\iota,\tilde{s},x,k}_{r}\right)\right\rVert_\frac{p}{2}\nonumber\\
&\leq 
8b
e^{2c^2\left[\sqrt{T}+p\right]^2T}
\left[\sqrt{T}+p\right]
e^\frac{1.5\bar{c}T}{p}\varphi^\frac{1}{p}(x)\lvert s-\tilde{s}\rvert^\frac{1}{2}+c\left\lVert
\mathcal{D}^{\iota,s,x,k}_{r}-
\mathcal{D}^{\iota,\tilde{s},x,k}_{r}\right\rVert_{\frac{p}{2}}.\label{f30}
\end{align}
Similarly,
for all
$k\in [1,d]\cap\Z$,
$s\in [0,T]$, 
$\tilde{s}\in [s,T]$,
$r\in [\tilde{s},T]$, 
$x\in \R^d$ we have that
\begin{align}
&
\left\lVert
\left((\operatorname{D}\!\mu)(\mathcal{X}^{\iota,s,x}_{r})
\right)\!\left(
\mathcal{D}^{\iota,s,x,k}_{r}
\right)-\left((\operatorname{D}\!\mu)(\mathcal{X}^{\iota,\tilde{s},x}_{r})
\right)\!\left(
\mathcal{D}^{\iota,\tilde{s},x,k}_{r}
\right)
\right\rVert_\frac{p}{2}\nonumber\\
&\leq 
8b
e^{2c^2\left[\sqrt{T}+p\right]^2T}
\left[\sqrt{T}+p\right]
e^\frac{1.5\bar{c}T}{p}\varphi^\frac{1}{p}(x)\lvert s-\tilde{s}\rvert^\frac{1}{2}+c\left\lVert
\mathcal{D}^{\iota,s,x,k}_{r}-
\mathcal{D}^{\iota,\tilde{s},x,k}_{r}\right\rVert_{\frac{p}{2}}.\label{f31}
\end{align}
Next,
\eqref{f07} shows for all
$k\in [1,d]\cap\Z$,
 $s\in [0,T]$, 
$\tilde{s}\in [s,T]$,
$t\in [\tilde{s},T]$, 
$x\in \R^d$ that
\begin{align}
\mathcal{D}^{\iota,s,x,k}_t-\mathcal{D}^{\iota,\tilde{s},x,k}_t
&=
\int_{s}^{\tilde{s}}
\left((\operatorname{D}\!\mu)(\mathcal{X}^{\iota,s,x}_{r})
\right)\!\left(
\mathcal{D}^{\iota,s,x,k}_{r}
\right)dr
+
\int_{s}^{\tilde{s}}\left((\operatorname{D}\!\sigma)(\mathcal{X}^{\iota,s,x}_{r})\right)
\!\left(\mathcal{D}^{\iota,s,x,k}_{r}\right)dW_r\nonumber\\
&\quad 
+\int_{\tilde{s}}^{t}\left((\operatorname{D}\!\mu)(\mathcal{X}^{\iota,s,x}_{r})
\right)\!\left(
\mathcal{D}^{\iota,s,x,k}_{r}
\right)-\left((\operatorname{D}\!\mu)(\mathcal{X}^{\iota,\tilde{s},x}_{r})
\right)\!\left(
\mathcal{D}^{\iota,\tilde{s},x,k}_{r}
\right)dr\nonumber\\
&\quad +
\int_{\tilde{s}}^{t}
\left((\operatorname{D}\!\sigma)(\mathcal{X}^{\iota,s,x}_{r})\right)
\!\left(\mathcal{D}^{\iota,s,x,k}_{r}\right)-
\left((\operatorname{D}\!\sigma)(\mathcal{X}^{\iota,\tilde{s},x}_{r})\right)
\!\left(\mathcal{D}^{\iota,\tilde{s},x,k}_{r}\right)
dW_r.
\end{align}
Thus, the triangle inequality, \eqref{f29}, H\"older's inequality,
the Burkholder-Davis-Gundy inequality (cf. \cite[Lemma~7.7]{DPZ1992}), \eqref{f30}, and \eqref{f31} prove 
for all
$k\in [1,d]\cap\Z$,
 $s\in [0,T]$, 
$\tilde{s}\in [s,T]$,
$t\in [\tilde{s},T]$, 
$x\in \R^d$ that
\begin{align}
&
\left\lVert
\mathcal{D}^{\iota,s,x,k}_t-\mathcal{D}^{\iota,\tilde{s},x,k}_t
\right\rVert_{\frac{p}{2}}\nonumber\\
&\leq \xeqref{f29}\sqrt{2}c\left[\sqrt{T}+p\right]
e^{ c^2\left[\sqrt{T}+p\right]^2T}\sqrt{\tilde{s}-s}\nonumber\\
&\quad +\sqrt{T}\left[
\int_{\tilde{s}}^{t}
\left\lVert
\left((\operatorname{D}\!\mu)(\mathcal{X}^{\iota,s,x}_{r})
\right)\!\left(
\mathcal{D}^{\iota,s,x,k}_{r}
\right)-\left((\operatorname{D}\!\mu)(\mathcal{X}^{\iota,\tilde{s},x}_{r})
\right)\!\left(
\mathcal{D}^{\iota,\tilde{s},x,k}_{r}
\right)\right \rVert_{\frac{p}{2}}^2 dr \right]^\frac{1}{2}\nonumber\\
&\quad +p
\left[
\int_{\tilde{s}}^{t}
\left\lVert
\left((\operatorname{D}\!\sigma)(\mathcal{X}^{\iota,s,x}_{r})\right)
\!\left(\mathcal{D}^{\iota,s,x,k}_{r}\right)-
\left((\operatorname{D}\!\sigma)(\mathcal{X}^{\iota,\tilde{s},x}_{r})\right)
\!\left(\mathcal{D}^{\iota,\tilde{s},x,k}_{r}\right)\right\rVert^2_{\frac{p}{2}}
dr \right]^\frac{1}{2}\nonumber\\
&\leq \sqrt{2}c\left[\sqrt{T}+p\right]
e^{ c^2\left[\sqrt{T}+p\right]^2T}\sqrt{\tilde{s}-s}\nonumber\\
&\quad +\left[\sqrt{T}+p\right]\left[\xeqref{f30}\sqrt{T}\cdot 
8b e^{2c^2\left[\sqrt{T}+p\right]^2T}
\left[\sqrt{T}+p\right]
e^\frac{1.5\bar{c}T}{p}\varphi^\frac{1}{p}(x)\lvert s-\tilde{s}\rvert^\frac{1}{2}+
c\left[\int_{\tilde{s}}^{t}\left\lVert
\mathcal{D}^{\iota,s,x,k}_{r}-
\mathcal{D}^{\iota,\tilde{s},x,k}_{r}\right\rVert_{\frac{p}{2}}^2
\right]^\frac{1}{2}
\right]\nonumber\\
&\leq 8(b+c)e^{2c^2\left[\sqrt{T}+p\right]^2T}
\left[\sqrt{T}+p\right]^3
e^\frac{1.5\bar{c}T}{p}\varphi^\frac{1}{p}(x)\lvert s-\tilde{s}\rvert^\frac{1}{2}
+c\left[\sqrt{T}+p\right]
\left[\int_{\tilde{s}}^{t}\left\lVert
\mathcal{D}^{\iota,s,x,k}_{r}-
\mathcal{D}^{\iota,\tilde{s},x,k}_{r}\right\rVert_{\frac{p}{2}}^2
\right]^\frac{1}{2}.
\end{align}
This, \eqref{f09}, and a Gr\"onwall-type inequality (cf.  \cite[Corollary~2.2]{HN2022}) prove 
for all
$k\in [1,d]\cap\Z$,
 $s\in [0,T]$, 
$\tilde{s}\in [s,T]$,
$t\in [\tilde{s},T]$, 
$x\in \R^d$ that
\begin{align}
\left\lVert
\mathcal{D}^{\iota,s,x,k}_t-\mathcal{D}^{\iota,\tilde{s},x,k}_t
\right\rVert_{\frac{p}{2}}&\leq 
\sqrt{2}\cdot 
8(b+c)e^{2c^2\left[\sqrt{T}+p\right]^2T}
\left[\sqrt{T}+p\right]^3
e^\frac{1.5\bar{c}T}{p}\varphi^\frac{1}{p}(x)\lvert s-\tilde{s}\rvert^\frac{1}{2}\cdot 
e^{c^2\left[\sqrt{T}+p\right]^2T}\nonumber\\
&\leq 12(b+c)e^{3c^2\left[\sqrt{T}+p\right]^2T}
\left[\sqrt{T}+p\right]^3
e^\frac{1.5\bar{c}T}{p}\varphi^\frac{1}{p}(x)\lvert s-\tilde{s}\rvert^\frac{1}{2}.\label{f32}
\end{align}
Next,
the Burkholder-Davis-Gundy inequality (cf. \cite[Lemma~7.7]{DPZ1992}),
\eqref{f04}, and \eqref{f09}
 imply 
for all
$k\in [1,d]\cap\Z$,
 $s\in [0,T]$, 
$\tilde{s}\in [s,T]$,
$t\in [\tilde{s},T]$, 
$x\in \R^d$ that
\begin{align}
\left\lvert
\frac{1}{t-\tilde{s}}-\frac{1}{t-{s}}
\right\rvert 
\left\lVert
\int_{\tilde{s}}^{t}
\left(
\sigma^{-1}(\mathcal{X}^{\iota,\tilde{s},x}_{r})
\mathcal{D}^{\iota,\tilde{s},x,k}_{r}
\right)^\top dW_r\right\rVert_p
&\leq 
\frac{\lvert\tilde{s}-s\rvert}{(t-s)(t-\tilde{s})}
p
\left[
\int_{\tilde{s}}^{t}
\left\lVert
\sigma^{-1}(\mathcal{X}^{\iota,\tilde{s},x}_{r})
\mathcal{D}^{\iota,\tilde{s},x,k}_{r}\right\rVert_p^2dr
\right]^\frac{1}{2}\nonumber\\
&\leq \xeqref{f04}
\frac{\lvert\tilde{s}-s\rvert}{(t-s)(t-\tilde{s})}
pc
\left[
\int_{\tilde{s}}^{t}
\left\lVert
\mathcal{D}^{\iota,\tilde{s},x,k}_{r}\right\rVert_p^2dr
\right]^\frac{1}{2}
\nonumber\\
&\leq \frac{\lvert\tilde{s}-s\rvert}{(t-s)(t-\tilde{s})}
pc\sqrt{t-\tilde{s}} \xeqref{f09}\sqrt{2}
e^{ c^2\left[\sqrt{T}+p\right]^2T}\nonumber\\
&\leq 
\frac{\sqrt{2}pc
e^{ c^2\left[\sqrt{T}+p\right]^2T}\sqrt{\tilde{s}-s}}{\sqrt{t-s}\sqrt{t-\tilde{s}}}.\label{f36}
\end{align}
In addition,
\eqref{f10}, \eqref{f13}, \eqref{f09}, and the fact that $5\sqrt{2}\leq 8$
 imply 
for all
$k\in [1,d]\cap\Z$,
 $s\in [0,T]$, 
$\tilde{s}\in [s,T]$,
$r\in [\tilde{s},T]$, 
$x\in \R^d$ that
\begin{align}
&
\left\lVert
\left(\sigma^{-1}(\mathcal{X}^{\iota,\tilde{s},x}_{r})-
\sigma^{-1}(\mathcal{X}^{\iota,s,x}_{r})
\right)
\mathcal{D}^{\iota,\tilde{s},x,k}_{r}\right\rVert_{\frac{p}{2}}\nonumber\\
&\leq \xeqref{f10}
b \left\lVert
\mathcal{X}^{\iota,\tilde{s},x}_{r}-\mathcal{X}^{\iota,s,x}_{r}
\right\rVert_p
\left\lVert
\mathcal{D}^{\iota,\tilde{s},x,k}_{r}
\right\rVert_p\nonumber\\
&\leq b \xeqref{f13}\cdot 
5e^{c^2\left[\sqrt{T}+p\right]^2T}
\left[\sqrt{T}+p\right]
e^\frac{1.5\bar{c}T}{p}\varphi^\frac{1}{p}(x)\lvert s-\tilde{s}\rvert^\frac{1}{2}\cdot \xeqref{f09}\sqrt{2}
e^{ c^2\left[\sqrt{T}+p\right]^2T}\nonumber\\
&\leq 8be^{2c^2\left[\sqrt{T}+p\right]^2T}
\left[\sqrt{T}+p\right]
e^\frac{1.5\bar{c}T}{p}\varphi^\frac{1}{p}(x)\lvert s-\tilde{s}\rvert^\frac{1}{2}.\label{f33}
\end{align}
Next, \eqref{f04} and \eqref{f32} imply 
for all
$k\in [1,d]\cap\Z$,
 $s\in [0,T]$, 
$\tilde{s}\in [s,T]$,
$r\in [\tilde{s},T]$, 
$x\in \R^d$ that
\begin{align}
&
\left\lVert
\sigma^{-1}(\mathcal{X}^{\iota,s,x}_{r})
\left(
\mathcal{D}^{\iota,\tilde{s},x,k}_{r}-
\mathcal{D}^{\iota,{s},x,k}_{r}\right)
\right\rVert_\frac{p}{2}\nonumber\\
&
\leq \xeqref{f04}c
\left\lVert
\mathcal{D}^{\iota,\tilde{s},x,k}_{r}-
\mathcal{D}^{\iota,{s},x,k}_{r}
\right\rVert_\frac{p}{2}\nonumber\\
&\leq c\cdot \xeqref{f32}
12(b+c)e^{3c^2\left[\sqrt{T}+p\right]^2T}
\left[\sqrt{T}+p\right]^3
e^\frac{1.5\bar{c}T}{p}\varphi^\frac{1}{p}(x)\lvert s-\tilde{s}\rvert^\frac{1}{2}.\label{f34}
\end{align}
Thus, the triangle inequality and \eqref{f33} prove 
for all
$k\in [1,d]\cap\Z$,
 $s\in [0,T]$, 
$\tilde{s}\in [s,T]$,
$r\in [\tilde{s},T]$, 
$x\in \R^d$ that
\begin{align}
&
\left\lVert
\sigma^{-1}(\mathcal{X}^{\iota,\tilde{s},x}_{r})
\mathcal{D}^{\iota,\tilde{s},x,k}_{r}
-\sigma^{-1}(\mathcal{X}^{\iota,s,x}_{r})
\mathcal{D}^{\iota,s,x,k}_{r}
\right\rVert_{\frac{p}{2}}\nonumber\\
&\leq 
\left\lVert
\left(\sigma^{-1}(\mathcal{X}^{\iota,\tilde{s},x}_{r})-
\sigma^{-1}(\mathcal{X}^{\iota,s,x}_{r})
\right)
\mathcal{D}^{\iota,\tilde{s},x,k}_{r}\right\rVert_{\frac{p}{2}}
+
\left\lVert
\sigma^{-1}(\mathcal{X}^{\iota,s,x}_{r})
\left(
\mathcal{D}^{\iota,\tilde{s},x,k}_{r}-
\mathcal{D}^{\iota,{s},x,k}_{r}\right)
\right\rVert_{\frac{p}{2}}\nonumber\\
&\leq \xeqref{f33}8be^{2c^2\left[\sqrt{T}+p\right]^2T}
\left[\sqrt{T}+p\right]
e^\frac{1.5\bar{c}T}{p}\varphi^\frac{1}{p}(x)\lvert s-\tilde{s}\rvert^\frac{1}{2}\nonumber
\\
&\quad +\xeqref{f34}
12(b+c)ce^{3c^2\left[\sqrt{T}+p\right]^2T}
\left[\sqrt{T}+p\right]^3
e^\frac{1.5\bar{c}T}{p}\varphi^\frac{1}{p}(x)\lvert s-\tilde{s}\rvert^\frac{1}{2}\nonumber\\
&\leq 20(b+c)ce^{3c^2\left[\sqrt{T}+p\right]^2T}
\left[\sqrt{T}+p\right]^3
e^\frac{1.5\bar{c}T}{p}\varphi^\frac{1}{p}(x)\lvert s-\tilde{s}\rvert^\frac{1}{2}.\label{f35}
\end{align}
This and the Burkholder-Davis-Gundy inequality (cf. \cite[Lemma~7.7]{DPZ1992})
 show 
for all
$k\in [1,d]\cap\Z$,
 $s\in [0,T]$, 
$\tilde{s}\in [s,T]$,
$t\in [\tilde{s},T]$, 
$x\in \R^d$ that
\begin{align}
&
\left\lVert
\frac{1}{t-s}
\int_{\tilde{s}}^{t}
\left(
\sigma^{-1}(\mathcal{X}^{\iota,\tilde{s},x}_{r})
\mathcal{D}^{\iota,\tilde{s},x,k}_{r}
-\sigma^{-1}(\mathcal{X}^{\iota,s,x}_{r})
\mathcal{D}^{\iota,s,x,k}_{r}
\right)^\top dW_r \right\rVert_\frac{p}{2}\nonumber\\
&\leq \frac{\frac{p}{2}}{t-s}
\left[
\int_{\tilde{s}}^{t}
\left\lVert
\left(
\sigma^{-1}(\mathcal{X}^{\iota,\tilde{s},x}_{r})
\mathcal{D}^{\iota,\tilde{s},x,k}_{r}
-\sigma^{-1}(\mathcal{X}^{\iota,s,x}_{r})
\mathcal{D}^{\iota,s,x,k}_{r}
\right)^\top \right\rVert_{\frac{p}{2}}^2dr\right]^\frac{1}{2}\nonumber\\
&\leq \frac{\frac{p}{2}}{t-s}\sqrt{t-\tilde{s}}\cdot \xeqref{f35}
20(b+c)ce^{3c^2\left[\sqrt{T}+p\right]^2T}
\left[\sqrt{T}+p\right]^3
e^\frac{1.5\bar{c}T}{p}\varphi^\frac{1}{p}(x)\lvert s-\tilde{s}\rvert^\frac{1}{2}\nonumber\\
&\leq \frac{10(b+c)ce^{3c^2\left[\sqrt{T}+p\right]^2T}
\left[\sqrt{T}+p\right]^3
e^\frac{1.5\bar{c}T}{p}\varphi^\frac{1}{p}(x)\lvert s-\tilde{s}\rvert^\frac{1}{2}}{\sqrt{t-s}}\nonumber\\
&\leq 
\frac{10(b+c)ce^{3c^2\left[\sqrt{T}+p\right]^2T}
\left[\sqrt{T}+p\right]^4
e^\frac{1.5\bar{c}T}{p}\varphi^\frac{1}{p}(x)\lvert s-\tilde{s}\rvert^\frac{1}{2}}{\sqrt{t-s}\sqrt{t-\tilde{s}}}.\label{f37}
\end{align}
Next, the Burkholder-Davis-Gundy inequality (cf. \cite[Lemma~7.7]{DPZ1992}), \eqref{f04}, and \eqref{f09}
 imply 
for all
$k\in [1,d]\cap\Z$,
 $s\in [0,T]$, 
$\tilde{s}\in [s,T]$,
$t\in [\tilde{s},T]$, 
$x\in \R^d$ that
\begin{align}
\frac{1}{t-{s}}
\left\lVert
\int_{s}^{\tilde{s}}
\left(
\sigma^{-1}(\mathcal{X}^{\iota,{s},x}_{r})
\mathcal{D}^{\iota,{s},x,k}_{r}
\right)^\top dW_r \right\rVert_p
&\leq 
\frac{1}{t-{s}}p
\left[
\int_{s}^{\tilde{s}}
\left\lVert
\sigma^{-1}(\mathcal{X}^{\iota,{s},x}_{r})
\mathcal{D}^{\iota,{s},x,k}_{r}
 \right\rVert_p^2 dr 
\right]^\frac{1}{2}\nonumber\\
&\leq\xeqref{f04} \frac{1}{t-s}pc
\left[
\int_{s}^{\tilde{s}}
\left\lVert
\mathcal{D}^{\iota,{s},x,k}_{r}
 \right\rVert_p^2 dr 
\right]^\frac{1}{2}\nonumber\\
&\leq \frac{1}{t-s}pc\sqrt{\tilde{s}-s}\xeqref{f09}\sqrt{2}
e^{ c^2\left[\sqrt{T}+p\right]^2T}\nonumber\\
&\leq 
\frac{\sqrt{2}pc
e^{ c^2\left[\sqrt{T}+p\right]^2T}\sqrt{\tilde{s}-s}}{\sqrt{t-s}\sqrt{t-\tilde{s}}}.\label{f38}
\end{align}
Moreover, \eqref{f19} shows 
for all
$k\in [1,d]\cap\Z$,
 $s\in [0,T]$, 
$\tilde{s}\in [s,T]$,
$t\in (\tilde{s},T]$, 
$x\in \R^d$ that
\begin{align}
\mathcal{V}^{\iota,\tilde{s},x,k}_t-\mathcal{V}^{\iota,s,x,k}_t
&=\frac{1}{t-\tilde{s}}\int_{\tilde{s}}^{t}
\left(
\sigma^{-1}(\mathcal{X}^{\iota,\tilde{s},x}_{r})
\mathcal{D}^{\iota,\tilde{s},x,k}_{r}
\right)^\top dW_r-
\frac{1}{t-s}\int_{s}^{t}
\left(
\sigma^{-1}(\mathcal{X}^{\iota,s,x}_{r})
\mathcal{D}^{\iota,s,x,k}_{r}
\right)^\top dW_r\nonumber\\
&=\left(
\frac{1}{t-\tilde{s}}-\frac{1}{t-s}
\right)\int_{\tilde{s}}^{t}
\left(
\sigma^{-1}(\mathcal{X}^{\iota,\tilde{s},x}_{r})
\mathcal{D}^{\iota,\tilde{s},x,k}_{r}
\right)^\top dW_r\nonumber\\
&\quad +\frac{1}{t-s}
\int_{\tilde{s}}^{t}
\left(
\sigma^{-1}(\mathcal{X}^{\iota,\tilde{s},x}_{r})
\mathcal{D}^{\iota,\tilde{s},x,k}_{r}
-\sigma^{-1}(\mathcal{X}^{\iota,s,x}_{r})
\mathcal{D}^{\iota,s,x,k}_{r}
\right)^\top dW_r\nonumber\\
&\quad-
\frac{1}{t-s}\int^{\tilde{s}}_{s}
\left(
\sigma^{-1}(\mathcal{X}^{\iota,s,x}_{r})
\mathcal{D}^{\iota,s,x,k}_{r}
\right)^\top dW_r.
\end{align}
Hence, the triangle inequality, \eqref{f36}, \eqref{f37}, \eqref{f38}, and the fact that $2\sqrt{2}\leq 3$
prove for all
$k\in [1,d]\cap\Z$,
 $s\in [0,T]$, 
$\tilde{s}\in [s,T]$,
$t\in (\tilde{s},T]$, 
$x\in \R^d$ that
\begin{align}
\left\lVert
\mathcal{V}^{\iota,\tilde{s},x,k}_t-\mathcal{V}^{\iota,s,x,k}_t
\right\rVert_{\frac{p}{2}}
&\leq\xeqref{f36} \frac{\sqrt{2}pc
e^{ c^2\left[\sqrt{T}+p\right]^2T}\sqrt{\tilde{s}-s}}{\sqrt{t-s}\sqrt{t-\tilde{s}}}\nonumber\\
&\quad +\xeqref{f37}\frac{10(b+c)ce^{3c^2\left[\sqrt{T}+p\right]^2T}
\left[\sqrt{T}+p\right]^4
e^\frac{1.5\bar{c}T}{p}\varphi^\frac{1}{p}(x)\lvert s-\tilde{s}\rvert^\frac{1}{2}}{\sqrt{t-s}\sqrt{t-\tilde{s}}}\nonumber\\
&\quad \xeqref{f38}+\frac{\sqrt{2}pc
e^{ c^2\left[\sqrt{T}+p\right]^2T}\sqrt{\tilde{s}-s}}{\sqrt{t-s}\sqrt{t-\tilde{s}}}\nonumber\\
&\leq \frac{13(b+c)ce^{3c^2\left[\sqrt{T}+p\right]^2T}
\left[\sqrt{T}+p\right]^4
e^\frac{1.5\bar{c}T}{p}\varphi^\frac{1}{p}(x)\lvert s-\tilde{s}\rvert^\frac{1}{2}}{\sqrt{t-s}\sqrt{t-\tilde{s}}}
\end{align}
This shows \eqref{f39}.
This completes the proof of \cref{f10b}.
\end{proof}
\section{Complexity analysis}\label{s06}
In this section we study the MLP approximations which have been introduced in \eqref{a04d} in the case when $(\mathcal{X},\mathcal{Z})$ is replaced by the Euler-Maruyama approximations (see \eqref{a04l}--\eqref{a04j}). We will prove that when the coefficients satisfy \eqref{f04c}--\eqref{f10c} then the corresponding processes
$\mathcal{X}^{d,\theta,K,(\cdot),(\cdot)}_{(\cdot)}$,
$\mathcal{Z}^{d,\theta,K,(\cdot),(\cdot)}_{(\cdot)}$ defined in \cref{c34} below fulfill \eqref{a02e} as well as \eqref{a21}--\eqref{a70}. This allows us to combine \cref{a20} with \cref{a37a} and \cref{f10b} (see the proof of \cref{c34} below).
\begin{theorem}\label{c34}
Let  $\Theta=\cup_{n\in \N}\Z^n$,
$T\in (0,\infty)$, $\mathbf{k}\in [0,\infty)$,
$\beta,c\in [1,\infty)$. 
Let $\lVert \cdot\rVert\colon \cup_{k,\ell\in\N}\R^{k\times \ell}\to[0,\infty)$ satisfy for all $k,\ell\in\N$, $s=(s_{ij})_{i\in[1,k]\cap\N,j\in [1,\ell]\cap\N}\in\R^{k\times \ell}$ that
$\lVert s\rVert^2=\sum_{i=1}^{k}\sum_{j=1}^{\ell}\lvert s_{ij}\rvert^2$.
For every $d\in \N$
let $(L^d_i)_{i\in [0,d]\cap \Z}\in \R^{d+1}$ satisfy that $
\sum_{i=0}^{d}L^d_i\leq c$. 
For every
 $K\in \N$ 
let $\rdown{\cdot}_K\colon \R\to\R$ satisfy for all $t\in \R$ that 
$\rdown{t}_K=\max( \{0,\frac{T}{K},\frac{2T}{K},\ldots,T\} \cap (    (-\infty,t)\cup\{0\} )  )$.
For every $d\in \N$
let 
 $\Lambda^d=(\Lambda^d_{\nu})_{\nu\in [0,d]\cap\Z}\colon [0,T]\to \R^{1+d}$ satisfy for all $t\in [0,T]$ that $\Lambda^d(t)=(1,\sqrt{t},\ldots,\sqrt{t})$.
For every $d\in \N$
let $\pr^d=(\pr^d_\nu)_{\nu\in [0,d]\cap\Z}\colon \R^{d+1}\to\R$ satisfy
for all 
$w=(w_\nu)_{\nu\in [0,d]\cap\Z}$,
$i\in [0,d]\cap\Z$ that
$\pr^d_i(w)=w_i$. 
For every $d\in \N$, $k\in [1,d]\cap\Z$ let $e^d_k\in \R^d$ denote the $d$-dimensional vector with a $1$ in the $k$-th coordinate and $0$'s elsewhere.
For every $d\in \N$ let $f_d\in C( [0,T)\times\R^d\times \R^{d+1},\R)$, $g_d\in C(\R^d,\R)$, 
$\mu_d\in C^1(\R^d,\R^d)$, 
$\sigma_d\in C^1(\R^d,\R^{d\times d})$.
To shorten the notation we write 
for all $d\in \N$,
$t\in [0,T)$, $x\in \R^d$,
$w\colon[0,T)\times\R^d\to \R^{d+1} $ that
\begin{align}
(F_d(w))(t,x)=f_d(t,x,w(t,x)).
\end{align}
Assume
for all $d\in \N$,
$i\in [0,d]\cap\Z$,
$s\in [0,T)$,
$t\in [s,T)$, $r\in (t,T]$,
 $x,y,h\in \R^d$, $w_1,w_2\in \R^{d+1}$
 that $\sigma_d$ is invertible,
\begin{align}
\lVert\mu_d(0)\rVert+\lVert\sigma_d(0)\rVert\leq cd^c,\label{f04c}
\end{align}
\begin{align}
\max\left\{\left\lVert\left((\operatorname{D}\!\mu_d)(x)\right)\!(h)\right\rVert  ,\left\lVert \left((\operatorname{D}\!\sigma_d)(x)
\right)\!(h)
\right\rVert,\left\lVert \sigma_d^{-1}(x)h\right\rVert  \right\}\leq c\lVert h\rVert,\label{f04b}
\end{align}
\begin{align}
&
\max \left\{
\left\lVert
\left(
(\operatorname{D}\!\mu_d)(x)-
(\operatorname{D}\!\mu_d)(y)\right)\!(h)\right\rVert
,
\left\lVert
\left(
(\operatorname{D}\!\sigma_d)(x)-
(\operatorname{D}\!\sigma_d)(y)\right)\!(h)\right\rVert,\left\lVert \left[
(\sigma_d(x))^{-1}-(\sigma_d(y))^{-1}\right]h\right\rVert
\right\}\nonumber\\
&
\leq cd^c\lVert x-y\rVert\lVert h\rVert,
\label{f10c}
\end{align}
\begin{align}\label{a04g}
\lvert g_d(x)\rvert+ \lvert Tf_d(t,x,0)\rvert\leq \left[(cd^c)^2+c^2\lVert x\rVert^2\right]^\beta,
\end{align}
\begin{align}
&
\lvert
f_d(t,x,w_1)-f_d(t,y,w_2)\rvert\nonumber\\
&\leq \sum_{\nu=0}^{d}\left[
L_\nu^d\Lambda^d_\nu(T)
\lvert\pr^d_\nu(w_1-w_2) \rvert\right]
+\frac{1}{T}\frac{((cd^c)^2+c^2\lVert x\rVert^2)^\beta+((cd^c)^2+c^2\lVert y\rVert^2)^\beta}{2}\frac{\lVert x-y\rVert}{\sqrt{T}}
,\label{a01g}
\end{align} and
\begin{align}
\lvert g_d(x)-g_d(y)\rvert\leq \frac{((cd^c)^2+c^2\lVert x\rVert^2)^\beta+((cd^c)^2+c^2\lVert y\rVert^2)^\beta}{2}\frac{\lVert x-y\rVert}{\sqrt{T}}.\label{a19e}
\end{align}
Let $\varrho\colon \{(\tau,\sigma)\in [0,T)^2\colon\tau<\sigma \}\to\R$ satisfy for all $t\in [0,T)$, $s\in (t,T)$ that
\begin{align}
\varrho(t,s)=\frac{1}{\mathrm{B}(\tfrac{1}{2},\tfrac{1}{2})}\frac{1}{\sqrt{(T-s)(s-t)}}.\label{a98b}
\end{align}
Let $ (\Omega,\mathcal{F},\P, (\F_t)_{t\in [0,T]})$ be a filtered probability space which satisfies the usual conditions.
Let $\mathfrak{r}^\theta\colon \Omega\to (0,1) $, $\theta\in \Theta$, be independent and identically distributed\ random variables and satisfy for all
$b\in (0,1)$
that 
\begin{align}
\P(\mathfrak{r}^0\leq b)=
\frac{1}{\mathrm{B}(\tfrac{1}{2},\tfrac{1}{2})}
\int_{0}^b\frac{dr}{\sqrt{r(1-r)}}.
\end{align}
For every $d\in \N$
let $W^{d,\theta}\colon [0,T]\times \Omega\to \R^d$, $\theta\in \Theta$, be independent $(\F_t)_{t\in [0,T]}$-Brownian motions with continuous sample paths. Assume for every $d\in \N$ that $(W^{d,\theta})_{\theta\in \Theta}$ and $(\mathfrak{r}^\theta)_{\theta\in \Theta}$ are independent.
For every $\theta\in \Theta$, $d,K\in \N$, $s\in [0,T]$, $x\in \R^d$,
$k\in [1,d]\cap\Z$
let $(\mathcal{X}^{d,\theta,K,s,x}_t)_{t\in [s,T]},(\mathcal{D}^{d,\theta,K,s,x,k}_t)_{t\in[s,T]}\colon [s,T]\times \Omega\to \R^d$
be  $(\F_t)_{t\in[s,T]}$-adapted stochastic processes with continuous sample paths which satisfy for all $t\in [s,T]$ that $\P$-a.s.\ we have that
\begin{align}\label{a04l}
\mathcal{X}^{d,\theta,K,s,x}_t=x+\int_{s}^{t} \mu_d(\mathcal{X}^{d,\theta,K,s,x}_{\max \{s, \rdown{r}_K\}})\,dr
+
\int_{s}^{t} \sigma_d(\mathcal{X}^{d,\theta,K,s,x}_{\max \{s, \rdown{r}_K\}})\,dW_r^{d,\theta}
\end{align}
and
\begin{align}
\mathcal{D}^{d,\theta,K,s,x,k}_t&=e^d_k+\int_{s}^{t}
\left( (\operatorname{D}\!\mu_d)(\mathcal{X}^{d,\theta,K,s,x}_{\max \{s, \rdown{r}_K\}})
\right)\!
\left(\mathcal{D}^{d,\theta,K,s,x,k}_{\max \{s, \rdown{r}_K\}}\right)
dr\nonumber\\
&\quad 
+
\int_{s}^{t}\left( (\operatorname{D}\!\sigma_d)(\mathcal{X}^{d,\theta,K,s,x}_{\max \{s, \rdown{r}_K\}})
\right)
\!
\left(\mathcal{D}^{d,\theta,K,s,x,k}_{\max \{s, \rdown{r}_K\}}\right)
dW_r^{d,\theta}.
\end{align}
For every 
$\theta\in \Theta$,
$d,K\in \N$,
$s\in [0,T)$, $t\in (s,T]$, $x\in \R^d$ let 
$\mathcal{V}^{d,\theta,K,s,x}_t= (\mathcal{V}_t^{d,\theta,K,s,x,k})_{k\in [1,d]\cap\Z}\colon \Omega\to\R^d $,
$\mathcal{Z}_t^{d,\theta,K,s,x}
=(\mathcal{Z}_t^{d,\theta,K,s,x,k})_{k\in [0,d]\cap\Z}
\colon \Omega\to \R^{d+1}$
 satisfy that
\begin{align}
\mathcal{V}^{d,\theta,K,s,x}_t=\frac{1}{t-s}\int_{s}^{t}
\left(
\sigma^{-1}_d(\mathcal{X}^{d,\theta,K,s,x}_{\max\{s,\rdown{r}_K\}})
\mathcal{D}^{d,\theta,K,s,x}_{\max\{s,\rdown{r}_K\}}
\right)^\top dW_r^{d,\theta}\label{c56}
\end{align}
and $\mathcal{Z}_t^{d,\theta,K,s,x}=(1,\mathcal{V}_t^{d,\theta,K,s,x} )$.
Let $
U^{d,\theta}_{n,m,K}\colon [0,T)\times \R^d\to \R^{d+1}$, 
$d\in \N$,
$d,n,m,K\in \Z$, $\theta\in \Theta$, satisfy for all $d,n,m,K\in \N$, 
$\theta\in \Theta$,
$t\in [0,T)$, $x\in\R^d$ that
$U_{-1,m,K}^{d,\theta}(t,x)=U^{d,\theta}_{0,m,K}(t,x)=0$ and
\begin{align} \begin{split} 
&
U^{d,\theta}_{n,m,K}(t,x)= (g_d(x),0)+\sum_{i=1}^{m^n}
\frac{g_d(\mathcal{X}^{d,(\theta,0,-i),K,t,x}_T )-g_d(x) }{m^n}\mathcal{Z}^{d,(\theta,0,-i),K,t,x}_{T}\\
& +\sum_{\ell=0}^{n-1}\sum_{i=1}^{m^{n-\ell}} \frac{\left(F_d(U^{d,(\theta,\ell,i)}_{\ell,m,K})-
\1_\N(\ell)
F_d(U^{d,(\theta,\ell,-i)}_{\ell-1,m,K})\right)\!\left(t+(T-t) \mathfrak{r}^{(\theta,\ell,i)},
\mathcal{X}^{d,(\theta,\ell,i),K,t,x}_{t+(T-t) \mathfrak{r}^{(\theta,\ell,i)}}\right)
\mathcal{Z}^{d,(\theta,\ell,i),K,t,x}_{t+(T-t) \mathfrak{r}^{(\theta,\ell,i)}}}{m^{n-\ell}\varrho(t,t+(T-t)\mathfrak{r}^{(\theta,\ell,i)})}.\end{split}\label{a04j}
\end{align}
For every $d\in \N$ let $\mathfrak{e}_d,\mathfrak{f}_d, \mathfrak{g}_d\in [0,\infty)$ satisfy that 
\begin{align}
\mathfrak{e}_d+\mathfrak{f}_d+\mathfrak{g}_d\leq cd^c.\label{c25}
\end{align} Let $\mathfrak{C}^d_{n,m,K}\in [0,\infty)$, $n,m\in \Z$, $d,K\in \N$, satisfy for all 
$n\in \Z$,
$d,m,K\in \N$ that 
\begin{align}\label{c22}
\mathfrak{C}^d_{n,m,K}\leq m^n(K\mathfrak{e}_d+\mathfrak{g}_d)\1_{\N}(n)
+\sum_{\ell=0}^{n-1}\left[m^{n-\ell}(K\mathfrak{e}_d+\mathfrak{f}_d+\mathfrak{C}_{\ell,m,K}^d+\mathfrak{C}^d_{\ell-1,m,K})\right].
\end{align}
Then the following items hold.
\begin{enumerate}[(i)]
\item\label{i01b} For all $d\in \N$ there exists an up to indistinguishability continuous random field $(X^{d,s,x}_{t})_{s\in [0,T],t\in [s,T], x\in \R^d}\colon \{(\sigma,\tau)\in [0,T]\colon \sigma\leq \tau \} \times \R^d\times \Omega\to \R^d$ such that for all $s\in [0,T]$, $x\in \R^d$ we have that $(X^{d,s,x}_t)_{t\in [s,T]} $ is $(\F_t)_{t\in [s,T]}$-adapted and such that for all $s\in [0,T]$, $t\in [s,T]$, $x\in \R^d$ we have $\P$-a.s.\ that
\begin{align}
X^{d,s,x}_t = x+\int_{s}^{t}\mu_d (X^{d,s,x}_r)\,dr
+\int_{s}^{t}\sigma_d (X^{d,s,x}_r)\,dW_r^{d,0}.
\end{align}

\item \label{i06} For all $d\in \N$, $s\in [0,T]$, $x\in \R^d$ there exists an 
$(\F_{t})_{t\in [s,T]}$-adapted
 stochastic process
$(D_t^{d,s,x})_{t\in[s,T]}=(D_t^{d,s,x,k})_{k\in [1,d]\cap\Z}\colon [s,T]\times \Omega\to \R^{d\times d}$ 
which satisfies for all $t\in [s,T]$, $k\in [1,d]\cap\Z$ that $\P$-a.s.\ we have that
\begin{align}
D_t^{d,s,x,k}=e^d_k+\int_{s}^{t}\left((\operatorname{D}\!\mu_d)(X^{d,s,x}_r)\right)\!
(D_r^{d,s,x,k})\,dr+
\int_{s}^{t}\left((\operatorname{D}\!\sigma_d)(X^{d,s,x}_r)
\right)\!(D_r^{d,s,x,k})\,dW_r^{d,0}.
\end{align}
\item \label{i08}For all
$d\in \N$, $s\in [0,T)$, $x\in \R^d$ there exist  
$(\F_{t})_{t\in (s,T]}$-adapted
 stochastic processes $({V}_t^{d,s,x})_{t\in (s,T]}= 
({V}_t^{d,s,x,k})_{t\in (s,T],k\in [1,d]\cap\Z}\colon (s,T]\times \Omega\to\R^d
$ and 
$({Z}_t^{d,s,x})_{t\in (s,T]} \colon (s,T]\times \Omega\to \R^{d+1}$
 which satisfy for all $t\in (s,T]$ that $\P$-a.s.\ we have that
\begin{align}
V^{d,s,x}_t=\frac{1}{t-s}\int_{s}^{t}(\sigma_d^{-1}(X^{d,s,x}_r)D^{d,s,x}_r)^\top\,dW_r^{d,0}
\end{align}
and ${Z}_t^{d,s,x}=(1, {V}_t^{d,s,x})$.
\item\label{i09} For all $d\in \N$ there exists a unique continuous function $u_d\colon [0,T)\times \R^d\to \R^{d+1}$ such that for all $t\in [0,T)$, $x\in \R^d$
we have that
\begin{align}\label{c01c}
\max_{\nu\in [0,d]\cap\Z}\sup_{\tau\in [0,T), \xi\in \R^d}
\left[\Lambda^d_\nu(T-\tau)\frac{\lvert\pr^d_\nu(u_d(\tau,\xi))\rvert}{\left[(cd^c)^2+c^2\lVert x\rVert^2\right]^\beta}\right]<\infty,
\end{align}
\begin{align}\max_{\nu\in [0,d]\cap\Z}\left[
\E\!\left [\left\lvert g_d(X^{d,t,x}_{T} )\pr_\nu^d(Z^{d,t,x}_{T})\right\rvert \right] + \int_{t}^{T}
\E \!\left[
f_d(r,X^{d,t,x}_{r},u_d(r,X^{d,t,x}_{r}))\pr^d_\nu(Z^{d,t,x}_{r})\right]dr\right]<\infty,
\end{align}
and
\begin{align}
u_d(t,x)=\E\!\left [g_d(X^{d,t,x}_{T} )Z^{d,t,x}_{T} \right] + \int_{t}^{T}
\E \!\left[
f_d(r,X^{d,t,x}_{r},u_d(r,X^{d,t,x}_{r}))Z^{d,t,x}_{r}\right]dr.\label{b05c}
\end{align}
 
\item \label{c33}  For all  $d\in \N$  we have that

\begin{align}
\limsup_{n\to\infty}\sup_{\nu\in [0,d]\cap\Z,t\in [0,T), x\in [-\mathbf{k},\mathbf{k}]^d}\left[
\Lambda_\nu^d(T-t)\left\lVert\pr^d_{\nu}\!\left(U^{d,0}_{n,n^\frac{1}{3},n^\frac{n}{3}}(t,x)-{u}_d(t,x)\right)\right\rVert_2
\right]=0.
\end{align}
\item \label{c35} There exist
$(C_\delta)_{\delta\in (0,1)}\subseteq (0,\infty)$,
 $\eta\in(0,\infty)$,
$(N_{d,\varepsilon})_{d\in \N,\varepsilon\in (0,1)}\subseteq\N$
 such that for all 
$d\in \N$, $\delta,\varepsilon\in (0,1)$ we have that 
\begin{align}
\sup_{\nu\in [0,d]\cap\Z,t\in [0,T), x\in [-\mathbf{k},\mathbf{k}]^d}\left[
\Lambda_\nu^d(T-t)\left\lVert\pr^d_{\nu}\!\left(U^{d,0}_{N_{d,\varepsilon},\lvert N_{d,\varepsilon}\rvert^\frac{1}{3},\lvert N_{d,\varepsilon}\rvert^\frac{N_{d,\varepsilon}}{3}}(t,x)-{u}_d(t,x)\right)\right\rVert_2\right]<\varepsilon
\end{align}
and
\begin{align}
\mathfrak{C}^d_{N_{d,\epsilon}, \lvert N_{d,\varepsilon}\rvert^\frac{1}{3}, \lvert N_{d,\varepsilon}\rvert^{\frac{N_{d,\varepsilon}}{3}}}\leq C_\delta
\varepsilon^{-(4+\delta)} \eta d^\eta.
\end{align}
\end{enumerate}
\end{theorem}
\begin{proof}
[Proof of \cref{c34}]
First, the fundamental theorem of calculus and \eqref{f04b} imply for all $d\in \N$, $x,y\in \R^d$ that
\begin{align}
\lVert
\mu_d(x)-\mu_d(y)\rVert&= 
\left\lVert\int_{0}^{1}\frac{d}{ds}[\mu_d(y+s(x-y))]\,ds\right\rVert=
\left\lVert\int_{0}^{1}\left((\operatorname{D}\!\mu_d)(y+s(x-y))\right)\!(x-y)\,ds\right\rVert\nonumber\\
&\leq c\lVert x-y\rVert
\end{align}
and similarly,
\begin{align}
\lVert
\sigma_d(x)-\sigma_d(y)\rVert\leq c\lVert x-y\rVert.
\end{align}
This and a standard result on SDEs with Lipschitz continuous coefficients (see, e.g., 
\cite[Theorem 4.5.1]{Kun1990}) prove \eqref{i01b}.
Next, a standard result on the existence and uniqueness of the derivative process (see, e.g., \cite[Theorem 3.4]{Kun2004}) and the regularity assumptions of $\mu_d, \sigma_d $, $d\in \N$, imply \eqref{i06}--\eqref{i08}. 

Throughout the rest of this proof let $q,\mathbf{c},\bar{c}\in \R$ satisfy that
\begin{align}
q=40\beta,\quad \bar{c}=16q^2c^2, \quad
\mathbf{c}=
2\sqrt{2}qce^{ c^2\left[2q+\sqrt{T}\right]^2T}\label{c11}
\end{align}
and
for every $d\in \N$
let $\varphi_d\colon \R^d\to [1,\infty)$, $\mathbf{V}_d\colon [0,T]\times \R^d\to [1,\infty)$ satisfy for all $t\in [0,T]$, $x\in \R^d$ that
\begin{align}
\varphi_d(x)=\left(\mathbf{c}+48e^{86\mathbf{c}^6T^3}+e^{\mathbf{c}^2T}+ 4(cd^c)^2+4c^2\lVert x\rVert^2\right)^q\label{c10}
\end{align}
and
\begin{align}
\mathbf{V}_d(t,x)=
\left[
15
c((cd^c)c+c^2)
\left[\sqrt{T}+2q\right]^{10} e^{3c^2\left[\sqrt{T}+2q\right]^2T}
e^{\frac{1.5\bar{c}T}{2q}}
\right]
e^{\frac{1.5\bar{c}(T-t)}{40}}\varphi_d^\frac{1}{40}(x).\label{c10b}
\end{align}
Then \eqref{a04g} shows for all $t\in [0,T)$, $x\in \R^d$ that
\begin{align}\label{a04h}
\lvert g_d(x)\rvert+ \lvert Tf_d(t,x,0)\rvert\leq \left((cd^c)^2+c^2\lVert x\rVert^2\right)^\beta\leq \varphi_d^{\frac{\beta}{q}}(x)=\varphi_d^\frac{1}{40}(x)\leq \mathbf{V}_d(t,x).
\end{align}
Next, \eqref{a19e}, \eqref{c10}, the fact that $q=40\beta$, and \eqref{c10b} prove for all 
$d\in \N$,
$x,y\in \R^d$ that
\begin{align}
\lvert g_d(x)-g_d(y)\rvert
&\leq\xeqref{a19e} \frac{((cd^c)^2+c^2\lVert x\rVert^2)^\beta+((cd^c)^2+c^2\lVert y\rVert^2)^\beta}{2}\frac{\lVert x-y\rVert}{\sqrt{T}}\nonumber\\
&\leq \xeqref{c10}
\frac{\varphi_d^{\frac{\beta}{q}}(x)+\varphi_d^{\frac{\beta}{q}}(y)}{2}
\frac{\lVert x-y\rVert}{\sqrt{T}}\nonumber\\
&= 
\frac{\varphi_d^{\frac{1}{40}}(x)+\varphi_d^{\frac{1}{40}}(y)}{2}
\frac{\lVert x-y\rVert}{\sqrt{T}}\nonumber\\
&\leq \frac{\mathbf{V}_d(T,x)+\mathbf{V}_d(T,y)}{2}\frac{\lVert x-y\rVert}{\sqrt{T}}
.\label{a19f}
\end{align}
Similarly, we have for all $d\in\N$, $t\in [0,T)$, $x,y\in \R^d$, $w_1,w_2\in \R$ that
\begin{align}
&
\lvert
f_d(t,x,w_1)-f_d(t,y,w_2)\rvert\nonumber\\
&\leq \sum_{\nu=0}^{d}\left[
L_\nu^d\Lambda^d_\nu(T)
\lvert\pr^d_\nu(w_1-w_2) \rvert\right]
+\frac{1}{T}\frac{\mathbf{V}_d(t,x)+\mathbf{V}_d(t,y)}{2}\frac{\lVert x-y\rVert}{\sqrt{T}}
.\label{a01h}
\end{align}
Next,
\eqref{f04c},  the fact that $\forall\, x,y\in \R\colon  x+y\leq 2\sqrt{x^2+y^2}$, and \eqref{c10} imply for all 
$d\in \N$, $x\in \R^d$ that
\begin{align}\label{f04d}
\lVert\mu_d(0)\rVert+\lVert\sigma_d(0)\rVert
+c\lVert x\rVert
\leq cd^c+c\lVert x\rVert\leq 2((cd^c)^2+c^2\lVert x\rVert^2)^\frac{1}{2}=\varphi^\frac{1}{2q}(x).
\end{align}
Moreover, \eqref{c10}, \cite[Lemma 3.1]{HN2022}
(applied for every $d\in \N$ with $p\gets q$, $a\gets \mathbf{c}+48e^{86\mathbf{c}^6T^3}+e^{\mathbf{c}^2T}+ 4(cd^c)^2$, $c\gets 2c$, $V\gets \varphi_d$ in the notation of \cite[Lemma 3.1]{HN2022}), and the fact that
$\bar{c}=16q^2c^2$ prove for all 
$d\in \N$,
$x,y\in \R^d$ that
\begin{align}
\lvert
(D\varphi_d(x))(y)\rvert\leq 4qc \varphi_d^\frac{2q-1}{2q}(x)\lVert y\rVert\leq \bar{c}\varphi_d^\frac{2q-1}{2q}(x)\lVert y\rVert
\end{align} and
\begin{align}
\lvert
(D^2\varphi_d(x))(y,y)\rvert\leq 16q^2c^2\varphi_d^\frac{2q-2}{2q}(x)\lVert y\rVert^2= \bar{c}\varphi_d^\frac{2q-2}{2q}(x)\lVert y\rVert^2
\end{align}
where $D, D^2$ denote the first and second derivative operators.
This, \eqref{f04d},
\eqref{f04b}, \eqref{f10c}, and
\cref{f10b} (applied for every $d,K\in \N$ with 
$b\gets cd^c$,
$p\gets 2q$, 
$\mu\gets \mu_d$, $\sigma \gets \sigma_d$, $\delta\gets \rdown{\cdot}_K$
  in the notation of \cref{f10b}) prove that the following items hold.
\begin{itemize}
\item  For all $d,K\in \N$, $s\in [0,T]$, 
$t\in [s,T]$, $x\in \R^d$ we have that 
\begin{align}\label{f11b}
\E\!\left [\varphi_d(\mathcal{X}^{d,0,K,s,x}_t)\right]\leq e^{1.5\bar{c}\lvert t-s\rvert}\varphi_d(x).
\end{align}
\item  For all $d,K\in \N$, $s\in [0,T]$, 
$t\in [s,T]$, $x\in \R^d$ we have that
\begin{align}
\left\lVert\left\lVert
\mathcal{X}^{d,0,K,s,x}_t
-X^{d,s,x}_t\right\rVert
\right\rVert_{2q}\leq \sqrt{2}c \left[\sqrt{T}+2q\right]^2e^{c^2\left[\sqrt{T}+2q\right]^2T}
\left(e^{1.5\bar{c}T}\varphi_d(x)\right)^\frac{1}{2q}
\lvert t-s\rvert^\frac{1}{2}
\frac{\sqrt{T}}{\sqrt{K}}.\label{f12b}
\end{align}
\item  For all $d,K\in \N$, $s,\tilde{s}\in [0,T]$, $t\in [s,T]$, $\tilde{t}\in [\tilde{s},T]$, $x,\tilde{x}\in \R^d$ we have that
\begin{align}
&
\left\lVert
\left\lVert\mathcal{X}^{d,0,K,s,x}_t-
\mathcal{X}^{d,0,K,\tilde{s},\tilde{x}}_{\tilde{t}}\right\rVert
\right\rVert_{2q}\nonumber\\
&
\leq\sqrt{2} \lVert x-\tilde{x}\rVert e^{c^2\left[\sqrt{T}+2q\right]^2T}\nonumber\\
&\quad 
+5e^{c^2\left[\sqrt{T}+2q\right]^2T}
\left[\sqrt{T}+2q\right]
e^\frac{1.5\bar{c}T}{2q}\frac{\varphi_d^\frac{1}{2q}(x)+\varphi_d^\frac{1}{2q}(\tilde{x}) }{2}\left[\lvert s-\tilde{s}\rvert^\frac{1}{2}+
\lvert t-\tilde{t}\rvert^\frac{1}{2}
\right].\label{f13b}
\end{align}
\item 
For all 
$d,K\in \N$,
 $k\in [1,d]\cap\Z$,
 $s\in [0,T]$, $t\in [s,T]$, $x\in \R^d$ we have that
\begin{align}\label{f26b}
\left\lVert \mathcal{V}^{d,0,K,s,x,k}_t\right\rVert_{2q}\leq \frac{2\sqrt{2}qce^{ c^2\left[2q+\sqrt{T}\right]^2T}}{\sqrt{t-s}}.
\end{align} 
\item 
For all
$d,K\in \N$,
$k\in [1,d]\cap\Z$,
 $s\in [0,T]$, $t\in [s,T]$, $x,y\in \R^d$ we have that
\begin{align}
\left\lVert
\mathcal{V}^{d,0,K,s,x,k}_t
-\mathcal{V}^{d,0,K,s,y,k}_t\right\rVert_{\frac{q}{2}}\leq 
\frac{2cd^cc
\left[\sqrt{T}+2q\right]^3
e^{3c^2\left[\sqrt{T}+2q\right]^2T}\lVert x-y\rVert}{\sqrt{t-s}}.\label{f27b}
\end{align}
\item 
For all
$d,K\in \N$, $k\in [1,d]\cap\Z$,
 $s\in [0,T]$, $t\in [s,T]$, 
$x\in \R^d$ we have that
\begin{align}\label{f28b}
\left\lVert
\mathcal{V}^{d,0,K,s,x,k}_t
-
V^{d,s,x,k}_t\right\rVert_{q}
\leq 
\frac{15
c(cd^c c+c^2)
\left[\sqrt{T}+2q\right]^6 e^{3c^2\left[\sqrt{T}+2q\right]^2T}
\left(e^{1.5\bar{c}T}\varphi_d(x)\right)^\frac{1}{2q}
\frac{\sqrt{T}}{\sqrt{K}}}{\sqrt{t-s}}
.
\end{align}
\item 
For all
$k\in [1,d]\cap\Z$,
 $s\in [0,T]$, 
$\tilde{s}\in [s,T]$,
$r\in [\tilde{s},T]$, 
$x\in \R^d$ we have that
\begin{align}\label{f39b}
\left\lVert
V^{d,\tilde{s},x,k}_t-V^{d,s,x,k}_t
\right\rVert_{q}
\leq \frac{13(cd^c+c)ce^{3c^2\left[\sqrt{T}+2q\right]^2T}
\left[\sqrt{T}+2q\right]^4
e^\frac{1.5\bar{c}T}{2q}\varphi_d^\frac{1}{2q}(x)\lvert s-\tilde{s}\rvert^\frac{1}{2}}{\sqrt{t-s}\sqrt{t-\tilde{s}}}.
\end{align}
\end{itemize}
Now,
\eqref{f11b} proves for all $d,K\in \N$, $s\in [0,T]$, 
$t\in [s,T]$, $x\in \R^d$ that
\begin{align}
& \left\lVert
e^{\frac{1.5\bar{c}(T-t)}{20}}\varphi_d^\frac{1}{20}(\mathcal{X}^{d,0,K,s,x}_t)
\right\rVert_{20}
\leq 
\left(
\E \!\left[e^{1.5\bar{c}(T-t)}\varphi_d (\mathcal{X}^{d,0,K,s,x}_t)
\right]\right)^\frac{1}{20}\nonumber
\\
&\leq 
\left(
e^{1.5\bar{c}(T-t)}
e^{1.5\bar{c}(t-s)}
\varphi_d (x)
\right)^\frac{1}{20}=
e^{1.5\bar{c}(T-s)}\varphi_d^\frac{1}{20}(x).
\end{align}
This and \eqref{c10b}  imply for all $d,K\in \N$, $s\in [0,T]$, 
$t\in [s,T]$, $x\in \R^d$ that
\begin{align}\label{c09}
\left
\lVert \mathbf{V}_d(t, \mathcal{X}^{d,0,K,s,x}_t)\right\rVert_{20}\leq \mathbf{V}_d(s,x).
\end{align}
Next, \eqref{f12b},  \eqref{c10b}, the fact that $q\geq 20$, and Jensen's inequality show for all $d,K\in \N$, $s\in [0,T]$, 
$t\in [s,T]$, $x\in \R^d$ that
\begin{align}
\left\lVert\left\lVert
\mathcal{X}^{d,0,K,s,x}_t
-X^{d,s,x}_t
\right\rVert\right\rVert_{20}\leq \frac{\sqrt{T}\mathbf{V}_d(t,x)}{\sqrt{K}}.
\end{align}
This proves for all 
$d\in \N$,
$s\in [0,T]$, 
$t\in [s,T]$, $x\in \R^d$  that 
$X^{d,s,x}_t=\lim_{K\to\infty}\mathcal{X}^{d,0,K,s,x}_t$ in probability. This, \eqref{c09}, the fact that for all $d\in \N$, $\mathbf{V}_d$ is continuous, and Fatou's lemma imply for all $d\in\N$, $s\in [0,T]$, 
$t\in [s,T]$, $x\in \R^d$ 
that\begin{align}
\left\lVert \mathbf{V}_d(t,X_t^{d,s,x})\right\rVert_{20}
\leq \mathbf{V}_d(s,x).\label{c09b}
\end{align}
Next, \eqref{f28b}, \eqref{c10b}, Jensen's inequality, and the fact that $q\geq 40$ prove for all $d,K\in \N$, $s\in [0,T)$, 
$t\in (s,T]$, $x\in \R^d$, $k\in [1,d]\cap\Z$ that
\begin{align}
\left\lVert
V^{d,0,K,s,x,k}_t
-
V^{d,s,x,k}_t\right\rVert_{20}\leq \frac{\frac{\sqrt{T}}{\sqrt{K}}\mathbf{V}_d(t,x)}{\sqrt{T}\sqrt{t-s}}
\end{align}
Therefore, the fact that $\forall\, d,K\in \N, s\in [0,T), x\in \R^d\colon \P(\mathcal{Z}^{d,0,K,s,x}=(1,\mathcal{V}^{d,0,K,s,x}))=1$,  the fact that $\forall\,d\in\N, s\in [0,T), x\in \R^d\colon \P ({Z}^{d,s,x}=(1,{V}^{d,s,x}))=1$, and the definition of $\Lambda^d$, $d\in \N$, prove for all $d,K\in \N$, $s\in [0,T)$, 
$t\in (s,T]$, $x\in \R^d$, $i\in [0,d]\cap\Z$ that
\begin{align}
\left\lVert
\pr_i^d(
\mathcal{Z}^{d,0,K,s,x}_t
-
Z^{d,s,x}_t)\right\rVert_{20}\leq \frac{\frac{\sqrt{T}}{\sqrt{K}}\mathbf{V}_d(t,x)}{\sqrt{T}\Lambda_i^d(t-s)}. \label{c12}
\end{align}
In addition, \eqref{f26b}, \eqref{c11}, Jensen's inequality, and the fact that $q\geq 40$ show for all $d,K\in \N$, $s\in [0,T)$, 
$t\in (s,T]$, $x\in \R^d$, $k\in [1 ,d]\cap\Z$ that
\begin{align}\label{f26c}
\left\lVert \mathcal{V}^{d,0,K,s,x,k}_t\right\rVert_{20}\leq \frac{\mathbf{c}}{\sqrt{t-s}},
\end{align} 
Thus, the fact that $\forall\, d,K\in \N, s\in [0,T), x\in \R^d\colon \P(\mathcal{Z}^{d,0,K,s,x}=(1,\mathcal{V}^{d,0,K,s,x}))=1$,  the fact that $\forall\, s\in [0,T), x\in \R^d\colon \P ({Z}^{d,s,x}=(1,{V}^{d,s,x}))=1$,
and the definition of $\Lambda^d$, $d\in \N$,
imply for all $d,K\in \N$, $s\in [0,T)$, 
$t\in (s,T]$, $x\in \R^d$, $i\in [0,d]\cap\Z$ that
\begin{align}\label{f26d}
\left\lVert \pr_i^d(\mathcal{Z}^{d,0,K,s,x}_t)\right\rVert_{20}\leq \frac{\mathbf{c}}{\Lambda^d_i(t-s)}.
\end{align} 
This and \eqref{c12} prove for all $d\in \N$, $s\in [0,T)$, 
$t\in (s,T]$, $x\in \R^d$, $i\in [0,d]\cap\Z$ that
\begin{align}\label{f26e}
\left\lVert \pr_i^d({Z}^{d,s,x}_t)\right\rVert_{20}\leq \frac{\mathbf{c}}{\Lambda_i^d(t-s)}.
\end{align} 
Next,
\eqref{f13b}, Jensen's inequality, the fact that $q\geq 40$, and \eqref{c11} show for all $d,K\in \N$, $s\in [0,T]$, 
$t\in [s,T]$, $x,y\in \R^d$ that
\begin{align}
\left\lVert\left\lVert
\mathcal{X}_t^{d,0,K,s,x}-
\mathcal{X}_t^{d,0,K,s,y}
\right\rVert\right\rVert_{20}\leq \mathbf{c} \lVert x-y\rVert.\label{c13}
\end{align}
Moreover, \eqref{f12b},  Jensen's inequality, the fact that $q\geq 40$, 
\eqref{c10},
and \eqref{c10b} imply for all $d,K\in \N$, $s\in [0,T]$, 
$t\in [s,T]$, $x\in \R^d$  that
\begin{align}
\left\lVert\left\lVert
\mathcal{X}^{d,0,K,s,x}_t
-X^{d,s,x}_t
\right\rVert\right\rVert_{20}\leq \tfrac{\sqrt{T}}{\sqrt{K}}\mathbf{V}_d(t,x).\label{c15}
\end{align}
This, \eqref{c13}, and the fact that $\forall\,d\in \N\colon \mathbf{c}\leq \mathbf{V}_d$ (see \eqref{c10}--\eqref{c10b}) imply for all $d,K\in \N$, $s\in [0,T]$, 
$t\in [s,T]$, $x\in \R^d$  that
\begin{align}
\left\lVert\left\lVert
{X}_t^{d,s,x}-
{X}_t^{d,s,y}
\right\rVert\right\rVert_{20}\leq \mathbf{c} \lVert x-y\rVert
\leq \frac{\mathbf{V}_d(t,x)+\mathbf{V}_d(t,y)}{2}\lVert x-y\rVert
.\label{c13b}
\end{align}
Next, \eqref{f27b}, Jensen's inequality, the fact that $q\geq 40$, 
\eqref{c11}, and \eqref{c10} show for all $d,K\in \N$, $s\in [0,T)$, 
$t\in (s,T]$, $x\in \R^d$, $k\in [1,d]\cap\Z$ that
\begin{align}
\left\lVert
\mathcal{V}^{d,0,K,s,x,k}_t
-\mathcal{V}^{d,0,K,s,y,k}_t\right\rVert_{20}\leq \frac{\mathbf{V}_d(t,x)+\mathbf{V}_d(t,y)}{2}
\frac{\lVert x-y\rVert}{\sqrt{T}\sqrt{t-s}}.
\end{align}
Hence, the fact that $\forall\, d,K\in \N,s\in [0,T), x\in \R^d\colon \P(\mathcal{Z}^{d,0,K,s,x}=(1,\mathcal{V}^{d,0,K,s,x}))=1$ and the definition of $\Lambda^d$, $d\in\N$, prove for all $d,K\in \N$, $s\in [0,T)$, 
$t\in (s,T]$, $x\in \R^d$, $i\in [0,d]\cap\Z$ that
\begin{align}
\left\lVert\pr_i^d(
\mathcal{Z}^{d,0,K,s,x}_t
-\mathcal{Z}^{d,0,K,s,y}_t)\right\rVert_{20}\leq \frac{\mathbf{V}_d(t,x)+\mathbf{V}_d(t,y)}{2}
\frac{\lVert x-y\rVert}{\sqrt{T}\Lambda^d_{i}(t-s)}.
\end{align}
This and \eqref{c12} show for all $d,K\in \N$, $s\in [0,T)$, 
$t\in (s,T]$, $x\in \R^d$, $i\in [0,d]\cap\Z$ that
\begin{align}
\left\lVert\pr_i^d(
{Z}^{d,s,x}_t
-{Z}^{d,s,y}_t)\right\rVert_{20}\leq \frac{\mathbf{V}_d(t,x)+\mathbf{V}_d(t,y)}{2}
\frac{\lVert x-y\rVert}{\sqrt{T}\Lambda^d_{i}(t-s)}.\label{c14}
\end{align}
Next,
\eqref{f39b}, Jensen's inequality, the fact that $q\geq 40$, \eqref{c10}, and \eqref{c10b} imply for all $d,K\in \N$, $s\in [0,T)$, 
$\tilde{s}\in [s,T)$, 
$t\in (\tilde{s}, T]$,
$x\in \R^d$, $i\in [0,d]\cap\Z$  that
\begin{align}\label{f39c}
\left\lVert
V^{d,\tilde{s},x,k}_t-V^{d,s,x,k}_t
\right\rVert_{q}
\leq \frac{\mathbf{V}_d(\tilde{s},x)+\mathbf{V}_d(s,x)}{2}\frac{\sqrt{\tilde{s}-s}}{\sqrt{t-\tilde{s}}\sqrt{t-s}}.
\end{align}
Thus, the fact that   $\forall\,d\in \N,s\in [0,T),x\in \R^d\colon \P({Z}^{d,s,x}=(1,{V}^{d,s,x}))=1$ and the definition of $\Lambda^d$, $d\in \N$, show for all $d,K\in \N$, $s\in [0,T)$, 
$\tilde{s}\in [s,T)$, 
$t\in (\tilde{s}, T]$,
$x\in \R^d$, $i\in [0,d]\cap\Z$ 
that
\begin{align}\label{c74b}
\left\lVert
\pr_i^d(Z^{d,s,x}_t-Z^{d,\tilde{s},x}_t)\right\rVert_{\exponentZ}\leq \frac{\mathbf{V}_d(\tilde{s},x)+\mathbf{V}_d(s,x)}{2}\frac{\sqrt{\tilde{s}-s}}{\sqrt{t-\tilde{s}}\Lambda^d_i(t-s)}.
\end{align}
Moreover, 
\eqref{f13b}, Jensen's inequality, the fact that $q\geq 40$, the fact that $\forall\,d,K\in \N,s\in[0,T], x\in\R^d\colon  \P(\mathcal{X}^{d,0,K,s,x}_s=s)=1$, \eqref{c10}, and \eqref{c10b} imply for all $d,K\in \N$, $s\in [0,T]$,
$t\in [s,T]$,
 $x\in \R^d$ that
\begin{align}
\left\lVert
\left\lVert\mathcal{X}^{d,0,K,s,x}_t-x\right\rVert
\right\rVert_{20}\leq \mathbf{V}_d(s,x)\sqrt{t-s}.\label{c16}
\end{align}
This and \eqref{c15} prove for all $d\in \N$, $s\in [0,T]$,
$t\in [s,T]$,
 $x\in \R^d$ that
\begin{align}\label{c17}
\left\lVert
\left\lVert{X}^{d,s,x}_t-x\right\rVert
\right\rVert_{20}\leq \mathbf{V}_d(s,x)\sqrt{t-s}.
\end{align}
Now, \cref{a37a} (applied for every $d,K\in \N$ with 
$\delta\gets \frac{\sqrt{T}}{\sqrt{K}}$, $\exponentV\gets 20$, 
$\exponentX\gets 20$, $\exponentZ\gets 20$, 
$c\gets \mathbf{c}$,
$V\gets \mathbf{V}_d$,
$(L_i)_{i\in [1,d]\cap\Z}\gets(L_i^d)_{i\in [1,d]\cap\Z}$,
$\Lambda\gets \Lambda^d$, 
$\pr\gets \pr^d$,
$f\gets f_d$, $g\gets g_d$,
$\mathcal{X}\gets \mathcal{X}^{d,0,K,(\cdot),(\cdot)}_{(\cdot)}$,
$\mathcal{Z}\gets \mathcal{Z}^{d,0,K,(\cdot),(\cdot)}_{(\cdot)}$,
${X}\gets \mathcal{X}^{d,(\cdot),(\cdot)}_{(\cdot)}$,
${Z}\gets \mathcal{Z}^{d,(\cdot),(\cdot)}_{(\cdot)}$ 
in the notation of \cref{a37a}), \eqref{a04h}--\eqref{a01h}, \eqref{c09}, \eqref{c09b}, \eqref{f26e}, \cref{c13b}, \eqref{c14}, the flow property of $X$ as solution to the SDE in \eqref{i01b}, \eqref{c15}, and \eqref{c12} imply that
\begin{itemize}
\item 
there exist unique measurable functions $u_d,\mathfrak{u}_{d,K}\colon [0,T)\times \R^d\to \R^{d+1}$, $d,K\in \N$, such that for all 
$d,K\in \N$,
$t\in [0,T)$, $x\in \R^d$
we have that
\begin{align}
\max_{\nu\in [0,d]\cap\Z}\sup_{\tau\in [0,T), \xi\in \R^d}
\left[\Lambda^d_\nu(T-\tau)\frac{\lvert\pr^d_\nu(u_d(\tau,\xi))\rvert}{\mathbf{V}_d(\tau ,\xi)}\right]<\infty,\label{c19}
\end{align}
\begin{align}
\max_{\nu\in [0,d]\cap\Z}\sup_{\tau\in [0,T), \xi\in \R^d}
\left[\Lambda^d_\nu(T-\tau)\frac{\lvert\pr^d_\nu(\mathfrak{u}_K(\tau,\xi))\rvert}{\mathbf{V}_d(\tau ,\xi)}\right]<\infty,
\end{align}
\begin{align}\max_{\nu\in [0,d]\cap\Z}\left[
\E\!\left [\left\lvert g_d(X^{d,t,x}_{T} )\pr^d_\nu(Z^{d,t,x}_{T})\right\rvert \right] + \int_{t}^{T}
\E \!\left[\left\lvert
f_d(r,X^{d,t,x}_{r},u_d(r,X^{d,t,x}_{r}))\pr^d_\nu(Z^{d,t,x}_{r})\right\rvert\right]dr\right]<\infty,\label{c20}
\end{align}
\begin{align}
&\max_{\nu\in [0,d]\cap\Z}\Biggl[
\E\!\left [\left\lvert g_d(\mathcal{X}^{d,0,K,t,x}_{T} )\pr^d_\nu(\mathcal{Z}^{d,0,K,t,x}_{T})\right\rvert \right]\nonumber\\&\qquad\qquad\qquad + \int_{t}^{T}
\E \!\left[\left\lvert
f_d(r,\mathcal{X}^{d,0,K,t,x}_{r},\mathfrak{u}_{d,K}(r,\mathcal{X}^{d,0,K,t,x}_{r}))\pr^d_\nu(\mathcal{Z}^{d,0,K,t,x}_{r})\right\rvert\right]dr\Biggr]<\infty,
\end{align}
\begin{align}
u_d(t,x)=\E\!\left [g_d(X^{d,t,x}_{T} )Z^{d,t,x}_{T} \right] + \int_{t}^{T}
\E \!\left[
f_d(r,X^{d,t,x}_{r},u_d(r,X^{d,t,x}_{r}))Z^{d,t,x}_{r}\right]dr,
\label{c18}
\end{align} and
\begin{align}
\mathfrak{u}_{d,K}(t,x)&=\E\!\left [g_d(\mathcal{X}^{d,0,K,t,x}_{T} )\mathcal{Z}^{d,0,K,t,x}_{T} \right]\nonumber\\&\quad + \int_{t}^{T}
\E \!\left[
f_d(r,\mathcal{X}^{d,0,K,t,x}_{r},\mathfrak{u}_{d,K}(r,\mathcal{X}^{d,0,K,t,x}_{r}))\mathcal{Z}^{d,0,K,t,x}_{r}\right]dr
\end{align}
and
\item 
for all $d\in \N$,
$\nu\in [0,d]\cap \Z$, $t\in [0,T)$, $y\in \R^d$ we have that
\begin{align}
\Lambda^d_\nu(T-t)\lvert
\pr^d_{\nu}( {u}_d(t,y) - \mathfrak{u}_{d,K}(t,y))\rvert
\leq  \frac{\frac{\sqrt{T}}{\sqrt{K}}e^{\mathbf{c}^2T}}{\sqrt{T}}\mathbf{V}_d^9(t,y).
\end{align}
\end{itemize}
This and the fact that $\forall\,d\in \N\colon e^{\mathbf{c}^2T}\leq \mathbf{V}_d$ (see \eqref{c10}--\eqref{c10b}) imply for all 
$d\in \N$,
$\nu\in [0,d]\cap \Z$, $t\in [0,T)$, $y\in \R^d$ that
\begin{align}
\Lambda^d_\nu(T-t)\lvert
\pr^d_{\nu}( {u}_d(t,y) - \mathfrak{u}_{d,K}(t,y))\rvert
\leq  \frac{\mathbf{V}_d^{10}(t,y)}{\sqrt{K}}.\label{c21}
\end{align}
Furthermore, \cref{b37} (applied for every $d\in \N$ with $\exponentV\gets 20$, $\exponentX\gets 20$, $\exponentZ\gets 20$, 
$c\gets \mathbf{c}$,
$(L_i)_{i\in [0,d]\cap\Z}\gets (L_i^d)_{i\in [0,d]\cap \Z}$,
$\Lambda\gets \Lambda^d$, $\pr\gets \pr^d$, 
$f\gets f_d$, $g\gets g_d$,
$V\gets \mathbf{V}_d$, 
$X\gets X^{d,(\cdot), (\cdot)}_{(\cdot)}$,
$Z\gets Z^{d,(\cdot), (\cdot)}_{(\cdot)}$
in the notation of \cref{b37}), \eqref{a04h}--\eqref{a01h}, \eqref{c09b}, \eqref{f26e}, \eqref{c13b}, \eqref{c14}, and \eqref{c17} prove for all $d\in \N$  that $u_d $ is continuous. This, \eqref{c19}, \eqref{c20}, and \eqref{c18}
show \eqref{i09}.

Next, \cref{a20} (applied for every $d,K\in \N$ with $\exponentV\gets 20$, $\exponentX\gets 20$, $\exponentZ\gets 20$, $c\gets \mathbf{c}$,
$(L_i)_{i\in [0,d]\cap\Z}\gets (L_i^d)_{i\in [0,d]\cap \Z}$,
$\Lambda\gets \Lambda^d$, $\pr\gets \pr^d$, 
$f\gets f_d$, $g\gets g_d$,
$V\gets \mathbf{V}_d$, 
$
\mathcal{X}\gets
\mathcal{X}^{d,(\cdot),K,(\cdot),(\cdot )}_{(\cdot)}$, $
\mathcal{Z}\gets
\mathcal{Z}^{d,(\cdot),K,(\cdot),(\cdot )}_{(\cdot)}$, 
$U\gets U^{d,(\cdot)}_{(\cdot),(\cdot),K}$, $\exponentFirstNorm\gets 3$ in the notation of \cref{a20}), \eqref{a04h}--\eqref{a01h}, \eqref{c09}, \eqref{c16}, \eqref{f26d}, \eqref{c10b},  the fact that
$\forall\,d,K\in \N,s\in [0,T), x\in \R^d\colon \P(\mathcal{Z}^{d,0,K,s,x}=(1,\mathcal{V}^{d,0,K,s,x}))$,  \eqref{c56}, and \eqref{a04j} prove 
for all
$d,K,n,m\in \N$, $t\in [0,T)$, $\nu\in [0,d]\cap\Z$, $x\in \R^d$  that
\begin{align}
\Lambda^d_\nu(T-t)\left\lVert\pr^d_{\nu}\!\left(U^{d,0}_{n,m,K}(t,x)-\mathfrak{u}_{d,K}(t,x)\right)\right\rVert_2\leq 
e^\frac{m^3}{6}
m^{-\frac{n}{2}}8^ne^{n\mathbf{c}^2T} \mathbf{V}_d^3(t,x).
\end{align}
Hence, the triangle inequality and \eqref{c21} show 
for all
$d,K,n,m\in \N$, $t\in [0,T)$, $\nu\in [0,d]\cap\Z$, $x\in \R^d$  that
\begin{align}
&\Lambda^d_\nu(T-t)\left\lVert\pr^d_{\nu}\!\left(U^{d,0}_{n,m,K}(t,x)-{u_d}(t,x)\right)\right\rVert_2\nonumber\\
&\leq \Lambda^d_\nu(T-t)\left\lVert\pr^d_{\nu}\!\left(U^{d,0}_{n,m,K}(t,x)-\mathfrak{u}_{d,K}(t,x)\right)\right\rVert_2
 +\Lambda^d_\nu(T-t)\left\lvert
\pr^d_{\nu}( {u}_d(t,x) - \mathfrak{u}_{d,K}(t,x))\right\rvert
\nonumber\\
&\leq e^\frac{m^3}{6}
m^{-\frac{n}{2}}8^ne^{n\mathbf{c}^2T} \mathbf{V}_d^3(t,x)+ \frac{\mathbf{V}^{10}_d(t,x)}{\sqrt{K}}\nonumber\\
&\leq \left[
e^\frac{m^3}{6}
m^{-\frac{n}{2}}8^ne^{n\mathbf{c}^2T}+\frac{1}{\sqrt{K}}
\right]\mathbf{V}^{10}_d(t,x).
\end{align}
This implies for all $d,n\in \N$ that
\begin{align}
&\sup_{\nu\in[0,d]\cap\Z, t\in [0,T),x\in \R^d}
\frac{\Lambda_\nu^d(T-t)\left\lVert\pr^d_{\nu}\!\left(U^{d,0}_{n,n^\frac{1}{3},n^\frac{n}{3}}(t,x)-{u}_d(t,x)\right)\right\rVert_2}{\mathbf{V}_d^{10}(t,x)}\nonumber\\
&=
\left[e^\frac{n}{6} n^{-\frac{n}{6}}8^ne^{n\mathbf{c}^2T}+n^{-\frac{n}{6}}\right]\leq 
e^\frac{n}{6} n^{-\frac{n}{6}}9^ne^{n\mathbf{c}^2T}.\label{c24}
\end{align}
Therefore, \eqref{c10} and \eqref{c10b} show for all $d\in \N$ that
\begin{align}\label{c23}
&\limsup_{n\to\infty}\sup_{\nu\in [0,d]\cap\Z,t\in [0,T), x\in [-\mathbf{k},\mathbf{k}]^d}\left[
\Lambda_\nu^d(T-t)\left\lVert\pr^d_{\nu}\!\left(U^{d,0}_{n,n^\frac{1}{3},n^\frac{n}{3}}(t,x)-{u}_d(t,x)\right)\right\rVert_2
\right]\nonumber\\
&\leq \limsup_{n\to \infty}
\sup_{\nu\in [0,d]\cap\Z,t\in [0,T), x\in [-\mathbf{k},\mathbf{k}]^d}\left[
e^\frac{n}{6} n^{-\frac{n}{6}}9^ne^{n\mathbf{c}^2T}
\mathbf{V}_d^{10}(t,x)\right]=0.
\end{align}
This proves \eqref{c33}.

For the next step let $(N_{d,\varepsilon})_{d\in \N,\varepsilon\in (0,1)}, (C_\delta)_{\delta\in (0,1)}\subseteq [0,\infty]$ satisfy for all $d\in \N$, $\varepsilon \in (0,1)$ that
\begin{align}
N_{d,\varepsilon}&=\inf\Biggl\{n\in \N\colon \sup_{k\in [0,\infty)\cap\Z,
\nu\in [0,d]\cap \Z, t\in [0,T), x\in [-\mathbf{k},\mathbf{k}]^d}\nonumber\\
&\qquad\qquad\qquad\qquad \Lambda_\nu^d(T-t)\left\lVert\pr^d_{\nu}\!\left(U^{d,0}_{k,k^\frac{1}{3},k^\frac{k}{3}}(t,x)-{u}_d(t,x)\right)\right\rVert_2\leq \varepsilon\Biggr\}\label{c30}
\end{align} and
\begin{align}\label{c31}
C_\delta=\sup_{n\in \N}
\frac{\left(e^\frac{n}{6} 9^ne^{n\mathbf{c}^2T}
\right)^{4+\delta}6^{n+1}n^\frac{2}{3}}{n^\frac{\delta n}{6}}.
\end{align}
Then \eqref{c23} shows for all $d\in \N$, $\varepsilon\in (0,1)$ that
$N_{d,\varepsilon}\in \N$. Furthermore, for all $\delta \in (0,1)$ we have that $C_\delta<\infty.$
Next, \cite[Lemma~3.14]{beck2020overcomingElliptic}, \eqref{c22}, and \eqref{c25}
imply for all $d,K,m,n\in \N$ that 
\begin{align}
\mathfrak{C}^d_{n,m,K}\leq \frac{K\mathfrak{e}+\mathfrak{g}_d+K\mathfrak{e}_d+\mathfrak{f}_d}{2}(3m)^n\leq K(\mathfrak{e}_d+\mathfrak{g}_d+\mathfrak{f}_d)(3m)^n\leq Kcd^c(3m)^n.\label{c26}
\end{align}This proves for all $d,K,m,n\in \N$ that 
\begin{align}
\mathfrak{C}^d_{n+1,(n+1)^\frac{1}{3},(n+1)^\frac{n+1}{3}}
&\leq (n+1)^\frac{n+1}{3}cd^c
(3(n+1)^\frac{1}{3})^{n+1}\nonumber\\
&
=cd^c3^{n+1} (n+1)^{\frac{2(n+1)}{3}}\nonumber\\
&\leq cd^c3^{n+1}(2n)^{\frac{2(n+1)}{3}}\nonumber
\\
&\leq cd^c 6^{n+1}n^{\frac{2(n+1)}{3}}\nonumber\\
&=cd^c6^{n+1}n^\frac{2}{3}n^\frac{2n}{3}.\label{c27}
\end{align}
Hence, \eqref{c24} shows for all $d,n\in \N$, $\delta\in (0,1)$ that
\begin{align}
&
\left[\sup_{\nu\in[0,d]\cap\Z, t\in [0,T),x\in \R^d}
\frac{\Lambda_\nu^d(T-t)\left\lVert\pr^d_{\nu}\!\left(U^{d,0}_{n,n^\frac{1}{3},n^\frac{n}{3}}(t,x)-{u}_d(t,x)\right)\right\rVert_2}{\mathbf{V}_d^{10}(t,x)}\right]^{4+\delta}\mathfrak{C}^d_{n+1,(n+1)^\frac{1}{3},(n+1)^\frac{n+1}{3}}\nonumber\\
&\leq \left(e^\frac{n}{6} n^{-\frac{n}{6}}9^ne^{n\mathbf{c}^2T}
\right)^{4+\delta}cd^c6^{n+1}n^\frac{2}{3}n^\frac{2n}{3}\nonumber\\
&\leq \frac{\left(e^\frac{n}{6} 9^ne^{n\mathbf{c}^2T}
\right)^{4+\delta}6^{n+1}n^\frac{2}{3}}{n^\frac{\delta n}{6}}cd^c\nonumber\\
&\leq C_\delta cd^c.\label{c28}
\end{align}
Next, \eqref{c10} and \eqref{c10b} imply that there exists $\eta\in(0,\infty)$ such that for all 
$d\in \N$ we have that  
$
\sup_{t\in [0,T),x\in [-\mathbf{k},\mathbf{k}]^d}
\mathbf{V}_d(t,x)\leq 
\eta d^\eta$.
This and \eqref{c28} imply that there exists $\eta\in(0,\infty)$ such that for all $d\in \N$, $\delta\in (0,1)$ we have that
\begin{align}
&\sup_{\nu\in[0,d]\cap\Z, t\in [0,T),x\in [-\mathbf{k},\mathbf{k}]^d}
\left[\Lambda_\nu^d(T-t)\left\lVert\pr^d_{\nu}\!\left(U^{d,0}_{n,n^\frac{1}{3},n^\frac{n}{3}}(t,x)-{u}_d(t,x)\right)\right\rVert_2\right]^{4+\delta}\mathfrak{C}^d_{n+1,(n+1)^\frac{1}{3},(n+1)^\frac{n+1}{3}}\nonumber\\
&\leq C_\delta cd^c
\left[\sup_{t\in [0,T),x\in [-\mathbf{k},\mathbf{k}]^d}
\mathbf{V}_d^{10}(t,x)\right]^{4+\delta}\leq 
C_\delta cd^c
\left[\sup_{t\in [0,T),x\in [-\mathbf{k},\mathbf{k}]^d}
\mathbf{V}_d^{10}(t,x)\right]^5\leq 
C_\delta \eta d^\eta.\label{c29}
\end{align}
In addition,
\eqref{c26} shows for all $d\in \N$, $\varepsilon\in (0,1)$ that if
$N_{d,\varepsilon}=1$ then
\begin{align}\label{c32}
\varepsilon^{4+\delta}\mathfrak{C}_{N_{d,\epsilon}, \lvert N_{d,\varepsilon}\rvert^\frac{1}{3}, \lvert N_{d,\varepsilon}\rvert^{\frac{N_{d,\varepsilon}}{3}}}\leq \mathfrak{C}_{1,1,1}\leq 3cd^c.
\end{align}
Moreover, \eqref{c30} and \eqref{c29} implies  that there exists $\eta\in(0,\infty)$ such that for all 
$d\in \N$, $\varepsilon\in (0,1)$ we have that  
if
$N_{d,\varepsilon}\geq 2$
then
\begin{align}
&
\varepsilon^{4+\delta}\mathfrak{C}^d_{N_{d,\epsilon}, \lvert N_{d,\varepsilon}\rvert^\frac{1}{3}, \lvert N_{d,\varepsilon}\rvert^{\frac{N_{d,\varepsilon}}{3}}}\nonumber\\
&\leq 
\sup_{\nu\in[0,d]\cap\Z, t\in [0,T),x\in [-\mathbf{k},\mathbf{k}]^d}
\left[\Lambda_\nu^d(T-t)\left\lVert\pr^d_{\nu}\!\left(U^{d,0}_{N_{d,\varepsilon}-1,(N_{d,\varepsilon}-1)^\frac{1}{3},(N_{d,\varepsilon}-1)^\frac{N_{d,\varepsilon}-1}{3}}(t,x)-{u}_d(t,x)\right)\right\rVert_2\right]^{4+\delta}\nonumber\\
&\quad \mathfrak{C}^d_{N_{d,\epsilon}, \lvert N_{d,\varepsilon}\rvert^\frac{1}{3}, \lvert N_{d,\varepsilon}\rvert^{\frac{N_{d,\varepsilon}}{3}}}\nonumber\\
&\leq C_\delta \eta d^\eta
\end{align}
This and \eqref{c32} prove
that there exists $\eta\in(0,\infty)$ such that for all 
$d\in \N$, $\varepsilon\in (0,1)$ we have that 
\begin{align}
\varepsilon^{4+\delta}\mathfrak{C}^d_{N_{d,\varepsilon}, \lvert N_{d,\varepsilon}\rvert^\frac{1}{3}, \lvert N_{d,\varepsilon}\rvert^{\frac{N_{d,\varepsilon}}{3}}}\leq C_\delta \eta d^\eta.
\end{align}
Thus, the fact that $\forall\,d\in \N,\varepsilon\in (0,1)\colon N_{d,\varepsilon }<\infty$, \eqref{c30}, and the fact that $\forall \,\delta\in (0,1)\colon  C_\delta<\infty$ establish \eqref{c35}. The proof  of \cref{c34} is thus completed.
\end{proof}

Finally, we provide the proof of \cref{c34b}.

\begin{proof}[Proof of \cref{c34b}]
First, the fundamental theorem of calculus and \eqref{f04f} prove for all $d\in \N$, $x,y\in \R^d$ that
\begin{align}
\lVert
\sigma_d(x)-\sigma_d(y)\rVert&= 
\left\lVert\int_{0}^{1}\frac{d}{ds}[\sigma_d(y+s(x-y))]\,ds\right\rVert=
\left\lVert\int_{0}^{1}\left((\operatorname{D}\!\sigma_d)(y+s(x-y))\right)\!(x-y)\,ds\right\rVert\nonumber\\
&\leq c\lVert x-y\rVert\label{d01}
\end{align}
and similarly
\begin{align}\label{d04}
\lVert
\mu_d(x)-\mu_d(y)\rVert\leq c\lVert x-y\rVert.
\end{align}
Next, \eqref{c37} shows for all $d\in \N$, $x,h\in \R^d$ that
\begin{align}
\lVert h\rVert^2
=h^\top \sigma_d^{-1}(x)\sigma_d(x)(\sigma_d(x))^\top ((\sigma_d(x))^{-1})^\top h
\geq \frac{1}{c}\left\lVert((\sigma_d(x))^{-1})^\top h\right\rVert^2.
\end{align}
This implies for all $d\in \N$, $x,h\in \R^d$ that
\begin{align}
\left\lVert((\sigma_d(x))^{-1})^\top h\right\rVert^2\leq c\lVert h\rVert^2.\label{d03}
\end{align}
Therefore, the fact that for all $d\in \N$, $\Sigma\in \R^{d\times d}$ we have that the operator norm of $\Sigma$ is equal to the operator norm of $\Sigma^\top$ proves for all 
$d\in \N$, $x,h\in \R^d$ that
\begin{align}\label{d02}
\left\lVert((\sigma_d(x))^{-1}) h\right\rVert^2\leq c\lVert h\rVert^2.
\end{align}
Next,
the Cauchy-Schwarz inequality implies for all $d\in \N$, $s=(s_{ij})_{i,j\in [1,d]\cap\Z}\in \R^{d\times d}$, $h=(h_j)_{j\in [1,d]\cap\Z}\in \R^d$ that
\begin{align}
\lVert sh\rVert^2=\sum_{i=1}^{d}\left\lvert\sum_{j=1}^{d}s_{ij}h_j\right\rvert^2\leq 
\sum_{i=1}^{d}\left[\left(\sum_{j=1}^{d}\lvert s_{ij}\rvert^2\right)
\left(\sum_{j=1}^{d}\lvert h_j\rvert^2\right)\right]= \lVert s\rVert^2\lVert h\rVert^2.
\end{align}
This, \eqref{d02},  \eqref{d01}, and the fact that $c\geq 1$ show for all $d\in \N$, $x,y,h\in \R^d$ that
\begin{align}
\left\lVert
(\sigma^{-1}_d(y)-\sigma^{-1}_d(x))h\right\rVert
&=
\left\lVert
\sigma_d^{-1}(y)(\sigma_d(x)-\sigma_d(y))\sigma_d^{-1}(x)h
\right\rVert\nonumber\\
&\leq c
\left\lVert(\sigma_d(x)-\sigma_d(y))\sigma_d^{-1}(x)h
\right\rVert\nonumber\\
&\leq c \left\lVert \sigma_d(x)-\sigma_d(y)\right\rVert
\left\lVert\sigma_d^{-1}(x)h\right\rVert\nonumber\\
&\leq c^3 \left\lVert x-y\right\rVert \lVert h\rVert.
\end{align}
Hence, \cref{c34}, \eqref{d02}, and the assumption of \cref{c34b} prove that the following items hold. 
\begin{itemize}
\item  For all $d\in \N$ there exists an up to indistinguishability continuous random field $(X^{d,s,x}_{t})_{s\in [0,T],t\in [s,T], x\in \R^d}\colon \{(\sigma,\tau)\in [0,T]\colon \sigma\leq \tau \} \times \R^d\times \Omega\to \R^d$ such that for all $s\in [0,T]$, $x\in \R^d$ we have that $(X^{d,s,x}_t)_{t\in [s,T]} $ is $(\F_t)_{t\in [s,T]}$-adapted and such that for all $s\in [0,T]$, $t\in [s,T]$, $x\in \R^d$ we have $\P$-a.s.\ that
\begin{align}
X^{d,s,x}_t = x+\int_{s}^{t}\mu_d (X^{d,s,x}_r)\,dr
+\int_{s}^{t}\sigma_d (X^{d,s,x}_r)\,dW_r^{d,0}.\label{j01b}
\end{align}
\item For all $d\in \N$, $s\in [0,T]$, $x\in \R^d$ there exists an 
$(\F_{t})_{t\in [s,T]}$-adapted
 stochastic process
$(D_t^{d,s,x})_{t\in[s,T]}=(D_t^{d,s,x,k})_{k\in [1,d]\cap\Z}\colon [s,T]\times \Omega\to \R^{d\times d}$ 
which satisfies for all $t\in [s,T]$, $k\in [1,d]\cap\Z$ that $\P$-a.s.\ we have that
\begin{align}
D_t^{d,s,x,k}=\int_{s}^{t}\left((\operatorname{D}\!\mu_d)(X^{d,s,x}_r)\right)\!
(D_r^{d,s,x,k})\,dr+
\int_{s}^{t}\left((\operatorname{D}\!\sigma_d)(X^{d,s,x}_r)
\right)\!(D_r^{d,s,x,k})\,dW_r^{d,0}.
\end{align}
\item \label{j08}For all
$d\in \N$, $s\in [0,T)$, $x\in \R^d$ there exist  
$(\F_{t})_{t\in (s,T]}$-adapted
 stochastic processes $({V}_t^{d,s,x})_{t\in (s,T]}= 
({V}_t^{d,s,x,k})_{t\in (s,T],k\in [1,d]\cap\Z}\colon (s,T]\times \Omega\to\R^d
$ and 
$({Z}_t^{d,s,x})_{t\in (s,T]} \colon (s,T]\times \Omega\to \R^{d+1}$
 which satisfy for all $t\in (s,T]$ that $\P$-a.s.\ we have that
\begin{align}
V^{d,s,x}_t=\frac{1}{t-s}\int_{s}^{t}(\sigma_d^{-1}(X^{d,s,x}_r)D^{d,s,x}_r)^\top\,dW_r^{d,0}
\end{align}
and ${Z}_t^{d,s,x}=(1, {V}_t^{d,s,x})$.
\item\label{j09} For all $d\in \N$ there exists a unique continuous function $u_d\colon [0,T)\times \R^d\to \R^{d+1}$ such that for all $t\in [0,T)$, $x\in \R^d$
we have that
\begin{align}\label{c01d} 
\max_{\nu\in [0,d]\cap\Z}\sup_{\tau\in [0,T), \xi\in \R^d}
\left[\Lambda^d_\nu(T-\tau)\frac{\lvert\pr^d_\nu(u_d(\tau,\xi))\rvert}{\left[(cd^c)^2+c^2\lVert x\rVert^2\right]}\right]<\infty,
\end{align}
\begin{align}\max_{\nu\in [0,d]\cap\Z}\left[
\E\!\left [\left\lvert g_d(X^{d,t,x}_{T} )\pr_\nu^d(Z^{d,t,x}_{T})\right\rvert \right] + \int_{t}^{T}
\E \!\left[
f_d(r,X^{d,t,x}_{r},u_d(r,X^{d,t,x}_{r}))\pr^d_\nu(Z^{d,t,x}_{r})\right]dr\right]<\infty,
\end{align}
and
\begin{align}
u_d(t,x)=\E\!\left [g_d(X^{d,t,x}_{T} )Z^{d,t,x}_{T} \right] + \int_{t}^{T}
\E \!\left[
f_d(r,X^{d,t,x}_{r},u_d(r,X^{d,t,x}_{r}))Z^{d,t,x}_{r}\right]dr.\label{b05d}
\end{align}
 
\item   For all $d\in \N$ we have that

\begin{align}\label{j33}
\limsup_{n\to\infty}\sup_{\nu\in [0,d]\cap\Z,t\in [0,T), x\in [-\mathbf{k},\mathbf{k}]^d}\left[
\Lambda_\nu^d(T-t)\left\lVert\pr^d_{\nu}\!\left(U^{d,0}_{n,n^\frac{1}{3},n^\frac{n}{3}}(t,x)-{u}_d(t,x)\right)\right\rVert_2
\right]=0.
\end{align} 
\item \label{j35} There exist
$(C_\delta)_{\delta\in (0,1)}\subseteq (0,\infty)$,
 $\eta\in(0,\infty)$,
$(N_{d,\varepsilon})_{d\in \N,\varepsilon\in (0,1)}\subseteq\N$
 such that for all 
$d\in \N$, $\delta,\varepsilon\in (0,1)$ we have that 
\begin{align}\label{c61}
\sup_{\nu\in [0,d]\cap\Z,t\in [0,T), x\in [-\mathbf{k},\mathbf{k}]^d}\left[
\Lambda_\nu^d(T-t)\left\lVert\pr^d_{\nu}\!\left(U^{d,0}_{N_{d,\varepsilon},\lvert N_{d,\varepsilon}\rvert^\frac{1}{3},\lvert N_{d,\varepsilon}\rvert^\frac{N_{d,\varepsilon}}{3}}(t,x)-{u}_d(t,x)\right)\right\rVert_2\right]<\varepsilon
\end{align}
and
\begin{align}\label{c62}
\mathfrak{C}^d_{N_{d,\epsilon}, \lvert N_{d,\varepsilon}\rvert^\frac{1}{3}, \lvert N_{d,\varepsilon}\rvert^{\frac{N_{d,\varepsilon}}{3}}}\leq C_\delta\varepsilon^{-(4+\delta)} \eta d^\eta.
\end{align}
\end{itemize}
Next, 
\eqref{a01i},
the Cauchy-Schwarz inequality, and the fact that $\forall\,d\in \N\colon \sum_{\nu=0}^{d}L_\nu^d\leq c$ imply for all $d\in \N$, $t\in [0,T)$, $w_1,w_2\in \R^{d+1}$ that
\begin{align}
&
\lvert
f_d(t,x,w_1)-f_d(t,y,w_2)\rvert\nonumber\\
&\leq \sum_{\nu=0}^{d}\left[
L_\nu^d\Lambda^d_\nu(T)
\lvert\pr^d_\nu(w_1-w_2) \rvert\right]
+\frac{1}{T}c\frac{\lVert x-y\rVert}{\sqrt{T}}
\nonumber\\
&\leq L_0^d\left\lvert \pr^d_0(w_1-w_2)\right\rvert
+\left(\sum_{\nu=1}^{d}\lvert L^d_\nu\rvert^2\right)^\frac{1}{2}
\left(\sum_{\nu=1}^{d}T\lvert \pr_\nu^d(w_1-w_2)\rvert^2\right)^\frac{1}{2}
+\frac{c}{T\sqrt{T}}\lVert x-y\rVert\nonumber\\
&\leq 
c\left\lvert \pr^d_0(w_1-w_2)\right\rvert
+c
\left(\sum_{\nu=1}^{d}T\lvert \pr_\nu^d(w_1-w_2)\rvert^2\right)^\frac{1}{2}
+\frac{c}{T\sqrt{T}}\lVert x-y\rVert.\label{d05}
\end{align}
In addition, \eqref{d01} and \eqref{d04}
imply for all $d\in \N$, $x\in \R^d$ that
\begin{align}\label{d07}
\lVert\mu_d(x)\rVert\leq \lVert\mu_d(0)\rVert+c\lVert x\rVert\leq cd^c+c\lVert x\rVert=c (d^c+\lVert x\rVert)\leq cd^c(1+\lVert x\rVert)
\end{align}
and similarly
\begin{align}\label{d08}
\lVert\sigma_d(x)\rVert\leq cd^c(1+\lVert x\rVert).
\end{align}
Furthermore, \eqref{f10f} shows for all $d\in \N$, $k\in [1,d]\cap\Z$, $x,y\in \R^d$ that
\begin{align}
&
\max \left\{
\left\lVert
(\operatorname{D}_k\! \mu_d)(x)-
(\operatorname{D}_k\! \mu_d)(y)\right\rVert
,
\left\lVert
\left(
(\operatorname{D}_k\!\sigma_{d})
(x)-
(\operatorname{D}_k\!\sigma_{d})(y)\right)\!(h)\right\rVert
\right\}\nonumber
\\
&=
\max \left\{
\left\lVert
\left(
(\operatorname{D}\!\mu_d)(x)-
(\operatorname{D}\!\mu_d)(y)\right)\!(e_k)\right\rVert
,
\left\lVert
\left(
(\operatorname{D}\!\sigma_d)(x)-
(\operatorname{D}\!\sigma_d)(y)\right)\!(e_k)\right\rVert
\right\}
\leq c\lVert x-y\rVert,
\end{align}
where $\operatorname{D}_k$, $k\in [1,d]\cap\Z$, denote the partial derivatives.
Thus, an existence and uniqueness result on 
viscosity solution (see, e.g.,
\cite[Proposition~5.1]{NW2023} and \cite[Proposition~5.2]{NW2023}), \eqref{d05}, \eqref{a19g}, \eqref{d01}, \eqref{d04}, \eqref{d07}, \eqref{d08}, \eqref{b04a},  \eqref{c37}, 
\eqref{j01b}--\eqref{b05d},
and the regularity assumptions of $\mu_d$, $\sigma_d$, $d\in \N$, show
for all $d\in \N$  that $v_d:=\pr_0^d(u_d)$ is the unique viscosity solution to the following 
semilinear PDE of parabolic type:
\begin{align}
&
\frac{\partial v_d}{\partial t}(t,x)
+\left\langle (\nabla_x v_d)(t,x),\mu_d(x)\right\rangle
+\frac{1}{2}\operatorname{trace}\!
\left(\sigma_d(x)[\sigma_d(x)]^\top \operatorname{Hess}_xv_d(t,x)\right)\nonumber
\\
&\quad \qquad\qquad+f_d(t,x,v_d(t,x), (\nabla_xv_d)(t,x))
=0
\quad 
\forall\, t\in (0,T), x\in \R^d,\label{c36b}
\\
&
v_d(T,x)=g_d(x) \quad \forall\, x\in \R^d\label{c38b}
\end{align}
and 
$\nabla_xv_d=(\pr^d_1(u_d), \pr^d_2(u_d),\ldots, \pr^d_d(u_d))$.
Combining this with 
\eqref{j33}--\eqref{c62} establishes 
 \eqref{k36}--\eqref{k36b}. The proof of \cref{c34b} is thus completed.
\end{proof}

{\bibliographystyle{acm}
\bibliography{Reference-Sizhou}
}


\end{document}